\documentclass[11pt,a4paper,leqno]{amsart}
\usepackage{amsmath,amsfonts,amssymb,amsthm, color}
\usepackage[T1]{fontenc}
\usepackage{verbatim}

\usepackage[active]{srcltx}
\usepackage{mathrsfs} 
\advance\textwidth by 2.5cm 
\advance\oddsidemargin by -1.0cm \advance\evensidemargin by -1.0cm

\newcommand\del[1]{}
\newcommand\think[1]{}
\newcommand\new[1]{}
\newcommand\zus[1]{}

\newcommand\comd[1]{} 
\newcommand\Redd[1]{} 

\def\bdm{\begin{displaymath}}
\def\edm{\end{displaymath}}
\def\bea{\begin{eqnarray}}
\def\eea{\end{eqnarray}}

\newtheorem{theorem}{Theorem}[section]

\newtheorem{lem}[theorem]{Lemma}
\newtheorem{defn}[theorem]{Definition}
\newtheorem{prop}[theorem]{Proposition}
\newtheorem{coro}[theorem]{Corollary}

\newtheorem{remark}{Remark}

\numberwithin{equation}{section}

\begin{document}

\date{\today}

\title[FSNSEs  \today]
{On the fractional stochastic Navier-Stokes equations on the Torus and on bounded domains.}

\author[L.{} Debbi]{Latifa Debbi.\\ {}\\  \SMALL This work is supported by  Alexander von Humboldt Foundation.}

\email{ldebbi@ymath.uni-bielefeld.de}
\address{Fakultat fur Mathematik, Universitat Bielefeld. Universitatsstratsse 25, 33615, Bielefeld. Germany.}

\maketitle

\begin{abstract}
In this work, we introduce and study the well-posedness of the multidimensional fractional stochastic Navier-Stokes equations  
on bounded domains and on the torus (Briefly dD-FSNSE).  We prove  the existence of a martingale solution for the general regime. We establish the uniqueness in the case a martingale solution enjoys  a condition of Serrin's type on the fractional Sobolev spaces. If an $L^2-$ local weak (strong in probability) solution exists and enjoys  conditions of Beale-Kato-Majda type, this solution is global and unique. These conditions are automatically satisfied for the 2D-FSNSE on the torus if the initial data has $ H^1-$regularity and the diffusion term satisfies growth and Lipschitz conditions corresponding to $ H^1-$spaces. The case of 2D-FSNSE on the torus is studied separately. In particular, we established thresholds for the global existence, uniqueness,  space and time regularities of the weak (strong in probability) solutions in the subcritical regime. 

\del{In this work, we introduce and study the wellposedness of the multidimensional fractional stochastic Navier-Stokes equation  
on bounded domains and on the torus (Briefly dD-FSNSE).  For the subcritical regime, we establish some thresholds for which  maximal local $L^q-$mild and $L^2-$weak-strong solutions exist and have a specified required space regularity. We prove that under conditions of Beale-Kato-Majda type, these solutions are global and unique. These conditions are automatically satisfied for the 2D-FSNSE on the torus if the initial data has $ H^1-$regularity and the diffusion term satisfies growth and Lipschitz conditions corresponding to $ H^1-$spaces.  The case of 2D-FSNSE on the torus is studied separately and improved thresholds, guaranties the global existence, the uniqueness, the space and time regularities of weak (strong in probability) solutions have been obtained. In particular, we prove for the subcritical and critical regimes the existence and the uniqueness of $L^2-$global weak 
solution with continuous trajectories. A space regularity result is also obtained for the subcritical regime. For the general regime, we prove  the existence of a martingale solution  and we prove the uniqueness under a condition of Serrin's type on the fractional Sobolev spaces.}

\del{In this work, we introduce and study the wellposedness of the multidimensional fractional stochastic Navier-Stokes equation  
on bounded domains and on the Torus (Briefly FSNSE).
In particular, we investigate two approaches to define this equation corresponding to the difficulties coming from the presence of the boundaries. We prove; the existence, uniqueness and the space regularity of a global weak (strong in probability) solution for the $2D-$FSNSE on the Torus in 
the subcritical and critical regimes, $ \alpha \in [1, 2]$ resp. $ \alpha \in (1, 2]$ resp. $ \alpha \in [\frac43, 2]$  and the local  existence and the
regularity of $ L^q-$mild solution for the $dD-$FSNSE,  on bounded domains and on the Torus, 
with $ d\geq 2$,
$ \alpha \in (1+\frac dq, 2]$ and $ q>d$. Under an extra condition on the noise, this solution is global and unique for the 2D-FSNSE on the Torus. If in addition, the mild solution of the $dD-$FSNSE, with $ d\geq 2$, enjoy a kind of Giga-Matsui-Sawada condition, then it is also global and unique.
We also prove the local existence of weak  (strong in probability) solution of the $dD-$FSNSE on bounded domains and on the Torus, with $ d\in \{2, 3\}$ and $ \alpha \in (1+\frac{d-1}3, 2]$. These solutions are global for the FSNSEs on the 2D-torus and on the 3D-torus under  a kind of Beale-Kato-Majda condition. The existence of a martingale solution for the general regime and the uniqueness 
of the martingale and the weak solutions under a similar Serrin's conditions on the Sobolev spaces, are also obtained for the $dD-$FSNSE on bounded domains and on the Torus.}

\vspace{0.25cm}
Keywords: {Fractional stochastic Navier-Stokes equation, classical Navier-Stokes equation,
fractional stochastic vorticity Navier-Stokes equation, 
Q-Wiener process, trace class operators, subcritical, critical, supercritical, dissipative and hyperdissipative regimes, 
martingale, \del{mild, }global and local  weak-strong solutions, 
Riesz transform, Serrin's condition, Beale-Kato-Majda condition, fractional Sobolev spaces, 
\del{$\gamma-$radonifying operators,
UMD Banach spaces of type 2,}pseudo-differential operators, Skorokhod embedding theorem, Faedo-Galerkin approximation, compactness method, representation theorem.}

\vspace{0.15cm}

Subjclass[2000]: {58J65, 60H15, 35R11.}
\end{abstract}

\section{Introduction}\label{sec-intro}

The Navier-Stokes equation (briefly NSE) has been derived,  more than one century ago, by the engineer C.L. Navier 
to describe the motion of an incompressible Newtonian fluid. Later, it has been reformulated by the mathematician-physicist 
G. H. Stokes.  Since that time, this equation continues to attract  a great deal of attention 
due to its mathematical  and physical importance. This equation appears,
alone or coupled with other active and passive scalar equations, in the study of many phenomena, see e.g. the list of references in this 
work and in \cite{Debbi-scalar-active}. The 3D-stochastic Navier-Stokes equation (briefly 3D-SNSE) is the most realistic model in fluid dynamics and  for many other physical purposes, see e.g. \cite{Temam-NS-Functional-95}. The 2D-SNSE is
used as an approximation of the 3D model when the velocity of the fluid belongs to a plane\del{ (surface velocity)}, as  is the case for  basins and oceans.
The 2D-SNSE on a bounded domain $ O\subset \mathbb{R}^2$
governs the flow of a fluid which fills in infinite cylinder of cross-section  $ O$ and moves parallel to the plane of $ O$.
Physically, NSE on the torus is not a realistic model, but it is used for some idealizations and for homogenization problems in turbulence see e.g. \cite{Foias-book-2001}.\del{ps29-30} Mathematically, basic questions like the existence and the uniqueness of a global
smooth solution of the dD-NSE is still an open problem for $ d\geq 3$. In particular, for $ d=3$, the statement above is one formulation of the so called the millennium problem of the Navier-Stokes equation. 
 The main difficulty in the study of the dD-NSE is related to the nonlinear term.
In particular, as this latter comes from kinematical considerations, i.e. deduced from a mathematical calculus,  it is no possible to
change or to replace it.

\vspace{0.25cm}

In this work, we deal with the d-dimensional fractional stochastic Navier-Stokes equation (dD-FSNSE) on bounded domains and on the torus.  One of the benefits  of the study of the fractional Navier-stokes equation  is  to contribute in the understanding of  millennium problem. In fact, this last is regarded as a dimensional problem due to the fact that, contrarily to the 3D-NSE, the 2D-NSE is well-posed and well understood.  However, it is also known that the 3D-hyperdissipative  Navier-Stokes equation admits a  global
 classical solution, provided that the order of dissipation is greater 
than or equal to $ \frac52$, see e.g. \cite{Lions-methodes,  WUJ-06}. 
In \cite{KatzPavlo-Caffa-02}, the authors established a cheap
Caffarelli-Kohn-Nirenberg inequality for the 3D-hyperdissipative Navier-Stokes equation. 
This latter is also used to regularize the classical NSE, see e.g. 
\cite{Lions-methodes, Mattingly-sinai02}.
 Therefore, in addition to the dimension, the problem of the dD-NSE could also be regarded as 
a dissipative problem as well. Moreover, the author claims that the 2D-FSNSE behaves, i.e. exhibits difficulties, like the classical 3D-NSE, see the proof in Section \ref{sec-Domain}. The global existence of the solution for the 2D-FSNSE is obtained by using the vorticity regularization effect. This proves one of the classical conveniences stating that the main differences between the 2D \& the 3D NSE appears in the vorticity, see e.g. \cite{Chemin-Book-98}. To support more the authors's claim above, we draw attention to the undimensional similarity between the $3D-$vorticity NSE,  see e.g. \cite{Chemin-Book-98, Giga-al-Globalexistence-2001} and the no-free divergence mode scalar active equation studied in \cite{Debbi-scalar-active}. For this latter, we are not able to prove the global existence.
 
\vspace{0.25cm}

 Recently, the author studied a class of  fractional stochastic active scalar equations generalizing, among many other equations, the $2D-$fractional stochastic vorticity Navier-Stokes equation and the dD-stochastic quasi-geostrophic equation \cite{Debbi-scalar-active}. In particular, thresholds to ensure the existence, uniqueness and the regularities of several kinds of solutions have been established. The author characterized, among others, the following two intrinsic thresholds $ \alpha_0(d, q):= 1+\frac dq,\; q>d$ and $ \alpha_0(d):= 1+\frac{d-1}3$, for $ d\in \{2, 3\}$, 
which guarantee the existence of the $ L^q-$ respectively the $ L^{\frac{3d}{d-1}}$-mild solutions (weak-strong as well). Other critical dissipation values are also obtained according to the different Sobolev regularities required for the solutions.\del{ It is also importance here to mention the differential dimension of the domain of the fractional operator,\footnote{See the definition of the differential dimension of a Sobolev space in \cite{R&S-96}.} $ c_{dd}:= \alpha-\frac dq$, which is in this work strictly bigger than 1. In the classical NSE, the differential dimension coresponding to the 2D-NSE ($\alpha=2,\; d=q=2$)  and to the 3D-NSE ($\alpha=2,\; d=3,\; q=2$) are $ c_{dd}= 1$ and 
$ c_{dd}= \frac12$ respectively.}

\vspace{0.25cm}

Motivated by the results in \cite{Debbi-scalar-active}, we try to make precise some balance relationships between the dissipation order, the dimension and the regularity of the solutions and establish  dissipative thresholds for the well-posedness of the $dD-$FSNSE. In this work, we consider the Hilbert setting. To the best knowledge of the author, the present work is the first in the target to 
study the well-posedness of the $dD-$Navier-Stokes equation, fractional and classical, from this triple-view, i.e. simultaneously taking into account the dissipation, the regularity and the dimension, quantifying the balance between them and establishing optimal thresholds.

\del{in this work, we make precise a balance relationship between the dissipation order, the dimension and the
regularity of the solutions and establish  dissipative thresholds for the well-posedness of the $dD-$FSNSE. In particular, 
we characterize the following two intrinsic thresholds $ \alpha_0(d, q):= 1+\frac dq,\; q>d$ and 
$ \alpha_0(d):= 1+\frac{d-1}3$, for $ d\in \{2, 3\}$, 
which guarantee the local existence of the $ L^q-$mild respectively $L^2-$ duality-martingale solutions for the dD-FSNSEs. Other critical dissipation values are also obtained according to the different Sobolev regularities required for the solutions.\del{ It is also importance here to put a great emphasis on the optimal role of the space $ L^{\frac{3d}{d-1}}$\del{for the existence of the weak and the mild solutions and to put emphasis on} and on the critical differential dimension. The critical differential dimension 
we are able to deal with in this work is $c_{cdd} =1$, i.e. the differential dimension of the domain of the 
fractional operator,\footnote{See the definition of the differential dimension of a Sobolev space in \cite{R&S-96}.} $ c_{dd}:= \alpha-\frac dq$, is strictly bigger than 1. In the classical NSE, the differential dimension coresponding to the 2D-NSE ($\alpha=2,\; d=q=2$)  and to the 3D-NSE ($\alpha=2,\; d=3,\; q=2$) are $ c_{dd}= 1 (= c_{cdd})$ and 
$ c_{dd}= \frac12 (<c_{cdd} =1)$ respectively.} It is also importance here to put a great emphasis on the optimal role of the space $ L^{\frac{3d}{d-1}}$\del{for the existence of the weak and the mild solutions and to put emphasis on} and on the differential dimension of the domain of the fractional operator,\footnote{See the definition of the differential dimension of a Sobolev space in \cite{R&S-96}.} $ c_{dd}:= \alpha-\frac dq$, which is in this work strictly bigger than 1. In the classical NSE, the differential dimension coresponding to the 2D-NSE ($\alpha=2,\; d=q=2$)  and to the 3D-NSE ($\alpha=2,\; d=3,\; q=2$) are $ c_{dd}= 1$ and 
$ c_{dd}= \frac12$ respectively.

\del{ One of the possible ways to improve actual results for the  Navier-Stokes equation, classical and fractional, would be to  elaborate technical tools valid for a critical differential dimension  $c_{cdd}\leq1$.} One of the possible ways to improve actual results for the  Navier-Stokes equation, classical and fractional, would be to develop the techniques and the tools used for the FNSE, as those presented in this work, to be valid for a differential dimension  $c_{dd}\leq1$. To the best knowledge of the author, the present work is the first to 
study the well-posedness of the $dD-$Navier-Stokes equation, fractional and classical, from this triple-view, i.e. simultaneously taking into account the dissipation, the regularity and the dimension, quantifying the balance between them\del{ dissipation, the regularity and the dimension} and establishing optimal thresholds.}

\del{Consequently, it is easy to deduce the optimal role of the $ L^{\frac{3d}{d-1}}-$space for the existence of the weak and the mild solutions simultaneously.
 In particular,  the  $ L^6-$space plays an such role for the 2D-FNSEs. Furthermore,  Let us also enphase that  
 To the best knowledge of the author, the present work is the first to 
study the well posdeness of the $dD-$Navier-Stokes equation, fractional and classical,  from this triplet-view, 
\del{from this triplet-perspectives, from this triplet-framework,}  quantify the balance between the dissipation, regularity and dimension and establish, following this techniques used, optimal thresholds,
after the author's first work about the scalar active equations \cite{Debbi-scalar-active}.}
\vspace{0.25cm}

To further clarify what is new in the present work, it is of great importance to point out some features and some delicate problems 
related to the FSNSE. Some of these problems are inherited  from the classical NSE. Other problems for general fractional stochastic partial differential equations have been discussed in
\cite{Debbi-scalar-active} see also \cite{DebbiDozzi1}.

\vspace{0.25cm}
The energy method applied for the dissipative PDEs is based on the ability to control the  kinetic energy $e(u)$ 
and the
enstrophy energy $ E(u)$\footnote{strophy comes from Greek and means rotation.} of the solution $u$. 
Recall that
\begin{equation}\label{eq-energy-enstrophy}
 e(u):= \frac12\int_O|u|^2dx, \;\;\;\; E(u):= \int_O|\nabla u|^2dx.
\end{equation}
The control of these quantities for the classical NSE emerges from the structure of
the equation itself. However, for the FNSE,
a priori, there is no guarantee about the control of the enstrophy energy. The structure of this equation\del{FNSE} guarantees 
only the control of
a weaker Sobolev norm. This fact is again  due to the
weakness of the fractional dissipation. In some special cases like the $ 2D-$FNSE on the torus, see Section 
\ref{sec-Torus}, 
the control of the
enstrophy energy emerges from the structure of the fractional equation.  
This improvement  and also the improvement  of the results in this case are consequences of the 
$ H^{1, 2}$-orthogonality. These facts generalize the classical features  
known for the classical NSE on the 2D-torus, see e.g. \cite{Temam-NS-Functional-95}\del{\cite[ps 19-20]{Temam-NS-Functional-95}}.

\vspace{0.25cm}

\del{As a consequence of the structure of the FSNSE, the estimation of the nonlinear term is one of the delicate features of the} 

A delicate technical feature of FNSE is  the  estimation of the nonlinear term. In fact, as the structure of 
the equation cannot initially
guarantee  the boundedness of the enstrophy energy,  mathematically, we are not allowed to estimate terms  
by $ H^1-$norm. Moreover, contrarily to the classical NSE, where the $ H^1-$space plays a common role for the 
linear and the nonlinear terms, the components of the  Gelfand triple corresponding to 
the FNSE are not automatically coherent with respect to the two terms. More precisely, the nonlinear term is not bounded on the domain of definition of the fractional Stokes operator. Indeed, this latter is larger than $H^1$,  see details in Section 
\ref{sec-nonlinear-prop} and in Remark \ref{Rem-1}.\del{ the domain of definition of the fractional operator is so large that the nonlinear term could not  be bounded on, see details in Section 
\ref{sec-nonlinear-prop}.} Therefore, an extension of the nonlinear term is needed. In order to construct a coherent Gelfand triple, to extend and to estimate the nonlinear term, we have established more 
refined estimates via fractional
 Sobolev spaces of order less than one. These estimates are completely new.

\vspace{0.25cm}

\del{We end this enumeration of novelties and features by mentioning the following two  questions, which are simple to resolve but important to deal with.} We add to this enumeration of novelties and features the following two  questions, which are simple to resolve but important to deal with. To introduce the FSNSE defined on $ O= \mathbb{R}^d$ or $ O= \mathbb{T}^d$, 
we take the fractional power of the Stokes operator, which is here equal to minus the Laplacian. Moreover, these equations take more advantage of the facts that the fractional power of the Stokes operator is  defined as a pseudodifferential operator and 
commutes with the Helmholtz projection and with the partial differential operators $ \partial_j, j=1,\cdots,d$. Contrarily to these two cases, the situation for the FSNSE on a bounded domain, $ O\subset \mathbb{R}^d$
is much more involved. In fact, it is well known that in this case, the Stokes operator is different than the Laplacian.
To introduce the FNSE on  a bounded domain we can use two approaches by 
taking either the fractional power of the Stokes operator or by taking the fractional power of minus 
the Dirichlet-boundary Laplacian operator and than apply Helmholtz projection. 
Initially, due to the effect of the boundaries and to the application of  Helmholtz projection, we cannot conclude, a priori, whether or not the two approaches yield the same equation.\del{equivalent equations\footnote{For simplicity, one can referee to  the two equations by fractional version of the stochastic NSE (FSNSE) 
respectively the stochastic version of the fractional NSE (SFNSE), however, we could not do such way here as later we shall prove that the two equations are equivalent}} In particular, it is intuitively seen that the fractional equation obtained by the first  approach is more theoretical and the equation obtained by the second approach is more suitable for physical modeling, see more discussion in sections \ref{sec-formulation} \& \ref{sec-Torus} and Appendix \ref{Appendix-Equivalence}. In this work, we introduce both equations and prove that they are well defined and equivalent. The author does not know any works considering  deterministic or stochastic FNSE on bounded domains. 

\vspace{0.25cm}

To prove the global existence of the weak solutions for the 2D-FSNSE on the torus, 
we use the regularization effect of the vorticity and the results from \cite{Debbi-scalar-active}.
For the classical NSE, the evolution equation describing the vorticity is obtained by the application of the 
curl operator on the pathwise velocity equation, see e.g. \cite{Chemin-Book-98, Majda-Bertozzi-02, Marchioro-Puvirenti-Vortex-84}. As the fractional operator is nonlocal, it is of 
great importance to\del{ check whether or not this  equation is still valid} derive the vorticity equation corresponding to the FSNSE.\del{ For the 2D-FSNSE on the torus, we obtain, without difficulties, the fractional stochastic vorticity NSE by application of the curl operator to the abstract integral FSNSE}  We obtain, without difficulties, the 2D-fractional stochastic vorticity NSE by application of the curl operator to the abstract integral 2D-FSNSE. In particular, we investigate, in a rigorous way,
 the curl  of the stochastic term and the composition of the  curl 
and the fractional Stokes operators,\del{. The main ingredients are 
 Fourier transform and the commutativity property of the  fractional 
Stokes operator and the partial differential operators $ \partial_j, j=1,\cdots,d$,} see Appendix \ref{sec-Passage Velocity-Vorticity}. The study  of the FSNSE on a bounded domain
is more difficult in both classical and fractional cases. In fact, 
it is well known that, when boundaries are present for the classical 
NSE (either deterministic or
stochastic), there is no simple boundary condition to impose on the vorticity in such a way that the 
velocity satisfies the right boundary conditions, see e.g. \cite{ 
Giga-vorticity-sing-NS-2011, Giga-al-Globalexistence-2001, Majda-Bertozzi-02, Marchioro-Puvirenti-Vortex-84}.
In the fractional case, a new difficulty emerges due to the  fact that  the boundaries are also included in the definition of the fractional operator.\del{it is even difficult to make rigorously the composition of the curl and the fractional Stokes operator.} 
Therefore, due to these multiple difficulties and to the fact that we need results already proved in the work in progress 
\cite{Debbi-scalar-activeBD}, we  postpone the study of this case. 

\vspace{0.25cm}

\del{We end this enumeration of novelties by drawing attention to the scheme introduced in Section \ref{sec-1-approx-local-solution} to get the local mild solution and to the approximation method used in Section \ref{sec-Domain} to get the maximal  weak solution, see related scheme in  \cite{Kunze-Neereven-Cont-parm-reaction-diffusion-2012} and related approximation in \cite{Mikulevicius-H1-NS-solution-2004}. To the best knowledge of the author, these two tools are new. 

\vspace{0.25cm}}

Recently, the deterministic fractional Navier-Stokes equation has been studied in some works using analytical and probabilistic tools, see e.g. 
\cite{Cannone-Fract-NS-08, Dong-Li- Optimal-GNSE-2009, KatzPavlo-Caffa-02, Tao-global-Regu-Logarithm-2009,
 WUJ-Global-regu-2011, WUJ-06, Zhang-FNSE-stocastic-tool-2012}.\del{ In all these works, the authors  consider the
deterministic FNSE 
on the whole space  $ \mathbb{R}^d$.} The existence and the uniqueness of a local solution for the 
FNSE in Besov space in the subcritical regime and under conditions on the regularity of the 
initial data, have been proved in \cite[Theorems 6.2 \& 6.3]{WUJ-06}. If moreover, the Besov norm of the initial data is dominated by the viscosity, \del{ the author
proved in \cite[Theorems 6.1]{WUJ-06} that }the solution is global \cite[Theorems 6.1]{WUJ-06}. In \cite{Cannone-Fract-NS-08}, 
the authors studied the 2D-FNSE and proved the existence and the uniqueness of a global solution in some Besov spaces.  They also proved that the family of viscosity fractional diffusion solutions converges in $ L^q-$space (with $ q$  depends on $\alpha$)  to the unique
solution of Euler equation. In particular, for the subcritical regime
the convergence is obtained in Besov space. The convergence rates in  
both regimes have been established as well.  In \cite{Dong-Li- Optimal-GNSE-2009}, the authors used the smoothing property of the fractional Oseen kernel, to establish the space analyticity and the decay estimates of the local mild solution of the 
FNSE in the subcritical regime.  The results are proved in time weighted space. The stochastic Lagrangian particle approach has been used in 
\cite{Zhang-FNSE-stocastic-tool-2012}
to prove the local existence and the uniqueness of the solution of the subcrirical  
NSE driven by 
the infinitesimal generator of a L\'evy semigroup. The author assume that the real part of the L\'evy-Khintchine 
formula behaves as a fractional power symbol and that the initial condition has $H^{1, q}-$regularity. The solution conserves the 
$H^{1, q}-$regularity, satisfies the nonlocal NSE in  distribution sense and when the dimension $ d=2$, 
the solution is global \cite[Theorem 3.6 \& 2.4]{Zhang-FNSE-stocastic-tool-2012}. 
In the periodic case and under the large viscosity condition, the author proved that the solution 
is global \cite[Theorem 5.1]{Zhang-FNSE-stocastic-tool-2012}. In addition to the references about the hyperdissipative regime,
\cite{KatzPavlo-Caffa-02, Lions-methodes, Mattingly-sinai02} cited above, we mention here also the references
\cite{Tao-global-Regu-Logarithm-2009, WUJ-Global-regu-2011}, where the authors treated the regularity properties of the solution of the hyperdissipative regime FNSE respectively of 
Magnetohydrodynamic equations with dissipation order $ \alpha\geq 1+d/2$.

\vspace{0.25cm}

As mentioned above, the aim of this work is to study the multi-dimesional fractional stochastic Navier-Stokes equation (dD-FSNSE) on  bounded domains in $ \mathbb{R}^d$ and on the torus $ \mathbb{T}^d$,  with $ d\geq 2$. We investigate the existence, the uniqueness and the regularity of weak (strong in probability) solution for the critical and subcritical 2D-FSNSE on the torus, martingale solution for general regime dD-FSNSE. In particular, we established, in the fractional framework, conditions of Serrin's and of Beale-Kato-Majda type ensuring the global existence and the uniqueness of weak-strong solutions. The threshold $\alpha_0(d): = 1+\frac{d-1}{3}$ and the Sobolev order $\frac{d+2-\alpha}{4}$  also emerge. We do not assume  any restrictions neither on the viscosity nor on the initial condition (smallness or regularity). The local solutions can start from an $ L^2-$ \del{(and even less than the ${L}^q-$regularity, see Remark \ref{Rem-initial-data-mild-solu}) }initial data.\del{The  $H^{1, q}-$regularity of the initial condition is required to prove the global existence. In \cite{Debbi-scalar-active}, the author proved, in a general framework the existence and the uniqueness of the global mild and weak solutions for the 2D-vorticity FSNSE on the torus, for $ \alpha\in (1+\frac dq, 2]$. The solution starts from an $ \mathbb{L}^q-$regularity (and even less than the $ \mathbb{L}^q-$regularity, see Remark \ref{Rem-initial-data-mild-solu}) of the initial data. 
Here, we extend the results to cover all the subcritical and critical regimes, for the 2D-FSNSE on the torus when the smoothness of the initial data is of $ H^1$. In addition, we prove also the  existence of a maximal local weak solution for $ \alpha \in (1+\frac{d-1}{3}, 2]$.} The results obtained in this work  cover not only
our scopes of interest, which are the subcritical, critical and supercritical regimes and the stochastic case, but they are also
valid for the deterministic case and for the dissipative\del{ (Laplacian dissipation)} and the hyperdissipative regimes. In some places, we need the condition $ \alpha < 2$, but in these cases the same result can be proved for $ \alpha\geq2$ by using classical and simpler methods.\del{ the calculus there can be done differently by more simple ways.}

\del{In addition to the subcritical, critical and supercritical regimes and to stochastic case, we are interested in 
the results of this work are still
valid for the deterministic case and for the Laplacian and the hyperdissipation regimes.}

\vspace{0.25cm}

The paper is organized as follow, in Section \ref{sec-formulation}, we introduce rigorously the FSNSE. We prove in Appendix \ref{Appendix-Equivalence} that the two approaches described above yield to the same equation. The main definitions and results are presented in Section 
\ref{sec-Results}.  Section \ref{sec-nonlinear-prop} is devoted to the study of the nonlinear term. The proofs of the results are distributed in the remaining sections 
\ref{sec-Torus}-\ref{sec-global-weak-solution} and Appendices \ref{Appendix-Equivalence}-\ref{Appendix-Sobolev} according to the subtitles.

  \vspace{0.25cm}

\noindent {\bf Preliminary Notations \& General Remarks}
Let  $ \mathbb{N}_k:=\del{\mathbb{N}-}\{j\in \mathbb{N},\; s.t.\; j>k\}$\del{, $ \mathbb{Z}^d:=\{(k_1, k_2, \cdots, k_d),\; k_j\in \}$} and 
$ \mathbb{Z}^d_0:=\mathbb{Z}^d-\{0\}$. 
For $ d \in \mathbb{N}_0$, we denote by $ \mathbb{T}^d$  the $d-$dimensional torus and by $ D(\mathbb{T}^d)$ the set of infinitely differentiable scalar-valued (complex) functions on $ \mathbb{T}^d$.\del{$ C^k(\mathbb{T}^d), k\in \mathbb{N}_0, $ is the set of  $k-$differentiable functions on $ \mathbb{T}^d$,} By a domain " $ O$" we mean an open non empty set.
For either $ O=\mathbb{T}^d $ or $ O\subset \mathbb{R}^d$ bounded, we define $ H_l^{\beta, q}(O):= (H^{\beta, q}(O))^l, l\in \mathbb{N}_0, 
\beta \in \mathbb{R}, 1<q<\infty$, in particular for $ \beta =0$,  $L^q_l(O):= (L^q(O))^l$. Recall that $ H^{\beta, q}(O)$, according to $ O$, are either the Sobolev spaces on a bounded domain or the null average periodic Sobolev spaces on the torus. $ C_0^\infty(O) $ is the set of infinitely differentiable real functions\del{ with values in $\mathbb{R}$}\del{ or in $\mathbb{C}$),} with compact support on the bounded domain $ O\subset \mathbb{R}^d(O)$, $ \mathring{H}_d^{\beta, q}(O), \beta \in \mathbb{R}_+, 1<q<\infty$ is the completion of $ C_0^\infty(O) $ in $ H_d^{\beta, q}(O)$, with $ O\subset \mathbb{R}^d$ bounded. $ \partial_{x_j}$ stands for the partial derivative
with respect to the component $ x_j$, sometimes we also use the notation  $ \partial_j$. We use the notation $ |\cdot|_{X}$ to indicate the norm in $ X$. For simplicity, we denote the norm of a matrix by the corresponding scalar  space notation of the components or by a symbol of this space. The Sobolev norms used are those defined by Riesz-potential. \del{We denote by $\mathcal{L}(X)$ the set of  bounded linear operators on a Banach space $ X$.  
The notation $|.|_{L^{r_1}\rightarrow  L^{r_2}}$
stands for the usual norm of bounded operators from  $L^{r_1}$ to $  L^{r_2}$.} 
The classification of the subcritical,  critical and
superctitical regimes corresponds to  $ \alpha \in (1, 2)$, $ \alpha =1$ and $ \alpha \in (0, 1)$ respectively.
The dissipative ( sometimes called also the Laplacian dissipation) and the hyperdissipative regimes correspond to $ \alpha =2$ respectively to 
$ \alpha >2$. \del{The abbreviations (FSNSE), (SFNSE)and (FNSE) are used respectively for 
fractional 
stochastic Navier-Stokes equation, the
stochastic fractional Navier-Stokes equation, the deterministic fractional 
stochastic Navier-Stokes equation.}The abbreviations (FSNSE), (SNSE) and (FNSE) are used respectively for 
fractional 
stochastic Navier-Stokes equation, the
stochastic Navier-Stokes equation and the deterministic fractional 
stochastic Navier-Stokes equation. The abbreviation i.i.d  means independent and identically distributed.
 $ \{a_1, a_2\}\leq_k b$ (respectively $ \{a_1, a_2\}\geq_k b$) means $ a_k\leq b$, $ a_j< b, \; j\neq k$ 
and $ a_1=a_2 < b$ (respectively $
a_k\geq b$, $ a_j> b, \; j\neq k$ and $ a_1=a_2 > b$).  The expression $ q\leq_\infty q_0$ means 
$ q\leq q_0<\infty$ and $ q< q_0=\infty$.\del{ $ \max^1_>\{a, b\} $ equals $ a$, if $ a>b$ and strictly greater 
than $ b$ if  $ a\leq b$. A mathematical notion with statement is noted by the usual notation with subscription
 "s" e.g. $ [_s=[$ if the statement "s" is satisfied and $ [_s=($ if not.}
We say that  $ q^*$ is the  conjugate of $q $, if for $ 1<q<\infty$,
 $ q^*$ satisfies the equation $ \frac 1q+\frac{1}{q^*} = 1$  and for $ q=1$ respectively 
$ q=\infty$, $ q^*=\infty$ respectively $ q^*=1$.
We define, in distribution sense, the curl of a vector \del{distribution} field 
 $ v=(v_1, v_2)$ by $ curl v := \partial_1v_2 -\partial_2v_1 $. The vorticity matrix of a $dD-$vector field $ v $ on $ \mathbb{R}^d$\del{ to $ \mathbb{R}^d$} is the null diagonal, antisymmetric matrix defined by $ \Omega(v):= ((\Omega(v))_{i, j})_{1\leq i, j\leq d}$, where $ (\Omega(v))_{i, j}:= \partial_i v_j-\partial_j v_i$. For $ d=2$, the vorticity $ \Omega(v)$ is identified to the scalar function $curl v$ and for $ d=3$ to the transpose of the $3D-$vector function $ (\partial_2 v_3-\partial_3 v_2, \partial_1 v_3-\partial_3 v_1, \partial_1 v_2-\partial_2 v_1)$. In Appendix \ref{Sobolev pointwise multiplication-Bounded-Domain}, we have proved that if a Sobolev\del{ embedding and} pointwise multiplication estimate is satisfied for Sobolev spaces on $ \mathbb{R}^d$ and if $ O\subset \mathbb{R}^d$ is a "good" bounded domain, then this pointwise multiplication estimate is also valid for Sobolev spaces on bounded domains. Therefore, in many cases, we referee directely to the source of the estimate  on $ \mathbb{R}^d$.\del{ Moreover, we
also pass from vectorial valued functions to real values functions and vise-versa without 
mentioning every time.}  We use the Einstein summation
convention. Constants vary from line to line and we often delete
their dependence on parameters.

\section{Formulation of the problem.}\label{sec-formulation}

To introduce the fractional stochastic Navier-Stokes equation\del{ on a bounded domain
$ O \subset \mathbb{R}^d$ and on the torus $O=\mathbb{T}^d $, $ d\in \mathbb{N}_1$} we are interested in, 
let us first recall the following classical dD-deterministic Navier-Stokes equation on a bounded domain $ O \subset \mathbb{R}^d$, $ d\in\mathbb{N}_1$
\begin{equation}\label{Eq-classical-SNSE-O-p}
\Bigg\{
\begin{array}{lr}
\partial_tu= \nu \Delta u + (u.\nabla) u - \nabla\pi + f, \;\; t>0, \;
x \in O , \\
div u=0,\;\;\;  \text{(incompressible condition)},\\
\end{array}
\end{equation}
with no-slip boundary condition 
\begin{equation}\label{Eq-Brut-SNSE-BC}
u/ \partial O=0\\
\end{equation}
and initial condition
\begin{equation}\label{Eq-Brut-SNSE-IC}
u(0) = u_0.
\end{equation}
The unknown quantity\del{in Navier-Stokes equation} is the vector $ (u, \pi)$. The vector $u:=(u_j(t, x))_{1\leq j\leq d}$ and 
the scalar $ \pi:=p(t, x)$ describe  respectively the motion velocity and the pressure of
an incompressible fluid evaluated at time $ t$ and at point $ x\in O$. 
The positive constant $ \nu>0$  (later we take, for simplicity, $ \nu =1$) is the viscosity of the fluid  and $ f $ is an external force, which could be random and  could depend on 
the velocity $ u$.
The notation $  (u.\nabla) u$ stands for the product of $ u$ and the gradient Matrix $ (\partial_iu_j)_{1\leq i, j\leq d}$. 
The no-slip boundary condition means that the fluid  is in
a domain which is  bounded by solid impermeable walls. For simplicity, we assume that

\vspace{0.15cm}
  \textit{'' $ O$ is an open bounded and connected set with a $ C^\infty$ boundary
$ \partial O$  and}\\
\vspace{-0.15cm}

\hspace{1.75cm}  \textit{such that $ O$ is on only one side of $ \partial O$.``}
\vspace{0.15cm}

\noindent It is well documented that to deal mathematically with Navier-Stokes equation, we have to split up the
problem  (\ref{Eq-classical-SNSE-O-p}-\ref{Eq-Brut-SNSE-IC})  in to $u-$ respectively $ \pi-$unkown problems.\del{, where the
 unknown in each is respectively $ u$ and $ \pi$.} In this aim, we introduce the following spaces
\begin{equation}
 \mathbb{L}^q(O) := \text{completion in} \;\;  L^q_d(O):= (L^q(O))^d \;\;\;  \;\; \text{ of }
\;\;\; \{ u\in (C_0^\infty(O))^d; div u= 0\},
\end{equation}
\del{where $ C_0^\infty(O) $ is the set of infinitely differentiable functions with compact support (see Notations above).}
\begin{equation}
 \mathcal{Y}_q (O):= \{\nabla p,\;\; p\in H^{1, q}(O)\},
\end{equation}
\del{where $H^{s, q}(O)$ is the Sobolev space of order $ s\in \mathbb{R}$  with  $H^{0, q}(O)=L^q(O) $.} 
Then we get the Helmholtz decomposition
\begin{equation}\label{direct-sum}
  L^q_d(O) = \mathbb{L}^q(O) \oplus\mathcal{Y}_q(O), 
\end{equation}
where the notation $ \oplus$ stands the direct sum, see e.g.  
\cite{Amann-solvability-NSE-2000, Farwig-Sohr-L-p-theory-2005, Foias-book-2001, Fujiwara-Morimoto-L-r-Helmholtz-decomposition-77,
 Giga-Solu-Lr-NS-85, Gigaweak-strong83}. In the case $ q=2$,
the sum above reduced to the orthogonal decomposition see e.g. \cite{Temam-NS-Functional-95, Temam-NS-Main-79} .
Explicitly, $ \mathbb{L}^q(O) $ is given by, see e.g. 
\cite{Daprato-Debussche-Martingale-Pbm-2-3-NS08, Sohr-Generalized-resolvent-94, Foias-book-2001, Gigaweak-strong83, Sohr-al-Imaginary-power-Helmholtz-98} and
\cite[p. 104]{Temam-Inf-dim-88},
\begin{eqnarray}\label{eq-def-main-L-q}
\mathbb{L}^q(O) =\{ u\in L_d^q(O); div u=0, \; on\; O,\; \; u\cdot \vec{n} =0, \; on \; \partial O\},\nonumber\\
\vec{n} \;\;\; \text{is the unit interior normal vector to}\;\; \partial O.
\end{eqnarray}
\noindent We denote by $\Pi_q$ the continuous Helmholtz projection, see e.g. \cite{Abel-Helmut-bound-Stokes-02,  Fujiwara-Morimoto-L-r-Helmholtz-decomposition-77, Gigaweak-strong83,
 Lemari-book-NS-probelems-02},
\begin{equation}\label{Helmoltz-projection}
\Pi_q: L_d^q(O) \rightarrow  \mathbb{L}^q(O).
\end{equation}
It is easy to prove, using \eqref{direct-sum} and the embedding property of the $L^q-$spaces on $ O$\del{ ( with $ O\subset \mathbb{R}^d$ being bounded or 
$ O=\mathbb{T}^d$)} that the restriction of $ \Pi_{q'}$ on $L_d^q(O)$, with $ q'\leq q$,
coincides with $ \Pi_{q}$. Consequently,  we will omit later the dependence in $ q$. The notations $ -A_q^D$ and $ A_q^S$ stand for
the Laplacian with Dirichlet boundary condition respectively
Stokes operator, i.e., see e.g. \cite{Sohr-Generalized-resolvent-94,  Flandoli-Schmalfub-99, Foias-book-2001, Giga-Solu-Lr-NS-85, Giga-Doamian-fract-Stokes-Laplace,  
Temam-NS-Functional-95},
\begin{eqnarray}\label{def-A-D}
 A_q^D = -\Delta \;\;\; \text{with} \;\;  D(A_q^D)&=&   H_d^{2, q}(O) \cap \mathring{H}_d^{1, q}(O)\nonumber\\
&=&\{u\in H_d^{2, q}(O):= (H^{2, q}(O))^d; \; u/\partial O=0\}.
\end{eqnarray}
respectively
\begin{equation}\label{def-A-S}
A_q^S:= -\Pi_q \Delta, \;\; D(A_q^S)= D(A_q^D) \cap \mathbb{L}^q(O).
\end{equation}
\noindent Let us also recall, see e.g. \cite{Fujiwara-Morimoto-L-r-Helmholtz-decomposition-77, Giga-Solu-Lr-NS-85}, that $ \mathbb{L}^q(O)$ is a closed subspace of $ L_d^q(O)$, the operator $ \Pi_q$\del{$: L_d^q(O) \rightarrow L_d^q(O)$} defined on $ L_d^q(O)$(for simplicity we keep the same notation) is bounded and its dual is $ (\Pi_q)^*= \Pi_{q^*},\; ( 1/q+1/{q^*}=1)$ and 
\begin{equation}
(A_q^S)^* = A_{q^*}^S, \;\; (\mathbb{L}^q(O))^*= \mathbb{L}^{q^*}(O), \; ( 1/q+1/{q^*}=1).
\end{equation}
\noindent  Applying Helmholtz projection $ \Pi $ on the two sides of Equation \eqref{Eq-classical-SNSE-O-p},
we get\del{ (\ref{Eq-classical-SNSE-O-p}- \ref{Eq-Brut-SNSE-IC})} on $ \mathbb{L}^q(O)$,
\begin{equation}\label{Eq-classical-SNSE-O}
\Bigg\{
\begin{array}{lr}
\partial_tu= -\nu A^S u + B(u) + \tilde{f}, \;\; t>0, \;
x \in O , \\
 u(0)= \Pi u_0,\\
\end{array}
\end{equation}
where $ \tilde{f}:= \Pi f$ and
\begin{equation}\label{eq-B-projection}
 B(u):= \Pi((u.\nabla) u).
\end{equation}

If $ \tilde{f}$ is random, then Equation \eqref{Eq-classical-SNSE-O} is called stochastic Navier-Stokes equation.

\noindent For $ O=\mathbb{T}^d$, $ d\in \mathbb{N}_1$, we consider the Navier-Stokes problem \eqref{Eq-classical-SNSE-O-p} and \eqref{Eq-Brut-SNSE-IC} and we use the zero space average Lebesgue and Sobolev spaces.  Physically, this condition is meaningful when
the volume forces have zero space average.\del{, see more details  in Section \ref{sec-Comments}} The above calculus remains also 
valid for  $ O=\mathbb{T}^d$ with 
\begin{equation}\label{eq-def-Torus-Lq} 
\mathbb{L}^q(\mathbb{T}^d):= \{ u \in L_d^q(\mathbb{T}^d) := (L^q(\mathbb{T}^d))^d, div u=0 \},\;\;  1<q<\infty,
\end{equation}
\begin{equation}\label{eq-def-Torus-Hs}
\mathbb{H}^{\beta, q}(\mathbb{T}^d):= H_d^{\beta, q}(\mathbb{T}^d)\cap \mathbb{L}^q(\mathbb{T}^d),\; \beta\in \mathbb{R}_+, 1<q<\infty,
\end{equation}
where $ (L^q(\mathbb{T}^d))^d $ and $ (H^{\beta, q}(\mathbb{T}^d))^d,\;\; \beta\in \mathbb{R}, 1<q<\infty $ are the 
corresponding vectorial spaces  of
the following null average Lebesgue and periodic Riesz potential spaces, see e.g.\del{ \cite[ps.44-48]{Foias-book-2001}}
\cite{Debbi-scalar-active, Foias-book-2001, Schmeisser-Tribel-87-book, Sickel-Pontwise-Torus, 
Sickel-periodic spaces-85, Sinai-Mattg-Gibbs-2001},
\del{\begin{eqnarray}
 L^{q}(\mathbb{T}^d):= \{\!\!\!\! &{}&\!\!\!\! f: \mathbb{T}^d \rightarrow \mathbb{C}; f(x):=
\sum_{k\in\mathbb{Z}^d}c_ke^{ikx},\;s.t. \;\;c_0=0, c_{-k}=
\overline{c_k} \;\; \text{and}\nonumber\\
&{}&\;\;
|f|_{L^{q}}:= | \sum_{k\in\mathbb{Z}_0^d}c_ke^{ik\cdot}|_{L^q} <\infty\},
\end{eqnarray}
respectively,
\begin{eqnarray}\label{def-H-s-q}
 H^{\beta, q}(\mathbb{T}^d):= \{\!\!\!\!&{}&\!\!\!\!\!f\in D'(\mathbb{T}^d),\; s.t.\; \hat{f}(0)=0, \; \hat{f}(-k) = \overline{\hat{f}(k)}\; \text{and}\nonumber\\
 &{}&\; |f|_{H^{\beta,q}}:=
| \sum_{k\in\mathbb{Z}_0^d}\del{(1+|k|^2)^\frac \beta2}|k|^\beta\hat{f}(k)e^{ik\cdot}|_{L^q}<\infty\},
\end{eqnarray}
where $ D'(\mathbb{T}^d)$ is the topological dual of $ D(\mathbb{T}^d)$,  $ (c_k:= \hat{f}(k))_{k\in\mathbb{Z}^d}$ 
is the sequence of Fourier
coefficients corresponding to  $ f$,
\begin{equation}\label{Fourier-coeff-Distr-}
c_k=\hat{f}(k):= (2\pi)^{-d} f(e^{ik\cdot}).
\end{equation}
\noindent The notation $ \overline{c_k}$ stands for the complex conjugate. If $ f \in L^q(\mathbb{T}^d)\subset D'(\mathbb{T}^d)$, then
\begin{equation}\label{Fourier-coeff}
c_k=\hat{f}(k):= (2\pi)^{-d}\langle f, e^{ik\cdot}\rangle = (2\pi)^{-d}\int_{\mathbb{T}^d}f(x)e^{-ixk}dx,
\end{equation}
where the brackets in \eqref{Fourier-coeff} stand for the duality, in particular, it also denotes the
scalar product in the Hilbert space  $ L^2(\mathbb{T}^d)$.
An equivalent definition to the spaces $ H^{s, q}(\mathbb{T}^d)$ could be given by using the Bessel potential see
\cite{Debbi-scalar-active}. Moreover, the divergence free condition could be written as, see e.g. \cite{Foias-book-2001},
\begin{equation}\label{div-free-Torus}
 div u = 0  \Leftrightarrow \langle \hat{u}(k),  k\rangle_{\mathbb{Z}_0^d} = 0, \; \forall k \in \mathbb{Z}_0^d.
\end{equation}}
\begin{eqnarray}
 L^{q}(\mathbb{T}^d):= \{\!\!\!\! &{}&\!\!\!\! f: \mathbb{T}^d \rightarrow \mathbb{C}; f(x):=
\sum_{k\in\mathbb{Z}^d}c_ke^{ikx},\;s.t. \;\;c_0=0 \; \text{and}\;
|f|_{L^{q}}:= | \sum_{k\in\mathbb{Z}_0^d}c_ke^{ik\cdot}|_{L^q} <\infty\},\nonumber\\
\end{eqnarray}
respectively,
\begin{eqnarray}\label{def-H-s-q}
 H^{\beta, q}(\mathbb{T}^d):= \{\!\!\!\!&{}&\!\!\!\!\!f\in D'(\mathbb{T}^d),\; s.t.\; \hat{f}(0)=0,\; \text{and}\; |f|_{H^{\beta,q}}:=
| \sum_{k\in\mathbb{Z}_0^d}\del{(1+|k|^2)^\frac \beta2}|k|^\beta\hat{f}(k)e^{ik\cdot}|_{L^q}<\infty\},\nonumber\\
\end{eqnarray}
where $ D'(\mathbb{T}^d)$ is the topological dual of $ D(\mathbb{T}^d)$,  $ (c_k:= \hat{f}(k))_{k\in\mathbb{Z}^d}$ 
is the sequence of Fourier
coefficients corresponding to  $ f$,
\begin{equation}\label{Fourier-coeff-Distr-}
c_k=\hat{f}(k):= (2\pi)^{-d} f(e^{ik\cdot}).
\end{equation}
\noindent If $ f \in L^q(\mathbb{T}^d)\subset D'(\mathbb{T}^d)$, then
\begin{equation}\label{Fourier-coeff}
c_k=\hat{f}(k):= (2\pi)^{-d}\langle f, e^{ik\cdot}\rangle = (2\pi)^{-d}\int_{\mathbb{T}^d}f(x)e^{-ixk}dx,
\end{equation}
where the brackets in \eqref{Fourier-coeff} stand for the duality, in particular, it also denotes the
scalar product in the Hilbert space  $ L^2(\mathbb{T}^d)$.
An equivalent definition to the spaces $ H^{s, q}(\mathbb{T}^d)$ could be given by using the Bessel potential see
\cite{Debbi-scalar-active}. Physically, as the velocity is a real function, one can add to the definition of $ L^{q}(\mathbb{T}^d)$ and $H^{\beta, q}(\mathbb{T}^d)$ the condition $ \hat{f}(-k) = \overline{\hat{f}(k)}$, where the notation\del{ $ \overline{c_k}$} $\overline{\hat{f}(k)}$ stands for the complex conjugate, see e.g. \cite{Foias-book-2001,  Temam-NS-Functional-95}. The techniques developed here and in \cite{Debbi-scalar-active} are valid for both the  complex and the real cases.  Moreover, the divergence free condition could be written as, see e.g. \cite{Foias-book-2001},
\begin{equation}\label{div-free-Torus}
 div u = 0  \Leftrightarrow \langle \hat{u}(k),  k\rangle_{\mathbb{Z}_0^d} = 0, \; \forall k \in \mathbb{Z}_0^d.
\end{equation}
Recall also that in this case (i.e. $ O=\mathbb{T}^d$) and thanks to \eqref{eq-def-Torus-Hs} we have, see e.g. \cite{Foias-book-2001,  Temam-NS-Functional-95} for $ q=2$\del{p52, p9}
\begin{equation}\label{def-doman-As-torus}
 D(A_q^S) = \mathbb{H}^{2, q}(\mathbb{T}^d).
\end{equation}
\noindent 
The equation characterizing the pressure $ \pi$ is derived  
by applying the  divergence operator on both sides of Equation \eqref{Eq-classical-SNSE-O-p}, then we get 
\begin{equation}\label{eq-pressure}
\Delta \pi = div((u\cdot \nabla) u) + div f.
\end{equation}
For brevity reasons, we keep the study of the pressure $ \pi$ beyond the scope of the present work. More discussions about the resolution of  Equation \eqref{eq-pressure} and the conditions ensuring the uniqueness of the solution, could be found e.g. in \cite{Chemin-Book-98, Foias-book-2001, Giga-al-Globalexistence-2001, Lemari-book-NS-probelems-02}. 
\del{Furthermore, we mention that the initial condition
problem \eqref{Eq-classical-SNSE-O-p} and \eqref{Eq-Brut-SNSE-IC} on the torus  and on bounded domains with periodic boundary conditions are equivalent, see e.g. \cite{ DaPrato-Debussche-NS-periodic-2002, Foias-book-2001, Temam-NS-Functional-95, Temam-Inf-dim-88}.}
\vspace{0.25cm}

We assume that $ d\in \mathbb{N}_1$ and either $ O= \mathbb{T}^d$ or $ O\subset \mathbb{R}^d$ is a bounded domain. We 
define the d-dimensional fractional stochastic Navier-Stokes equation  (dD-FSNSE) on $ O$ by replacing the Stokes  
operator $ A^S$ in Equation \eqref{Eq-classical-SNSE-O} by  $ A_\alpha:= (A^S)^\frac\alpha2$, i.e. the dD-FSNSE  
is then given by 
\begin{equation}\label{Main-stoch-eq}
\Bigg\{
\begin{array}{lr}
 du(t)= \left(-\nu
A_{\alpha}u(t) + B(u(t))\right)dt+ G(u(t))dW(t), \; 0< t\leq T,\\
u(0)= u_0,
\end{array}
\end{equation}
where  $ B$ is given by \eqref{eq-B-projection}, 
$ W:= (W(t), t\in [0, T])$ is a Wiener process, $ G$ is a map from $ \mathbb{L}^q(O)$ to a set of bounded operators to be precise later and we assume that the initial data is of divergence free, i.e. $ u_0:= \Pi u_0(\cdot)= \Pi u(0, \cdot)$. To prove that Equation \eqref{Main-stoch-eq}, with $ A_\alpha$ being defined either by 
$ (A^S)^\frac\alpha2$, for $ O=\mathbb{T}^d$ and  $ O\subset\mathbb{R}^d$ bounded,  or by $ \Pi(A^D)^\frac\alpha2\Pi$ in the case $ O\subset\mathbb{R}^d$ bounded, are well defined, we investigate 
simultaneously, some intrinsic properties of the Stokes operator 
$ A^S $ and of 
the Laplacian operator with Dirichlet boundary condition $ A^D$. Later on, we prove  that the two equations are equivalent. 
\del{To prove that 
Equation \eqref{Main-stoch-eq} is well defined, we investigate some intrinsic properties of Stokes operator 
$ A^S $ in both cases
$ O=\mathbb{T}^d$ and  $ O\subset\mathbb{R}^d$ being bounded. Simultaneously, we investigate also similar properties for 
the  Laplacian operator with Dirichlet boundary conditions $ A^D$ in the case $ O\subset\mathbb{R}^d$ bounded.
Our aim behind these two investigations is to prove that 
\eqref{Main-stoch-eq}, with $ A_\alpha$ being defined either by 
$ (A^S)^\frac\alpha2$ or by $ \Pi(A^D)^\frac\alpha2$ and that the two equations are equivalent. }
\begin{theorem}\label{Prop-1-Laplace-Stokes}
[\cite[Lemma 1.1]{Giga-Solu-Lr-NS-85}, \cite[Lemma 2.1]{Gigaweak-strong83}, \cite[Theorem 2]{Giga-semigroup-S-81} 
and 
\cite{Giga-Doamian-fract-Stokes-Laplace, Taylor-PDE-III}.]\del{\cite[p 28-L-p Spectral theory]{Taylor-PDE-III}}
\noindent The operators $ A^S $ and $ A^D$ are densely defined, have bounded inverse 
($0$ is in the resolvent)  and the corresponding semi groups $ (e^{-tA^S})_{t\geq 0}$
respectively  $ (e^{-tA^D})_{t\geq 0}$ are \del{holomorphic of class $ C_0$.} analytic on $\mathbb{L}^q(O)$ 
respectively $ L_d^q(O)$,
 where $ \mathbb{L}^q(O) $ is defined by either \eqref{eq-def-main-L-q} or by \eqref{eq-def-Torus-Lq}.
\end{theorem}
\noindent Consequently, as $ A^S$  and $ A^D$ are  the infinitesimal generators of analytic semigroups, then
we can define the fractional power of $ A^\beta, \beta \in \mathbb{R}$,  where $ A $ stands either for $ A^S$  
or $ A^D$,
 see e.g. \cite[Definition 6.7]{Pazy-83}, \cite[Chap. IX]{Yosida} and \cite{Giga-Doamian-fract-Stokes-Laplace},
\begin{defn}
For all $ \beta >0$, we define $ A^{\beta}$, the fractional power of the operator $ A$, as the inverse of
\begin{equation}\label{Eq-def-A-alpha}
 A^{-\beta}:= \frac{1}{\Gamma(\beta)} \int_0^\infty z^{\beta-1} e^{-zA}dz,
\end{equation}
where the Dunford integral in RHS of \eqref{Eq-def-A-alpha} converges in the uniform operator topology.
\end{defn}
\noindent Moreover, the domain of $A^S$  is given by the following complex interpolation, see e.g. \cite{Abel-Helmut-bound-Stokes-02, Foias-book-2001,  Fujiwara-Morimoto-L-r-Helmholtz-decomposition-77, Giga-Solu-Lr-NS-85, Giga-Doamian-fract-Stokes-Laplace, Gigaweak-strong83, Taylor-PDE-III,  Temam-NS-Functional-95, Triebel-interpolation-theory},
\del{\cite[p28]{Taylor-PDE-III}}
\begin{theorem}\label{theorem-domains-A-D-A-S}
 For every $ 0<\beta <2$, we have 
\begin{itemize}

\item  For $ O\subset\mathbb{R}^d$ bounded
\begin{equation}
D((A^D)^\frac\beta2)= [L_d^q(O), D(A^D)]_{\frac\beta2}= \mathring{H}_d^{\beta, q}(O).
\end{equation}
\begin{equation}\label{eq-embedded-domain-A-beta}
\del{\mathbb{H}^{\beta, q}(O):=} D((A^S)^\frac\beta2)= [\mathbb{L}^q(O), D(A^S)]_{\frac\beta2}= D((A^D)^\frac\beta2)\cap \mathbb{L}^q(O) \hookrightarrow  
H_d^{\beta, q}(O)\cap \mathbb{L}^q(O).
\end{equation}
where $ \hookrightarrow $ means continuously embedded. 

\item For $ O= \mathbb{T}^d$,
\begin{equation}\label{eq-embedded-domain-A-beta-torus}
D((A^S)^\frac\beta2)= [\mathbb{L}^q(\mathbb{T}^d), D(A^S)]_{\frac\beta2}
\del{=   H_d^{\beta, q}(O)\cap \mathbb{L}^q(O)}= \mathbb{H}^{\beta, q}(\mathbb{T}^d).
\end{equation}
\end{itemize} 
\end{theorem}
\noindent  Recall that for $ O\subset \mathbb{R}^d$ bounded, see e.g. \cite{Lions-Magenes-72-I},
$$
\mathring{H}_d^{\beta, q}(O)= H_d^{\beta, q}(O),\; \text{for} \; \beta\leq \frac dq\;\; \text{and}\;
\mathring{H}^{\beta, q}(O) \subsetneqq H_d^{\beta, q}(O),\; \text{for} \; \beta> \frac dq.
$$
\noindent \del{We define the Riesz potential Sobolev spaces,}To identify the notations in formulae \eqref{eq-embedded-domain-A-beta} and \eqref{eq-embedded-domain-A-beta-torus} and the defninition in \eqref{eq-def-Torus-Hs}, we define for $ O\subset \mathbb{R}^d$ being bounded\del{ or $ O=\mathbb{T}^d$, we denote,} \begin{equation}
\mathbb{H}^{\beta, q}(O):= D(( A_q^S)^\frac\beta2), \; \beta\in\mathbb{R},\; 1<q<\infty.
\end{equation}
\noindent For $ O=\mathbb{T}^d$, this notation has already been used for the Riesz potential Sobolev spaces \eqref{eq-def-Torus-Hs}. It is important to mention  that the Dirichlet boundary condition is included in the definition of $ \mathbb{H}^{\beta, q}(O)$ in the case  $ O$ being bounded. Moreover, we have, see e.g. \cite{ Giga-Solu-Lr-NS-85},
\begin{equation}
(\mathbb{H}^{\beta, q}(O))_{\mathbb{L}^q}^*=\mathbb{H}^{-\beta, q^*}(O)\;\; \text{and}\;\; (\mathbb{H}^{\beta, q}(O))_{L^q_d}^*=H_d^{-\beta, q^*}(O).
\end{equation} 
For further discussion see e.g. \cite{Abel-Helmut-bound-Stokes-02, Fujiwara-Morimoto-L-r-Helmholtz-decomposition-77, Giga-Solu-Lr-NS-85, Giga-Doamian-fract-Stokes-Laplace, Gigaweak-strong83} and the references therein.\del{Moreover, by a classical interpolation method, we construct} Using a standard proof like in \cite[Lemma 2.1 \& Lemma 2.2]{Gigaweak-strong83}, see also 
\cite[Theorem 1.7.7]{Pazy-83} and \cite{Yosida}, we infer that (bellow $ A=A^S$ but the same result remains true for $ A=A^D$ and $ \mathbb{L}^q(O)$ replaced by $L_d^q(O)$),
\begin{lem}\label{Lem-semigroup}
The operator $ A^\frac\alpha2:= (A^S)^\frac\alpha2$ is the infinitesimal generator of an analytic semi group
$(e^{-tA^\frac\alpha2} )_{t\geq 0}$ on
$ \mathbb{L}^q(O)$. Moreover, we have for $ \beta \geq 0$,
\begin{equation}\label{eq-semigp-property}
 | A^\frac\beta2 e^{-tA_\alpha}|_{\mathcal{L}(\mathbb{L}^q)}\leq c t^{-\frac\beta\alpha}.
\end{equation}
\end{lem}
\noindent Furthermore, we  recall, see \cite{DaPrato-Debussche-NS-periodic-2002, Debbi-scalar-active, Sohr-Generalized-resolvent-94,  Foias-book-2001, Flandoli-Schmalfub-99}, \cite[ps. 283, 303]{Taylor-PDE-I},
\cite[Chap. II]{Temam-Inf-dim-88}, that $A_2: D(A_2)\rightarrow \mathbb{L}^2(O)$ is an isomorphism,
\del{where $ \mathbb{L}^2(O) $ is defined by either \eqref{eq-def-main-L-q}, for $ O$ being a bounded domain or by \eqref{eq-def-Torus-Lq-Hs} 
bellow for
$ O= \mathbb{T}^d$.}the inverse $A^{-1}$ is self adjoint and
thanks to the compact embedding of $ D(A)$ in
$\mathbb{L}^2(O)$, we conclude that  $A^{-1}$ is compact in $\mathbb{L}^2(O)$.
Hence, there exists an orthonormal basis $( e_j)_{j\in \mathbb{N}}\subset D(A)$ consisting of eigenfunctions of
 $ A^{-1}$ and such that the sequence of eigenvalues
$ (\lambda_j^{-1})_{j\in \mathbb{N}}$ with $ \lambda_j>0 $,  converges to zero. Consequently, $( e_j)_{j\in \mathbb{N}}$ is also
a sequence of eigenfunctions of $ A$  corresponding to the eigenvalues $ (\lambda_j)_{j\in \mathbb{N}}$.
The operator  $ A$ is  positive, self adjoint on $\mathbb{L}^2(O)$ and densely defined. Using the spectral decomposition,
we  construct  positive and negative fractional
powers  $A^{\frac\beta2}, \; \beta \in \mathbb{R} $. In particular, as the spectrum of $ A$ is
reduced to the discrete one, we get an elegant representation for  $(A^{\frac\beta2}, D(A^{\frac\beta2}))$. In fact, let  $\beta \geq 0$, then, see e.g. \cite{Flandoli-Schmalfub-99},
\begin{eqnarray}\label{construction-of-fract-bounded}
\mathbb{H}^{\beta, 2}(O):= D(A^\frac\beta2)&=& \{v\in \mathbb{L}^2(O), \; s.t.\; |v|^2_{D(A^\frac\beta2)}:=
\sum_{j\in \mathbb{N}} \lambda_j^{\beta} \langle v, e_j \rangle^2<\infty\}, \nonumber\\
A^\frac\beta2 v &=& \sum_{j\in \mathbb{N}} \lambda_j^{\frac\beta2}\langle v, e_j \rangle e_j,\;  
\forall v\in D(A^\frac\beta2),
\end{eqnarray}
\noindent with $ (\langle v, e_j \rangle := \hat{v}(j))_j$ is the sequence of Fourier coefficients in the case $ O=\mathbb{T}^d$. 
Furthermore, it is easy to see that
\begin{equation}\label{basis}
A^{\frac\alpha2} e_k := \lambda_k^{\frac\alpha2} e_k, \; k\in \mathbb{N}.
\end{equation}

Now, we introduce the stochastic term. We fix the stochastic basis $ (\Omega, \mathcal{F}, P, \mathbb{F}, W)$, where
$ (\Omega, \mathcal{F}, P) $  is a complete probability space, $\mathbb{F} := (\mathcal{F}_t)_{t\geq 0}$
is a filtration satisfying the usual conditions, i.e. $(\mathcal{F}_t)_{t\geq 0}$ is an increasing right continuous filtration containing all null sets.
The stochastic process $ W:= (W(t), t\in [0, T])$ is a Wiener process with covariance operator $ Q $ being a positive symmetric trace 
class on $ \mathbb{L}^2(O)$\del{$ L_d^2(O)$}. By a  Wiener process on an abstract Hilbert space $ H$, we mean, see e.g.
\cite[Definition 2.1]{Sundar-Sri-large-deviation-NS-06} and \cite{Millet-Chueshov-Hydranamycs-2NS-10, DaPrato-Zbc-92},
\begin{defn}
A stochastic process $ W:= (W(t), t\in [0, T])$ is said to be an $H$-valued 
$ \mathcal{F}_t-$adapted  Wiener process with covariance operator $ Q$, if
\begin{itemize}
\item for all $ 0\neq h\in H$, the process $  (|Q^\frac12 h|^{-1}\langle W(t), h\rangle, t\in [0, T])$ is a standard one dimensional Brownian motion,
\item for all $  h\in H$, the process $  (\langle W(t), h\rangle, t\in [0, T])$ is a martingale adapted to $\mathbb{F}$.
\end{itemize}
\end{defn}
\noindent Otherwise,  the process
$ W:= (W(t), t\in [0, T])$ is a mean zero
 Gaussian process defined on the filtered probability space $ (\Omega,
\mathcal{F}, P, \mathbb{F} )$ with time  stationary independent  increments and covariance function given by:
\begin{equation}\label{Eq-Cov-W}
\mathbb{E}[\langle W(t), f\rangle \langle W(s), g\rangle]= (t\wedge s)\langle Qf, g\rangle,   \;\;\;   t,s \geq 0,\; f, g \in H.
\end{equation}
\noindent  Formally,  we write $ W$ as the sum of an infinite series
\begin{equation}\label{Seri-W}
W(t):= \sum_{j\in\Sigma}\beta_j(t)Q^\frac12e_j,
\end{equation}
\del{\begin{equation}\label{Seri-W}
W(t):= \sum_{j\in\Sigma} q_j^\frac12\beta_j(t)e_j,
\end{equation}}
where $\Sigma= \mathbb{Z}_0^d $, if  $ O= \mathbb{T}^d$, or $ \Sigma= \mathbb{N}_0$, if  $ O\subset \mathbb{R}^d$ bounded,
$(\beta_j)_{j\in \Sigma}$ is an i.i.d.  sequence of  real Brownian motions and
$ (e_j)_{j\in \Sigma}$ is any orthonormal basis, here we consider the basis of the Stokes eigenvalues.  For more illustration one can assume that the basis $ (e_j)_{j\in \Sigma}$ diagonalizes simultaneously the Stokes operator $ A$
and the Covariance  $ Q$ with  $ (q_j)_{j\in \Sigma}$  is the sequence
of the eigenvalues of  $ Q$,
\begin{equation}\label{cond-Q-tr}
 Qe_j = q_j e_j, \;\;\; \text{and} \;\;\; tr(Q):= \sum_{j\in \Sigma} q_j<\infty.
\end{equation}
\noindent The following approximation result, see e.g. \cite{Millet-Chueshov-Hydranamycs-2NS-10, DaPrato-Zbc-92}\del{p99},  will be used later in  some proofs,
\begin{equation}\label{eq-W-n}
W(t)= \lim_{n\rightarrow \infty} W_n(t)\;\; in \; L^2(\Omega; H),\;  \text{ where }\;  
W_n(t):= \sum_{|j|\leq n} \beta_j(t)Q^\frac12e_j.
\end{equation}

\noindent In \del{a great part of} this work, we consider the stochastic Ito integral in Hilbert spaces. 
In particular, we define the following Hilbert spaces
\begin{eqnarray}\label{eq-def-H-0}
 H_0:&=& Q^{\frac12}(H) \;\;\;  \text{endowed with the scalar product} \nonumber\\
&{}&  \langle \phi, \psi\rangle_{H_0}:=
\langle Q^{-\frac12} \phi, Q^{-\frac12}\psi\rangle_{H}, \;\; \forall \phi,\;  \psi\in H_0,
\end{eqnarray}
\begin{eqnarray}\label{eq-def-L-Q}
 L_Q(H):&=& \{S: H \rightarrow H\, s.t.\; SQ^\frac12 \;\text{ is a Hilbert Schmidt operator}\},\nonumber\\
&{}& \langle S_1, S_2\rangle_{L_Q}:= tr (S_2^*QS_1), \; \forall S_1, S_2 \in L_Q(H)
\end{eqnarray}
and
\begin{eqnarray}\label{eq-def-P-T}
\mathcal{P}_T(H):&=&\{\sigma \in L^2(\Omega\times [0, T];\; L_Q(H)),\;\;\text{ predictable processes\del{progressively measurable  processes}}\}, \nonumber\\
&{}& \langle \sigma_1, \sigma_2 \rangle_{\mathcal{P}_T} := \mathbb{E}\int_0^T    \langle \sigma_1(s), \sigma_2(s) \rangle_{L_Q}ds.
\end{eqnarray}
\noindent To be more precise, $ \mathcal{P}_T(H)$ is the set of equivalence classes of predictable processes, but, here we follow the custom to do not make a difference between a class of processes and a process representing this class, see e.g. \cite{Karatzas-Book}.
It is to be noted that as the operator $ Q$ is of trace class, then the canonical injection $ i: H_0\rightarrow H$ is a Hilbert-Schmidt operator. Moreover, we have  $ ii^*=Q$. It is well known that the stochastic integral, $ (\int_0^t \sigma(s)dW(s), t\in [0, T])$,
is well defined  for all $ \sigma \in \mathcal{P}_T(H)$, see e.g. \cite{DaPrato-Zbc-92}.\del{ In this setup,  we consider
in Section \ref{sec-Torus} the space $ \mathcal{P}_T(\mathbb{H}^{1, 2}(O))$ and in Section  \ref{sec-Domain}\del{ and  \ref{sec-Marting-solution},} the space $ \mathcal{P}_T(\mathbb{L}^{2}(O))$. In Section \ref{sec-Marting-solution}, we use a subset of $ \mathcal{P}_T(\mathbb{L}^{2}(O))$ consisting of progressively measurable processes.}  In this setup,  we consider
in Section \ref{sec-Torus} the space $ \mathcal{P}_T(\mathbb{H}^{1, 2}(O))$ and in sections  \ref{sec-Domain} and  \ref{sec-Marting-solution}, the space $ \mathcal{P}_T(\mathbb{L}^{2}(O))$. As we shall refer to  results from \cite{Debbi-scalar-active}\del{In Section \ref{sec-global-mild-solution}{sec-1-approx-local-solution}}, where stochastic integrals have been considered in the  Banach spaces $ \mathbb{H}^{\delta, q}(O)$ with $ 2<q<\infty$ and $\delta \geq0$ and we shall prove some results in the general framework of $L^q-$spaces, we give here some definitions about this notion. It is well known that the spaces $ \mathbb{H}^{\delta, q}(O), q\geq 2$ are UMD Banach spaces of type 2\del{ (They are the finite product of UMD Banach spaces of type 2.}. It is well known, see e.g.
\cite{Neerven-2010, Neerven-Evolution-Eq-08,  Nerveen-stochasticIntegral-07} and the references therein, that for a separable UMD Banach space of type 2 $ X$ and a Hilbert space  $ H$,
the stochastic integral with respect to $W$ \del{in this setting}is  well defined\del{ and the necessary tools such as the Banach version of the Burkholder-Davis-Gundy inequality are disposable,} provided that the integrator $ \sigma:[0, T]\times \Omega \rightarrow \mathcal{L}(H, X) $ is an \del{$\mathbb{L}^2-$}$H-$strongly measurable, (i.e. $ \sigma$ is the pointwise limit of a sequence simple functions), $ \mathcal{F}_t-$adapted process which takes values in the space of $\gamma-$radonifying operators $ R_\gamma(\mathbb{L}^{2}, \mathbb{H}^{\delta, q})$, see e.g. \cite[Theorem 3.6 \& Corollary 3.10]{Nerveen-stochasticIntegral-07},\del{  We define for general separable Banach and Hilbert spaces $ X$  respectively $ H$,}
\begin{eqnarray}\label{eq-def-R-Q}
R_Q(H, X):&=&\{S: H \rightarrow X, s.t.\; SQ^\frac12 \in  R_\gamma(H, X): \text{ the set of }\; 
\gamma-\text{radonifying  operators}\},\nonumber\\
&{}&  ||S||^2_{R_Q}:= ||S Q^\frac12||^2_{R_\gamma} := \mathbb{E}'|\sum_{j\in\Sigma}
\gamma_jS Q^\frac12 h_j|^2_{X},\; \forall S \in R_Q(H, X),
\end{eqnarray}
where  $ (\gamma_j)_{j\in \Sigma}$ is a sequence of independent standard
real-valued Gaussian random variables on a probability space
$(\Omega', \mathcal{F}', P')$ and  $(h_j)_{j\in \Sigma}$ is any orthonormal basis. Moreover, the necessary tools such as the Banach versions of the Ito isometry and the Burkholder-Davis-Gundy inequalities are also disposable, see for more details \cite{Nerveen-stochasticIntegral-07}. Similarly as above we can define 
\begin{eqnarray}\label{eq-def-BP-T}
\mathcal{P}_T(H, X):&=& \{ \sigma \in L^2(\Omega\times [0, T],  R_Q(H, X)),\; \text{H-strongly measurable $\mathcal{F}_t-$adapted  processes }\}, \nonumber\\
&{}& ||\sigma||^2_{\mathcal{P}_T(H, X)} := \mathbb{E}\int_0^T ||\sigma(s)||^2_{R_Q}ds = \mathbb{E}\int_0^T ||\sigma(s)Q^\frac12 ||^2_{R_\gamma}ds.
\end{eqnarray}
\noindent Recall that  for $ X=H$ being a Hilbert space,  $R_Q(H, H) =  L_Q(H)$ and  $ \mathcal{P}_T(H, H) = \mathcal{P}_T(H)$. 
To simplify the notations, we use later on the subscript  
$ \mathcal{P}_T$.\del{, for both $ X$ being either  Hilbert or  Banach spaces.} Furthermore for the same reason above, we introduce the set of diffusion terms we are dealing with in the general framework of Banach spaces, see
for a comparison purpose the conditions in
\cite{Millet-Chueshov-Hydranamycs-2NS-10, Mikulevicius-H1-NS-solution-2004, Flandoli-3DNS-Dapratodebussche-06, Flandoli-Gatarek-95,  Roeckner-Zhang-tamedNS-12,  Roeckner-Zhang-tamedNS-09, Sundar-Sri-large-deviation-NS-06},

{\bf Assumption $(\mathcal{C})$} For fixed $ 2\leq q <\infty$ and $ \delta \geq 0$, we assume that
the operator
$$
G: \mathbb{L}^2(O)\rightarrow R_\gamma(\mathbb{L}^{2}, \mathbb{H}^{\delta, q})
$$
satisfies,
\begin{itemize}
 \item Lipschitz condition: For all $ R>0$, there exists a constant $C_R>0$,  s.t.
\begin{equation}\label{Eq-Cond-Lipschitz-Q-G}
||G(u)-G(v)||_{R_Q}:=
||(G(u)-G(v))Q^\frac{1}{2}||_{R_\gamma(\mathbb{L}^{2}, \mathbb{H}^{\delta, q})}\leq C_R|u-v|_{\mathbb{H}^{\delta, q}},
\;\;\; \forall |u|_{\mathbb{H}^{\delta, q}}, |v|_{\mathbb{H}^{\delta, q}}\leq R,
\end{equation}
\item Linear growth: There exists a constant $ c>0$, s.t
\begin{equation}\label{Eq-Cond-Linear-Q-G}
||G(u)||_{R_Q}:= ||G(u)Q^\frac{1}{2}||_{R_\gamma(\mathbb{L}^{2}, \mathbb{H}^{\delta, q})}\leq c(1+|u|_{\mathbb{H}^{\delta, q}}),
\;\;\; \forall u\in \mathbb{H}^{\delta, q}(O).
\end{equation}
\end{itemize}
The parameters $ q$ and $\delta$ are chosen independently for every result. \del{To prove the global existence of a mild solution for the dD-FSNSE, we need the following extra condition, 
see for a comparison purpose the conditions in \cite{Daprato-Debussche-Martingale-Pbm-2-3-NS08, Flandoli-3DNS-Dapratodebussche-06, Flandoli-Gatarek-95, Roeckner-Zhang-tamedNS-09}.} It is of great interest to mention here that for simplicity reasons and without any loss of generality, 
we have assumed that the diffusion term $ G(u)$ acts on $ \mathbb{L}^2(O)$. Otherwise, we can use $ \Pi G(u)$.

\del{\noindent {\bf Assumption $(\mathcal{C}_b)$}: We assume that\del{for  $ u\in \mathbb{L}^q(O)$, $ G(u)$ maps 
$ D(A_q^\frac12)$ in to $ D(A_q^{\frac12-\frac\eta2})$, the operator 
$ A^{\frac12-\frac\eta2}G(u) \in R_Q(\mathbb{L}^2, \mathbb{H}^{-\theta, q})$, with 
$ \frac\eta\alpha+\frac\theta2<\frac12-\frac1p$, $p>2$ and   
there exists a constant $c:=c_{\eta, \theta,  q}>0$,
s.t.} the mapping   $ u\in \mathbb{L}^q(O) \mapsto A^{\frac12-\frac\eta2}G(u) \in R_Q(\mathbb{L}^2, \mathbb{H}^{-\theta, q})$, with 
$ \frac\eta\alpha+\frac\theta2<\frac12-\frac1p$, $p>2$ is bounded and   
there exists a constant $c:=c_{\eta, \theta,  q}>0$,
s.t.
\begin{equation}\label{Eq-Cond-unif-bound}
\sup_{u}||A_{q}^{\frac12-\frac\eta2}G(u)||_{R_Q(\mathbb{L}^2, \mathbb{H}^{-\theta, q})} \leq c_{\eta, \theta,  q}.
\end{equation} 

\del{\noindent {\bf Assumption $(\mathcal{C}_c)$}: For  $ u\in \mathbb{H}^{1, 2}(O)$, $ curlG(u) \in R_Q(\mathbb{L}^2, L_1^q)$, $ 2\leq q<\infty$  and satisfies \eqref{Eq-Cond-Lipschitz-Q-G} and \eqref{Eq-Cond-Linear-Q-G} with 
$R_Q(\mathbb{L}^2, \mathbb{H}^{\delta, q})$ in the LHSs is repalced by $ R_Q(\mathbb{L}^2, L_1^q)$ and $ \delta=1$ in the RHSs.}\del{ $ \mathbb{L}^2$ and $ \mathbb{H}^{\delta, q}$ are replaced by 
$ L^2_1$ and  $ L^q_1$ respectively.}

\noindent It is easy to see that the set of diffusions $ G$ satisying Assumption  $(\mathcal{C}_b)$ is not empty. In fact, for $ G=I$, Assumption $(\mathcal{C}_b)$ is then  given by 
\begin{equation}
tr(A^{1-\eta}Q)<\infty,
\end{equation}
which is a weaker condition than \cite[Cond. 1.3.]{Daprato-Debussche-Martingale-Pbm-2-3-NS08}. Moreover,  we introduce the following examples.
Let us take $ \theta =0$ and denote by
\begin{equation}\label{eq-def-sigma-k}
\sigma^k(u):= G(u)Q^{\frac12}e_k= q_{k}^\frac12G(u)e_k,  for\;\;  k\in\Sigma.
\end{equation}
Then the following condition 
\begin{equation}\label{falndoli-example-1}
 \sum_{j\in \Sigma}q_j\sigma_j^2<\infty, \;\;\; \text{with}\;\;\; \sigma_j^2:= \sup_{u}
 |G(u(s))e_j|^2_{\mathbb{H}^{1-\eta, q}}
\end{equation}
is sufficient to prove that $(\mathcal{C}_b)$ is satisfied. Remark that the condition above is weaker than the condition \cite[H2]{Roeckner-Zhang-tamedNS-09}. The following example is a slit generalization of the example introduced in \cite{Flandoli-Gatarek-95}, 
\begin{equation}
 G(u)x= ((G(u)x)^j)_{j}:= 
 (\sum_{k\in \Sigma} \xi_k^j(u)\langle x, e_j\rangle e_k^j)_j,\;\; j=1,\cdots , d,\del{ \;\; \text{with the condition here}\;\; 
\sum_{j\in \Sigma}q_j\sigma_j^2\lambda_j^{1-\eta}|e_j|^2_{\mathbb{L}^q}<\infty}
\end{equation}
where $(\xi_k^j(u))_k$, $ j=1,\cdots , d$ are sequences of real kernels. 
 Then the condition \eqref{falndoli-example-1} becomes 
 \begin{equation}
 \sum_{k\in \Sigma}q_k\xi_k^2\lambda_k^{1-\eta}<\infty, \;\;\; \text{with}\;\;\; \xi_k^2:= \sup_{u}
 |\xi_k(u)|_{\mathbb{R}^d}^2.
\end{equation}
\noindent It is of great interest to mention here that for simplicity reasons and without any loss of generality, 
we have assumed that the diffusion term $ G(u)$ acts on $ \mathbb{L}^2(O)$. Otherwise, we can use $ \Pi G(u)$. The following lemma plays an important role in the estimations of the stocastic term.

\begin{lem}\label{lem-est-z-t}
Let us fix $ \alpha\in (0, 2]$, $ 2\leq q <\infty$,  $ p> 2$ and 
$ u \in L^p(\Omega; L^\infty(0, T; \mathbb{H}^{\beta, q}(O)))$, $ \beta \geq 0$, an $ \mathbb{F}-$adapted stochastic process.
We introduce the Ornstein-Uhlenbeck stochastic process,
\begin{equation}\label{Eq-z-t}
 z(t):=  \int_0^te^{-A_\alpha (t-s)}G
(u(s))W(ds).
\end{equation}
Assume $ G$ satisfies \eqref{Eq-Cond-Linear-Q-G}. Then 
$ z\in L^p(\Omega; C([0, T]; \mathbb{H}^{\beta+\eta, q}(O)))$,\del{$ (z(t), t\in [0, T])$ has a.s. continuous trajectories and for
all} for $ 0\leq \eta <\alpha(\frac12-\frac1p)$ and there exists a constant $c$, such that
\begin{equation}\label{Eq-est-z-t}
  \mathbb{E}\sup_{[0, T]}| z(t)|^p_{\mathbb{H}^{\beta+\eta, q}}\leq c(1+\mathbb{E}\sup_{[0, T]}|u(s)|_{\mathbb{H}^{\beta, q}}^p).
 \end{equation}
Moreover, if $ G$ satisfies Assumption $(\mathcal{C}_b)$,
 then there exists a constant $ c:= c_{\eta, \theta,  q}>0$, s.t.
\begin{equation}\label{Eq-nabla-z-t}
   \mathbb{E}\sup_{[0, T]}| \nabla z(t)|^p_{q}\leq c_{\eta, \theta,  q}<\infty.
 \end{equation}
\end{lem}
\begin{proof}
 The first statement is a straightforward application of \cite[Proposition 4.2]{Neerven-Evolution-Eq-08}. To prove the second statement, we apply again \cite[ Proposition 4.2]{Neerven-Evolution-Eq-08} and use Assumption $(\mathcal{C}_b)$, we infer for $ 0<\gamma <1/2$ and 
 $ \theta $ and $ \eta $ satisfying $ \frac\eta\alpha+\frac\theta2<\gamma-\frac1p$,  that 
 \begin{eqnarray}\label{Eq-nabla-z-t-proof}
    \mathbb{E}\sup_{[0, T]}| \nabla z(t)|^p_{q}\del{&\leq&
    c\mathbb{E}\sup_{[0, T]}| \int_0^tA^\frac12e^{-A_\alpha (t-s)}G
(u(s))W(ds)|^p_{\mathbb{L}^{q}}\nonumber\\}
&\leq&
    c\mathbb{E}\sup_{[0, T]}|\int_0^tA^{\frac12-\frac\eta2}e^{-A_\alpha (t-s)}G
(u(s))W(ds)|^p_{\mathbb{H}^{2\frac\eta\alpha, q}}\nonumber\\
&\leq&
    c\int_0^T\mathbb{E}\big(\int_0^t(t-s)^{-2\gamma}||A^{\frac12-\frac\eta2}G
(u(s))||_{R_Q(\mathbb{L}^2, \mathbb{H}^{-\theta, q}}^2ds\big)^\frac p2dt\leq c.
\del{&\leq&
    c\int_0^T\mathbb{E}\big(\int_0^t(t-s)^{-2\gamma}||A^{\frac12-\frac\eta2}e^{-A_\alpha (t-s)}G
(u(s))W(ds)|^p_{R_Q(\mathbb{L}^2, \mathbb{H}^{-\theta, q}}ds\big)dt\nonumber\\}
 \end{eqnarray} 
\end{proof}}

\del{\noindent {\bf Assumption $(\mathcal{C}_b)$}: We assume that the mapping   $ u\in \mathbb{L}^q(O) \mapsto G(u) \in R_Q(\mathbb{L}^2,\; \mathbb{L}^{q})$ is bounded and   
there exists a constant $c:=c_{\eta, \theta,  q}>0$ such that 
\begin{equation}\label{Eq-Cond-unif-bound}
\sup_{u}||G(u)||_{R_Q(\mathbb{L}^2,\; \mathbb{L}^{q})} \leq c_{q}.
\end{equation} 
\noindent It is easy to see that the set of diffusions $ G$ satisfying Assumption  $(\mathcal{C}_b)$ is not empty. In fact, for $ G=I$, Assumption $(\mathcal{C}_b)$ is then  given by 
\begin{equation}\label{imed-cond}
tr(Q)<\infty.
\end{equation}
The condition \eqref{imed-cond} is weaker than the condition  \cite[Cond. 1.3]{Daprato-Debussche-Martingale-Pbm-2-3-NS08}. Moreover,  we introduce bellow some other examples. We denote by
\begin{equation}\label{eq-def-sigma-k}
\sigma^k(u):= G(u)Q^{\frac12}e_k= q_{k}^\frac12G(u)e_k,\;  for\;\;  k\in\Sigma.
\end{equation}
The following condition 
\begin{equation}\label{falndoli-example-1}
 \sum_{j\in \Sigma}q_j\sigma_j^2<\infty, \;\;\; \text{with}\;\;\; \sigma_j^2:= \sup_{u}
 |G(u(s))e_j|^2_{\mathbb{L}^{q}}
\end{equation}
is sufficient to prove that $(\mathcal{C}_b)$ is satisfied. Remark that the condition above is weaker than the condition \cite[H2]{Roeckner-Zhang-tamedNS-09}. The following example has been introduced in \cite{Flandoli-Gatarek-95}, but our condition here is weaker, 
\begin{equation}
 G(u)x= ((G(u)x)^j)_{j}:= 
 (\sum_{k\in \Sigma} \xi_k^j(u)\langle x, e_j\rangle e_k^j)_j,\;\; j=1,\cdots , d,\del{ \;\; \text{with the condition here}\;\; 
\sum_{j\in \Sigma}q_j\sigma_j^2\lambda_j^{1-\eta}|e_j|^2_{\mathbb{L}^q}<\infty}
\end{equation}
where $(\xi_k^j(u))_k$, $ j=1,\cdots , d$ are sequences of real kernels. 
 The condition \eqref{falndoli-example-1} becomes 
 \begin{equation}
 \sum_{k\in \Sigma}q_k\xi_k^2<\infty, \;\;\; \text{with}\;\;\; \xi_k^2:= \sup_{u}
 |\xi_k(u)|_{\mathbb{R}^d}^2.
\end{equation}
\noindent It is of great interest to mention here that for simplicity reasons and without loss of generality, we have assumed that the diffusion term $ G(u)$ acts on $ \mathbb{L}^2(O)$. Otherwise, we can use $ \Pi G(u)$. The following lemma plays an important role in the estimations of the stochastic term. We give it here in a general framework see also \cite{Debbi-scalar-active}.

\begin{lem}\label{lem-est-z-t}
Let us fix $ \alpha\in (0, 2]$, $ 2\leq q <\infty$,  $ p> 2$ and 
$ u \in L^p(\Omega; L^\infty(0, T; \mathbb{H}^{\beta, q}(O)))$, $ \beta \geq 0$, a strongly measurable $ \mathbb{F}-$adapted stochastic process.
We introduce the Ornstein-Uhlenbeck stochastic process
\begin{equation}\label{Eq-z-t}
 z(t):=  \int_0^te^{-A_\alpha (t-s)}G
(u(s))W(ds).
\end{equation}
Assume that $ G$ satisfies Assumption $(\mathcal{C})$ \del{\eqref{Eq-Cond-Linear-Q-G} }(with $ \delta$ here is replaced by $ \beta$). Then $ z$ is well defined, has a strongly measurable $ \mathbb{F}-$adapted version, $ z\in L^p(\Omega; C([0, T]; \mathbb{H}^{\beta+\eta, q}(O)))$,\del{$ (z(t), t\in [0, T])$ has a.s. continuous trajectories and for
all} and for $ 0\leq \eta <\alpha(\frac12-\frac1p)$ there exists a constant $c>0$ such that
\begin{equation}\label{Eq-est-z-t}
  \mathbb{E}\sup_{[0, T]}| z(t)|^p_{\mathbb{H}^{\beta+\eta, q}}\leq c(1+\mathbb{E}\sup_{[0, T]}|u(s)|_{\mathbb{H}^{\beta, q}}^p).
 \end{equation}
Moreover, if $ G$ satisfies Assumption $(\mathcal{C}_b)$ (with $ \delta$ here is replaced by $ \beta$), then there exists a constant $ c:= c_{q}>0$ such that 
\begin{equation}\label{Eq-nabla-z-t}
   \mathbb{E}\sup_{[0, T]}|z(t)|^p_{\mathbb{L}^q}\leq c_{q}<\infty.
 \end{equation}
\end{lem}
\begin{proof}
Thanks to the strong measurablility and  $ \mathbb{F}-$adaptdness of the  stochastic process $(u(t), t\in [0, T]) $, the continuity of $ G$ and to Proposition A.1 in  \cite{Neerven-Evolution-Eq-08}, we prove that the process $ (z(t), t\in [0, T])$ is well defined and has strongly measurable $ \mathbb{F}-$adapted version. The
first estimation \eqref{Eq-est-z-t} is a straightforward application of  \cite[Proposition 4.2]{Neerven-Evolution-Eq-08}, of the type 2 inequality and of the definition of the fractional power of $A$. To prove the second estimation \eqref{Eq-nabla-z-t}, we apply again the same ingredients above\del{\cite[ Proposition 4.2]{Neerven-Evolution-Eq-08}} and use Assumption $(\mathcal{C}_b)$, we infer for $ 0<\gamma <1/2$ satisfying $ 0<\gamma-\frac1p$,  that 
 \begin{eqnarray}\label{Eq-nabla-z-t-proof}
    \mathbb{E}\sup_{[0, T]}| z(t)|^p_{\mathbb{L}^q}
\del{&\leq&
    c\mathbb{E}\sup_{[0, T]}|\int_0^tA^{\frac12-\frac\eta2}e^{-A_\alpha (t-s)}G
(u(s))W(ds)|^p_{\mathbb{H}^{2\frac\eta\alpha, q}}\nonumber\\}
&\leq&
    c\int_0^T\mathbb{E}\big(\int_0^t(t-s)^{-2\gamma}||G
(u(s))||_{R_Q(\mathbb{L}^2,\; \mathbb{L}^{q})}^2ds\big)^\frac p2dt\leq c.
\del{&\leq&
    c\int_0^T\mathbb{E}\big(\int_0^t(t-s)^{-2\gamma}||A^{\frac12-\frac\eta2}e^{-A_\alpha (t-s)}G
(u(s))W(ds)|^p_{R_Q(\mathbb{L}^2, \mathbb{H}^{-\theta, q}}ds\big)dt\nonumber\\}
 \end{eqnarray} 
\end{proof}}

\del{Now, Equation \eqref{Main-stoch-eq} with $ A_\alpha:= (A^S)^{\frac{\alpha}{2}}$ is well defined. 
This equation can be seen as the fractional version of the stochastic Navier-Stokes equation (FSNSE). 
A stochastic version of the fractional Navier-Stokes equation (SFNSE)
 can also be constructed, by taking 
$ A_\alpha:= \Pi(-\Delta)^{\frac{\alpha}{2}}$ on $ \mathbb{L}^q(O)$. 
For simplicity, let us keep in mind, for a short time, the two different terminologies (FSNSE and SFNSE) and the distinction between the two equations.\del{ and referee to  the two equations by fractional version of the stochastic NSE (FSNSE) 
respectively the stochastic version of the fractional NSE (SFNSE).} Later on, we shall prove that the two equations are equivalent. Recall that by a fractional Navier-Stokes equation (FNSE), we mean Equation 
 \eqref{Eq-classical-SNSE-O-p}, with $ \Delta$ replaced by $ -(-\Delta)^\frac\alpha2$. The SFNSE is a 
stochastic perturbation of FNSE. Thanks to theorems \ref{Prop-1-Laplace-Stokes}-\ref{Lem-semigroup} 
and the calculus above, the SFNSE is also well defined.
The main question now is whether or not the two equations FSNSE and SFNSE are equivalent. Let us  before moving to this last question,  
clarify some pratical and theoritical matters for the two versions. 
Indeed, as the derivation of the equations describing physical phenomena, is mainley based on the deterministic case, it is intuitively seen that the stochastic version of the fractional Navier-Stokes equation (SFNSE) is more suitable for physical modelings, rather than the fractional version of the stochastic Navier-Stokes equation (FSNSE). One can see e.g. \cite{SugKak} for the derivation of the deterministic fractional Burgers equations, see also \cite{Caffarelli-2009}. Moreover, as we shall prove the equivalence of the two versions, we conclude that the FSNSE, seems intuitevly more theoretical, is  of practical intereset as well.

\noindent  In the case $ O= \mathbb{T}^d$, the operators  $ \Delta $ and $ div$ are commuting. Therefore, the Stokes operator $-A^S$
is the Laplacian $ \Delta$ on $ \mathbb{L}^q( \mathbb{T}^d)$, see e.g.
\cite{Gigaweak-strong83},
\cite[p. 48]{Foias-book-2001}, \cite[p. 9]{Temam-NS-Functional-95} and \cite[p 105]{Temam-Inf-dim-88} for the Torus,
 \cite{Kato-Ponce-86} for $ O=\mathbb{R}^d$  and see also a direct  proof in Section \ref{sec-Torus}. Combining this statment with the results in Theorem \ref{theorem-domains-A-D-A-S}, we conclude that  
\begin{eqnarray}\label{eq-A-S-alpha-Delta-alpha}
D((A^S)^\frac\alpha2) &=& D(\Pi_q(-\Delta)^\frac\alpha2\Pi_q) = H_d^{\alpha, q}(\mathbb{T}^d)\cap \mathbb{L}^q(\mathbb{T}^d),\nonumber\\
(A^S)^\frac\alpha2 u &=& \Pi(-\Delta)^\frac\alpha2 u= (-\Delta)^\frac\alpha2 u, \;\; \forall u \in D((A^S)^\frac\alpha2).
\end{eqnarray}
This proves that the two equations FSNSE and SFNSE, on the torus, are ''equivalent''. 
In the case $ O\subset \mathbb{R}^d$  is bounded, the Stokes operator $ A^S$ is not equal to $ -\Delta$. In fact,
 as  we can not expect that for general $ u \in D(A^S)$, we can get
$ \Delta u\cdot \vec{n} =0 $ on $\partial O$, it is not obvious whether or not $ \Delta u \in \mathbb{L}^q(O)$. Our claim is that $ (A^S)^\frac\alpha2 =  \Pi(A^D)^\frac\alpha2 \Pi$.
In fact, thanks to \eqref{def-A-D} and \eqref{def-A-S}, it is easy to deduce that 
 $ A^S = \Pi A^D\Pi$, see also \cite{Fujiwara-Morimoto-L-r-Helmholtz-decomposition-77}.\del{In fact, this latter is a simple 
consequence of \eqref{def-A-D} and \eqref{def-A-S}.
using \eqref{def-A-D} and \eqref{def-A-S}, we get
\begin{eqnarray}
 D(\Pi A^D\Pi)&=& D(A^D)\cap \mathbb{L}^q(O) = D(A^S)\nonumber\\
\Pi A^D u &=& -\Pi \Delta u = A^Su, \;\;\;  \forall u \in  D(\Pi A^D)\cap \mathbb{L}^q(O).
\end{eqnarray}
In the second step, we prove that $ (A^S)^\frac\alpha2 =  \Pi(A^D)^\frac\alpha2 \Pi$.}
Using Theorem \ref{theorem-domains-A-D-A-S}, we infer that
\begin{equation}
 D(\Pi (A^D)^\frac\alpha2\Pi)=  D((A^D)^\frac\alpha2)\cap \mathbb{L}^q(O) = D((A^S)^\frac\alpha2).
\end{equation}
\del{Recall the following definition of the negative power of $ A$, with  
$ A$ could be either $ A^S$ or $ A^D$,  see e.g. 
\cite{Giga-Doamian-fract-Stokes-Laplace, Pazy-83}, 
\begin{equation}\label{negative-fractional-A}
 A^{-\frac\alpha2} u = \frac1{2\pi i}\int_\Gamma z^{-\frac\alpha2}(A-zI)^{-1} u dz, \; \forall u \in \mathbb{L}^q(O).
\end{equation}
where  $ \Gamma $ is the path running the resolvent set from $ \infty e^{-i\theta}$ to  $ \infty e^{i\theta}$, $ <0\theta <\pi$,  avoiding the 
negative real axis and the origin and such that the branch $  z^{-\frac\alpha2}$ is taken to be positive for real for real positive values of $z$.
The integral in the RHS of \eqref{negative-fractional-A-D} converges in the uniform operator topology.}
Moreover, using the definition of the negative power of $ A^S$ and $ A^D$ via the resolvent, see e.g. 
\cite{Giga-Doamian-fract-Stokes-Laplace, Pazy-83} and the definition of the Helmholtz projection, we infer that 
\begin{equation}\label{negative-fractional-A-D}
 (A^D)^{-\frac\alpha2} \Pi^{-1}u = \frac1{2\pi i}
 \int_\Gamma z^{-\frac\alpha2}(A^D-zI)^{-1}\Pi^{-1}u dz,\;\;\; \forall u \in \mathbb{L}^q(O),
\end{equation}
where  $ \Gamma $ is the path running the resolvent set from $ \infty e^{-i\theta}$ to  
$ \infty e^{i\theta}$, $ 0<\theta <\pi$,  avoiding the 
negative real axis and the origin and such that the branch $  z^{-\frac\alpha2}$ is taken to be positive
for real  positive values of $z$.
The integral in the RHS of \eqref{negative-fractional-A-D} 
converges in the uniform operator topology. Therefore, for all $ u \in \mathbb{L}^q(O)$,
\begin{eqnarray}\label{negative-fractional-A-D-last}
 (A^D)^{-\frac\alpha2} \Pi^{-1}u &=& \frac1{2\pi i}\int_\Gamma z^{-\frac\alpha2}(\Pi A^D-z\Pi)^{-1}u dz
 =\frac1{2\pi i}\int_\Gamma z^{-\frac\alpha2}(A^S-zI)^{-1}u dz
= (A^S)^{-\frac\alpha2}.\nonumber\\
\end{eqnarray}
As the operators $(A^S)^{-1}$ and $(A^D)^{-1}$ are one-to-one this achieved the equivalence between the FSNSE and the 
SFNSE.} 

Now, thanks to theorems \ref{Prop-1-Laplace-Stokes}-\ref{Lem-semigroup} 
and to the calculus above, "equations" \eqref{Main-stoch-eq} with either $ A_\alpha:= (A^S)^{\frac{\alpha}{2}}$ or $ A_\alpha:= \Pi(-\Delta)^{\frac{\alpha}{2}}$ on $ \mathbb{L}^q(O)$  are well defined. The main question is whether or not the two equations are equivalent. The answer is yes. The proof and further discussions\del{, in particular  about the theoretical and the practical interest of the equation are discussed} are presented in Appendix \ref{Appendix-Equivalence}. We end this section by the following assumption on the initial condition

\noindent {\bf Assumption $(\mathcal{B})$}:  Assume that the initial condition $u_0$ is an 
$\mathcal{F}_0-$random variable satisfying
\begin{equation}\label{Eq-initial-cond}
u_0\in L^{p}(\Omega, \mathcal{F}_0, P; \mathbb{H}^{\delta_0, q_0}(O)),
\end{equation}
with  $ p\geq 2$ and either $2\leq q_0\leq \infty$ and $ \delta_0= 0$ or  $2\leq q_0< \infty$ and $ \delta_0> 0$.

\del{\noindent {\bf Assumption $(\mathcal{B}_\alpha)$}:  Assume that the initial condition $u_0$ is an 
$\mathcal{F}_0-$random variable s.t
\begin{equation}\label{assumption-B-alpha}
u_0\in L^2(\Omega, \mathcal{F}_0, P, \mathbb{H}^{\frac{d+2-\alpha}{4}, 2}(O)).
\end{equation}}
 \section{Definitions of solutions and Results.}\label{sec-Results}
In this section, we give the different  definitions of solutions we are interested in. 
\del{\begin{defn}\label{def-variational solution}
Let  $ H$ be a separable Hilbert space and let us give the Gelfand triple
\begin{equation}
 V \hookrightarrow H\hookrightarrow V^*,
\end{equation}
where $ V$ is a  separable reflexive Banach space with topological dual  $ V^* $.
Assume that $ u_0 \in L(\Omega, \mathcal{F}_0, P; H)$.
The couple $(u, \tau)$, where $(u(t), t\in [0, T])$ is an $H$-valued $ \mathbb{F}-$adapted stochastic process and $(\tau)$ is a predictable stoping time,  is called a local weak solution of Equation
\eqref{Main-stoch-eq}, iff
 \begin{equation}\label{eq-solu-weak-stopped-1}
 u(t) = u(t\wedge \tau),\; P-a.s.\;\; \forall t\in [0, T],
 \end{equation}
 \begin{equation}\label{eq-set-solu-weak}
 u(\cdot, \omega) \in \del{L^\infty(0, T; H) \cap}L^2(0, T; V)\cap C([0, T]; H)\;\;\; a.s.
 \end{equation}
\del{for a given $ \delta \geq 0$ and} such that $ P-a.s.$ the following identity holds, for all $ t\in [0, T]$ and for all $ \varphi \in V$,
\begin{eqnarray}\label{Eq-weak-Solution}
\langle u(t), \varphi\rangle_H &=& \langle u_0, \varphi\rangle_H + \int_0^{t\wedge \tau} {}_{V^*}\langle A^\frac\alpha2 u(s), \varphi_{V} \rangle ds+
\int_0^{t\wedge \tau} \langle B(u(s), \varphi), u(s)\rangle ds\nonumber\\
& + & {}_{V^*}\langle\int_0^{t\wedge \tau}G(u(s))dW(s), \varphi\rangle_{V}.
\end{eqnarray}
In the case  $ \tau=T, \; P-a.s.$, we say that $(u(t), t\in [0, T]) $ is a weak solution(global).
\end{defn}}
\begin{defn}\label{def-variational solution}
Let  $ H$ be a separable Hilbert space and let
$ V, V_1, V_2$ be  separable reflexive Banach spaces such that,
\begin{equation}\label{vspaces-def-veriational}
V_2 \hookrightarrow V \hookrightarrow H\cong H^* \hookrightarrow V^* \hookrightarrow V_1,
\end{equation}
with  $ V^*$  being the topological dual of $ V$.
Assume that $ u_0\in L^p(\Omega, \mathcal{F}_0, P, H)$.
A $ \mathcal{F}_t-$adapted $H$-valued\del{  progressively measurable}  stochastic process $(u(t), t\in [0, T])$  is called a weak solution of Equation \eqref{Main-stoch-eq}, if
 \begin{equation}\label{eq-set-solu-weak}
 u(\cdot, \omega) \in L^\infty(0, T; H) \cap L^2(0, T; V)\cap C([0, T]; V_1)\;\;\; P-a.s.
 \end{equation}
and for all $ t\in [0, T]$,  the following identity holds  $ P-a.s.$,  for all $ \varphi \in V_2$,
\begin{eqnarray}\label{Eq-weak-Solution}
\langle u(t), \varphi\rangle_H &=& \langle u_0, \varphi\rangle_{H} + \int_0^t {}_{V_2^*}\langle A^\frac\alpha2 u(s), \varphi \rangle_{V_2} ds+
\int_0^t {}_{V_2^*}\langle B(u(s)), \varphi)\rangle_{V_2} ds\nonumber\\
& + & {}_{V_2^*}\langle\int_0^tG
(u(s))dW(s), \varphi\rangle_{V_2}.
\end{eqnarray}
\end{defn}
\del{\begin{defn}\label{def-maxi-variational solution}
Let $ \tau $ be a predictable stopping time such that $ P(\tau >0)=1$ and let $(u(t), t\in [0, T])$  be  a 
progressively measurable $H$-valued  stochastic process, $(u(t), t\in [0, T])$, where $ H$ is a separable Hilbert space. The couple $ (u, \tau)$ is called 
a local weak solution of Equation \eqref{Main-stoch-eq}, iff
\begin{equation}\label{eq-solu-weak-stopped-1}
 u(t) = u(t\wedge \tau),\; P-a.s.\;\; \forall t\in [0, T],
 \end{equation}
and for $ 0<S<\tau $ the stopped process $(u(t\wedge S), t\in [0, T])$ is a weak solution of Equation \eqref{Main-stoch-eq} in the 
sense of Definition \ref{def-variational solution}.
 \end{defn}}
\del{\begin{defn}\label{def-maxi-variational solution}
Let $ \tau $ be a predictable stopping time such that $ P(\tau >0)=1$, i.e. there exists an increasing sequence of stopping times $(\tau_n)_n\nearrow \tau$ and let $(u(t), t\in [0, T])$  be  a 
progressively measurable $H$-valued  stochastic process,  where $ H$ is a separable Hilbert space. The couple $ (u, \tau)$ is called 
a local weak solution of Equation \eqref{Main-stoch-eq}, iff 
\begin{equation}\label{eq-solu-weak-stopped-1}
 u(t) = u(t\wedge \tau),\; P-a.s.\;\; \forall t\in [0, T]
 \end{equation}
and for all $n $ the stopped process $(u(t\wedge \tau_n), t\in [0, T])$ is a weak solution of the stopped Equation \eqref{Eq-weak-Solution}, in the sense of Definition \ref{def-variational solution}.\\ The local solution $ (u, \tau)$ is said to be maximal if there exists a seperable Hilbert space $ V_0\hookrightarrow H$, such that 
\begin{equation}
\limsup_{t\nearrow \tau}|u(t)|_{V_0}=\infty \;\; on\;\; \{\tau<\infty\}\;\; P-a.s.
\end{equation}
\end{defn}}
\del{\begin{defn}\label{def-maxi-variational solution}Let  $ H$ be a separable Hilbert space and let $ \tau $ be a predictable stopping time such that $ P(\tau >0)=1$, i.e. there exists an increasing sequence of stopping times $(\tau_n)_n\nearrow \tau$ and let $(u(t), t\in [0, T])$  be  an $H$-valued  $ \mathcal{F}_t-$adapted stochastic process. The couple $ (u, \tau)$ is called 
a local weak solution of Equation \eqref{Main-stoch-eq}, iff 
\begin{equation}\label{eq-solu-weak-stopped-1}
 u(t) = u(t\wedge \tau),\; \;\; \forall t\in [0, T],  \;\; P-a.s.
 \end{equation}
and for all $n $ the stopped process $(u(t\wedge \tau_n), t\in [0, T])$ is a weak solution of the stopped Equation \eqref{Eq-weak-Solution} in the sense of Definition \ref{def-variational solution}.\\ The local solution $ (u, \tau)$ is said to be maximal if there exists a seperable Hilbert space $ V_0\hookrightarrow H$ such that 
\begin{equation}
\limsup_{t\nearrow \tau}|u(t)|_{V_0}=\infty \;\; on\;\; \{\tau<\infty\}\;\; P-a.s.
\end{equation}
\end{defn}}
\del{\begin{defn}\label{def-maxi-variational solution}Let  $ H$ be a separable Hilbert space and let $ \tau $ be a \del{predictable} stopping time,\del{ (i.e. there exists an increasing sequence of stopping times  $(\tau_n)_n\nearrow \tau$)}   such that $ P(\tau >0)=1$ and let $(u(t), t\in [0, T])$  be  an $H$-valued  $ \mathcal{F}_t-$adapted stochastic process. The couple $ (u, \tau)$ is called 
a local weak solution of Equation \eqref{Main-stoch-eq} if
\begin{equation}\label{eq-solu-weak-stopped-1}
 u(t) = u(t\wedge \tau),\; \;\; \forall t\in [0, T],  \;\; P-a.s.
 \end{equation}
and the stopped process $(u(t\wedge \tau), t\in [0, T])$ is a weak solution of the stopped Equation \eqref{Eq-weak-Solution} in the sense of Definition \ref{def-variational solution}.\\ The local solution $ (u, \tau)$ is said to be maximal if\del{ there exists a seperable Hilbert space $ V_0\hookrightarrow H$ such that} 
\begin{equation}
\limsup_{t\nearrow \tau}|u(t)|_{H}=\infty  \;\; on\;\; \{\tau<T\}.
\end{equation}
\end{defn}}

\begin{defn}\label{def-maxi-variational solution}Let  $ H$ be a separable Hilbert space and let $ \tau $ be a \del{predictable} stopping time,\del{ (i.e. there exists an increasing sequence of stopping times  $(\tau_n)_n\nearrow \tau$)}   such that $ P(\tau >0)=1$ and let $(u(t), t\in [0, T])$  be  a time strongly continuous  $H$-valued  $ \mathcal{F}_t-$adapted stochastic process. The couple $ (u, \tau)$ is called 
a local weak solution of Equation \eqref{Main-stoch-eq} if
\begin{equation}\label{eq-solu-weak-stopped-1}
 u(t) = u(t\wedge \tau),\; \;\; \forall t\in [0, T],  \;\; P-a.s.
 \end{equation}
and the stopped process $(u(t\wedge \tau), t\in [0, T])$ is a weak solution of the stopped Equation \eqref{Eq-weak-Solution} in the sense of Definition \ref{def-variational solution}.\\ The local solution $ (u, \tau)$ is said to be maximal if\del{ there exists a seperable Hilbert space $ V_0\hookrightarrow H$ such that} 
\begin{equation}
\limsup_{t\nearrow \tau}|u(t)|_{H}=\infty  \;\; on\;\; \{\tau<T\}.
\end{equation}
\end{defn}

\begin{defn}\label{def-martingle-solution}
The multiple  $ (\Omega^*, \mathcal{F}^*, P^*, \mathbb{F}^*, W^*, u^*)$, where
$ (\Omega^*, \mathcal{F}^*, \mathbb{P}^*, \mathbb{F}^*, W^*)$ is a stochastic basis with
$ W^* $ being  a $ Q-$Wiener process of trace class and $ u^*:=(u^*(t), t\in[0, T])$ being an \del{$ \mathcal{B}([0, t])\times \mathcal{F}^*_t-$, $\mathcal{F}^*_t\times \mathcal{B}((0, t))-$ progressevely measurable} $\mathcal{F}^*_t-$adapted stochastic process, is called a martingale solution of Equation \eqref{Main-stoch-eq},  if  $ \theta^*$ is a solution of Equation \eqref{Main-stoch-eq} in
the sense of Definition \ref{def-variational solution}
on  the basis $ (\Omega^*, \mathcal{F}^*, P^*, \mathbb{F}^*, W^*)$.
\end{defn}

\del{We will also give some results concerning the global existence of mild solutions, see Definition in  \cite[Pages 969-967]{Neerven-Evolution-Eq-08}. 

\begin{defn}\label{def-mild-solution}
Let $ X$ be an UMD-Banach space of type 2. Assume that $ u_0:\Omega\rightarrow X$ is strongly $ \mathcal{F}_0-$measurable\del{$ \in L^p(\Omega, \mathcal{F}_0, P; X)$}.  A strongly measurable $ \mathcal{F}_t-$adapted $ X$-valued
stochastic process, $(u(t), t\in [0, T])$,  is called a mild solution of Equation \eqref{Main-stoch-eq} if
\begin{itemize}
\item for all $ t\in [0, T]$, $ s \mapsto e^{-(t-s)A_\alpha}B(u(s))$ is in $ L^0(\Omega, L^1(0, t: X))$,
\item for all $ t\in [0, T]$, $ s \mapsto e^{-(t-s)A_\alpha}G(u(s))$ is $ H-$strongly measurable $ \mathcal{F}_t-$adapted ans a.s. in $ R_Q(H, X)$,   
\item $ \forall t \in [0, T]$, $ P-a.s.$ the following equality holds in $ X$
\begin{equation}\label{Eq-Mild-Solution}
u(t)= e^{-A_\alpha t}u_0 +
\int_0^t e^{-A_\alpha (t-s)}B(u (s))ds + \int_0^te^{-A_\alpha (t-s)}G
(u(s))W(ds).
\end{equation}
\end{itemize}
\end{defn}

\begin{defn}\label{def-local-mild-solution}
Let \del{$ p\geq 2$ and }$ X$ be an UMD-Banach space of type 2. Assume that\del{ $ u_0 \in L^p(\Omega, \mathcal{F}_0, P; X)$}$ u_0:\Omega\rightarrow X$ is strongly $ \mathcal{F}_0-$measurable. A local mild solution of Equation \eqref{Main-stoch-eq} is a couple $(u, \tau_\infty)$, where  $\tau_\infty\leq T\;  a.s. $ being a predictable stopping time (see \cite{Revuz-Yor}) and
 $(u(t), t\in [0, \tau_\infty))$ is a strongly measurable $ \mathcal{F}_t-$adapted $ X$-valued stochastic process such that
\begin{itemize}
 \item  there exists an increasing sequence of stopping times
$ (\tau_n)_{n\in \mathbb{N}}$, s.t. $ \tau_n\nearrow \tau_\infty $.
\item For all $ n\in \mathbb{N}$, the process $ (u(t\wedge \tau_n), t\in [0, T])$  satisfies \del{\eqref{eq-cond-norm-mild-solution} and} the stopped  \eqref {Eq-Mild-Solution} equation,
i.e. for all $ n \in \mathbb{N}$ and $ \forall t \in [0, T]$ the following equation is satisfied 
\begin{eqnarray}\label{Eq-Mild-Solution-stoped}
u(t\wedge \tau_n)= e^{-A_\alpha (t\wedge \tau_n)}u_0 &+&
\int_0^{(t\wedge \tau_n)} e^{-A_\alpha (t\wedge \tau_n-s)}B(u (s\wedge \tau_n))ds\nonumber \\
& + &\int_0^te^{-A_\alpha (t-s)}1_{[0, \tau_n)}(s)G
(u(s\wedge \tau_n))W(ds), \; P-a.s. 
\end{eqnarray}
\end{itemize}
\end{defn}}

\begin{remark}\label{Rem-1}
\begin{itemize}
\item As a consequence of the condition\del{ \eqref{eq-cond-cont-mild-solution} and} \eqref{eq-set-solu-weak},\del{$ u(\cdot, \omega) \in L^\infty(0, T; X) \cap C([0, T]; X_1)$ with $ X_1\hookrightarrow X$,} the trajectories of the \del{mild solutions in definitions \ref{def-mild-solution} and \ref{def-local-mild-solution} and the}weak solutions in the definition \ref{def-variational solution} \del{and  \ref{def-maxi-variational solution} are $X-$weakly respectively} are $H-$weakly continuous and $ P-a.s.$, \del{$ u(t)\in X$ respectively }$ u(t)\in H$ for all $ t\in [0, T]$, see e.g. \cite{Flandoli-Gatarek-95, Temam-NS-Main-79}.
\item Remark that for $ V_2 =V$ in \eqref{vspaces-def-veriational}, we get the classical Gelfant triple and the well known classical definition of the weak solution (one can forget about $ V_1$). We shall see that this classical formulation is also valued for the fractional case provided that $\alpha \geq \alpha_0(d):= 1+\frac{d-1}{3}$, see Section \ref{sec-Domain} and Section \ref{sec-Marting-solution}. See  also \cite{Debbi-scalar-active} for a similar definition and conditions for the $ L^q$-spaces with $ 2\leq q<\infty$. The main feature for the sub-super and critical regimes \del{difficulty, for which we need the space $ V_2 \hookrightarrow V$ and $ V_2 \neq V$,} is that contrarily to the dissipative and the hyperdissipative regimes ($\alpha\geq 2$), the decrease of the values of $ \alpha$ makes e.g. the spaces $ \mathbb{H}^{\frac\alpha2, 2}(O)$ and their duals $ \mathbb{H}^{-\frac\alpha2, 2}(O)$ approaching simultaneously the space $ \mathbb{L}^2(O)$ and thus approaching each other. Therefore the difficulty to give a sense to the fourth term in \eqref{Eq-weak-Solution} arises. According to our calculus, the values  $ \alpha_0(d)$ makes a threshold which characterizes the two phenomena.
\item The solution in Definition \ref{def-variational solution} is known in the literature either as a strong or a weak or a weak-strong solution, see e.g. \cite{Millet-Chueshov-Hydranamycs-2NS-10, Roeckner-Zhang-tamedNS-12, Roeckner-Zhang-tamedNS-09, Sundar-Sri-large-deviation-NS-06}. In fact, this
solution is strong in probabilistic sense and weak in the analytic sense. In this work, we use initially the terminology weak. In some places, if there is need to recall, we also use the terminology weak-strong.
\item The Definition \ref{def-maxi-variational solution} is used in \cite{Mikulevicius-H1-NS-solution-2004} in more general framework, see also similar definitions in \cite{Daprato-Debussche-Martingale-Pbm-2-3-NS08, Mikulevicius-H1-NS-solution-2009, Mikulevicius-H1-NS-solution-2004}. 
\del{\item The Definition of the local solution is based on the idea to stop each term of Equation \eqref {Eq-Mild-Solution}.
Contrary to other terms, it is not obvious how to stop the stochastic term; that is due to the fact that the integrand should be adapted.
The  stopped formula in  \eqref {Eq-Mild-Solution-stoped} is obtained in \cite{Brzez-Beam-eq}.}
\end{itemize}
\end{remark}

\noindent The main results of this work are

\del{\begin{theorem}\label{Main-theorem-mild-solution-d}\del{[Multidimensional case \& subcritical regime]}
Let $ d \in \mathbb{N}_{1}$, $ T>0$ and  $ p> 2$, $ \alpha \in ]1, 2] $\del{$ \delta \geq 0$} and $ d<q\leq_\infty q_0\leq \infty$ be fixed.
Assume $ u_0$\del{ $ u_0 \in L^p(\Omega, \mathcal{F}_0, P; D(A^\frac\delta2_{q_0}))$} and  $ G$  satisfying assumptions $(\mathcal{B})$ and $(\mathcal{C})$ respectively,
 with\del{  $\frac d{\alpha-1}< q\leq_\infty q_0$ and $ \frac d{\alpha-1} < q_0\leq \infty$ and} $ 0\leq \delta\leq \delta_0$. \del{and that  $ \alpha \in ( 1+ \frac dq,  2]$ with}Then
\begin{itemize}

\vspace{0.35cm}

\item  { (3.6.1)  [Local mild solution for dD-FSNSE.]} For $ \alpha \in (1+\frac dq, 2]$, Equation \eqref{Main-stoch-eq} has a local mild solution $ (u, \tau_\infty)$ in the sense of Definition \ref{def-local-mild-solution}, \del{satisfying for all stopping time $ S<\tau_\infty$,
\del{\begin{equation}\label{Eq-proper-local-solution}
 u \in L^p(\Omega;  L^\infty([0, S]; \mathbb{H}^{\delta_1, q}(O))\cap L(\Omega;  C([0, S];
\mathbb{H}^{-\delta'', q}(O)),
 \end{equation}}
with $ X:=D(A_q^{\frac\delta2})=\mathbb{H}^{\delta, q}(O)$, $  X_1:=D(A_{q}^{-\frac{\delta''}2}) = \mathbb{H}^{\delta'', q}(O)$
with $ X:=D(A_q^{\frac\delta2})=\mathbb{H}^{\delta, q}(O)$, $  X_1:=D(A_{q^*}^{-\frac{\delta''}2}) = \mathbb{H}^{\delta'', q^*}(O)$ and $\delta_1 \leq \delta$ and $\delta''\geq \alpha+1+\frac dq-\delta$.
\begin{equation}\label{Eq-proper-local-solution}
 u \in L^p(\Omega;  L^\infty([0, \tau_\infty); \mathbb{H}^{\delta, q}(O))\cap L(\Omega;  C([0, \tau_\infty);
\mathbb{H}^{-\delta'', q}(O)),
 \end{equation}}
with $ X:=D(A_q^{\frac\delta2})=\mathbb{H}^{\delta, q}(O)$, $  X_1:=D(A_{q}^{-\frac{\delta''}2}) = \mathbb{H}^{-\delta'', q}(O)$  and $\delta''\geq \alpha+1+\frac dq-\delta$.

\vspace{0.35cm}

\item { (3.6.2) [Global mild solution for the 2D-FSNSE on the torus.]} If  $ O= \mathbb{T}^2$ and in addition,\del{$ curl u_0\in L^p(\Omega, \mathcal{F}_0, P; L^{q}(\mathbb{T}^2)$ and} $ G$ satisfies also Assumption $(\mathcal{C})$ for $\delta$ replaced by $ 1$ with $ p\geq q$ and $u_0$ satisfies 
\begin{equation}\label{eq-curl-u-0-torus}
curl u_0 \in L^p(\Omega, \mathcal{F}_0, P; L^{q_0}(\mathbb{T}^d)),\,
\end{equation}
with  $p\geq \max\{q_0, 4\}$, then Equation \eqref{Main-stoch-eq} admits a unique global mild solution satisfying\del{  for $ P-a.s.$,}
\del{\begin{equation}
 u(\cdot, \omega) \in  L^q(\Omega; L^\infty(0, T ; \mathbb{H}^{1, q}(\mathbb{T}^2)), \;\; P-a.s.
\end{equation}}
\begin{equation}\label{propty-of-2D-global-mild-sol}
 \mathbb{E}\sup_{[0, T]}|u(t)|^q_{\mathbb{H}^{1, q}}+ \mathbb{E}\int_0^T|u(t)|^2_{\mathbb{H}^{1+\frac\alpha2, 2}}dt<\infty
\end{equation}
provided that one of the following cases is satisfied

{\bf case 1. }$ \alpha \in (\frac43, 2]$, $6\leq q_0 \leq \infty$ and $\frac2{\alpha-1} < q\leq 6$.

{\bf case 2. } $6 < q\leq \min\{q_0, \frac{4}{2-\alpha}\}$ and  $ 2-\frac{4}{q} <\alpha \leq 2$.

{\bf case 3. }  $  \frac{2}{\alpha-1}< q \leq q_0 \leq 6$ and  $1+\frac{2}{q}< \alpha \leq 2$.
\del{\begin{equation}\label{est-u-torus-H-1-q-1}
 \mathbb{E}\sup_{[0, T]}|u(t)|^2_{\mathbb{H}^{1, q}}<\infty.
\end{equation}
Assume $ \alpha \in (\frac43, 2]$ and
\begin{equation}\label{eq-curl-u-0-torus}
curl u_0 \in L^p(\Omega, \mathcal{F}_0, P; L^{q_0}(\mathbb{T}^d)),\,
\end{equation}
with $6\leq q_0 \leq \infty$,  $p\geq q_0$ and supposed large enough,  $ G$ satisfying
Assumption $ (\mathcal{C}_1)$ with $ \delta =0, \frac2{\alpha-1} < q\leq 6$ and that
Equation \eqref{Main-stoch-eq} admits a weak solution in the sense of Definition
\ref{def-variational solution} satisfying
\eqref{eq-set-solu-weak-torus} and \eqref{cond-solu-torus-H1}. Then the solution satisfies the following estimation

\noindent Furthermore,  Estimation \eqref{est-u-torus-H-1-q-1} is also satisfied in the following cases

{\bf case 1. } $6 \leq q\leq \min\{q_0, \frac{4}{2-\alpha}\}$ and  $ 2-\frac{4}{q} <\alpha \leq 2$.

{\bf case 2. }  $  \frac{2}{\alpha-1}< q \leq q_0 \leq 6$ and  $1+\frac{2}{q}< \alpha \leq 2$.

  $ curl u_0\in L^p(\Omega, \mathcal{F}_0, P; L^{q}(\mathbb{T}^2)$
 and $ G$ satisfies Assumption $(\mathcal{C})$ for $\delta$ replaced by $ 1$ and Assumption $(\mathcal{C}_b)$, then Equation \eqref{Main-stoch-eq} admits a unique global mild solution satisfying for $ P-a.s.$,}

\vspace{0.35cm}

\item { (3.6.3) [Global unique mild solution for the dD-FSNSE.]} Assume that $ d\geq 2$, $ O$ is either bounded or $ O= \mathbb{T}^2$,  $ G$ satisfies in addition\del{Assumption $(\mathcal{C})$ for $\delta$ replaced by $ 1$ and} Assumption $(\mathcal{C}_b)$, with $ p\geq q$\del{ (prefer to be large)} and Equation \eqref{Main-stoch-eq} admits a local mild solution $(u, \tau_\infty)$ which enjoys 
\begin{equation}\label{cond-global-mild-solu}
\mathbb{E}\sup_{[0, \tau_\infty)}\int_0^t (t-s)^{-\frac{d}{\alpha q}}|\nabla u(s)|_{q}ds\leq c<\infty,
\end{equation}
then Equation \eqref{Main-stoch-eq} admits an $ \mathbb{L}^q-$valued  unique global mild solution.
\end{itemize}
\end{theorem}}

\begin{theorem}\label{Main-theorem-strog-Torus}
Let $ O=\mathbb{T}^2$, $ \alpha \in [1, 2]$ and  $ u_0$  satisfying Assumption $(\mathcal{B})$, with $ \delta_0\geq 1$, $ q\geq 2$ and $ p\geq 4$.
\del{\begin{equation}\label{cond-torus-u0}
u_0\in L^p(\Omega, \mathcal{F}_0, P; \mathbb{H}^{1, 2}(\mathbb{T}^2) ).
\end{equation}}
\begin{itemize}

\vspace{0.35cm}

\item { (\ref{Main-theorem-strog-Torus}.1) [Existence of weak-strong solution.]}  Assume that $ G$ satisfies  Assumption $ (\mathcal{C})$ with $ q=2$,\del{ $ \delta=1$} $ \delta\in \{0, 1\}$ and $ C_R$ being independent of $ R$
(global Lipschitz). Then for $ \alpha \in [\frac43, 2]$, Equation \eqref{Main-stoch-eq} admits a weak solution (strong in probability)
$ (u(t), t\in [0, T])$ in the sense of Definition \ref{def-variational solution}, with the corresponding Gelfant triple
\begin{equation}\label{Gelfand-triple-Torus}
 \mathbb{H}^{1+\frac\alpha2, 2} (\mathbb{T}^2)\hookrightarrow \mathbb{H}^{1, 2}(\mathbb{T}^2)\hookrightarrow
(\mathbb{H}^{1+\frac\alpha2, 2}(\mathbb{T}^2))^*=\mathbb{H}^{1-\frac\alpha2, 2}(\mathbb{T}^2),
\end{equation}
\del{and which }satisfying
\begin{equation}\label{eq-set-solu-weak-torus}
u(\cdot, \omega) \in L^\infty(0, T; \mathbb{H}^{1, 2}(\mathbb{T}^2))\cap L^2(0, T; \mathbb{H}^{1+\frac \alpha2, 2}(\mathbb{T}^2))\cap
C([0, T]; \mathbb{L}^{2}(\mathbb{T}^2))\;\; P-a.s.
 \end{equation}
and \del{for all $ p'\leq p$}
\begin{eqnarray}\label{cond-solu-torus-H1}
\mathbb{E}\Big(\sup_{[0, T]}|u(t)|^{p}_{\mathbb{H}^{1, 2}}&+&
\int_0^T|u(t)|^2_{\mathbb{H}^{1+\frac\alpha2, 2}} dt\Big)<\infty.
\end{eqnarray}
\del{\item Assume $ \alpha \in [\frac{p}{p-1}, 2]$ and Equation \eqref{Main-stoch-eq} admits a weak solution in the sense of Definition
\ref{def-variational solution} satisfying
\eqref{eq-set-solu-weak-torus} and \eqref{cond-solu-torus-H1}. Assume $ G$ satisfying  Assumption $ (\mathcal{C})$ with $ q=2$,
$ \delta= 0$ (local Lipschitz). Then the solution is unique.}

\del{If in addition,  $ G$ satisfies \eqref{Eq-Cond-Linear-Q-G}, with $ \delta =1+\frac\alpha2$ and $ q=2$ and  $ \alpha \in (\frac43, 2]$, then 
\begin{equation}\label{Eq-H-1-regu-2D-torus}
u(\cdot, \omega)\in C([0, T]; \mathbb{H}^{1, 2}( \mathbb{T}^2)).\;\; a.s.
\end{equation}}

\vspace{0.35cm}

\item { (\ref{Main-theorem-strog-Torus}.2) [Uniqueness of the weak solution.]}  Assume\del{ $ \alpha \in [1, 2]$ and} that $ G$ satisfies  
Assumption $ (\mathcal{C})$ with\del{ ( \eqref{Eq-Cond-Lipschitz-Q-G}, with} $ q=2$ and $ \delta= 0$ (local Lipschitz). Then if for $ \alpha\in[1, 2]$,
Equation \eqref{Main-stoch-eq} admits a weak solution in the sense of Definition \ref{def-variational solution} satisfying
\eqref{eq-set-solu-weak-torus} and \eqref{cond-solu-torus-H1}, pathwise uniqueness holds.\del{ it is unique.}

\vspace{0.35cm}

\item { (\ref{Main-theorem-strog-Torus}.3) [Space regularity of the weak solution.]}  Assume that $ \alpha \in (1, 2]$, $ G$ satisfies
Assumption $ (\mathcal{C})$ with $ \delta =1$ and $\frac2{\alpha-1} < q<\infty$, $ u_0$ satisfies 

\begin{equation}\label{eq-curl-u-0-torus}
curl u_0 \in L^p(\Omega, \mathcal{F}_0, P; L^{q_0}(\mathbb{T}^d)),\,
\end{equation}
and Equation \eqref{Main-stoch-eq} admits a weak solution $(u(t), t\in [0, T])$ in the sense of Definition
\ref{def-variational solution} satisfying
\eqref{eq-set-solu-weak-torus} and \eqref{cond-solu-torus-H1}. Then $(u(t), t\in [0, T])$ satisfies 

\begin{equation}\label{propty-of-2D-global-mild-sol}
 \mathbb{E}\sup_{[0, T]}|u(t)|^q_{\mathbb{H}^{1, q}}+ \mathbb{E}\int_0^T|u(t)|^2_{\mathbb{H}^{1+\frac\alpha2, 2}}dt<\infty,
\end{equation}
\del{\eqref{propty-of-2D-global-mild-sol},} for $ \alpha$, $ q$ and $ q_0$ follow one of the cases 

{\bf case 1. }$ \alpha \in (\frac43, 2]$, $6\leq q_0 \leq \infty$ and $\frac2{\alpha-1} < q\leq 6$.

{\bf case 2. } $6 < q\leq \min\{q_0, \frac{4}{2-\alpha}\}$ and  $ 2-\frac{4}{q} <\alpha \leq 2$.

{\bf case 3. }  $  \frac{2}{\alpha-1}< q \leq q_0 \leq 6$ and  $1+\frac{2}{q}< \alpha \leq 2$.

\del{
{\bf case 1. }$ \alpha \in (\frac43, 2]$, $6\leq q_0 \leq \infty$ and $\frac2{\alpha-1} < q\leq 6$.

{\bf case 2. } $6 \leq q\leq \min\{q_0, \frac{4}{2-\alpha}\}$ and  $ 2-\frac{4}{q} <\alpha \leq 2$.

{\bf case 3. }  $  \frac{2}{\alpha-1}< q \leq q_0 \leq 6$ and  $1+\frac{2}{q}< \alpha \leq 2$.}
\del{\item (3.7.3) Assume $ \alpha \in (\frac43, 2]$ and
\begin{equation}\label{eq-curl-u-0-torus}
curl u_0 \in L^p(\Omega, \mathcal{F}_0, P; L^{q_0}(\mathbb{T}^d)),\,
\end{equation}
with $6\leq q_0 \leq \infty$,  $p\geq q_0$ and supposed large enough,  $ G$ satisfying
Assumption $ (\mathcal{C}_c)$ with $\frac2{\alpha-1} < q\leq 6$ and that
Equation \eqref{Main-stoch-eq} admits a weak solution in the sense of Definition
\ref{def-variational solution} satisfying
\eqref{eq-set-solu-weak-torus} and \eqref{cond-solu-torus-H1}. Then the solution satisfies the following estimation
\begin{equation}\label{est-u-torus-H-1-q-1}
 \mathbb{E}\sup_{[0, T]}|u(t)|^2_{\mathbb{H}^{1, q}}<\infty.
\end{equation}
\noindent Furthermore,  Estimation \eqref{est-u-torus-H-1-q-1} is also satisfied in the following cases

{\bf case 1. } $6 \leq q\leq \min\{q_0, \frac{4}{2-\alpha}\}$ and  $ 2-\frac{4}{q} <\alpha \leq 2$.

{\bf case 2. }  $  \frac{2}{\alpha-1}< q \leq q_0 \leq 6$ and  $1+\frac{2}{q}< \alpha \leq 2$.}
\end{itemize}
\del{Moreover, if  $ G$ satisfies also the following conditions:
\noindent there exists a constant $ c>0$, such that
\begin{equation}\label{cond-torus-lingrow}
|[\sum_{k\in\mathbb{Z}^d_0} q_k|curl(G(u)e_k)|^2]^\frac12|_{L^q}  \leq c(1+ |curl u|_{L^{q_0}}), \; \forall u\in \mathbb{H}^{1, q}(\mathbb{T}^2).
\end{equation}
For all $ R>0$, there exists $ c_R>0$, such that
\begin{equation}\label{cond-torus-lipschitz}
|[\sum_{k\in\mathbb{Z}^d_0} |curl((G(u) -G(v))Q^\frac12 e_k)|^2]^\frac12|_{L^q}  \leq c_R|curl (u-v)|_{L^q},
\; \forall |u|_{\mathbb{H}^{1, q}}, |v|_{\mathbb{H}^{1, q}} \leq R,
\end{equation}
where $ q_0\geq 6$ and  $ curl u \in L_2^6(\mathbb{T}^2)$. Then $(u(t), t\in[0, T])$, satisfying
\begin{equation}\label{eq-T-2-strong-solu}
u(\cdot, \omega) \in L^\infty(0, T;  \mathbb{H}^{1, q}(\mathbb{T}^2))\cap L^2(0, T;  \mathbb{H}^{1+\frac\alpha2, 2}
(\mathbb{T}^2))\cap C([0, T]; \mathbb{H}^{\frac\alpha2, 2}(\mathbb{T}^2)). a.s.
\end{equation}}
\end{theorem}

\begin{theorem}\label{Main-theorem-martingale-solution-d} \del{[Multidimensional case\& general regime\del{Martingale solution}]}
Let $ d \in \mathbb{N}_{1}$,  $ \alpha \in (0, 2] $ and $ T>0$ be fixed.
Assume that $ u_0 $ and $ G$  satisfy Assumption  $(\mathcal{B})$ respectively Assumption $(\mathcal{C})$ (\eqref{Eq-Cond-Linear-Q-G} with $ \delta =0, q=2\leq_\infty q_0\leq \infty$ and $ p\geq 4$). Then  
 
\begin{itemize}
\item { (\ref{Main-theorem-martingale-solution-d}.1) [Martingale solution.]}  Equation \eqref{Main-stoch-eq}
has a martingale solution, \\
$ (\Omega^*, \mathcal{F}^*, P^*, \mathbb{F}^*, W^*, u^*)$,  in the sense of  Definition \ref{def-martingle-solution}, satisfying \eqref{eq-set-solu-weak}, \eqref{Eq-weak-Solution}, with $ V_2=V=\mathbb{H}^{\frac\alpha2, 2}(O)$, $ H=\mathbb{L}^{2}(O)$, $ V_1=\mathbb{H}^{-\delta', 2}(O)$, with $\delta'> 1+\frac d2$ and satisfies \del{for all $ p'\leq p$the following estimate}
\begin{equation}\label{eq-mart-l-2solu}
\mathbb{E}\sup_{[0, T]}|u(t)|^p_{\mathbb{L}^2}+ \mathbb{E}\int_0^{T}|u(t)|^2_{\mathbb{H}^{\frac\alpha2, 2}}\leq c<\infty.
\end{equation}
\del{\begin{eqnarray}\label{cond-solu-torus-l-2}
\mathbb{E}\Big(\sup_{[0, T]}|u(t)|^{p'}_{\mathbb{H}^{1, 2}}&+&
\int_0^T|u(t)|^{p'-2}_{\mathbb{H}^{1, 2}}\Big(|u(t)|^2_{\mathbb{H}^{1+\frac\alpha2, 2}} +
|u(t)|_{\mathbb{H}^{\beta, q_1}}^2 \Big)dt \nonumber\\
&+& \int_0^T|u(t)|^4_{\mathbb{H}^{1, 2}}dt + \int_0^T|u(t)|^{\frac{\alpha}{\eta}}_{\mathbb{H}^{1+\eta, 2}}dt\Big)<\infty,
\end{eqnarray}
where $ \beta \leq 1+\frac\alpha2-\frac d2+\frac d{q_1}$, $ 2\leq q<\infty$ and $ 0<\eta\leq \frac\alpha2$.} In particular, for $ 1+\frac{d-1}{3}<\alpha \leq 2$, we can take $ V_2= \mathbb{H}^{\frac\alpha2, 2}(O)$, $ V_1=\mathbb{H}^{-\delta', 2}(O)$.

\del{\item {\bf (\ref{Main-theorem-martingale-solution-d}.2) Regularity of the martingale solution.} 
  If in addition $ \alpha \in [\alpha_0:=1+\frac{d-1}{3}, 2]$, then 
\begin{equation}\label{cont-cond-martg}
 u^*(\cdot, \omega)\in C([0, T]; \mathbb{L}^2(O)),\; P^*-a.s.
\end{equation}}
\item { (\ref{Main-theorem-martingale-solution-d}.2) [Uniqueness  of the martingale solution.]}
If $ G$  satisfies  \eqref{Eq-Cond-Lipschitz-Q-G} with $ q=2, \delta =0$ and Equation \eqref{Main-stoch-eq} has a martingale solution $ u^*$ satisfying the following condition
\begin{equation}\label{uniquness-con-martg}
 P^*(u^*(\cdot, \omega) \in L^{\frac{4\alpha}{3\alpha-d-2}}(0, T; \mathbb{H}^{\frac{d+2-\alpha}{4}, 2}(O)))=1,
\end{equation}
then pathwise uniqueness holds and consequently $ u^*$ is the unique global strong-weak solution.\del{  corollary of the pathwise uniqueness  in (3.9.2), the martingale solution becomes probabilistically strong.}
\end{itemize}

\del{\noindent Moreover, the solution $ (\theta(t), t\in [0, T])$ satisfies \eqref{eq-set-solu-weak},
\eqref{Eq-weak-Solution} and \eqref{Eq-reg-theta-mild}, for  $ d(1-\frac2q)\leq \alpha \leq 2$ and either
\vspace{-0.35cm}
\begin{itemize}
\item $ d\leq \alpha$ and   $ 2\leq q\leq_\infty q_0$, or
\item $ d> \alpha$ and $ 2\leq q \leq_\infty\min\{q_0, \frac{2d}{d-\alpha}\}$,
\end{itemize}
\begin{itemize}
\item $ d=1$, $ \alpha\geq 1$ and   $ 2\leq q\leq_\infty q_0$, or
\item $ d=2$, $ \alpha= 2$ and   $ 2\leq q\leq_\infty q_0$, or
\item $ d\geq \alpha$ and $ 2\leq q \leq q_0\leq \frac{2d}{d-2} $.
\end{itemize}}
\end{theorem}

\begin{theorem}\label{Main-theorem-boubded-2}
Let $ d\in\{2, 3\}$, $ O=\mathbb{T}^d$ or $ O\subset \mathbb{R}^d$ bounded, $ \alpha \in (1+\frac{d-1}3, 2]$ and  $ u_0$  satisfies
\del{\begin{equation}\label{cond-torus-u0}
u_0\in L^q(\Omega, \mathcal{F}_0, \mathbb{P}; \mathbb{H}^{1, q}(O) ).
\end{equation} 
\eqref{Eq-initial-cond}} Assumption $(\mathcal{B})$, with $ q_0=2$, $ \delta_0\geq 0$ and $ p\geq 4$.
Assume that $ G$ satisfies Assumption $ (\mathcal{C})$ with $ C_R=c$ independent of $ R$, $ q=2$  and $ \delta =0$. Then 

\begin{itemize}

\del{\item { (\ref{Main-theorem-boubded-2}.1) [Maximal local weak solution.]} 
Equation \eqref{Main-stoch-eq} admits a maximal local solution
$ (u, \xi)$ in the sense of Definition \ref{def-maxi-variational solution}, with the Gelfand triple 
\begin{equation}\label{Gelfand-triple-Domain}
 V_2=V:= \mathbb{H}^{\frac\alpha2, 2} (O)\hookrightarrow \mathbb{L}^{2}(O)\hookrightarrow \mathbb{H}^{-\frac\alpha2, 2}(O)=:V^* 
\end{equation}
and $ V^*\hookrightarrow V_1:=\mathbb{H}^{-\delta', 2}(O), \;\; \delta'>1+\frac d2$, satisfying for a sequence of stopping times $\xi_m \nearrow \xi$\del{ $ m\in \mathbb{N}_0$, and for a sequence $ (c_m)_m$} and for all $m\geq 1$,\del{, there exists a constant $ c_m>0$, s.t.\del{ and such that}
 \begin{equation}\label{eq-set-solu-weak-L-2}
 u(\cdot, \omega) \in L^\infty(0, T; \mathbb{H}^{1, 2}(O))\cap L^2(0, T; \mathbb{H}^{1+\frac \alpha2, 2}(O))\cap
C([0, T]; \mathbb{H}^{-?, 2}(O))\;\;\; a.s.
 \end{equation}}
\del{\begin{equation}\label{eq-weak-l-2solu}
\mathbb{E}\sup_{[0, T\wedge \xi_m]}|u(t)|^p_{\mathbb{L}^2}+ \mathbb{E}\int_0^{T\wedge \xi_m}|u(t)|^2_{\mathbb{H}^{\frac\alpha2, 2}}dt\leq c_m<\infty.
\end{equation}}
\begin{equation}\label{eq-weak-H-d-alpha-solu}
\mathbb{E}\sup_{[0, T\wedge \xi_m]}|u(t)|^p_{\mathbb{H}^{\frac{d+2-\alpha}{4}, 2}}+ \mathbb{E}\int_0^{T\wedge \xi_m}|u(t)|^2_{\mathbb{H}^{\frac{d+2-\alpha}{4}+\frac\alpha2, 2}}dt\leq c_m<\infty.
\end{equation}

\vspace{0.25cm}}

\item { (\ref{Main-theorem-boubded-2}.1) [Global weak solution for the 2D-FSNSE on the torus.]}

for $ O=\mathbb{T}^2$, $ G$ satisfies, in addition, Assumption $(\mathcal{C})$ for $\delta$ replaced by $ 1$ and $ u_0$ satisfies  \eqref{eq-curl-u-0-torus} and Equation \eqref{Main-stoch-eq} admits a local weak solution in the sense of Definition \ref{def-maxi-variational solution}, with $ V_2=V:=\mathbb{H}^{\frac\alpha2, 2}(O)$, $V_1:=\mathbb{H}^{-\delta', 2}(O)$, $ \delta'>1+\frac d2$ and $ H=\mathbb{L}^{2}(O)$, then this local solution becomes global in the sense of Definition \ref{def-variational solution} and satifies \eqref{propty-of-2D-global-mild-sol}, according to the values of $ p, \; q,\; \alpha$ in the cases $(\ref{Main-theorem-strog-Torus}.3)$ and $ q$ satisfies in addition that $1+\frac{2d}{\alpha}\leq q$.

\item { (\ref{Main-theorem-boubded-2}.2) [Global weak solution for the dD-FSNSE.]}  
If one of the  maximal solutions $ (u, \xi)$  enjoys\del{, for $ 4<q<\infty$} either
\begin{equation}\label{eq-bale-kato-majda-con}
\mathbb{E}\int_0^{T\wedge \xi}|\nabla u(t)|_q^{\frac{1}{1-\frac{2d}{\alpha q}}}dt\leq c<\infty
\end{equation}
or 
\begin{equation}\label{eq-other-bale-kato-majda-con}
\mathbb{E}\int_0^{T\wedge \xi}|u(t)|_{\mathbb{H}^{\frac{d+2-\alpha}{4}, 2}}^{\frac{4\alpha}{3\alpha-d-2}}dt\leq c<\infty
\end{equation}
then  $ T =\xi$ and the process $ (u(t), t\in [0, T])$ is the  unique global weak solution of Equation \eqref{Main-stoch-eq}.

\del{in addition assumptions $ (\mathcal{C})$ respectively 
$ (\mathcal{B})$, with  $ p\geq q_0\geq q\geq  4$, $\delta =1$ and $ C_R=c$ independent of $ R$. Then the local solution 
$ (u, \xi)$ is the unique global solution of \eqref{Main-stoch-eq}, i.e. $ \xi=T$ and $ (u(t), t\in [0, T])$ is the weak
solution of \eqref{Main-stoch-eq} in the sense of Definition \ref{def-variational solution}, with $ V=\mathbb{H}^{\frac\alpha2, 2}(O)$\del{,
$ V^*=\mathbb{H}^{\frac\alpha2, 2}(O)$} and $ H=\mathbb{L}^{2}(O)$ satisfying \eqref{propty-of-2D-global-mild-sol}.}
\del{\begin{equation}\label{eq-set-solu-weak-global}
u(\cdot, \omega) \in L^\infty(0, T; \mathbb{H}^{1, 2}(\mathbb{T}^2))\cap L^2(0, T; \mathbb{H}^{1+\frac \alpha2, 2}(\mathbb{T}^2))\cap
C([0, T]; \mathbb{L}^{2}(\mathbb{T}^2))\;\;\; a.s.
 \end{equation}
and for all $ p'\leq p$, we have
\begin{eqnarray}\label{cond-solu-global-H1}
\mathbb{E}\Big(\sup_{[0, T]}|u(t)|^{q}_{\mathbb{H}^{1, q}}&+&
\int_0^T|u(t)|^{2}_{\mathbb{H}^{1+\frac\alpha2, 2}}dt\del{
+(\int_0^T|u(t)|^_{\mathbb{H}^{1, q}}dt)^\frac12 + \int_0^T|u(t)|^{\frac{\alpha}{\eta}}_{\mathbb{H}^{1+\eta, 2}}dt\Big)}<\infty,
\end{eqnarray}}

\del{\vspace{0.25cm}

\item { (\ref{Main-theorem-boubded-2}.3) [Global weak solution for the dD-FSNSE.]} Assume that one of the  maximal solutions $ (u, \xi)$  enjoys\del{, for $ 4<q<\infty$} either
\begin{equation}\label{eq-bale-kato-majda-con}
\mathbb{E}\int_0^{T\wedge \xi}|\nabla u(t)|_q^{\frac{1}{1-\frac{2d}{\alpha q}}}dt\leq c<\infty
\end{equation}
or 
\begin{equation}\label{eq-other-bale-kato-majda-con}
\mathbb{E}\int_0^{T\wedge \xi}|u(t)|_{\mathbb{H}^{\frac{d+2-\alpha}{4}, 2}}^{\frac{4\alpha}{3\alpha-d-2}}dt\leq c<\infty
\end{equation}
then  $ T =\xi$ and the process $ (u(t), t\in [0, T])$ constructed in $(3.8.1)$ is the  unique global weak solution of Equation \eqref{Main-stoch-eq} satisfying \eqref{eq-weak-l-2solu} up to $ T$\del{$ \xi=T$} $(c_m=c)$.} 

\del{Assume that the maximal solution $ (u, \xi)$ constructed in $(3.8.1)$ enjoys\del{, for $ 4<q<\infty$} the following  condition
\begin{equation}\label{eq-bale-kato-majda-con}
\mathbb{E}\int_0^{T\wedge \xi}|\nabla u(t)|_q^{\frac{1}{1-\frac{2d}{\alpha q}}}dt\leq c<\infty.
\end{equation}
Then $ T =\xi$ and $ (u(t), t\in [0, T])$ is the  unique global weak solution of Equation \eqref{Main-stoch-eq} satisfying \eqref{eq-weak-l-2solu} up to $ T$\del{$ \xi=T$ $(c_m=c)$}.} 
\del{where (NOT YET FINISHED)$ \beta \leq 1+\frac\alpha2-\frac d2+\frac d{q_1}$, $ 2\leq q<\infty$ and $ 0<\eta\leq \frac\alpha2$.
 Furthermore, the couple $  (u, \tau) $ satisfies 
\begin{equation}\label{est-local-u-tau-principle}
 \mathbb{E}\big (\sup_{[0, \tau]}|u(s)|^p_{\mathbb{L}^2}+ \int_0^{\tau}|u(s)|^p_{\mathbb{L}^2}|u(s)|^2_{\mathbb{H}^{\frac\alpha2, 2}}ds\big)<\infty.
\end{equation}}
\end{itemize}

\end{theorem}

\del{\begin{remark}\label{Rem-2}
\begin{itemize}
 \del{\item In \cite{Debbi-scalar-active}, the author proved, in a general framework, the existence and the uniquness of the global mild solution of the 2D-vorticity FSNSE on the torus without the extra assumption $(\mathcal{C}_b)$.
  \item From Theorem \ref{Main-theorem-boubded-2}, we remark that the fractional 2D-NSE exibits similarity in the difficulties of the weel-posedness as the 3D-NSE, see further discussion in \ref{sec-Domain}. In the prsent work, the global existence of the solution of the 2D-FSNSE is obtained thanks to the regularization effect of the vorticity.\del{ This latter has been studied in a general framework in \cite{Debbi-scalar-active}.}}
    \del{The results in $ (3.6.3)$ and in \cite{Debbi-scalar-active} confirm and extend to the fractional case the well known difference between the classical 2D & 3D NSEs concerningthe regularization effect of the vorticity, see for the calssical case \cite{Giga-vorticity-sing-NS-2011, Giga-al-Globalexistence-2001, Lemari-book-NS-probelems-02, Marchioro-Puvirenti-Vortex-84}. Let us for simplicity illustrate this idea via the detrministic case. The $3D-$vorticity equation is given by,  see e.g.\cite{Giga-al-Globalexistence-2001},
    \begin{equation} 
\end{equation}}     
\del{ In fcat, following the same  calculus as in Appendix \ref{sec-Passage Velocity-Vorticity}, we get the fa$3D-$vorticity equation . Following a  calculus similar to that in Appendix \ref{sec-Passage Velocity-Vorticity}, we get the $3D-$vorticity equation on the torus as in 
  As it it  well known in the calssical theory of Navier-Stokes that one of the main diffence between the 2D and the 3D NSE is related to the regularity of the vorticity}
\item  }
\begin{remark}\label{Rem-2}
\begin{itemize}
\del{\item The result in Theorem \ref{Main-theorem-mild-solution-d} generalizes the existence result in \cite[Theorem 3.2]{Giga-Solu-Lr-NS-85}.}
\item The conditions  \del{\eqref{cond-global-mild-solu},} \eqref{eq-bale-kato-majda-con} and \eqref{eq-other-bale-kato-majda-con} are of Beale-Kato-Majda type, see e.g. \cite{Beale-kato-Majda-Paper-84, Majda-Bertozzi-02}\del{p115}. \del{In particular, the condition \eqref{cond-global-mild-solu} is weaker than the condition in \cite[Theorem 5]{Giga-al-Globalexistence-2001} and \cite[ Definition 2.1]{Mikulevicius-H1-NS-solution-2009}. }
\item The condition \eqref{uniquness-con-martg}, we have assumed for the uniqueness of the martingale solution is of  Serrin's type on Sobolev spaces\del{ rather than $ L^q-$spaces}, see similar extension of the Serrin's condition in \cite[Theorem 5.2]{Farwig-Sohr-Serrin-cond-2012}.  In particular, our condition in this work is weaker in the time integrability than the condition in \cite{Farwig-Sohr-Serrin-cond-2012}.\del{and identic for the space reguarity  the Sobolev space $ D(A^\frac14)$, for $ d=2$ and silitas }\del{ In fact, due to the weak dissipation the estimation of the  nolinear term could be} See also similar condition in \cite[Theorem 2.8]{Kiselev-Nazarov-Fractal-Burgers-08}.  
\item Remark that thanks to \cite[Theorem 2.6]{Debbi-scalar-active}, Appendix \ref{sec-Passage Velocity-Vorticity}, the conditions in 
(\ref{Main-theorem-boubded-2}.1)\del{Assumption $(\mathcal{C})$, with $\delta$ replaced by $ 1$ and $ u_0$ satisfies  \eqref{eq-curl-u-0-torus},} we conclude that the condition \eqref{uniquness-con-martg} is satisfied in the case $ O=\mathbb{T}^2$. Therefore, one can get the existence and the uniqueness of the global solution under these conditions for the 2D-fractional stochastic Navier-Stokes equation on the torus. In Section \ref{sec-Torus}, the results are more stronger. The aim to develop (\ref{Main-theorem-boubded-2}.1) is to show that the conditions \eqref{eq-bale-kato-majda-con} and \eqref{eq-other-bale-kato-majda-con}  sound natural.
\item Similarly, \cite[Theorem 2.6]{Debbi-scalar-active}, Appendix \ref{sec-Passage Velocity-Vorticity} and the conditions in 
(\ref{Main-theorem-boubded-2}.1) ensure that the condition \eqref{uniquness-con-martg} is satisfied. Therefore a unique global strong-weak solution exists in sense of Definition \ref{def-variational solution}, with $ V_2= V=\mathbb{H}^{\frac\alpha2, 2}(O)$, $ H=\mathbb{L}^{2}(O)$, $ V_1=\mathbb{H}^{-\delta', 2}(O)$, with $\delta'> 1+\frac d2$ and satisfies \eqref{propty-of-2D-global-mild-sol} according to the cases in (\ref{Main-theorem-strog-Torus}.3).
\item The condition $ q\geq 1+\frac{2d}{\alpha}$ in $ (\ref{Main-theorem-boubded-2}.2)$ is not optimal.
\del{\item The uniqueness for  the local mild and weak solutions in the set 
 \end{itemize}}

\end{itemize}
\end{remark}


\section{Properties of the nonlinear term}\label{sec-nonlinear-prop}
\noindent Our aim in this section is to study the nonlinear operator $ B$  defined by \eqref{eq-B-projection}. 
Here $ O$ denotes either the torus
$ \mathbb{T}^d$ or a bounded domain from $ \mathbb{R}^d$ with smooth boundary as mentioned  above. 
We define the bilinear operator
$ B:(\mathcal{D}(O))^2 \rightarrow \mathbb{L}^2(O)$ and the tri-linear form
$ b:(\mathcal{D}(O))^3 \rightarrow \mathbb{R}$ by,
\begin{equation}\label{Eq-def-B-theta1-Theata2}
B(u, v): = \Pi ((u\cdot \nabla) v), \;\;\;\; \forall (u, v),
\in (\mathcal{D}(O))^2
\end{equation}
respectively,
\begin{equation}\label{Eq-3lin-form}
b(u, \theta, v):= \langle B(u, v), v\rangle,\;\;\; \forall (u, \theta, v) \in (\mathcal{D}(O))^3,
\end{equation}
where the brackets in RHS of \eqref{Eq-3lin-form} stand for the scalar product in $ \mathbb{L}^2(O)$, see e.g. \cite{Amann-solvability-NSE-2000, Farwig-Sohr-L-p-theory-2005},
\begin{eqnarray}
 \mathcal{D}(O):&=& \{u\in (C^\infty(O))^d, div u=0\,\,  \text{ and $ u$ has a compact support when } O\,
 \text{is a bounded domain}\}\nonumber\\
 &=& \left\{
 \begin{array}{lr}
 (C_0^\infty(O))^d\cap \mathbb{L}^q(O), \;\; \text{when}\;\; O \;\; \text{is bounded},\nonumber\\
 (C^\infty(O))^d\cap \mathbb{L}^q(O), \;\; \text{when}\;\; O =\mathbb{T}^d.\nonumber\\
 \end{array}
 \right.
\end{eqnarray}
\del{\mathcal{D}(O):= \{\text{the set of infinitely differential functions, with compact support in the case } O
 \text{is a bounded domain and such that }  div u=0\}.}
The bilinear operator $ B $ and the trilinear form  $ b$\del{  as defined above \eqref{Eq-def-B-theta1-Theata2} and  
\eqref{Eq-3lin-form}} have several extensions based on the $ \mathbb{}H^{\beta, q}-$norm, with $ \beta\geq 1$, 
 see e.g. \cite{Temam-NS-Functional-95} and \cite[p. 97]{Foias-book-2001} for Hilbert spaces and
\cite{Giga-Solu-Lr-NS-85}  for Banach spaces and for a general survey.\del{ of many other results.} Unfortunately, 
 due to the weakness of the fractional dissipation in our equation these extensions are
useless for  our case.
 Let us before dealing with the extensions we are interested in here, recall the following intrinsic properties
\begin{equation}\label{Eq-3lin-propsym}
b(u, \theta, v)= -b(u, v, \theta), \;\; \forall u, v, \theta  \in \mathbb{H}^{1, 2}(O).
\end{equation}
Hence
\begin{equation}\label{Eq-3lin-propnull}
b(u, v, v)= 0, \;\; \forall u, v \in \mathbb{H}^{1, 2}(O).
\end{equation}
In particular, for $ O= \mathbb{T}^2$, we have also, see e.g. either \cite[Lemma 3.1]{Temam-NS-Functional-95} or  \cite[Lemma VI.3.1]{Temam-Inf-dim-88}.
\begin{equation}\label{vanishes-bilinear-tous-H1}
 \langle B(u), u\rangle_{\mathbb{H}^{1, 2}} = \langle B(u), u\rangle_{H_2^{1, 2}} =0,\;\;\; \forall u \in D(A):= \mathbb{H}^{2, 2}(\mathbb{T}^2).
\end{equation}
Now, we cite some basic lemmas.
\begin{lem}\label{Lem-classic}
For all  $ 1\leq j\leq d$, $ \eta\geq 0$ and  $1<q<\infty$,  the operators $ A^{-\frac12}\Pi\partial_{j}$ extends uniquely to a bounded
linear operator from  $ H_d^{\eta, q}(O)$ to $ \mathbb{H}^{\eta, q}(O)$.\del{Furthermore $ \partial_{j}IA^{-\frac12}$
is a bounded operator from $ \mathbb{L}^q(\mathbb{T}^d)$ to $ L_d^q(\mathbb{T}^d)$, where $ I $ denotes the injection
$ \mathbb{L}^q(\mathbb{T}^d)\subset L_d^q(\mathbb{T}^d)$.}
\end{lem}
\begin{proof}
For $ \eta \geq 0 $ and  $ O= \mathbb{T}^d$, we 
use Marcinkiewicz's theory for the pseudodifferential  operator $ A^{-\frac12}\Pi\partial_{j}$. In fact,
the  symbol of this latter in  Fourier modes is given by the matrix $ i|k|^{-1}k_j(\delta_{m, n}-|k|^{-2}k_mk_n)_{mn}$.
See also \cite{Debbi-scalar-active} and also \cite{Kato-Ponce-86} for similar calculus for the  case $ O= \mathbb{R}^d$. 
The case $ O$ bounded and $ \eta =0 $  has been proved in 
\cite[Lemma 2.1]{Giga-Solu-Lr-NS-85}. We claim here that the method in \cite{Giga-Solu-Lr-NS-85} and also 
the proof bellow are also valid 
for $ O= \mathbb{T}^d$.
For $ \eta \geq 1$  and $ O$ is either a bounded domain of $ \mathbb{R}^d$ or  $ O= \mathbb{T}^d$,
thanks to the properties of 
Helmholtz projection, we prove for all $ \beta\geq 0$ and $1< q<\infty$,
that $ \Pi: H_d^{\beta, q}(O) \rightarrow \mathbb{H}^{\beta, q}(O)$ is well defined and bounded. 
Using this statement and arguing
 as in the proof of \cite[Lemma 2.1]{Giga-Solu-Lr-NS-85}, we get the result for $ \eta\geq 1$. The result for 
 $ 0< \eta <1$ is a consequence of the interpolation for $ \eta= 0$ and $ \eta= 1$. 
\end{proof}
\del{\begin{remark}
 Following the same ideas in Lemma \ref{Lem-classic} and in \cite[Lemma 2.1]{Giga-Solu-Lr-NS-85}, 
 we can prove the result for $ \eta\geq -1$. To be checked.
\end{remark}}

\del{\begin{lem}\label{Giga-Mikayawa-solutionLr-NS}
Let $1<r_1\leq r_2<\infty$ and $ \frac{d}{2}(\frac{1}{r_1}-\frac{1}{r_2})\leq \epsilon <\frac{d}{2r_1}$. Then, the operator
$ A_{r_1}^{-\epsilon}$  extends uniquely to a bounded operator from
$  \mathbb{L}^{r_1}(O)$ to $ \mathbb{L}^{r_2}(O)$.
\end{lem}

\begin{proof}
Thanks to Theorem \ref{theorem-domains-A-D-A-S}\del{{eq-embedded-domain-A-beta}}, it is sufficient to show that
$$  \mathbb{H}^{2\epsilon, r_1}(O) \hookrightarrow H_d^{2\epsilon, r_1}(O)\cap \mathbb{L}^{r_1}(O)  \hookrightarrow \mathbb{L}^{r_2}(O).$$
This last follows from the Sobolev  embedding; $ H_d^{2\epsilon, r_1}(O) \hookrightarrow L_d^{r_2}(O)$,
\del{$  \mathbb{H}^{2\epsilon, r_1}(O) \hookrightarrow \mathbb{L}^{r_2}(O)$,} which is
satisfied provided $ \frac{d}{2}(\frac{1}{r_1}-\frac{1}{r_2})\leq \epsilon <\frac{d}{2r_1}$,
see e.g. \cite[Theorem 7.63 p 221 and 7.66 p 222]{Adams-Hedberg-94} for bounded domain and
 see either \cite[Theorem 3.5.4.ps.168-169 and Corollary 3.5.5.p 170]{Schmeisser-Tribel-87-book} or \cite{Sickel-periodic spaces-85} for  the torus. See also \cite{Amann-solvability-NSE-2000} for solenoidal Sobolev spaces.
\end{proof}

\noindent We call the property in Lemma \ref{Giga-Mikayawa-solutionLr-NS}, the $ L^{r_1}\rightarrow  L^{r_2}$
smoothing property of  $ A_{r_1}^{-\epsilon}$.  As an application of Lemma \ref{Giga-Mikayawa-solutionLr-NS}, we have
\begin{coro}\label{coro-lem-Giga-epsilon}
 For all $ \epsilon $, such that $ \frac{d}{2q}\leq \epsilon < \frac{d}{q}$, the operator $ A^{-\epsilon}: \mathbb{L}^{\frac q2}(O)  \rightarrow
\mathbb{L}^{q}(O)$ is bounded.
\end{coro}}

\noindent The following Lemma has been proved in \cite{Giga-Solu-Lr-NS-85} for bounded domain. Our claim is that
the same proof is also valid for $ O= \mathbb{T}^d$,\del{as  the Stokes operator in the case $ O= \mathbb{T}^d$ coincides with the Laplacian, the same proof is also valid for this case,} see
similar calculus in \cite{Debbi-scalar-active} and for $ q=2 $, see e.g. \cite[p 13]{Temam-NS-Functional-95}.

\begin{lem}\label{Lemma-bound-op-RGradient-Via-gamma} \cite[Lemma 2.2]{Giga-Solu-Lr-NS-85}
Let $ 0\leq \delta < \frac12+\frac d2(1-\frac1q)$. Then
\begin{eqnarray}
|A^{-\delta}\Pi (u. \nabla) v|_{\mathbb{L}^q} \leq M|A^{\nu} u|_{\mathbb{L}^q}
|A^{\rho} v|_{\mathbb{L}^q}.
\end{eqnarray}
with some constant $ M:= M_{\delta, \rho, \nu, q}$,\del{ the operator $ (u, v) \mapsto u . \nabla v $ is
continuously extended to $B: D(A^\nu) \times D(A^\rho)
\rightarrow D(A^{-\delta})$} provided that $ \nu, \rho >0$, $\delta+ \rho > \frac12$,
$\delta+ \nu + \rho \geq \frac{d}{2q}+ \frac12$. 
\end{lem}

\noindent As a corollary  of Lemma \ref{Lemma-bound-op-RGradient-Via-gamma}, we cite the following results, which will be generalized later on.\del{These results will be for a general
framework, to prove Theorem \ref{theo-gelfand-gene}.}
\del{\begin{coro}\label{coro-lem-Giga}
\noindent Let either $ O\subset \mathbb{R}^2$ a bounded domain with smooth boundary or
\begin{itemize}
 \item For all  $( u, v) \in \mathbb{H}^{1-\frac\alpha2,2}(O^2)\times \mathbb{H}^{1+\frac\alpha2,2}(O^2)$,
there exists a constant $ c>0$, such that
\begin{eqnarray}\label{Eq-B-L-2-est}
|u\nabla v|_{\mathbb{L}^2} &\leq& c|u|_{H^{1-\frac\alpha2,2}}|v
|_{H^{1+\frac\alpha2,2}}.
\end{eqnarray}
For all  $( u, v) \in (\mathbb{H}^{1-\frac\alpha4,2}(O^2))^2$,
there exists a constant $ c>0$, such that
\begin{eqnarray}\label{Eq-B-H-alpha-2-est}
|u\nabla v|_{\mathbb{H}^{-\frac\alpha2, 2}} &\leq& c|u|_{\mathbb{H}^{1-\frac\alpha4, 2}}|u|_{\mathbb{H}^{1-\frac\alpha4, 2}}.
\end{eqnarray}
\end{itemize}
\end{coro}
\begin{proof}
We can use several methods such as,
\begin{itemize}
 \item Lemma \ref{Lem-classic}. In this proof, we  take $  \delta =0, \rho = \frac12+\frac\alpha4$ and $ \nu = \frac12-\frac\alpha4 $  to get
\eqref{Eq-B-L-2-est} and  we take $  \delta =\frac\alpha4, \rho = \nu = \frac12-\frac\alpha8$ to get \eqref{Eq-B-H-alpha-2-est}.
\item or  \cite[Theorem 4.6.1 p 190 \& Proposition Tr. 2.3.5 p 14]{R&S-96} and Lemma \eqref{Theo-pointwiseMulti-Bounded-Domain},
for $ O\subset \mathbb{R}^2$ being a bounded domain. For the case $O= \mathbb{T}^2 $, we use
\cite[Theorem IV.2.2 (ii)]{Sickel-Pontwise-Torus}, the monotonicity property  \cite[Remark 4, p 164]{Schmeisser-Tribel-87-book} and
\cite[Theorem 3.5.4.ps.168-169]{Schmeisser-Tribel-87-book}.
\item or \cite[Lemma 2.1 and in p 13]{Temam-NS-Functional-95} with $  m_3=0, m_2= \frac\alpha2, m_1= 1-\frac\alpha2$  to prove  \eqref{Eq-B-L-2-est}.
\end{itemize}
\end{proof}}

\begin{coro}\label{coro-lem-Giga}
\noindent Let either $ O= \mathbb{T}^d$ or $ O\subset \mathbb{R}^d$ be a bounded domain. Then
\begin{itemize}
 \item For $ \alpha\in (0,2) $, there exists a constant $ c:= c(\alpha, d)>0$ such that for all
$( u, v) \in \mathbb{H}^{\frac d2-\frac\alpha2,2}(O)\times \mathbb{H}^{1+\frac\alpha2,2}(O)$
\begin{eqnarray}\label{Eq-B-L-2-est}
|B(u, v)|_{\mathbb{L}^2} &\leq& c|u|_{\mathbb{H}^{\frac d2-\frac\alpha2,2}}|v
|_{\mathbb{H}^{1+\frac\alpha2,2}}.
\end{eqnarray}
\item  For $ \alpha\in (0,2] $ there exists a constant $ c:= c(\alpha, d)>0$ such that for all $( u, v) \in (\mathbb{H}^{\frac{2+d-\alpha}{4},2}(O))^2$,
\begin{eqnarray}\label{Eq-B-H-alpha-2-est}
|B(u, v)|_{\mathbb{H}^{-\frac\alpha2, 2}} &\leq& c|u|_{\mathbb{H}^{\frac{2+d-\alpha}{4},2}}
|v|_{\mathbb{H}^{\frac{2+d-\alpha}{4},2}}.
\end{eqnarray}
\end{itemize}
\end{coro}
\del{\begin{proof}
Using  Lemma \ref{Lemma-bound-op-RGradient-Via-gamma} with $  \delta =0, \rho = \frac12+\frac\alpha4$ and $ \nu = \frac d4-\frac\alpha4 $  to get
\eqref{Eq-B-L-2-est} and  with $  \delta =\frac\alpha4, \rho = \nu = \frac{2+d-\alpha}{8}$ to get 
\eqref{Eq-B-H-alpha-2-est},
see also other proof in \cite[Lemma 2.1 and p 13]{Temam-NS-Functional-95}.\del{
 with $  m_3=0, m_2= \frac\alpha2, m_1= 1-\frac\alpha2$  to prove  \eqref{Eq-B-L-2-est}.}
\end{proof}}
\noindent The following results generalize 
\cite[Lemma 2.2]{Giga-Solu-Lr-NS-85} for $ O\subset \mathbb{R}^d$ and
\cite[Lemma 1.4]{Kato-Ponce-86} for $ O=\mathbb{R}^d$.

\del{\begin{prop}\label{Prop-Main-B}
Let $ 2<q<\infty$ and $ \eta\geq 0$. The bilinear form $B :(\mathcal{D}(O))^2
\rightarrow \mathbb{L}^q(O)$ given by Formula
\eqref{eq-B-projection}, extends uniquely, (we keep the same notation) to
\begin{eqnarray}\label{Eq-def-Bilinear-form}
B:&(\mathbb{H}^{\eta, q}(O))^2& \rightarrow \mathbb{H}^{\eta-1-\frac dq, q}(O)\nonumber\\
&(u, v) & \mapsto B(u,v):= \Pi((u\cdot\nabla) v).
\end{eqnarray}
Moreover, there exists a constant $ c:= c_{\eta, d, q}>0$  such that for all 
$ (u,  v) \in (\mathbb{H}^{\eta, q}(O))^2$,
\begin{equation}\label{eq-B-estimatoion}
 | B(u, v)|_{\mathbb{H}^{\eta-1-\frac dq, q}}\leq c |u|_{\mathbb{H}^{\eta, q}} |v|_{\mathbb{H}^{\eta, q}}.
\end{equation}
\end{prop}

\begin{proof}
Let us first mention that  to justify that the bilinear form $B :(\mathcal{D}(O))^2
\rightarrow \mathbb{L}^q(O)$  is well defined, one can use H\"older and Gagliardo-Nirenberg inequalities.
Remark also that, thanks to the free divergence,  we can rewrite  
$ B(u, v)$, for smooth functions $ u$ and $ v$, e.g. $ (u, v)\in (\mathcal{D}(O))^2$, as
\begin{equation}\label{B-theta-divergence-1}
B(u, v) = \Pi(\sum_{j=1}^d \partial_j(u_j v))= \Pi(\partial_j(u_j v)).
\end{equation}
\noindent Using  Corollary \ref{coro-lem-Giga-epsilon} and Lemma \ref{Lem-classic}, we get for  $ \eta =0$, after applying  H\"older inequality,
\begin{eqnarray}\label{est-1-semi-group-div-delta2not2-B}
|B(u, v)|_{\mathbb{H}^{-1-\frac dq, q}} &=& |A^{-\frac12-\frac d{2q}} \Pi\partial_j(u_j v)|_{\mathbb{L}^{q}}
\leq C| A^{-\frac12} \Pi\partial_j(u_j v)|_{\mathbb{L}^{\frac q2}} \nonumber\\
&\leq & C  |u_j v|_{\mathbb{L}^{\frac q2}} \leq  C |u|_{\mathbb{L}^q}| v|_{\mathbb{L}^q}
\end{eqnarray}
 and for $ 0<\eta < \frac dq$, thanks to
\cite[Theorem IV.2.2 (iii) \& Theorem III. 11. (ii)]{Sickel-Pontwise-Torus}  and
\cite[Theorem 3.5.4.ps.168-169]{Schmeisser-Tribel-87-book}, for $ O=\mathbb{T}^d$ and thanks to
\cite[Theorem 1, p. 176 and Proposition Tr 6, 2.3.5, p 14]{R&S-96} and Theorem 
\ref{Theo-pointwiseMulti-Bounded-Domain} bellow,  for the bounded domain,
\begin{eqnarray}\label{est-1-semi-group-div-delta2not2-B}
|B(u, v)|_{\mathbb{H}^{\eta-1-\frac dq, q}} &=& |A^{-\frac d{2q}} A^{\frac\eta2-\frac12}\Pi(\partial_j(u_j v))|_{\mathbb{L}^{q}}
\leq C\del{| A^{\frac\eta2}A^{-\frac12} \Pi\partial_j(u_j v)|_{\mathbb{L}^{\frac q2}} \nonumber\\
&\leq & C  }|u_j v_i|_{H^{\eta, \frac q2}} \leq  C |u|_{\mathbb{H}^{\eta, q}}| v|_{\mathbb{H}^{\eta, q}}.\nonumber\\
\end{eqnarray}

\noindent For $ \eta = \frac dq$, we use  Lemma \ref{Lem-classic},
\cite[Theorem IV.2.2 (ii)]{Sickel-Pontwise-Torus}, \cite[Theorem 3.5.4.ps.168-169]{Schmeisser-Tribel-87-book} 
and the monotonicity property in
\cite[Remark 4.p.164]{Schmeisser-Tribel-87-book} for $ O=\mathbb{T}^d$ and \cite[Theorem 4.6.1, p. 190 and Proposition Tr 6, 2.3.5, p 14]{R&S-96} and Theorem \ref{Theo-pointwiseMulti-Bounded-Domain}, we infer that
\begin{eqnarray}\label{est-1-semi-group-div-delta2not2-B-+2}
|B(u, v)|_{\mathbb{H}^{\eta-1-\frac dq, q}} \leq
 c|u_jv_i|_{L^{q}}
\leq  C  |u|_{\mathbb{H}^{\frac{d}{2q}, q}}|v|_{_{\mathbb{H}^{\frac{d}{2q}, q}}}
\leq  C  |u|_{\mathbb{H}^{\eta, q}}|v|_{\mathbb{H}^{\eta, q}}.
\end{eqnarray}
\noindent For $ \eta > \frac dq $, we use Lemma \ref{Lem-classic} (here we do not use 
Corollary \ref{coro-lem-Giga-epsilon}), the fact that 
$ \mathbb{H}^{\eta, q}(O)$  is an algebra, see e.g.
\cite[Theorem IV.2.2 (ii)]{Sickel-Pontwise-Torus} and \cite[Theorem 3.5.4.ps.168-169]{Schmeisser-Tribel-87-book} for 
$ O=\mathbb{T}^d$ and see e.g. 
\cite[Theorem 1, p. 221 and Proposition Tr 6, 2.3.5, p. 14]{R&S-96} for 
$ O\subset \mathbb{T}^d$ bounded, then we conclude \eqref{eq-B-estimatoion}.  Finally, by  density of the
space $ \mathcal{D}(O)$ in $ \mathbb{H}^{\eta, q}(O)$\del{, see the definition in Section \ref{sec-formulation},} the result is proved.
\end{proof}

\noindent The proof of Proposition \ref{Prop-Main-B} does not cover the case $ q=2$.
This is due to the fact that the results in
Lemmas \ref{Lem-classic}-\ref{Giga-Mikayawa-solutionLr-NS}  and other related tools like 
\cite[Theorem 3.5.4.ps.168-169]{Schmeisser-Tribel-87-book}
do not cover the space $ \mathbb{L}^1(O)$.
We give bellow other estimates for $ q=2$.\del{ This result is used to prove Theorem \ref{Main-theorem-martingale-solution-d}.}}
\begin{prop}\label{prop-Ben-q=2}
Let $\epsilon>0$, then
the bilinear form $ B$ extends uniquely $
B:(\mathbb{L}^{2}(O))^2 \rightarrow \mathbb{H}^{-1-\epsilon-\frac d2, 2}(O)$
and there exists a constant $ c:= c_{\alpha, \epsilon, d}$ such that  for all $ (u, v) \in (\mathbb{L}^{2}(O))^2$,
\begin{equation}\label{eq-B-estimatoion-q=2}
|B(u, v)|_{H^{-1-\epsilon-\frac d2, 2}}\leq c |u|_{\mathbb{L}^{2}} |v|_{\mathbb{L}^{2}}.
\end{equation}
\end{prop}
\noindent  We omit the proof here,  as a more general one will be given
in the proof of Lemma \ref{lem-bounded-W-gamma-p}.

\del{\begin{prop}\label{Prop-Main-I}
Assume that  $ 2< q<\infty$, $ \eta\geq 0$ and  $1 + \frac dq <\alpha\leq 2$
and  let $ u, v \in L^\infty(0, T; D(A^\frac{\eta}{2}_q))$. Then
\begin{equation}
 \int_0^. e^{-A_\alpha (.-s)}B(u(s), v(s)) ds\in L^\infty(0, T; D(A^\frac{\beta}{2}_q)),
\end{equation}
for all $ \beta $  satisfies
\begin{equation}
  0\leq \beta <\eta+ \alpha-1-\frac dq.
\end{equation}
Moreover, there exist  $ c, \mu >0$ (depend on $ d, q, \alpha, \beta$), such that
 the following estimate holds
\begin{equation}\label{est-semigroup-B-main-2}
 | A^\frac{\beta}{2}\int_0^t e^{-A_\alpha (t-s)}B(u(s), v(s)) ds|_{\mathbb{L}^q}\leq c t^{\mu}
|u|_{L^\infty(0, t; D(A^{\frac{\eta}{2}}_q))}|v|_{L^\infty(0, t; D(A^{\frac{\eta}{2}}_q))}.
\end{equation}
\end{prop}

\begin{proof}
Using Proposition \ref{Prop-Main-B} and the semigroup property \eqref{eq-semigp-property}, we infer that
\begin{eqnarray}\label{est-1-semi-group-second-term-1}
|\int_0^t e^{-A_\alpha (t-s)}B(u(s)\!\!\!\!&,& \!\!\!\!v(s))ds|_{D(A^{\frac\beta2}_q)}\nonumber\\
&\leq& C\int_0^t |A^{\frac{\beta-\eta+1+\frac dq}2} e^{-A_\alpha (t-s)}|_{\mathcal{L}(\mathbb{L}^q)}
|B(u(s), v(s))|_{\mathbb{H}^{\eta-1-\frac dq, q}} ds \nonumber\\
&\leq &  C  \int_0^t (t-s)^{-\frac{ 1}{\alpha}(\beta-\eta+1+\frac dq)}|u(s)|_{D(A_q^{\frac\eta2})}|v(s)|_{D(A_q^{\frac\eta2})}ds
\nonumber\\
&\leq &  C T^{1-\frac{ d}{\alpha}(\frac{1+\beta-\eta}{d}+\frac1q)}|u|_{L^\infty(0, T; D(A_q^{\frac\eta2}))}
|v|_{L^\infty(0, T; D(A_q^{\frac\eta2}))}.
\end{eqnarray}
\end{proof}}

\begin{prop}\label{prop-est-B-eta-q=2}
Let $\eta \geq 0$ and
 \begin{equation}\label{critical-value-alpha}
\alpha(d, \eta) :=
\Bigg\{
\begin{array}{lr}
\max\{\frac{d+2-2\eta}{3},\; 2\eta+2-d\},\;\;\; \  if\; \;  \eta \in [0, \frac{d}{2})\cap 
(\frac d2-2, \frac{d}{2}), \\
1, \;\;\; if \,\,\, \eta\geq \frac d2.
\end{array}
\end{equation}
\del{\footnote{Recall $ \max^1_>\{a, b\} $ equals $ a$, if $ a>b$ and greater than $ b$ if  $ a\leq b$}}
Then for either $ \alpha \in [\alpha(d, \eta), 2)$ with $(\eta \in [0, \frac{d-1}{2})\cap 
(\frac d2-2, \frac{d-1}{2})) 
\cup [\frac d2, \infty)$ or
$ \alpha \in (\alpha(d, \eta), 2)$ with $\eta \in [\frac{d-1}{2},  \frac{d}{2})$,
 the bilinear operator $ B$ extends uniquely
$$
B:(\mathbb{H}^{\eta +\frac\alpha2, 2}(O))^2 \rightarrow \mathbb{H}^{\eta -\frac\alpha2, 2}(O)$$
and there exists a constant $ c:= c_{\alpha, \eta, d}$ such that  for all 
$ (u, v) \in (\mathbb{H}^{\eta -\frac\alpha2, 2}(O))^2$,
\begin{equation}\label{eq-B-estimatoion-q=2-eta}
|B(u, v)|_{\mathbb{H}^{\eta -\frac\alpha2, 2}}\leq c |u|_{\mathbb{H}^{\eta +\frac\alpha2, 2}} |v|_{\mathbb{H}^{\eta +\frac\alpha2, 2}}.
\end{equation}
\end{prop}

\begin{proof}
Thanks to  Lemma \ref{Lem-classic}, there exists a constant $ c>0$, such that
\begin{eqnarray}\label{eq-fact-b-u-v}
 |B(u, v)|_{\mathbb{H}^{\eta-\frac\alpha2, 2}} &\leq& c |u_jv|_{\mathbb{H}^{\eta+1-\frac{\alpha}{2}, 2}}.
\end{eqnarray}
First, let us suppose that $ \eta \geq \frac d2$. Then  $\mathbb{H}^{\eta+1-\frac\alpha2, 2}(O)$ is an algebra, 
therefore
\del{$ \eta+1-\frac{\alpha}{2} >\frac d2$, then}
\begin{eqnarray}\label{est-B-estimatoion-q=2-eta}
 |B(u, v)|_{\mathbb{H}^{\eta-\frac\alpha2, 2}} &\leq& c |u|_{\mathbb{H}^{\eta+1-\frac{\alpha}{2}, 2}} 
 |v|_{\mathbb{H}^{\eta+1-\frac{\alpha}{2}, 2}}.
\end{eqnarray}
Then  Estimate \eqref{eq-B-estimatoion-q=2-eta} follows from \eqref{est-B-estimatoion-q=2-eta} 
by using the Sobolev embedding
$ \mathbb{H}^{\eta+\frac{\alpha}{2}, 2}(O) \hookrightarrow \mathbb{H}^{\eta+1-\frac{\alpha}{2}, 2}(O)$. 
This last is guaranteed thanks to the condition
$ \alpha\geq 1= \alpha(d, \eta)$.\\
\noindent For $ 0\leq \eta < \frac d2$, we combine \eqref{eq-fact-b-u-v} and either
\cite[Theorem 4.6.1.1, Proposition Tr 6, 2.3.5]{R&S-96} (see also Appendix
\ref{Sobolev pointwise multiplication-Bounded-Domain}) for $ O\subset \mathbb{R}^d$ being a  bounded domain or
\cite[Theorem IV.2.2]{Sickel-Pontwise-Torus} and \cite[Theorem 3.5.4 \& Remark 4 p 164]{Sickel-periodic spaces-85}
for $ O= \mathbb{T}^d$, then we get,
\begin{eqnarray}\label{eq-fact-b-u-v-last-est}
|B(u, v)|_{\mathbb{H}^{\eta-\frac\alpha2, 2}} &\leq& c |u|_{\mathbb{H}^{\frac{d+2+2\eta-\alpha}{4}, 2}}
| v|_{\mathbb{H}^{\frac{d+2+2\eta-\alpha}{4}, 2}}
\end{eqnarray}
provided that $ 2\eta+2-d<\alpha$. Moreover, under the condition $ \frac{d+2-2\eta}{3}\leq \alpha$, Estimate
\eqref{eq-B-estimatoion-q=2-eta} follows from \eqref{eq-fact-b-u-v-last-est} by using the
Sobolev embedding $ \mathbb{H}^{\eta+\frac{\alpha}{2}, 2}(O) \hookrightarrow 
\mathbb{H}^{\frac{d+2+2\eta-\alpha}{4}, 2}(O)$. This achieves the proof of \eqref{eq-B-estimatoion-q=2-eta}. The intervals in the definition of $\alpha(d, \eta) $ in Formula  \eqref{critical-value-alpha}
emerge thanks to the condition $ \eta >\frac d2-2$ which guaranties that $ \frac{d+2-2\eta}{3} <2$ and to the equivalence 
$ 2+2\eta-d<\frac{d+2-2\eta}{3}\Leftrightarrow \eta< \frac{d-1}{2}$.\del{, we deduce the corresponding 
values of $ \alpha$ and $ \eta$ as cited in the proposition.}
\del{ we infer the existence of a constant $ c>0$, such that
\begin{eqnarray}
 |B(u, v)|_{\mathbb{H}^{\eta-\frac\alpha2, 2}} &\leq& c |u|_{\mathbb{H}^{\frac{d+2+2\eta-\alpha}{4}, 2}}
| v|_{\mathbb{H}^{\frac{d+2+2\eta-\alpha}{4}, 2}} \leq c |u|_{\mathbb{H}^{\eta -\frac\alpha2, 2}} |v|_{\mathbb{H}^{\eta -\frac\alpha2, 2}}.
\end{eqnarray}}
\end{proof}

\noindent The investigation of the Gelfand triple corresponding to
the fractional Navier-Stokes equation for which $ B$ can be extended to a  bounded operator,
is\del{, as mentioned in the introduction,} one of the delicate questions of the theory of fractional nonlinear equations.
\del{here, we apply Proposition \ref{prop-est-B-eta-q=2} to discuss this question.}
To characterize this feature, let us first recall the following
classical Gelfand triple
 \begin{equation}
 V_c= D((A^\frac12))=\mathbb{H}^{1, 2}(O)\hookrightarrow  H:=\mathbb{L}^2(O) \tilde{=}  H^*\hookrightarrow  V_c^*,
 \del{:= \mathbb{H}^{-1, 2}(O)}
 \end{equation}
 where $ V_c^*$ is the dual of  $ V_c$.
The operators  $ A: V_c \rightarrow  V_c^*$ and  $ B: D(B):=V_c\times  H \rightarrow  V_c^*$ are bounded.
\del{ then both the operator $ A$ and the restriction of
$ B $ on $ V_c$ (denoted also by $ B$ ) are well defined and bounded from $V_c$ to $ V_c^*$.}
For the fractional case, we have for  $ \alpha >0$,
$A_\alpha: V:=D(A_\alpha^{\frac12})=\mathbb{H}^{\frac\alpha2, 2}(O) \rightarrow 
V^*=(\mathbb{H}^{\frac\alpha2, 2}(O))^*$  is bounded.  However, for $ \alpha<2$,
the space $ V\times H $ is larger than $ D(B)$. In particular,
$ \mathbb{H}^{1}(O) \subsetneq V$. Consequently, we need to extend uniquely the operator $ B $ to a
bounded operator from  $ V$ to $ V^*$. This extension is not possible for all values of 
$ \alpha \in (0, 2)$.\del{ The result is given in the following theorem}

\begin{theorem}\label{theo-gelfand-gene}
Let  $ \alpha \in [\alpha(d, \eta), 2]$, with $\eta \in ([0, \frac{d-1}{2})\cap (\frac{d}{2}, \frac{d-1}{2})) \cup [\frac d2, \infty)$\del{$\eta \in [\max\{0, \frac{d}{2}-2\},  
\frac{d-1}{2}) \cup [\frac d2, \infty)$} or
$ \alpha \in (\alpha(d, \eta), 2]$ with $\eta \in [\frac{d-1}{2},  \frac{d}{2})$
where $ \alpha(d, \eta)$ is defined by \eqref{critical-value-alpha}.
We introduce the following  Gelfant triple
\begin{equation}\label{gelfant-triple-eta}
 V_\eta:= \mathbb{H}^{\eta+\frac\alpha2, 2}(O)\hookrightarrow \mathbb{H}^{\eta, 2}(O)\hookrightarrow
 \mathbb{H}^{\eta-\frac\alpha2, 2}(O).
\end{equation}
Then
\begin{equation}
 B: V_\eta\times V_\eta:= (D(A^{\frac\eta2+\frac\alpha4}))^2= 
 (\mathbb{H}^{\eta+\frac\alpha2, 2}(O))^2\rightarrow V_\eta^* =
\mathbb{H}^{\eta-\frac\alpha2, 2}(O).
\end{equation}
 is bounded. Moreover, there exists a constant $ c:=c_{d, \alpha, \eta}>0$, such that
\begin{equation}\label{Eq-B-u-v-V-dual}
|B(u, v)|_{\mathbb{H}^{\eta-\frac\alpha2, 2}} \leq  c |u|_{\mathbb{H}^{\eta+\frac\alpha2, 2}}
|v|_{\mathbb{H}^{\eta+\frac\alpha2, 2}}.
\end{equation}
In particular, we have the following useful cases,
\begin{itemize}
\item  $ \eta =0$, $ d\in \{2, 3, 4\}$ and $ \alpha\in [\frac{d+2}{3}, 2]$.
\item  $ \eta =1$ and either  $ d=2$  and  $ \alpha \in [1, 2] $ or  $d=3$ and  $ \alpha \in (1, 2]$ or
$d\in \{4, 5, 6\}$ and  $ \alpha \in [\frac{d}{3}, 2]$.
\end{itemize}
\end{theorem}
\begin{proof}
 The proof is a straightforward application of Proposition \ref{prop-est-B-eta-q=2} and 
the classical case for $ \alpha =2$.
\end{proof}

\del{\noindent We mention here the following useful  cases,
\begin{coro}
\begin{itemize}
 \item For $ \eta =0$, $ d\in \{2, 3\}$, $ \alpha_0(d) := \alpha_0(d, 0) =\frac{d+2}{3}\leq \alpha <2 $, we consider the Gelfant triple
\begin{equation}\label{gelfant-triple-eta=0}
 \mathbb{H}^{\frac\alpha2, 2}(O)\hookrightarrow \mathbb{L}^{ 2}(O)\hookrightarrow\mathbb{H}^{-\frac\alpha2, 2}(O).
\end{equation}
Then the operator $ B $ is extended uniquely to a bounded operator
\begin{equation}
 B:  (V_0:= \mathbb{H}^{\frac\alpha2, 2}(O))^2\rightarrow V_0^* =
\mathbb{H}^{-\frac\alpha2, 2}(O).
\end{equation}
\item For $ \eta =1$, we consider the Gelfant triple
\begin{equation}\label{gelfant-triple-eta=1}
 \mathbb{H}^{1+\frac\alpha2, 2}(O)\hookrightarrow \mathbb{H}^{1, 2}(O)\hookrightarrow\mathbb{H}^{1-\frac\alpha2, 2}(O).
\end{equation}
Then for either  $ d=2$  and  $ \alpha \in [1, 2) $ or  $d=3$ and  $ \alpha \in (1, 2) $ or
$d\in \{4, 5\}$ and  $ \alpha \in (\frac{d}{3}, 2) $
the operator $ B $ is extended uniquely to a bounded operator.
\end{itemize}
\end{coro}}

\del{
In the following lemma, we give further estimations
for the nonlinear term and for the trilinear form
\begin{lem}\label{lem-B-different-q}
\begin{itemize}
 \item $(i)$ For $ 0\leq \eta <\frac{d}{2}$ and $ \alpha \in (2+2\eta-d, 2]$ there exists a constant $ c>0$, such that
\begin{equation}\label{main-est-b-B-eta-1}
 |B(u, v)|_{\mathbb{H}^{\eta-\frac\alpha2, 2}}\leq c |u|_{\mathbb{H}^{\frac{d+2+2\eta-\alpha}{4}, 2}}
| v|_{\mathbb{H}^{\frac{d+2+2\eta-\alpha}{4}, 2}}.
\end{equation}
 \item $(i')$ Let $ d\in \{2, 3\}$, $ 1<q<\frac{d}{d-1}$  and $ 0<\alpha\leq 2$. Then, there exists a constant $ c>0$, such that
\begin{equation}
 |B(u, v)|_{\mathbb{H}^{-\frac\alpha2, q}}\leq c |u|_{\mathbb{H}^{\frac\alpha2, 2}} | v|_{\mathbb{H}^{1-\frac\alpha2, 2}},
\end{equation}
provided $ \frac dq-d+\frac\alpha2>0$.
\item $(ii)$ For $ \frac{d-2}{3}\leq \alpha<2$,
\begin{eqnarray}\label{est-nonlinear-general}
| B(u, v)|_{{\mathbb{H}^{1-\frac{\alpha}{2}, 2}} } &\leq & | u|_{{\mathbb{H}^{\frac{d+2-\alpha}4, 2}}}| v|_{{\mathbb{H}^{\frac{d+2-\alpha}4, 2}}}.
\end{eqnarray}
\del{\item $(iii)$ For $ \frac{d-2}{3}\leq \alpha<2$ and $ \frac d2 <s<\alpha$,
\begin{eqnarray}\label{est-nonlinear-general}
| \langle B(u), v\rangle|_{{\mathbb{H}^{-\frac{\alpha}{2}, 2}} } &\leq & | u|_{{\mathbb{H}^{1-\frac\alpha2, 2}}}| v|_{{\mathbb{H}^{\frac{d+2-\alpha}4, 2}}}.
\end{eqnarray}}
\item $(iii)$ For $ \frac{d-2}{3}\leq \alpha<2$ and $ u\in \mathbb{H}^{1+\frac\alpha2, 2}$,
\begin{eqnarray}\label{est-tree-linear-H1}
|\langle B(u), u\rangle_{\mathbb{H}^{1, 2}}|&\leq & c\big(| u|_{{\mathbb{H}^{1+\frac\alpha2, 2}}}^2 + | u|^2_{{\mathbb{H}^{1, 2}}}\big).
\end{eqnarray}
\item $(iv)$  Let $ 1\leq \alpha \leq 2$. For all $ (u, w)\in \mathbb{H}^{1, 2}(O)\times \mathbb{H}^{\frac\alpha2, 2}(O)$,
\begin{eqnarray}\label{3linear-H1-H-1}
|\langle B(w), u\rangle_{\mathbb{L}^{2}}|&\leq & c| u|_{\mathbb{H}^{1, 2}}|w|^{\frac2\alpha}_{\mathbb{H}^{\frac{\alpha}{2}, 2}}|w|^{2\frac{1-\alpha}\alpha}_{{\mathbb{L}^{2}}}.
\end{eqnarray}
\end{itemize}
\end{lem}

\begin{proof}
\begin{itemize}
 \item $(i)$ Using Lemma \ref{Lem-classic} and
\cite[Theorem 4.6.1.1, Proposition Tr 6, 2.3.5]{R&S-96} for bounded domain (see also Appendix
\ref{Sobolev pointwise multiplication-Bounded-Domain}). For the Torus, we use again Lemma \ref{Lem-classic} and
\cite[Theorem IV.2.2]{Sickel-Pontwise-Torus} and \cite[Theorem 3.5.4 \& Remark 4 p 164]{Sickel-periodic spaces-85}, we infer the existence of
 a constant $ c>0$, such that
\begin{eqnarray}
 |B(u, v)|_{\mathbb{H}^{\eta-\frac\alpha2, 2}} &\leq& c |u_jv|_{\mathbb{H}^{1+\eta-\frac{\alpha}{2}, 2}}
\leq c |u|_{\mathbb{H}^{\frac{d+2+2\eta-\alpha}{4}, 2}}
| v|_{\mathbb{H}^{\frac{d+2+2\eta-\alpha}{4}, 2}}.
\end{eqnarray}

 \item $(i')$ Let $ 0\leq \eta <\frac{d}{2}$ and $ \alpha \in (2+2\eta-d, 2$, Using Lemmas \ref{Lem-classic}, \cite[Theorem 1.4.4.4 ]{R&S-96} and the condition $ \frac dq-d+\frac\alpha2>0$, we infer that
\begin{eqnarray}\label{Eq-B-H-alpha-2-est-d}
|B(u, v)|_{\mathbb{H}^{-\frac\alpha2, q}}&\leq& c |u_iv_j|_{\mathbb{H}^{1-\frac\alpha2, q}} \leq c |u_i|_{\mathbb{H}^{\frac\alpha2, 2}}
 | v_j|_{\mathbb{H}^{1-\frac\alpha2, 2}}.
\end{eqnarray}
\item $(ii)$ Thanks to the condition $ \frac{d-2}{3}\leq \alpha<2$, Estimation in \eqref{est-nonlinear-general} follows from
\cite[Theorem 4.6.1.1 (iii)p 190 \& Proposition Tr 6, 2.3.5 (vii)]{R&S-96}, \cite[p.177]{Adams-Hedberg-94} and Appendix
\ref{Sobolev pointwise multiplication-Bounded-Domain} in the case of
$ O\subset \mathbb{R}^d$ is bounded and \cite[Theorem iv.2.2 (ii)]{Sickel-Pontwise-Torus},
\cite[Theorem 3.5.4 (v) p. 108 \& Remark 4 p. 104]{Schmeisser-Tribel-87-book} for the case $ O= \mathbb{T}^d$.
\item $(iii)$ Using Lemma \ref{Lem-classic},\del{ \cite[Theorem 4.6.1.1]{R&S-96} (see also Appendix \ref{Sobolev pointwise multiplication-Bounded-Domain} )}
the fact that $ \mathbb{H}^{2-\frac\alpha2, 2}$ is a multiplicative algebra
\begin{eqnarray}\label{est-tree-linear-H1-first}
|\langle B(u), u\rangle_{\mathbb{H}^{1, 2}}|&\leq &| u|_{{\mathbb{H}^{1+\frac\alpha2, 2}}}^2| B(u)|_{{\mathbb{H}^{1-\frac\alpha2, 2}}}
\leq  c| u|_{{\mathbb{H}^{1+\frac\alpha2, 2}}}\sum_{j=1}^d| u^ju|_{{\mathbb{H}^{2-\frac\alpha2, 2}}}
\leq c| u|_{{\mathbb{H}^{1+\frac\alpha2, 2}}}|u|_{{\mathbb{H}^{2-\frac\alpha2, 2}}}^2,
\end{eqnarray}
where $ \max\{\frac d2, 2-\frac\alpha2\}<s<\alpha$. Recall that $ s$ exists thanks to the conditions >>>>>>>>>>>>>>>>>
\item $(iv)$  Let $ 1\leq \alpha \leq 2$ and  $ (u, w) \in (\mathcal{D}(O))^2$. Recall that  $ (u, w) \in (\mathcal{D}(O))^2$ is dense in
$\mathbb{H}^{1, 2}(O)\times \mathbb{H}^{\frac\alpha2, 2}(O)$. Using Lemma \ref{Lem-classic},
 H\"older inequality and  Gagliardo-Nirenberg inequality, we infer
\begin{eqnarray}\label{3linear-H1-H-1-est-1}
|\langle B(w), u\rangle_{\mathbb{L}^{2}}|&\leq &
| u|_{{\mathbb{H}^{1, 2}}}|B(w)|_{{\mathbb{H}^{-1, 2}}}\leq
c| u|_{{\mathbb{H}^{1, 2}}}|w_jw|_{{\mathbb{L}^{2}}}\leq
c| u|_{{\mathbb{H}^{1, 2}}}|w|^2_{{\mathbb{L}^{4}}}\leq
c| u|_{\mathbb{H}^{1, 2}}|w|^{\frac2\alpha}_{\mathbb{H}^{\frac{\alpha}{2}, 2}}|w|^{2\frac{\alpha-1}\alpha}_{{\mathbb{L}^{2}}}.\nonumber\\
\end{eqnarray}
\end{itemize}
\end{proof}

\del{Scours:
Using Lemma \ref{Lem-classic} and
\cite[Theorem 4.6.1.1, Proposition Tr 6, 2.3.5]{R&S-96} for bounded domain (see also Appendix
\ref{Sobolev pointwise multiplication-Bounded-Domain}). For the Torus, we use again Lemma \ref{Lem-classic} and
\cite[Theorem IV.2.2]{Sickel-Pontwise-Torus} and \cite[Theorem 3.5.4 \& Remark 4 p 164]{Sickel-periodic spaces-85}}

\subsection{Applications: the Gelfand triple corresponding to the FSNSE}
The notion of Gelfand triple is  one of the main tools in the study of nonlinear partial differential equations,
in particular in the study of deterministic and stochastic
Navier-Stokes  equation and in the application of the  monotonicity method see e.g \cite{Temam-NS-Main-79}.
To characterize the specificity of the  Gelfand triple corresponding to
the fractional Navier-Stokes  equation, let us first recall the following
classical Gelfand triple
 \begin{equation}
 V_c= D((A^\frac12))\hookrightarrow  H:=\mathbb{L}^2(O) \tilde{=}  H^*\hookrightarrow  V_c^*:= \mathbb{H}^{-1, 2}(O).
 \end{equation}
The operators  $ A: V_c \rightarrow  V_c^*$ and  $ B: V_c\times  V_c \rightarrow  V_c^*$ are bounded.
\del{ then both the operator $ A$ and the restriction of
$ B $ on $ V_c$ (denoted also by $ B$ ) are well defined and bounded from $V_c$ to $ V_c^*$.}
For the fractional case, we have for  $ \alpha >0$,
$A_\alpha: V=D(A_\alpha^{\frac12}) \rightarrow V^*=H^{-\frac\alpha2, 2}(O)$  is bounded.  However for $ \alpha<2$,
the space $ V $ is larger than the domain of definition of $ B$. In particular,
$ \mathbb{H}^{1}(O) \subsetneq V$. Consequently, we need to extend uniquely the operator $ B $ to a
bounded operator from  $ V$ to $ V^*$. This  is the content of the following proposition

\begin{theorem}\label{theo-gelfand-gene}
Let $ d\in \{2, 3\}$, $ 0\leq \eta\frac{d-4}{2}$ and\del{ such that $ \frac{d+2-6\eta}{3}< 2$. We define}
\del{\begin{equation}
\alpha_G(d, \eta) :=
\Bigg\{
\begin{array}{lr}
\frac{d+2-6\eta}{3},\;\;\; \  if\; \;  d+2-6\eta \geq 0, \\
0, \;\;\; otherwise.
\end{array}
\end{equation}}
\begin{equation}\label{alpha-critical}
\alpha_G(d, \eta) := 1+ \frac{d-1+2\eta}{3}.
\end{equation}
Assume $\alpha \in [\alpha_G(d, \eta), 2]$.
Then the operator $ B $ is extended uniquely to a bounded operator
(we keep the same notation)
\begin{equation}
 B: V_\eta:= D(A^{\frac\eta2+\frac\alpha4})= \mathbb{H}^{\eta+\frac\alpha2, 2}(O)\rightarrow V_\eta^* = D(A^{\frac\eta2-\frac\alpha4})=
\mathbb{H}^{\eta-\frac\alpha2, 2}(O).
\end{equation}
\end{theorem}

\begin{proof}
\del{Using Lemma \ref{Lemma-bound-op-RGradient-Via-gamma} with $ \delta := \frac\eta2 +\frac{\alpha}{4}$,
$ \rho=\nu = \frac{d+2--2\eta-\alpha}{8}$ and $ p= 2$,  we conclude
the existence of a constant $ c>0$, such that
\begin{equation}\label{Eq-B-u-v-V-dual}
|B(u, v)|_{V_\eta^*} \leq  c|u |_{D(A^{\frac{d+2-2\eta- \alpha}4})} |v|_{D(A^{\frac{d+2 -2\eta-\alpha}4})}.
\end{equation}
Thanks to the condition $ \alpha \geq \alpha_0(d)$, we get
\begin{equation}\label{Eq-B-u-v-V-dual}
|B(u, v)|_{V_\eta^*} \leq  c |u|_{V_\eta}|v|_{V_\eta}.
\end{equation}}

\noindent Using Estimation \eqref{main-est-b-B-eta-1} and the Sobolev embedding ( remark that ) we conclude
the existence of a constant $ c>0$, such that
\begin{equation}\label{Eq-B-u-v-V-dual}
|B(u, v)|_{V_\eta^*}\leq  c|u |_{D(A^{\frac{d+2-2\eta- \alpha}4})} |v|_{D(A^{\frac{d+2 -2\eta-\alpha}4})}.
\end{equation}
Thanks to the condition $ \alpha \geq \alpha_0(d)$, we get
\begin{equation}\label{Eq-B-u-v-V-dual}
|B(u, v)|_{V_\eta^*} \leq  c |u|_{V_\eta}|v|_{V_\eta}.
\end{equation}
\end{proof}
In particular, we have
\begin{coro}
\begin{itemize}
 \item For $ d\in \{2, 3\}$, $ \alpha_G(d) := \alpha_G(d, 0) =\frac{d+2}{3}\leq \alpha <2 $,
the operator $ B $ is extended uniquely to a bounded operator
\begin{equation}
 B:  V_\eta:= D(A^{\frac\alpha4})= \mathbb{H}^{\frac\alpha2, 2}(O)\rightarrow V_\eta^* = D(A^{-\frac\alpha4})=
\mathbb{H}^{-\frac\alpha2, 2}(O).
\end{equation}
\item For $ d\in \{2, 3, \cdots, 9\}$, $ \alpha_G(d, 1) := \min\{0, \frac{d-4}{3}\}\leq \alpha <2 $, the operator $ B $
is extended uniquely to a bounded operator
\begin{equation}
 B:  V_\eta:= D(A^{\frac12+\frac\alpha4})= \mathbb{H}^{1+\frac\alpha2, 2}(O)\rightarrow V_\eta^* = D(A^{-\frac12-\frac\alpha4})=
\mathbb{H}^{-1-\frac\alpha2, 2}(O).
\end{equation}
\end{itemize}
\end{coro}}

\noindent Further estimations for the bilinear operator $ B$ and the trilinear form $ b$ are summarized in the following lemma
\begin{lem}\label{lem-further-estimation}
\begin{itemize}

\item $(i)$ Assume $0\leq  \eta <\frac d2$, and $ \alpha \in (2\eta+2-d, 2]$. Then for all  
$ (u, v)\in (\mathbb{H}^{\frac{d+2+2\eta-\alpha}{4}, 2}(O))^2$, we have 
 \begin{eqnarray}\label{B-u-v-h-alpha-2-d}
|B(u, v)|_{\mathbb{H}^{\eta-\frac\alpha2, 2}} &\leq& c |u|_{\mathbb{H}^{\frac{d+2+2\eta-\alpha}{4}, 2}}
| v|_{\mathbb{H}^{\frac{d+2+2\eta-\alpha}{4}, 2}}.
\end{eqnarray}
\item $(ii)$ For  $ \eta \geq \frac d2$ and   $u, v \in \mathbb{H}^{\eta+1-\frac\alpha2, 2}(O)$, we have
\begin{eqnarray}\label{est-B-estimatoion-q=2-eta-to-use}
 |B(u, v)|_{\mathbb{H}^{\eta-\frac\alpha2, 2}} &\leq& c |u|_{\mathbb{H}^{\eta+1-\frac{\alpha}{2}, 2}} 
 |v|_{\mathbb{H}^{\eta+1-\frac{\alpha}{2}, 2}}.
\end{eqnarray}
\item $(iii)$ Assume $ d \in \{2, 3, 4\}$ and  $ \frac d2\leq \alpha\leq 2$. For all 
$ (u, w)\in \mathbb{H}^{1, 2}(O)\times \mathbb{H}^{\frac\alpha2, 2}(O)$,
\begin{eqnarray}\label{3linear-H1-H-1}
|\langle B(w), u\rangle_{\mathbb{L}^{2}}|&\leq & c| u|_{\mathbb{H}^{1, 2}}
|w|^{\frac d{\alpha}}_{\mathbb{H}^{\frac{\alpha}{2}, 2}}
|w|^{\frac{2\alpha-d}\alpha}_{{\mathbb{L}^{2}}}.
\end{eqnarray}
\item $(iv)$   Assume $ d \in \{2,\cdots, 5\}$ and  $ \frac d3\leq \alpha<d$. 
For all $ (u, w)\in \mathbb{H}^{1+\frac\alpha2, 2}(O)\times \mathbb{H}^{\frac\alpha2, 2}(O)$,
\begin{eqnarray}\label{3linear-H1+alpha2-H-1}
|\langle B(w), u\rangle_{\mathbb{L}^{2}}|&\leq & c| u|_{\mathbb{H}^{1+\frac\alpha2, 2}}|w|^{\frac{d-\alpha}{\alpha}}_{\mathbb{H}^{\frac{\alpha}{2}, 2}}
|w|^{\frac{3\alpha-d}{\alpha}}_{{\mathbb{L}^{2}}}.
\del{|\langle B(w), u\rangle_{\mathbb{L}^{2}}|&\leq & c| u|_{\mathbb{H}^{1+\frac\alpha2, 2}}|w|^{\frac2\alpha-1}_{\mathbb{H}^{\frac{\alpha}{2}, 2}}
|w|^{\frac{3\alpha-2}\alpha}_{{\mathbb{L}^{2}}}.}
\end{eqnarray}
\item $(v)$  For all $ (u, v) \in \mathbb{H}^{\frac d{4q}, q}(O)$ with $ q>2$,
\begin{equation}\label{eq-est-H-1-d-q}
 |B(u, v)|_{\mathbb{H}^{-1, q}}\leq c |u|_{\mathbb{H}^{\frac d{2q}, q}}|v|_{\mathbb{H}^{\frac d{2q}, q}}. 
\end{equation}

\item $(vi)$  The following estiamte is a classical result. For all $ u\in \mathbb{H}^{1, 2}(O)$,
\begin{equation}\label{classical-B-H-1}
 |B(u)|_{\mathbb{H}^{-1, 2}}\leq c |u|_{\mathbb{H}^{1, 2}}|u|_{\mathbb{L}^{2}}. 
\end{equation}
\end{itemize}
\end{lem}

\begin{proof}

\begin{itemize}
\item $ (i)$-$(ii)$ Estimates \eqref{B-u-v-h-alpha-2-d} and \eqref{est-B-estimatoion-q=2-eta-to-use} are copies  of the estimates \eqref{eq-fact-b-u-v-last-est} respectively \eqref{est-B-estimatoion-q=2-eta} proved above without restriction conditions. We only 
emphasize them here.\del{ as they will be needed later, in particular, for $ \eta=0$ and $ \eta=1$.}
\item $(iii)$  Let  $ d \in \{2, 3, 4\}$,  $ \frac d2\leq \alpha\leq 2$ and  $ (u, w) \in (\mathcal{D}(O))^2$. Recall that  $ (u, w) \in (\mathcal{D}(O))^2$ is dense in
$\mathbb{H}^{1, 2}(O)\times \mathbb{H}^{\frac\alpha2, 2}(O)$. Using the group property\del{ of 
the powers of Stokes operator} $ (A^\beta)_{\beta\in\mathbb{R}}$,
H\"older inequality, Lemma \ref{Lem-classic},
 again H\"older inequality and  Gagliardo-Nirenberg inequality in the case $ \alpha>\frac d2$ and Sobolev embedding 
 in the case $ \alpha=\frac d2$, see e.g. \cite[Theorem 7.63 \& 7.66]{Adams-Hedberg-94} for bounded domain and 
 \cite[Theorem 3.5.4. \& Theorem 3.5.5.]{Schmeisser-Tribel-87-book} for the torus, we infer that
\begin{eqnarray}\label{3linear-H1-H-1-proof}
|\langle B(w), u\rangle_{\mathbb{L}^{2}}|&\leq &
| u|_{{\mathbb{H}^{1, 2}}}|B(w)|_{{\mathbb{H}^{-1, 2}}}\leq
c| u|_{{\mathbb{H}^{1, 2}}}|w_jw|_{{\mathbb{L}^{2}}}
\leq c| u|_{{\mathbb{H}^{1, 2}}}|w|^2_{{\mathbb{L}^{4}}} \nonumber\\
&\leq & c| u|_{\mathbb{H}^{1, 2}}
|w|^{\frac d{\alpha}}_{\mathbb{H}^{\frac{\alpha}{2}, 2}}
|w|^{\frac{2\alpha-d}\alpha}_{{\mathbb{L}^{2}}}.
\end{eqnarray}
\item $(iv)$  Let  $ d \in \{2, \cdots, 5\}$, $ \frac d3\leq \alpha<d$ and  $ (u, w) \in (\mathcal{D}(O))^2$. 
Using the group property of $ (A^\beta)_{\beta\in\mathbb{R}}$, H\"older inequality, Lemma \ref{Lem-classic},
 \cite[Theorem 4.6.1]{R&S-96} for the bounded domain and \cite[Theorem iv.2. ii]{Sickel-Pontwise-Torus} and 
\cite[Remark 4 p 164]{Schmeisser-Tribel-87-book} for $ O=\mathbb{T}^d$ and by interpolation, we infer that
\begin{eqnarray}\label{3linear-H1+alpha-H-1-proof}
|\langle B(w), u\rangle_{\mathbb{L}^{2}}|&\leq &
| u|_{{\mathbb{H}^{1+\frac\alpha2, 2}}}|B(w)|_{{\mathbb{H}^{-1-\frac\alpha2, 2}}}\leq
c| u|_{{\mathbb{H}^{1+\frac\alpha2, 2}}}|w_jw|_{{\mathbb{H}^{-\frac\alpha2, 2}}}
\leq c| u|_{{\mathbb{H}^{1+\frac\alpha2, 2}}}|w|^2_{{\mathbb{H}^{\frac{d-\alpha}4, 2}}} \nonumber\\
&\leq& 
c| u|_{\mathbb{H}^{1+\frac\alpha2, 2}}|w|^{\frac{d-\alpha}{\alpha}}_{\mathbb{H}^{\frac{\alpha}{2}, 2}}
|w|^{\frac{3\alpha-d}{\alpha}}_{{\mathbb{L}^{2}}}.
\end{eqnarray}
\item $(v)$ We use  Lemma \ref{Lem-classic},
\cite[Theorem IV.2.2 (ii)]{Sickel-Pontwise-Torus}, \cite[Theorem 3.5.4.ps.168-169]{Schmeisser-Tribel-87-book} 
and the monotonicity property in
\cite[Remark 4.p.164]{Schmeisser-Tribel-87-book} for $ O=\mathbb{T}^d$ and \cite[Theorem 4.6.1, p. 190 and Proposition Tr 6, 2.3.5, p 14]{R&S-96} and Theorem \ref{Theo-pointwiseMulti-Bounded-Domain}, we infer that
\begin{eqnarray}\label{est-1-semi-group-div-delta2not2-B-+2}
|B(u, v)|_{\mathbb{H}^{-1, q}} \leq
 c|u_jv_i|_{L^{q}}
\leq  C  |u|_{\mathbb{H}^{\frac{d}{2q}, q}}|v|_{_{\mathbb{H}^{\frac{d}{2q}, q}}}.
\end{eqnarray}

The proof of \eqref{eq-est-H-1-d-q}, follows from the first two esimates in 
\eqref{est-1-semi-group-div-delta2not2-B-+2}, with $ \eta=\frac dq$. 

\item $(vi)$ We use Lemma \ref{Lemma-bound-op-RGradient-Via-gamma}, H\"older inequality and  than Gaglairdo-Nirenberg inequality, we get 
 \begin{eqnarray}
 |B(u)|_{\mathbb{H}^{-1, 2}}&\leq& c |u_ju|_{\mathbb{L}^{2}} \leq c |u|^2_{\mathbb{L}^{4}}\leq 
|u|_{\mathbb{H}^{1, 2}}|u|_{\mathbb{L}^{2}}.
\end{eqnarray}
\end{itemize}

\end{proof}


\del{\section{2D-Fractional stochastic Navier-Stokes equation on the Torus with smooth data}\label{sec-Torus}}

\del{\section{Local mild solutions for the multi-dimensional FSNSEs.}\label{sec-1-approx-local-solution}

In this Section, we prove\del{ the first part of Theorem \ref{Main-theorem-mild-solution-d}} ($ 3.6.1$). The scheme to prove
 the existence of local mild solutions for nonlinear deterministic or stochastic partial differential equations is now standard, see e.g. \cite{DaPrato-Zbc-92,  Neerven-Evolution-Eq-08}. The key idea is to apply a fixed point theorem or equivalently  Picard iterations. In \cite{Kunze-Neereven-Cont-parm-reaction-diffusion-2012},  the authors  constructed  an approximative scheme\del{ to prove the existence of a local solution} for an abstract semilinear stochastic evolution equation on a Banach space. The scheme is based on the assumption that the linear, the nonlinear and the diffusion terms of this equation\del{, $ A, B, G$,} can be approximated in appropriate way.  In particular, they applied this scheme on some\del{ to prove the global existence for the} reaction-diffusion equations.\del{ and other related models.} In the deterministic case, one of the elegant schemes for NSE is based on the approximation of the integral representation of the equation, more precisely of the mild solution, see e.g. \cite{ Farwig-Sohr-L-p-theory-2005, Lemari-book-NS-probelems-02}. It is resumed in finding a space $\mathcal{E}_T$, called the admissible path space on which  the 
bilinear operator $ B_*$ defined by
 \begin{eqnarray}\label{B-*}
B_*\!\!\!\!\!\!\!\!\!\!&:&\!\!\!\!\!\!\!  \mathcal{E}_T \times  \mathcal{E}_T \rightarrow  \mathcal{E}_T \nonumber\\
&(u, v)& \mapsto \int_0^{\cdot}e^{-(\cdot-s)A_\alpha}B(u(s), v(s))ds = \int_0^{\cdot}e^{-(\cdot-s)A_\alpha}\Pi((u(s)\nabla) v(s))ds
 \end{eqnarray}
is locally bounded. The associated space of admissible initial conditions $E_T$ is then defined by
$ u_0 \in \mathcal{D}'(O)$ such that $(e^{-tA_\alpha}u_0, t\in[0, T])\in \mathcal{E}_T$.
The space $E_T$ is called adapted value space.\del{To the best knowledge of the author, such scheme does not exist for the stochastic case.} Our aim here is to extend this scheme for the stochastic case, in particular for the stochastic Navier-Stokes equation. We shall use the same notations, but we shall delete 
the terminologies "path space"  and "adapted", as these latter are used, in the stochastic framework, for other probabilistic notions. 
We construct the admissible space $\mathcal{E}_T$
 such as the operators $ B_*$ and $ G_*$ are locally  bounded, where $ B_*$ is given by \eqref{B-*} and $ G_*$ is defined by 
 \begin{eqnarray}\label{G-*}
G_*\!\!\!\!\!\!&:&\!\!\!\mathcal{E}_T \rightarrow  \mathcal{E}_T \nonumber\\
&u& \!\!\!\!\mapsto \int_0^{\cdot}e^{-(\cdot-s)A_\alpha}G(u(s)W(ds).
 \end{eqnarray}
\noindent For the fixed parameters $ 2< p <\infty$, $ 2< q<\infty$, $ \delta\geq 0$ and $ T>0$, we introduce the 
admissible space
\begin{eqnarray}\label{E-cal-T}
\mathcal{E}_T =\{\!\!\!&{}&\!\!\! (u(t), t\in [0, T]), \;\; D(A_q^{\frac {\delta}2})-\text{adapted strongly measurable stochastic
processes},  s.t.\nonumber\\
  &{}&|u |_{\mathcal{E}_{T}}:= \left(\mathbb{E}\sup_{[0, T]}|u(t)|^p_{D(A_q^{\frac{\delta}2})}\right)^\frac 1p<\infty\}
 \end{eqnarray}
and the admissible initial value space 
\begin{equation}\label{E-T}
 E_T:=\{\text{random variables}\; u_0 \text{ with values in} \; \mathcal{D}'(O), \; s.t.\;\; \del{(e^{-tA_\alpha}u_0)_{t\in[0, T]}\in \mathcal{E}_T e^{-\cdot A_\alpha}u_0 } (e^{-tA_\alpha}u_0, t\in[0, T])\in \mathcal{E}_T\}.
\end{equation}
\noindent Now, it is easy, using Lemma \ref{lem-est-z-t} and Propositions \ref{Prop-Main-I}, 
to prove the existence of the local mild solution. The space regularity of the local mild solution is  a consequence of the construction of the solution as an element of the space \eqref{E-cal-T} and the time regularity follows thanks to Lemma \ref{lem-est-z-t}, Proposition \ref{Prop-Main-I}, Assumption $\mathcal{B}$; \eqref{Eq-initial-cond} and the strong continuity
of the semigroup $ (e^{-tA^{\frac{\alpha}{2}}})_{t\geq 0}$ with $ X_1:= H^{-\delta'', q}(O)$ and $ \delta'' \geq \alpha+1+\frac{d}{q}-\delta$. For the convenience of the reader, we give details of this proof in Appendix \ref{appendix-local-solution}, see also \cite{Debbi-scalar-active}.
\begin{remark}\label{Rem-initial-data-mild-solu}
It is obvious that, in the frame work presented above, Assumption $(\mathcal{B})$ is not optimal and may be replaced by the condition $ u_0\in E_T$, where $ E_T$ is given by \eqref{E-T}.
\end{remark}}

\section{2D-FSNSEs on the Torus with smooth data.}\label{sec-Torus}
\del{\subsection{Introduction}}  In this section, we assume\del{ study the fractional stochastic Navier-Stokes equations on the 
2D-Torus,} $ O = \mathbb{T}^2$ and prove Theorem \ref{Main-theorem-strog-Torus}. As mentioned above, we have $ (A^S)^{\frac\alpha2}= (-\Delta)^\frac\alpha2$, $0<\alpha\leq 2$. Furthermore, $ (A^S)^{\frac\alpha2}$ can  also be
defined by \eqref{construction-of-fract-bounded} and  \eqref{basis} with the explicit orthonormal
basis of eigenvectors $ (e_k(\cdot):= \frac{k^\perp}{|k|}e^{ik\cdot})_{k\in \mathbb{Z}_0^d}$ and 
$(\langle v, e_k\rangle :=\hat{v}_{k})_{k\in \mathbb{Z}_0^d}$ being the sequence of Fourier coefficients, 
see e.g. \cite[p316]{Temam-Inf-dim-88}.\del{This last equality could also be deduced easily by using an orthonormal
basis of eigenvectors of $ A$, e.g. $ (e_k(\cdot):= \frac{k^\perp}{|k|}e^{ik\cdot})_{k\in \mathbb{Z}_0^d}$
and formulas \eqref{construction-of-fract-bounded} and  \eqref{basis}, see for similar remark \cite[p316]{Temam-Inf-dim-88}.
In this case  $(\langle v, e_k\rangle :=\hat{v}_{k})_{k\in \mathbb{Z}^d}$
is the sequence of Fourier coefficients). Therefore, the fractional version of the stochastic Navier-Stokes equation and
the stochastic version of the fractional Navier-Stokes equation are equivalent.}


For $ \alpha \in [1, 2]$, we fix the densely,
continuous embedding Gelfand triple \eqref{Gelfand-triple-Torus} and 
we use the following Faedo-Galerkin approximation.
Let us fix $ n\geq 1$ and introduce the projection $ P_n$, $ n\geq 1$ on the finite space 
$ H_n \subset \mathbb{L}^2(\mathbb{T}^2)$ generated by $\{e_k, k\in \mathbb{Z}_0^d,\;s.t.\; |k|\leq n\}$.
\del{We define,
\begin{equation}\label{eq-def-B-n-G-n}
 B_n(u):= \mathcal{X}_n(|u|_{\mathbb{H}^1})B(u), \;\;\; \text{and}\;\;\;\;  G_n(u):= \mathcal{X}_n(|u|_{\mathbb{H}^1})G(u).
\end{equation}}
The Faedo-Galerkin approximation scheme is defined for the process $(u_n(t), t\in [0, T])\in H_n $ by
\begin{equation}\label{FSBE-Galerkin-approxi}
\Bigg\{
\begin{array}{lr}
du_n(t)= (-A_\alpha u_n(t) + P_nB(u_n(t))dt + P_nG(u_n(t))\,dW_n(t), \; 0< t\leq T,\\
u(0) = P_nu_0 = u_{0n},
\end{array}
\end{equation}
where $ W_n(t):= P_nW(t)= \sum_{|j|\leq n} Q^\frac12e_j\beta_j(t)$ (for example 
$ (\sum_{|j|\leq n} q^\frac12_je_j\beta_j(t))$). Let us mention here that using a similar proof as in 
 \cite{Millet-Chueshov-Hydranamycs-2NS-10, Flandoli-3DNS-Dapratodebussche-06}, we can prove that 
 $ W_n(t)$ converges to $ W(t)$  in  the space
$ L^2(\Omega, \mathbb{H}^{1, 2}(\mathbb{T}^2))$, provided  $ A^\frac12Q^\frac12$ is a  Hilbert-Schmidt in
$ \mathbb{L}^2(\mathbb{T}^2)$.
Since the finite dimensionnal space stochastic differential equation  \eqref{FSBE-Galerkin-approxi} has locally  Lipschitz and  linear growth coefficients, then  Equation \eqref{FSBE-Galerkin-approxi} admits a unique strong solution
$(u_n(t), t\in [0, T])\in L^2(\Omega; C([0, T]; H_n) )$, see e.g. \cite{Millet-Chueshov-Hydranamycs-2NS-10, Krylov-simple-proof-90, Roeckner-Zhang-tamedNS-12} and the reference therien. Furthermore, we have the following result,
\begin{lem}\label{lem-unif-bound-theta-n-H-1}
Let $ \alpha\in (0, 2)$ and $ u_0\in L^{p}(\Omega, \mathbb{H}^{1, 2}(\mathbb{T}^d))$  with $ p\geq 4$. Then
the solutions $ (u_n(t), t\in [0, T])$ of equations \eqref{FSBE-Galerkin-approxi},  $ n\in \mathbb{N}_0$,
satisfy the following estimates 
 \begin{eqnarray}\label{Eq-Ito-n-weak-estimation-1-Torus}
\sup_{n}\mathbb{E}\Big(\sup_{[0, T]}|u_n(t)|^{p}_{\mathbb{H}^{1, 2}}&+&
\int_0^T|u_n(t)|^{p-2}_{\mathbb{H}^{1, 2}}\Big(|u_n(t)|^2_{\mathbb{H}^{1+\frac\alpha2, 2}} +
|u_n(t)|_{\mathbb{H}^{\beta, q_1}}^2 \Big)dt \nonumber\\
&+& \int_0^T|u_n(t)|^4_{\mathbb{H}^{1, 2}}dt+ \int_0^T|u_n(t)|^{\frac{\alpha}{\eta}}_{\mathbb{H}^{1+\eta, 2}}dt\Big)<\infty,
\end{eqnarray}
where $ \beta \leq 1+\frac\alpha2-\frac d2+\frac{d}{q_1}$, $ 2\leq q_1<\infty$ and $ \frac\alpha p<\eta\leq \frac\alpha2$.
\begin{eqnarray}\label{Eq-B-n-weak-estimation-1-Torus}
\sup_{n}\left(\mathbb{E}\int_0^T (|P_nB(u_n(t))|^{2}_{\mathbb{H}^{1-\frac\alpha2, 2}} +|A^\frac\alpha2 u_n(t)|^{2}_{\mathbb{H}^{1-\frac\alpha2, 2}} )dt\right)<\infty.
\end{eqnarray}
\end{lem}

\begin{proof}
First, we prove the following estimate
\begin{eqnarray}\label{Eq-Ito-n-weak-estimation-torus}
\sup_{n}\mathbb{E}\left(\sup_{[0, T]}|u_n(t)|^{2}_{\mathbb{H}^{1, 2}}
+\int_0^T|u_n(t)|^2_{\mathbb{H}^{1+\frac\alpha2, 2}} dt \right)&\leq&C<\infty.
\end{eqnarray}
\del{Thanks to the definitions of $B_n$ and $ G_n$ in \eqref{eq-def-B-n-G-n}
and Assumption $ (\mathcal{A})$, we  can see that $ u_n$ satisfies the assumptions of
 \cite[Lemma 5.1]{Krylov-L-p-2010}, then} By application of Ito's formula,\del{see also  \cite[Theorem 2.1]{Krylov-L-p-2010},}  we have
\begin{eqnarray}\label{Eq-first-norm-l-2n-torus}
|u_n(t)|^{2}_{\mathbb{H}^{1, 2}}&=&|u_{n0}|^{2}_{\mathbb{H}^{1, 2}}- 2 \int_0^t \langle
u_n(s), A^\frac\alpha2u_n(s)-P_nB(u_n(s))\rangle_{\mathbb{H}^{1, 2}} ds \nonumber\\
&+& 2 \int_0^t \langle u_n(s),
P_nG(u_n(s))dW_n(s)\rangle_{\mathbb{H}^{1, 2}} +\int_0^t\sum_{|j|\leq n}|P_nG(u_n(s))Q^\frac12 e_j|_{\mathbb{H}^{1, 2}}^2 ds.
\end{eqnarray}
\del{Using the commutator estimate, see e.g. \cite{Cordoba2-Pointwise-commautator, WUJ-06}, we have

\vspace{-0.35cm}

\begin{eqnarray}\label{Eq-commutator-estimate-torus}
\langle u_n(s),A_\alpha u_n(s)\rangle_{\mathbb{H}^{1, 2}} = \langle u_n(s),A_\alpha u_n(s)\rangle_{{H}_d^1} =\int_{\mathbb{T}^2}
u_n(s)(-\Delta)^{1+\frac\alpha2}u_n(s)dx\nonumber\\ &\geq& \int_{\mathbb{T}^2}
|(-\Delta)^{\frac12+\frac\alpha4}u_n(s)|^2dx\geq c |u_n(s)|^2_{\mathbb{H}^{1+\frac\alpha2, 2}}.\nonumber\\
\end{eqnarray}}
Using the semigroup property of $ (A^\beta)_{\beta\geq0}$ and the definition of the Sobolev spaces in Section \ref{sec-formulation}, we get

\vspace{-0.35cm}

\begin{eqnarray}\label{Eq-commutator-estimate-torus}
\langle u_n(s), A_\alpha u_n(s)\rangle_{\mathbb{H}^{1, 2}} =\del{ \langle u_n(s), A_\alpha u_n(s)\rangle_{{H}_d^1} = \int_{\mathbb{T}^2}
u_n(s)(-\Delta)^{1+\frac\alpha2}u_n(s)dx=} |u_n(s)|^2_{\mathbb{H}^{1+\frac\alpha2, 2}}.
\end{eqnarray}
The term $ \langle u_n(s),  P_nB(u_n(s))\rangle_{\mathbb{H}^{1, 2}}$ 
\del{($ \langle u_n(s),  P_nB(u_n(s))\rangle_{H_n^{1, 2}}$)} in the RHS of \eqref{Eq-first-norm-l-2n-torus} vanishes thanks to
\eqref{vanishes-bilinear-tous-H1}.\del{ pages $ \langle u_n(s),  P_nB_n(u_n(s))\rangle_{\mathbb{H}^1}=0$
$ \langle u,  B(u, u)\rangle_{\mathbb{H}^1}=0$.} To estimate the stochastic term in \eqref{Eq-first-norm-l-2n-torus}, 
we  use the stochastic isometry, Minkowski and H\"older inequalities, the contraction property of $ P_n$ and
Assumption $ (\mathcal{C})$ (\eqref{Eq-Cond-Linear-Q-G}, with $ q=2$ and $ \delta =1$),   we get
\begin{eqnarray}\label{Eq-second-norm-l-2n-start-Torus}
\mathbb{E}\sup_{[0, T]}|\int_0^t \langle u_n(s)\!\!\!\!\!\!&,&\!\!\!\!
P_nG(u_n(s))dW_n(s)\rangle_{\mathbb{H}^{1, 2}} |\nonumber \\
&\leq & c\mathbb{E}\left(\int_0^T\sum_{k\leq n}
\big(\int_{\mathbb{T}^d}|A^\frac12u_n(s)|| A^\frac12P_nG(u_n(s))Q^\frac12e_k|dx\big)^2 ds\right)^\frac12\nonumber\\
&\leq & c\mathbb{E}\left(\int_0^T|u_n(s)|^2_{\mathbb{H}^{1, 2}}|| G(u_n(s))Q^\frac12||_{HS(\mathbb{H}^{1, 2})}^2ds\right)^\frac12\nonumber\\
&\leq & c\mathbb{E}\left(\int_0^T
\big(|u_n(s)|^2_{\mathbb{H}^{1, 2}}+ |u_n(s)|^4_{\mathbb{H}^{1, 2}})ds\right)^\frac12.
\end{eqnarray}

\vspace{-0.35cm}

\noindent Than, we use Young and H\"older inequalities ($\epsilon <1$), we infer that
\begin{eqnarray}\label{Eq-third-norm-l-2n-torus}
\mathbb{E}\sup_{[0, T]}|\int_0^t \langle u_n(s)\!\!\!\!\!\!&,&\!\!\!\!
P_nG(u_n(s))dW_n(s)\rangle_{\mathbb{H}^{1, 2}} |\nonumber \\
&\leq & c\mathbb{E}\left(\sup_{[0,
T]}|u_n(s)|_{\mathbb{H}^{1, 2}}\left[\int_0^T (1+
|u_n(s)|^{2}_{\mathbb{H}^{1, 2}})
ds\right]^\frac12\right)\nonumber\\
&\leq& \epsilon\mathbb{E}\sup_{[0, T]}|u_n(s)|^{2}_{\mathbb{H}^{1, 2}}+
C\mathbb{E}\int_0^T (1+ |u_n(s)|^{2}_{\mathbb{H}^{1, 2}})
ds\nonumber\\
&\leq& \epsilon\mathbb{E}\sup_{[0, T]}|u_n(s)|^{2}_{\mathbb{H}^{1, 2}}+
C\mathbb{E}\int_0^T \sup_{\tau\in[0, s]}|u_n(\tau)|^{2}_{\mathbb{H}^{1, 2}}
ds +C.
\end{eqnarray}
\noindent For the last term in the RHS of  \eqref{Eq-first-norm-l-2n-torus}, we use Assumption $
(\mathcal{C})$ (\eqref{Eq-Cond-Linear-Q-G}, with $ q=2$ and $ \delta =1$), we infer the existence of a 
positive constant $ c>0$ such that,
\begin{eqnarray}\label{Eq-last-Term-Last-estimate-torus}
|\int_0^t\int_{\mathbb{T}^2} \sum_{|j|\leq n}|A^\frac12 P_nG(u_n(s))Q^\frac12 e_j|^2 dx ds|
&\leq& \int_0^t\|G(u_n(s))Q^\frac12\|^2_{HS(\mathbb{H}^{1, 2})}ds\nonumber\\
&\leq & c\int_0^t(1+\sup_{\tau\in[0, s]}|u_n(\tau)|_{\mathbb{H}^{1, 2}}^2)ds.
\end{eqnarray}
Now, we replace \eqref{Eq-commutator-estimate-torus}, \eqref{Eq-third-norm-l-2n-torus} and
\eqref{Eq-last-Term-Last-estimate-torus} in \eqref{Eq-first-norm-l-2n-torus}, we get
\begin{eqnarray}\label{Eq-second-norm-l-2n-torus}
\mathbb{E}\left[\sup_{[0, T]}|u_n(t)|^{2}_{\mathbb{H}^{1, 2}}+   \int_0^T|u_n(s)|^2_{\mathbb{H}^{1+\frac\alpha2, 2}}ds \right]&\leq &C\left(1+
\mathbb{E}|u_{n0}|^{2}_{\mathbb{H}^{1, 2}} + 
C\int_0^T\mathbb{E}\sup_{\tau\in[0, s]}|u_n(\tau)|_{\mathbb{H}^{1, 2}}^{2}ds\right).\nonumber\\
\end{eqnarray}
By application of Gronwall's lemma for the function  $ \mathbb{E}\sup_{[0, T]}|u_n(t)|^{2}_{\mathbb{H}^{1, 2}}$, 
we get the estimation of the first term in the LHS of \eqref{Eq-Ito-n-weak-estimation-torus}, 
(recall that $\mathbb{E}|u_{n0}|^{2}_{\mathbb{H}^{1, 2}}\leq \mathbb{E}|u_{0}|^{2}_{\mathbb{H}^{1, 2}}$). 
The second term in \eqref{Eq-Ito-n-weak-estimation-torus} is
then deduced from  \eqref{Eq-second-norm-l-2n-torus} and the uniform boundedness of 
 $ \mathbb{E}\sup_{[0, T]}|u_n(t)|^{2}_{\mathbb{H}^{1, 2}}$. 
Now,\del{ we assume without lose of generality that  $ p=p'$ (
for $ p'<p$ the result is easily obtained from Estimate  \eqref{eq-estim-energy-Gene-2} bellow, 
by using H\"older inequality).} we prove
\begin{equation}\label{eq-estim-energy-Gene-2}
\mathbb{E}\sup_{[0, T]}|u_n(t)|^p_{\mathbb{H}^{1, 2}} +
\mathbb{E}\int_0^{T}|u_n(t)|^{p-2}_{\mathbb{H}^{1, 2}}|u_n(s)|_{\mathbb{H}^{1+\frac\alpha2, 2}}^2 ds \leq c(1+ \mathbb{E}|u_0|_{\mathbb{H}^{1, 2}}^p).
\end{equation}
By application of Ito's formula to the process
$ (|u_n(t)|^2_{\mathbb{H}^{1, 2}}, t\in[0, T])$ given by \eqref{Eq-first-norm-l-2n-torus}
we get, see for similar calculus e.g 
\cite{Millet-Chueshov-Hydranamycs-2NS-10, Flandoli-Gatarek-95, Sundar-Sri-large-deviation-NS-06},
\begin{eqnarray}\label{stoch-eq-product-Ito-Form-Lp-norm-SDE-torus}
| u_n(t)|^p_{\mathbb{H}^{1, 2}} &+& p
\int_0^t| u_n(s)|^{p-2}_{\mathbb{H}^{1, 2}}|u_n(s)|_{\mathbb{H}^{1+\frac\alpha2, 2}}^2 ds\nonumber\\
&\leq& |u_0|^p_{\mathbb{H}^{1, 2}}
+ \frac p2\int_0^t | u_n(s)|^{p-2}_{\mathbb{H}^{1, 2}}||P_n G(u_n(s))Q^\frac12||_{HS(\mathbb{H}^{1, 2})}^2 ds \nonumber\\
&+& \frac p2 (\frac p2-1)\int_0^t| u_n(s)|^{p-4}_{\mathbb{H}^{1, 2}}|Q^\frac12G^*(u_n(s))u_n(s)|_{\mathbb{H}^{1, 2}}^2ds\nonumber\\
&+&\frac p2 \int_0^t| u_n(s)|^{p-2}_{\mathbb{H}^{1, 2}}\langle u_n(s), P_nG(u_n(s))dW(s)\rangle_{\mathbb{H}^{1, 2}}.
\end{eqnarray}
We argue as above and use Assumption $ (\mathcal{C})$ ( \eqref{Eq-Cond-Linear-Q-G}, with $ q=2$ and $ \delta=1$). 
In particular, for the third term in the RHS of
\eqref{stoch-eq-product-Ito-Form-Lp-norm-SDE-torus}, we follow a similar calculus as in
\eqref{Eq-second-norm-l-2n-start-Torus} and \eqref{Eq-third-norm-l-2n-torus},
we infer that
\begin{eqnarray}\label{Eq-last-part-L-2power-p-torus}
\mathbb{E}\sup_{[0,T]}&{}&\left|\int_0^t | u_n(s)|^{p-2}_{\mathbb{H}^{1, 2}}\langle u_n(s), P_nG(u_n(s))u_n(s),
dW(s)\rangle_{\mathbb{H}^{1, 2}} \right|\nonumber\\
&\leq & C\mathbb{E}\left(\sup_{[0,T]}|u_n(s)|^\frac p2_{\mathbb{H}^{1, 2}}\left[  \int_0^{T} |u_n(s)|^{p-4}_{
\mathbb{H}^{1, 2}} (1+ |u_n(s)|^2_{\mathbb{H}^{1, 2}})ds\right]^\frac12\right)\nonumber\\
&\leq & C_1\mathbb{E}\sup_{[0,T]}|u_n(s)|^p_{
\mathbb{H}^{1, 2}}+ C_2\mathbb{E}\int_0^{T} |u_n(s)|^\frac p2_{
\mathbb{H}^{1, 2}}(1+ |u_n(s)|^2_{\mathbb{H}^{1, 2}})ds\nonumber\\
&\leq & C\mathbb{E}\sup_{[0,T]}|u_n(s)|^p_{
\mathbb{H}^{1, 2}}+ C\mathbb{E}\int_0^{T} \sup_{[0, s]}|u_n(r)|^p_{\mathbb{H}^{1, 2}}ds.
\end{eqnarray}
Hence,
\begin{eqnarray}\label{Eq-1-term-Lp-norm-SDE-Grn-torus}
\mathbb{E}\sup_{[0, T]}|u_n(t)|^p_{\mathbb{H}^{1, 2}} &+& c
\mathbb{E}\int_0^{T}| u_n(s)|^{p-2}_{\mathbb{H}^{1, 2}}|u_n(s)|_{\mathbb{H}^{1+\frac\alpha2, 2}}^2 ds\nonumber\\
&\leq& \mathbb{E}|u_0|^p_{\mathbb{H}^{1, 2}}
+ c\int_0^{T} \mathbb{E}\sup_{[0, s]}|u_n(r)|^{p}_{\mathbb{H}^{1, 2}}ds.
\end{eqnarray}
By application of Gronwall's lemma, we conclude that the first term in the LHS of \eqref{eq-estim-energy-Gene-2} is uniformly bounded in $n$. 
The total estimate  \eqref{eq-estim-energy-Gene-2}  follows easily from the above statement and from \eqref{Eq-1-term-Lp-norm-SDE-Grn-torus}.
The uniformity boundedness of the third, fourth and last  terms in
LHS of \eqref{Eq-Ito-n-weak-estimation-1-Torus} is a consequence of the application of Sobolev embedding
see e.g. \cite{Adams-Hedberg-94, R&S-96, Schmeisser-Tribel-87-book, Sickel-periodic spaces-85} and Appendix \ref{Appendix-Sobolev}
 respectively  H\"older inequality respectively Sobolev interpolation and
Estimate \eqref{Eq-Ito-n-weak-estimation-torus}.

\del{By application of Gronwall's lemma, we conclude that the first term in the LHS of \eqref{Eq-Ito-n-weak-estimation-1-Torus}, with $ p'=p$,
is uniformly bounded. The uniformity boundedness of the third, fourth and last  terms in
LHS of \eqref{Eq-Ito-n-weak-estimation-1-Torus} is a consequence of the application of Sobolev embedding
see e.g. see \cite{Adams-Hedberg-94, R&S-96, Schmeisser-Tribel-87-book, Sickel-periodic spaces-85} and Appendix \ref{appendix-Sobolev-Embd}
 respectively  H\"older inequality respectively Sobolev interpolation and
Estimation \eqref{Eq-Ito-n-weak-estimation-torus}.}

\vspace{0.15cm}

\noindent Now we prove Estimate \eqref{Eq-B-n-weak-estimation-1-Torus}.
Thanks to the contraction of $ P_n$ on $ \mathbb{H}^{1-\frac\alpha2, 2}(\mathbb{T}^2)$, \eqref{est-B-estimatoion-q=2-eta} with $ \eta=1$ and \del{the Sobolev embedding,} the interpolation,   
we infer that for all sequence 
 $(u_n)_n$ satisfying \eqref{Eq-Ito-n-weak-estimation-1-Torus}, we have 
\begin{eqnarray}\label{Eq-B-n-weak-estimation-1-Torus-deter-2step}
\int_0^T |P_nB(u_n(t))|^{2}_{\mathbb{H}^{1-\frac\alpha2, 2}} dt&\leq& C \int_0^T |u_n(t)|^4_{\mathbb{H}^{2-\frac\alpha2, 2}}dt\nonumber\\
&\leq& \int_0^T |u_n(t)|^{8\frac{\alpha-1}{\alpha}}_{\mathbb{H}^{1, 2}} |u_n(t)|^{4\frac{(2-\alpha)}{\alpha}}_{\mathbb{H}^{1+\frac\alpha2, 2}}dt\leq c<\infty,
\end{eqnarray}
provided $ \alpha \in [\frac43, 2]$. Moreover, it is easy to see that 
\begin{eqnarray}\label{Eq-A-alpha-weak-estimation-1-Torus-deter-2step}
\int_0^T |A^\frac\alpha2 u_n(t)|^{2}_{\mathbb{H}^{1-\frac\alpha2, 2}} dt&\leq& \del{C \int_0^T |u_n(t)|^2_{\mathbb{H}^{\alpha,2}} dt\leq}
 c\int_0^T|u_n(t)|^2_{\mathbb{H}^{1+\frac\alpha2,2}}dt\leq c<\infty.
\end{eqnarray}
\end{proof}

\subsection*{Proof of the existence}\label{sec-subsection-existence-Torus} We shall follow for this proof a quiet standard scheme, see e.g.
\cite{Millet-Chueshov-Hydranamycs-2NS-10, Krylov-Rozovski-monotonocity-2007, Rockner-Pevot-06,  Roeckner-Zhang-tamedNS-12,
Sundar-Sri-large-deviation-NS-06}, but we shall use completely different  estimates. These latter are of fractional type and have been developed in Section \ref{sec-nonlinear-prop}.\del{completely different than the classical case} We shall focus more on the key estimates and on the main features of our equation. In deed, thanks to Lemma \ref{lem-unif-bound-theta-n-H-1} and Assumption $(\mathcal{C})$, with $ q=2$ and $ \delta =1$,
we conclude the existence of
a subsequence (we keep the same notation ) $(u_n)_n$,\del{ and adapted measurable processes $ u, F_1$ and $G_1$, such that} 
\begin{equation}\label{eq-u-first-belonging}
 u\in L^2(\Omega\times [0, T]; \mathbb{H}^{1+\frac\alpha2, 2}(\mathbb{T}^2))\cap
L^p(\Omega, L^\infty([0, T];
 \mathbb{H}^{1, 2}(\mathbb{T}^2))),
\end{equation}
\begin{equation}
F_1 \in L^2(\Omega\times [0, T]; \mathbb{L}^{2} (\mathbb{T}^2)) \;\; \text{and}\;\; 
G_1\in L^2(\Omega\times [0, T]; L_Q(\mathbb{H}^{1, 2} (\mathbb{T}^2)), s.t.
\end{equation}
\begin{itemize}
\item (1) $u_n \rightarrow u$ weakly in $ L^2(\Omega\times [0, T],; \mathbb{H}^{1+\frac\alpha2, 2}(\mathbb{T}^2))$.
\item (2) $u_n \rightarrow u$ weakly-star in $ L^p(\Omega, L^\infty([0, T]; \mathbb{H}^{1, 2}(\mathbb{T}^2)))$.
\item (3) $P_n(F(u_n):= (A^\frac\alpha2 + B)(u_n))\rightarrow F_1$ weakly in $ L^2(\Omega\times [0, T]; 
\mathbb{H}^{1-\frac\alpha2, 2} (\mathbb{T}^2))$.
\item (4) $u_n \rightarrow u$ weakly in $ L^{\frac{\alpha}{\eta}}(\Omega\times [0, T]; \mathbb{H}^{1+\eta, 2}(\mathbb{T}^2))$, for all
$ \frac\alpha p<\eta \leq \frac\alpha2 $.
\del{\item (5) $u_n(T) \rightarrow \xi$ weakly in $ L^2(\Omega, \mathbb{H}^{1, 2}(\mathbb{T}^2)))$.}
\item (5) $P_nG(u_n)\rightarrow G_1$ weakly in $ L^2(\Omega\times [0, T]; L_Q(\mathbb{H}^{1, 2} (\mathbb{T}^2))$.
\end{itemize}
The statements $ (1)-(3)$ are straightforward consequence of Lemma \ref{lem-unif-bound-theta-n-H-1}. 
Statement $ (4)$ is a consequence of  the combination of Lemma \ref{lem-unif-bound-theta-n-H-1} 
and the Sobolev interpolation.  Statement $ (5)$  holds 
thanks to the fact that $ P_n$ contracts the $\mathbb{H}^{1, 2}$-norm, 
Assumption $(\mathcal{C})$ (with $ q=2$ and $ \delta=1$) and the uniform boundedness of
 $ u_n$ in $ L^2(\Omega\times[0, T], \mathbb{H}^{1, 2}(\mathbb{T}^2))$.

Now, we construct a process $ (\tilde{u}(t), t\in [0, T])$ as
\begin{equation}\label{eq-def-u-tilde}
\tilde{u}(t) = u_0+ \int_0^t F_1(s)ds + \int_0^tG_1(s)dW(s)
\end{equation}
and prove that $ u= \tilde{u}, dt\times dP- a.e.$. Indeed, using Statement $ (1)$, Equation \eqref{FSBE-Galerkin-approxi} and  Fubini theorem, we infer that for all
$ \varphi \in L^\infty(\Omega\times[0, T], \mathbb{R})$ and $v\in \cup_{n}H_n $,
\begin{eqnarray}\label{eq-u=u-tilde-1}
 \mathbb{E}\int_0^T\langle u(t), \varphi(t) v\rangle_{\mathbb{L}^2} dt &=&
\lim_{n\rightarrow +\infty} \mathbb{E}\int_0^T\langle u_n(t), \varphi(t) v\rangle_{\mathbb{L}^2} dt\nonumber\\
&=& \lim_{n\rightarrow +\infty}\Big[\mathbb{E}\int_0^T \Big(\langle u_n(0), \varphi(t) v\rangle_{\mathbb{L}^2} +
\langle P_nF(u_n(t)), \int_t^T\varphi(s)ds v\rangle_{\mathbb{L}^2}\\
&+&\langle \int_0^t P_nG(u_n(s))dW_n(s), \varphi(t) v\rangle_{\mathbb{L}^2} \Big)dt\Big].\nonumber
\end{eqnarray}
The convergence of the terms in the RHS of \eqref{eq-u=u-tilde-1} to the terms in the RHS of \eqref{eq-def-u-tilde} is as follow. The first term is a consequence of the convergence of 
$ P_n \rightarrow I_{\mathbb{L}^2}$ with respect to the bounded linear operator topology on 
$ \mathbb{L}^2(\mathbb{T}^2)$ and the application of Lebesgue dominated convergence theorem. The second 
term converges thanks to Statement $ (3)$ and the last one converges thanks to the stochastic isometry,
Statement $ (5)$,  Lebesgue dominated  convergence theorem and Lemma \ref{lem-unif-bound-theta-n-H-1},
\del{ and to
the fact that the map $ \mathcal{T}: \mathcal{P}_T(\mathbb{L}^2) \ni \sigma \mapsto \int_0^T\sigma(s)dW(s) \in L^2(\Omega)$ is 
linear continuous} see e.g.
\cite{Millet-Chueshov-Hydranamycs-2NS-10}. Therefore,
\begin{eqnarray}\label{eq-u=u-tilde-2}
\mathbb{E}\int_0^T\langle u(t)- \Big(u(0) +
\int_0^t F_1(s)ds +  \int_0^t G_1(s)dW(s)\Big), \varphi(t) v\rangle dt=0.
\end{eqnarray}
To achieve the proof of the existence, we have to prove that  $F_1= A^\frac\alpha2 \tilde{u}+ B(\tilde{u})$ and $ G_1= G(\tilde{u})$. 
First, we prove the following key estimates
\begin{itemize}
 \item $ (\mathcal{K}_1)$  The local monotonicity property: There exists a constant $ c>0$ such that 
$\forall u, v \in \mathbb{H}^{1+\frac\alpha2}(\mathbb{T}^2)$,
\begin{eqnarray}\label{eq-a-alpha-B-G}
 -2\langle A_\alpha(u-v), u-v\rangle_{\mathbb{L}^2} &+&  2\langle B(u)-B(v), u-v\rangle_{\mathbb{L}^2} +
|| G(u) - G(v)||^2_{L_Q(\mathbb{L}^{2})}\nonumber\\
&\leq & c (1+ | v|^{\frac{2\alpha}{3\alpha-2}}_{\mathbb{H}^{1+\frac\alpha2, 2}})| u-v|^2_{\mathbb{L}^2}.
\end{eqnarray}
\item $ (\mathcal{K}_2)$ For all 
$ \psi\in L^\infty([0, T], \mathbb{R}_+)$ and\del{$ v\in L^2(\Omega\times [0, T]; H_n)$}
$ v\in L^2(\Omega\times [0, T]; \mathbb{H}^{1+\frac\alpha2}(\mathbb{T}^2))$,
\begin{eqnarray}\label{eq-def-Z-n}
 Z_n:= \int_0^T\psi(t)dt\mathbb{E}\big\{\int_0^te^{-r(s)}\big(\!\!\!\!\!&-&\!\!\!\!\!r'(s)|u_n(s)- v(s)|^2_{\mathbb{L}^2}+ 
|| P_nG(u_n(s)) - P_nG(v(s))||^2_{L_Q(\mathbb{L}^{2})}\nonumber\\
&+&2\langle F(u_n(s)) - F(v(s)), u_n(s)- v(s))\rangle_{\mathbb{L}^2}\big)ds\big\}\leq 0,
\end{eqnarray}
\end{itemize}
where $ r'(t):= c(1+ | v(t)|^{\frac{2\alpha}{3\alpha-2}}_{\mathbb{H}^{1+\frac\alpha2, 2}})$ and $c>0$ is a 
constant relevantly chosen.
In fact,  by using 
\begin{equation}\label{formula-B-v1-B-v2}
 B(u_1, u_1) - B(u_2, u_2) = B(u_1, u_1-u_2)+ B(u_1-u_2, u_2),
\end{equation}
\del{an elementary calculus as in \eqref{formula-B-v1-B-v2},} Property \eqref{Eq-3lin-propnull},  
H\"older inequality, Estimate 
\eqref{Eq-B-L-2-est}\del{{Eq-B-H-alpha-2-est}}, interpolation in the case $ 1< \alpha <2 $ (recall that
$ 1\leq \alpha \leq 2 \Rightarrow \mathbb{H}^{\frac\alpha2, 2}(\mathbb{T}^2) \hookrightarrow \mathbb{H}^{1-\frac\alpha2, 2}(\mathbb{T}^2)$)
and Young inequality, we infer that
\begin{eqnarray}
|\langle B(u)-B(v), u-v\rangle_{\mathbb{L}^2}| &= & |\langle B(u-v, v), u-v\rangle_{\mathbb{L}^2}|\leq
| B(u-v, v)|_{\mathbb{L}^2}|u-v|_{\mathbb{L}^2}\nonumber\\
&\leq& c|v|_{\mathbb{H}^{1+\frac\alpha2, 2}} |u-v|_{\mathbb{L}^2} |u-v|_{\mathbb{H}^{1-\frac\alpha2, 2}}\nonumber\\
&\leq& c|v|_{\mathbb{H}^{1+\frac\alpha2, 2}} |u-v|^{\frac{3\alpha-2}{\alpha}}_{\mathbb{L}^2}
|u-v|^{\frac{2-\alpha}{\alpha}}_{\mathbb{H}^{\frac\alpha2, 2}}\nonumber\\
&\leq& c|v|^{\frac{2\alpha}{3\alpha-2}}_{\mathbb{H}^{1+\frac\alpha2, 2}}|u-v|^{2}_{\mathbb{L}^2} + \frac12|u-v|^{2}_{\mathbb{H}^{\frac\alpha2, 2}}.
\end{eqnarray}
Moreover, thanks to the semigroup property of $ (A^\beta)_{\beta\geq0}$, we infer that 
\begin{eqnarray}
 -\langle A_\alpha(u-v), u-v\rangle_{\mathbb{L}^2} = - |u-v|^{2}_{\mathbb{H}^{\frac\alpha2, 2}}.
\end{eqnarray}
Therefore, there exists a constant $ c>0$ such that
\begin{eqnarray}\label{eq-last-exitence-Torus-H-1}
 -\langle A_\alpha(u-v), u-v\rangle_{\mathbb{L}^2} +  \langle B(u)-B(v), u-v\rangle_{\mathbb{L}^2} \leq
 -\frac12|u-v|^{2}_{\mathbb{H}^{\frac\alpha2, 2}} + c|v|^{\frac{2\alpha}{3\alpha-2}}_{\mathbb{H}^{1+\frac\alpha2, 2}}|u-v|^{2}_{\mathbb{L}^2}.\nonumber\\
\end{eqnarray}
Combining \eqref{eq-last-exitence-Torus-H-1} and Assumption $ (\mathcal{C})$: (\eqref{Eq-Cond-Lipschitz-Q-G} with 
$ q=2$, $ \delta=0$ and $ C_R:=C$),\del{ as mentioned in Theorem \ref{Main-theorem-strog-Torus},} 
we easily get \eqref{eq-a-alpha-B-G}. In particular, thanks to 
the contraction of $ P_n$ on $ \mathbb{L}^2(\mathbb{T}^{2})$,  Estimate \eqref{eq-a-alpha-B-G} is still valid when replacing 
$ u, v$ and $G$  by $ u_n$ (recall $ u_n$ is the solution of Equation \eqref{FSBE-Galerkin-approxi}), $v\in L^{2}(\Omega\times[0, T]; \mathbb{H}^{1+\frac\alpha2}(\mathbb{T}^{2}))$
and  $ P_nG$ respectively.\del{and thanks to 
the contraction of $ P_n$ on $ \mathbb{L}^2(O)$, we infer that 
Estimate \eqref{eq-a-alpha-B-G} is valid as well.}  Furthermore, estimating the LHS of  
\eqref{eq-def-Z-n} by  \eqref{eq-a-alpha-B-G} endowed with these latter variables, we get $ Z_n\leq 0$. Consequently $ (\mathcal{K}_1)$ and $ (\mathcal{K}_2)$ are proved.\del{with $ u= u_n(t)$ and $ v= v(t)$}
Let us also mention and recall the following two statements

\noindent (a)- If $ v\in L^{2}(\Omega\times[0, T]; \mathbb{H}^{1+\frac\alpha2}(\mathbb{T}^{2}))$,
 with $ \alpha\in [1, 2]$, then 
\begin{equation}
 \mathbb{E}\int_0^T|v(t)|^{\frac{2\alpha}{3\alpha-2}}_{\mathbb{H}^{1+\frac\alpha2, 2}}dt \leq 
 \mathbb{E}\int_0^T(1+|v(t)|^{2}_{\mathbb{H}^{1+\frac\alpha2, 2}})dt < \infty.
\end{equation}

\noindent (b)- If a sequence $(f_n)_n$ in a Hilbert space $ H$ converges weakly to $f$ then $|f|_{H}\leq \liminf_{n\rightarrow \infty}|f_n|_H$.

\vspace{0.25cm}

\noindent Now, we take $ \psi$ and $ r(t)$ as defined above. Thanks to the equality  $ u= \tilde{u}, dt\times dP- a.e.$,
statements $ (1)$ \& (b) and Fubini's theorem, we infer that 
\begin{eqnarray}\label{eq-equality-norms-1-Torus}
\int_0^T\psi(t)dt\mathbb{E}|u(s)|^2_{\mathbb{L}^2}e^{-r(t)}- \mathbb{E}|u_0|^2_{\mathbb{L}^2}
&=&\int_0^T\psi(t)dt\mathbb{E}|\tilde{u}(s)|^2_{\mathbb{L}^2}e^{-r(t)}- \mathbb{E}|u_0|^2_{\mathbb{L}^2} \nonumber\\
&\leq &\liminf_{n\rightarrow \infty}\int_0^T\psi(t)dt\mathbb{E}|u_n(s)|^2_{\mathbb{L}^2}e^{-r(t)}- \mathbb{E}|u_0|^2_{\mathbb{L}^2}.
\end{eqnarray}
By application of the Ito formula to the Ito process $ \tilde{u} $ given by \eqref{eq-def-u-tilde} 
and using\del{the first equality in \eqref{eq-equality-norms-1-Torus},} 
the equality  $ u= \tilde{u}, dt\times dP- a.e.$ and the elementary identity 
\begin{equation}\label{elme-identity-Hilbert}
 \forall f, g \in H, | f|^2_{H}= |f-g|^2_{H} +2\langle f-g, g \rangle_{H} +| g|^2_{H}
\end{equation}
with $ f= u(s)$, $ g= v(s)$ and\del{``$
|u(s)|^2_{\mathbb{L}^2} = |u(s)-v(s)|^2_{\mathbb{L}^2}
+2\langle u(s)- v(s), v(s)\rangle_{\mathbb{L}^2} + |v(s)|^2_{\mathbb{L}^2}$''  with} $ v\in 
L^{2}(\Omega\times[0, T]; \mathbb{H}^{1+\frac\alpha2}(\mathbb{T}^{2}))$, we get
\del{\begin{eqnarray}\label{eq-elementary-calculs}
\mathbb{E}\int_0^te^{-r(s)}r'(s)|u(s)|^2_{\mathbb{L}^2}ds = \mathbb{E}\int_0^te^{-r(s)}r'(s)
\big(|u(s)-v(s)|^2_{\mathbb{L}^2}
+2\langle u(s)- v(s), v(s)\rangle_{\mathbb{L}^2} + |v(s)|^2_{\mathbb{L}^2}\big)ds.\nonumber\\
\end{eqnarray}
By application of Ito formula on Ito process $ \tilde{u} $ and using the first equality in 
\eqref{eq-equality-norms-1} and the
identity \eqref{eq-elementary-calculs}, we get}
\begin{eqnarray}\label{eq-equality-norms-Torus}
\mathbb{E}|u(s)|^2_{\mathbb{L}^2}e^{-r(t)}&-&\mathbb{E}|u_0|^2_{\mathbb{L}^2}
=\mathbb{E}\int_0^t2e^{-r(s)}\langle F_1(s),  u(s)\rangle_{\mathbb{L}^2}+ \mathbb{E}\int_0^te^{-r(s)}||G_1(s)||^2_{L_Q(\mathbb{L}^2)}ds \nonumber\\
&- &  \mathbb{E}\int_0^te^{-r(s)}r'(s)\big(|u(s)-v(s)|^2_{\mathbb{L}^2}
+2\langle u(s)- v(s), v(s)\rangle_{\mathbb{L}^2} + |v(s)|^2_{\mathbb{L}^2}\big)ds.
\end{eqnarray}
Similarly, we get Identity \eqref{eq-equality-norms-Torus} for $ \mathbb{E}|u_n(s)|^2_{\mathbb{L}^2}e^{-r(t)}$ with 
$ u, F_1 $\del{in fact it is P_nF_1, but because the product <P_nF_1, u_n>, the P_n is killed} and $ G_1$ in the RHS of \eqref{eq-equality-norms-Torus} are respectively replaced by $ u_n, F(u_n), P_nG(u_n)$ (Recall that $F:= A^\frac\alpha2 + B$). Replacing Identity \eqref{eq-equality-norms-Torus} for 
$ \mathbb{E}|u(s)|^2_{\mathbb{L}^2}e^{-r(t)}$ in the LHS of the first equality in \eqref{eq-equality-norms-1-Torus} 
and Identity \eqref{eq-equality-norms-Torus} for
$ \mathbb{E}|u_n(s)|^2_{\mathbb{L}^2}e^{-r(t)}$ in the RHS of  the second Inequality \eqref{eq-equality-norms-1-Torus} and arranging terms
(in particular, we introduce the term $ G(v(s))$ and use the elementary identity \eqref{elme-identity-Hilbert}), we
infer that
\begin{eqnarray}\label{eq-equality-norms-Torus-sum-of-u}
\mathcal{E}:= \int_0^T\psi(t)dt\mathbb{E}\big\{\int_0^te^{-r(s)}\Big[\!\!\!\!&\!\!\!\!2\!\!\!\!&\!\!\!\!\langle F_1(s),  u(s)\rangle_{\mathbb{L}^2}+
||G_1(s)||^2_{L_Q(\mathbb{L}^2)}ds \nonumber\\
&- & r'(s)\big(|u(s)-v(s)|^2_{\mathbb{L}^2}
+2\langle u(s)- v(s), v(s)\rangle_{\mathbb{L}^2} \big)\Big] ds\big\}\nonumber\\
&\leq & \liminf_{n\rightarrow \infty}\big(Z_n + Y_n + X_n\big).
\end{eqnarray}
where $ Z_n$ is given by \eqref{eq-def-Z-n},
\begin{eqnarray}
Y_n:&=& \int_0^T\psi(t)dt\mathbb{E}\big\{\int_0^t e^{-r(s)}\big(-2r'(s)
\langle u_n(s)- v(s), v(s)\rangle_{\mathbb{L}^2}+
2 \langle P_nG(u_n(s)), G(v(s))\rangle_{L_Q(\mathbb{L}^2)}\nonumber\\
&+&2\langle F(u_n(s)), v(s)\rangle_{\mathbb{L}^2} +2
\langle F(v(s)), u_n(s))\rangle_{\mathbb{L}^2}-2
\langle F(v(s)), v(s))\rangle_{\mathbb{L}^2}
\big)ds\big\},\nonumber\\
\end{eqnarray}
and
\begin{eqnarray}
X_n:&=& \int_0^T\psi(t)dt\mathbb{E}\big\{\int_0^t e^{-r(s)}\big(
2 \langle P_nG(u_n(s)), P_nG(v(s))- G(v(s))\rangle_{\mathbb{L}^2}
-||P_nG(v(s))||^2_{L_Q(\mathbb{L}^2)}\big)ds\big\},\nonumber\\
\end{eqnarray}
The sequences $ (Y_n)_n$ and $ (X_n)_n$  converge to $ Y$ and $X$ respectively,
thanks to statements $(1)-(3)$ and  $(5)$, the convergence of $ P_n \rightarrow I_{\mathbb{L}^2}$,  
Assumption $(\mathcal{C})$:( \eqref{Eq-Cond-Lipschitz-Q-G} with $ q=2$, $ \delta =0$ and $ C_R:=C$), Lemma \ref{lem-unif-bound-theta-n-H-1},   
Estimate \eqref{Eq-B-L-2-est}, similar calculus as in \eqref{Eq-B-n-weak-estimation-1-Torus-deter-2step} 
and \eqref{Eq-A-alpha-weak-estimation-1-Torus-deter-2step} and the Lebesgue dominated, where
\begin{eqnarray}
Y:&=& \int_0^T\psi(t)dt\mathbb{E}\big\{\int_0^t e^{-r(s)}\big(-2r'(s)\langle u(s)- v(s), v(s)\rangle_{\mathbb{L}^2}+
2 \langle G(s), G(v(s))\rangle_{L_Q(\mathbb{L}^2)}\nonumber\\
&+&2\langle F(s), v(s)\rangle_{\mathbb{L}^2} +2
\langle F(v(s)), u(s))\rangle_{\mathbb{L}^2}-2
\langle F(v(s)), v(s))\rangle_{\mathbb{L}^2}
\big)ds\big\} \nonumber\\
\end{eqnarray}
and
\begin{eqnarray}
X:&=& -\int_0^T\psi(t)dt\mathbb{E}\big\{\int_0^t e^{-r(s)}
||G(v(s))||^2_{L_Q}\big)ds\big\}.
\end{eqnarray}
Replacing $ X$ and $ Y$ in \eqref{eq-equality-norms-Torus-sum-of-u} and taking into account \eqref{eq-def-Z-n}, we conclude that
\begin{equation}
 \mathcal{E} -X-Y \leq \liminf_{n\rightarrow \infty}Z_n \leq 0.
\end{equation}
Therefore, we get
\begin{eqnarray}\label{eq-key-leq-0-1}
\int_0^T\psi(t)dt\mathbb{E}\big\{\int_0^te^{-r(s)}\big(-r'(s)|u(s)\!\!\!&-&\!\!\! v(s)|^2_{\mathbb{L}^2}+ \langle F_1(s) -
F(v(s)), u(s)- v(s))\rangle_{\mathbb{L}^2}\nonumber\\
&+&2|| G_1(s) - G(v(s))||^2_{L_Q}\big)ds\big\}\leq 0.
\end{eqnarray}
\del{Consequently, for $  v \in L^2([\Omega\times0, T], \mathbb{H}^{1+\frac\alpha2, 2}(\mathbb{T}^2))$, $ P_nv$ satisfies \eqref{eq-key-leq-0-1}. 
By approximation, \eqref{eq-key-leq-0-1} is also satisfied by  $  v \in L^2([\Omega\times0, T], \mathbb{H}^{1+\frac\alpha2, 2}(\mathbb{T}^2))$.
In deed, the passage to the limit of the first and third terms is guaranteed thanks to the convergence of the projection 
$ P_n$ to the identity $ I_{\mathbb{H}^{s, 2}}$ for all Sobeolev spaces $ \mathbb{H}^{s, 2}(O), s\geq 0$, in particular for  $ I_{\mathbb{L}^2}(O)$
and by Assumption $ (\mathcal{C})$. To prove the convergence of the second term, it is sufficient to prove and justify the convergence
of the term $ \langle F(P_nv(s)),  P_nv(s))\rangle_{\mathbb{L}^2}$. In fact, we rewrite this last as follow
\begin{eqnarray}\label{local-est-1}
 \langle F(P_nv(s)),  P_nv(s)\rangle_{\mathbb{L}^2}&-& \langle F(v(s)),  v(s)\rangle_{\mathbb{L}^2}\nonumber\\
&=& \langle F(P_nv(s))-F(v(s)),  v(s)\rangle_{\mathbb{L}^2} + \langle F(P_nv(s)),  v(s)- P_nv(s)\rangle_{\mathbb{L}^2}.
\end{eqnarray}
Using H\"older and  Minowksky inequalities, Estimation \eqref{Eq-B-L-2-est} and the contraction property of the projection 
$ P_n$ on all Sobeolev spaces $ \mathbb{H}^{s, 2}(O), s\geq 0$, we get
\begin{eqnarray}\label{local-est-2}
|\langle F(P_nv(s)),  v(s)- P_nv(s)\rangle_{\mathbb{L}^2}|&\leq & | v(s)- P_nv(s)|_{\mathbb{L}^2}|F(P_nv(s))|_{\mathbb{L}^2}\nonumber\\
&\leq &c| v(s)- P_nv(s)|_{\mathbb{L}^2}\big( |P_nv(s)|_{\mathbb{H}^{\alpha, 2}}+   |B(P_nv(s))|_{\mathbb{L}^2}  \big)\nonumber\\
\del{&\leq &c| v(s)- P_nv(s)|_{\mathbb{L}^2}\big( |v(s)|_{\mathbb{H}^{\alpha, 2}}+  
 |P_nv(s))|_{\mathbb{H}^{1-\frac\alpha2, 2}}|P_nv(s))|_{\mathbb{H}^{1+\frac\alpha2, 2}} \big)\nonumber\\}
&\leq &c| v(s)- P_nv(s)|_{\mathbb{L}^2}\big( |v(s)|_{\mathbb{H}^{\alpha, 2}}+  
 |v(s)|_{\mathbb{H}^{1-\frac\alpha2, 2}}|v(s)|_{\mathbb{H}^{1+\frac\alpha2, 2}} \big).
\end{eqnarray}
For the second term in \eqref{local-est-1}, we use H\"older and  Minoksky inequalities, \eqref{Eq-def-B-theta1-Theata2}, 
\eqref{Eq-B-H-alpha-2-est-d}, we infer 
\begin{eqnarray}\label{local-est-2}
|\langle F(P_nv(s))&-& F(v(s)),  v(s)\rangle_{\mathbb{L}^2}|
\leq  |v(s)|_{\mathbb{H}^{+1+\frac\alpha2, 2}} |F(v(s))- F(P_nv(s))|_{\mathbb{H}^{-1-\frac\alpha2, 2}}\nonumber\\
&\leq &c |v(s)|_{\mathbb{H}^{+1+\frac\alpha2, 2}}\big( |v(s)- P_nv(s)|_{\mathbb{L}^{2}}+ 
|B(v(s)-P_nv(s), P_nv(s))|_{\mathbb{H}^{-1-\frac\alpha2, 2}} \nonumber\\
&+& |B(v(s), P_nv(s)-v(s))|_{\mathbb{H}^{-1-\frac\alpha2, 2}}\big)\nonumber\\
&\leq &c |v(s)|_{\mathbb{H}^{1+\frac\alpha2, 2}}\big( |v(s)- P_nv(s)|_{\mathbb{L}^{2}}+ 
|v(s)|_{\mathbb{H}^{1-\frac\alpha4, 2}}|P_nv(s)-v(s)|_{\mathbb{H}^{1-\frac\alpha4, 2}}\big)\nonumber\\
&\leq &c (1+|v(s)|_{\mathbb{H}^{+1+\frac\alpha2, 2}})|P_nv(s)-v(s)|_{\mathbb{H}^{1-\frac\alpha4, 2}}.
\end{eqnarray}
Finally, the integrand in RHS of \eqref{eq-key-leq-0-1} with $ v(s)$ being replaced by $ P_nv(s)$, converges by application of the Lebesgue 
dominate convergence theorem.}
\noindent Now, we take $ v = u$ in $ L^2([\Omega\times0, T], \mathbb{H}^{1+\frac\alpha2, 2}(\mathbb{T}^2))$, we conclude from \eqref{eq-key-leq-0-1}, that
\del{$|| G(s) - G(v(s))||^2_{L_Q}$} $ G(s)= G(u(s)), \; ds\times dP-a.e.$. To get the equality $ F(s)= F(u(s)), \; ds\times dP-a.e.$, 
we consider Estimate \eqref{eq-key-leq-0-1} without the last term and we introduce
$ \tilde{v} \in L^\infty(\Omega\times[0, T], \mathbb{H}^{1+\frac\alpha2}(\mathbb{T}^2))$ and a
parameter $ \lambda \in [-1, +1]$. Replacing $ v$  and $ r'(s)$ by $ u-\lambda\tilde{v}$ respectively
$ r'_\lambda(s):= c(1+|u-\lambda\tilde{v}|^{\frac{2\alpha}{3\alpha-2}}_{\mathbb{H}^{1+\frac\alpha2}})$, we get
\begin{eqnarray}\label{eq-equality-F-Fn-*}
\mathbb{E}\int_0^Te^{-r_\lambda(s)}\big(-r'_\lambda(s)\lambda^2|\tilde{v}(s)|^2_{\mathbb{L}^2}+ 2\lambda\langle F(s) -
F(u(s)-\lambda \tilde{v}(s)), \tilde{v}(s))\rangle_{\mathbb{L}^2}\big)ds\leq 0.\nonumber\\
\end{eqnarray}
Dividing on $ \lambda<0$ and on $ \lambda>0$, we conclude that, when $ \lambda \rightarrow 0$, the limit of
the LHS of \eqref{eq-equality-F-Fn-*} exists and vanishes. Moreover, using the fact that
$ \tilde{v} \in L^\infty(\Omega\times[0, T], \mathbb{H}^{1+\frac\alpha2}(\mathbb{T}^2))$, 
the continuity of $ r_\lambda$ and 
$ r'_\lambda$  with respect to $ \lambda$ and the Lebesgue 
dominated convergence theorem, we conclude that the first term in $ \frac1\lambda$LHS of \eqref{eq-equality-F-Fn-*} 
vanishes and also 
\begin{eqnarray}\label{eq-equality-F-Fn-1}
\mathbb{E}\int_0^Te^{-r_0(s)}\langle F(s) -
F(u(s)), \tilde{v}(s))\rangle_{\mathbb{L}^2}ds= 0.
\end{eqnarray}
The justification of the use of the Lebesgue dominated convergence theorem is due to, the positivity of 
$ r_\lambda(s)$, Inequality
\del{$ |B(f)|_{\mathbb{H}^{-1, 2}} \leq |f|_{\mathbb{H}^{1, 2}}|f|_{\mathbb{L}^{2}}$}\eqref{classical-B-H-1}, Minikowskii inequality,
the conditions $ 0< \alpha\leq 2$ and $ |\lambda|\leq 1$,
the statements $ (1) \;\& \; (3)$ and the definition of $ \tilde{v}$. In fact,
\del{\begin{eqnarray}
 e^{-r_\lambda(s)}|\langle F(s) &-&
F(u(s)-\lambda \tilde{v}(s)), \tilde{v}(s))\rangle_{\mathbb{L}^2}|\nonumber\\
&\leq&
 |\tilde{v}(s)|_{\mathbb{L}^2}|F(s) -
F(u(s)-\lambda \tilde{v}(s))|_{\mathbb{L}^2} \nonumber\\
&\leq&
 |\tilde{v}(s)|_{\mathbb{L}^2}\big(|F(s)|_{\mathbb{L}^2} + |u(s)|_{\mathbb{H}^{\frac\alpha2, 2}}+ |\tilde{v}(s)|_{\mathbb{H}^{\frac\alpha2, 2}}
|B(u(s)-\lambda \tilde{v}(s))|_{\mathbb{L}^2} \nonumber\\
&\leq&
|\tilde{v}(s)|_{\mathbb{L}^2}\big(|F(s)|_{\mathbb{L}^2} + |u(s)|_{\mathbb{H}^{\frac\alpha2, 2}}+ |\tilde{v}(s)|_{\mathbb{H}^{\frac\alpha2, 2}}
|B(u(s)-\lambda \tilde{v}(s))|_{\mathbb{L}^2} \nonumber\\
&\leq&
 |\tilde{v}(s)|_{\mathbb{L}^2}|F(s) -
F(u(s)-\lambda \tilde{v}(s))|_{\mathbb{L}^2} \nonumber\\
\end{eqnarray}}
\begin{eqnarray}\label{just-domin-Torus}
 e^{-r_\lambda(s)}|\langle F(s) &-&
F(u(s)-\lambda \tilde{v}(s)), \tilde{v}(s))\rangle_{\mathbb{L}^2}|\nonumber\\
&\leq&
 |\tilde{v}(s)|_{\mathbb{H}^{1, 2}}\big[|F(s)|_{\mathbb{L}^2}+
|u(s)|_{\mathbb{H}^{\alpha-1, 2}}+ |\tilde{v}(s)|_{\mathbb{H}^{\alpha-1, 2}}+ |B(u(s)-\lambda \tilde{v}(s))|_{\mathbb{H}^{-1, 2}}\big] \nonumber\\
&\leq& c
 |\tilde{v}(s)|_{\mathbb{H}^{1, 2}}\big[|F(s)|_{\mathbb{L}^2}+
|u(s)|_{\mathbb{H}^{\alpha-1, 2}}+ |\tilde{v}(s)|_{\mathbb{H}^{\alpha-1, 2}}+
|u(s)|_{\mathbb{H}^{1, 2}}|u(s)|_{\mathbb{L}^{ 2}} \nonumber\\
&+& |\tilde{v}(s)|_{\mathbb{H}^{1, 2}}|\tilde{v}(s)|_{\mathbb{L}^{ 2}}
+|u(s)|_{\mathbb{H}^{1, 2}}|\tilde{v}(s)|_{\mathbb{L}^{ 2}} + |\tilde{v}(s)|_{\mathbb{H}^{1, 2}}|u(s)|_{\mathbb{L}^{ 2}}\big].
\end{eqnarray}
This ends the proof of the existence of a solution $ (u(t), t\in [0, T])$ belonging to the first intersection in \eqref{eq-set-solu-weak-torus} and satisfying by construction  \eqref{cond-solu-torus-H1}.\del{ and  $ u(\cdot, \omega) \in L^\infty(0, T; \mathbb{H}^{1, 2}(\mathbb{T}^2)) \cap 
L^2(0, T; \mathbb{H}^{1+\frac\alpha2, 2}(\mathbb{T}^2)), \; P-a.s.$}

\subsection*{Proof of the time regularity.}

To prove the continuity of the trajectories of the weak solution 
$ (u(t), t\in[0, T])$, i.e. $ u(\cdot, \omega)\in C([0, T], \mathbb{L}^2(O)), \; P-a.s.$, we apply \cite[Proposition VII.3.2.2]{Metivier-book-SPDEs-88}, see also
\cite[Proposition 2.5]{Sundar-Sri-large-deviation-NS-06}. We consider the dense Gelfand Triple 
$$ \mathbb{H}^{1, 2}(\mathbb{T}^2) \hookrightarrow \mathbb{L}^{2}(\mathbb{T}^2) \hookrightarrow 
(\mathbb{H}^{1, 2}(\mathbb{T}^2))^*=\mathbb{H}^{-1, 2}(\mathbb{T}^2).$$
Using  \eqref{classical-B-H-1}, Sobolev Embedding\del{,  the trivial identity $ |u|\leq 1+|u|^2$ } and  \eqref{cond-solu-torus-H1}, we infer that 
$ B(u(\cdot, \omega))\in L^2(0, T; \mathbb{H}^{-1, 2}(\mathbb{T}^2)),\; P-a.s.$. In fact,
\begin{equation}\label{cont-1}
\mathbb{E} \int_0^T|B(u(s))|_{\mathbb{H}^{-1, 2}}^{2}ds \leq c \mathbb{E} \int_0^T|u(s)|^4_{\mathbb{H}^{1, 2}}ds
\leq c (1+\mathbb{E}\sup_{[0, T]}|u(s)|^p_{\mathbb{H}^{1, 2}})<\infty. 
\end{equation}
And that for $\alpha\leq 2$,
\begin{equation}\label{cont-1-1}
\mathbb{E} \int_0^T|A_\alpha u(s)|^2_{\mathbb{H}^{-1, 2}}ds \leq c \mathbb{E} \int_0^T|u(s)|^2_{\mathbb{H}^{\alpha-1, 2}}ds
\del{\leq c \mathbb{E}\sup_{[0, T]}|u(s)|^2_{\mathbb{H}^{1, 2}}}\leq c (1+\mathbb{E}\sup_{[0, T]}|u(s)|^p_{\mathbb{H}^{1, 2}})<\infty. 
\end{equation}
Moreover, we prove that\del{ the trajectories of the 
continuous} the martingale $ M(t):= \int_0^tG(u(s))dW(s)$ belongs to $ L^2(\Omega, \del{L^\infty}C(0, T; \mathbb{H}^{1, 2}(\mathbb{T}^2)))$. In fact, we use 
Burkholdy-Davis-Gandy inequality, Assumption $ (\mathcal{C})$: (\eqref{Eq-Cond-Linear-Q-G}
with $ q=2$ and $ \delta=1$), \eqref{cond-solu-torus-H1}, we obtain 
\begin{eqnarray}\label{cont-2}
\mathbb{E}\sup_{[0, T]} |\int_0^tG(u(s))dW(s)|^2_{\mathbb{H}^{1, 2}}
&\leq& c \mathbb{E}\int_0^T|G(u(s))|^2_{L_Q(\mathbb{H}^{1, 2})}ds
\leq c \mathbb{E}\int_0^T(1+|u(s)|^2_{\mathbb{H}^{1, 2}})ds \nonumber\\
&\leq& c (1+ \mathbb{E}\sup_{[0, T]}|u(s)|^p_{\mathbb{H}^{1, 2}})<\infty.
\end{eqnarray}
Hence from \eqref{cond-solu-torus-H1}, \eqref{cont-1}, \eqref{cont-1-1}  and \eqref{cont-2}, 
we establish the existence of a subset $ \Omega'\subset \Omega$ (independent of "t"),  such that $ P(\Omega')=0$ and
$ F(u(\cdot, \omega))\in L^2(0, T; \mathbb{H}^{-1, 2}(\mathbb{T}^2))$, $  u(\cdot, \omega) \;\text{and}\;  M(\cdot, \omega) 
\in L^\infty(0, T; \mathbb{H}^{1, 2}(\mathbb{T}^2))$, $\forall \omega \in \Omega'^c$. These ingredients are enough to apply  \cite[Proposition VII.3.2.2]{Metivier-book-SPDEs-88}\del{\cite[Proposition 2.5.]{Sundar-Sri-large-deviation-NS-06}}, hence we get the result. It is important to mention that the property $  u(\cdot, \omega) \;\text{and}\;  M(\cdot, \omega) 
\in L^\infty(0, T; \mathbb{H}^{1, 2}(\mathbb{T}^2))$ is more what we need here. In deed, it is sufficient to prove $  u(\cdot, \omega) \;\text{and}\;  M(\cdot, \omega) 
\in L^2(0, T; \mathbb{H}^{1, 2}(\mathbb{T}^2))$. By the above two subsections, the proof of $ (3.7.1)$ is achieved.

\del{ Therefore, we have
\begin{eqnarray}\label{eq-equality-F-Fn-1}
\lim_{\lambda\rightarrow 0}\mathbb{E}\int_0^Te^{-r_\lambda(s)}\big\{2\langle F(s)&-&
F(u(s)), \tilde{v}(s))\rangle_{\mathbb{L}^2}+
\big(-r'_\lambda(s)\lambda|\tilde{v}(s)|^2_{\mathbb{L}^2}\nonumber\\
&+& 2\langle F(u(s)) - F(u(s)-\lambda \tilde{v}(s), \tilde{v}(s))\rangle_{\mathbb{L}^2}\big)\big\}ds= 0.
\end{eqnarray}
\begin{eqnarray}\label{eq-equality-F-Fn-1}
\lim_{\lambda\rightarrow 0}\mathbb{E}\int_0^Te^{-r_\lambda(s)}\big\{2\langle F(s)&-&
F(u(s)), \tilde{v}(s))\rangle_{\mathbb{L}^2}+
\big(-r'_\lambda(s)\lambda|\tilde{v}(s)|^2_{\mathbb{L}^2}\big\}ds\nonumber\\
&+& \lim_{\lambda\rightarrow 0}\mathbb{E}\int_0^T2e^{-r_\lambda(s)}\langle F(u(s)) - F(u(s)-\lambda \tilde{v}(s)), \tilde{v}(s))
\rangle_{\mathbb{L}^2}\big)ds= 0.
\end{eqnarray}}

\del{WITOUT STOP TIME\subsection{Proof of pathwise uniqueness.}\label{sec-subsection-uniqueness-Torus}
Let $ u^1$ and $ u^2$ be two weak solutions of Equation \eqref{Main-stoch-eq} with $ 1<\alpha \leq 2$ as constructed in Subsection
\ref{sec-subsection-existence-Torus}. Let $ w:=  u^1- u^2$, then  $ w$ satisfied the following equation
\begin{eqnarray}
 w(t) = \int_0^t \big(-A_\alpha w(s)+ B(w(s), u^1(s))&+& B(u^2(s), w(s))\big)ds \nonumber\\
&+& \int_0^t\big(G(u^1(s))- G(u^2(s))\big)dW(s).
\end{eqnarray}
Using Ito formula to the product $ e^{-r(t)}|w(t)|^2_{\mathbb{L}^{2}}$, with $(r(t), t\in [0, T])$ is a
real positive stochastic process to be precise later, Property \eqref{Eq-3lin-propnull},
condition \eqref{Eq-Cond-Lipschitz-Q-G} with $ q=2$ and $\delta =0$ (locally Lipschitz), 
Estimation \eqref{3linear-H1-H-1}, Young inequality  and arguing as in the proof of 
\eqref{Eq-Ito-n-weak-estimation-1-Torus} by replacing
the spaces $ \mathbb{H}^{1, 2}(\mathbb{T}^2)$ and $ \mathbb{H}^{1+\frac\alpha2, 2}(\mathbb{T}^2)$ by
$ \mathbb{L}^{2}(\mathbb{T}^2)$ respectively $ \mathbb{H}^{\frac\alpha2, 2}(\mathbb{T}^2)$.  we infer that 
\begin{eqnarray}\label{Torus-uniquenss-ito-formula}
\mathbb{E}\!\!\!\!&{}&\!\!\!\!e^{-r(t)}|w(t)|^2_{\mathbb{L}^{2}}+
2\mathbb{E}\int_0^{t}e^{-r(s)}|w(s)|^2_{\mathbb{H}^{\frac\alpha2, 2}}ds \nonumber\\
&\leq& \mathbb{E}\int_0^{t}e^{-r(s)}|| G(u^1(s))- G(u^2(s))||^2_{L_Q(\mathbb{L}^2)}ds\nonumber\\
& -& \mathbb{E}\int_0^{t} e^{-r(s)} \big(2\langle B(w(s),  w(s)),  u^1(s)\rangle-
r'(s)|w(s)|^2_{\mathbb{L}^{2}}\big)ds\nonumber\\
& \leq & c_N\mathbb{E}\int_0^{t}e^{-r(s)}\big(|w(s)|^2_{\mathbb{L}^2}+
|u^1(s)|_{\mathbb{H}^{1, 2}}|w(s)|^{\frac2\alpha}_{\mathbb{H}^{\frac{\alpha}{2}, 2}}
|w(s)|^{2\frac{\alpha-1}\alpha}_{{\mathbb{L}^{2}}}-
r'(s)|w(s)|^2_{\mathbb{L}^{2}}\big)ds\nonumber\\
& \leq & c_N
\mathbb{E}\int_0^{t}e^{-r(s)}\big(|w(s)|^2_{\mathbb{L}^2}+ 2 c|w(s)|^2_{\mathbb{H}^{\frac\alpha2, 2}}+
2c_1|u^1(s)|^{\frac{\alpha}{\alpha-1}}_{\mathbb{H}^{1, 2}}|w(s)|^2_{\mathbb{L}^{2}}-
r'(s)|w(s)|^2_{\mathbb{L}^{2}}\big)ds.\nonumber\\
\end{eqnarray}
Now, we choose $ c<1$ and $ r'(s)= 2c_1|u^1(s)|^{\frac{\alpha}{\alpha-1}}_{\mathbb{H}^{1, 2}}$ and replace in
\eqref{Torus-uniquenss-ito-formula},
we end up by the simple formula
\begin{eqnarray}
\mathbb{E}e^{-r(t)}|w(t)|^2_{\mathbb{L}^{2}}&+& 2(1-c)\mathbb{E}\int_0^{t}
e^{-r(s)}|w(s)|^2_{\mathbb{H}^{\frac\alpha2, 2}}ds
\leq c_N\mathbb{E}\int_0^{t}e^{-r(s)}|w(s)|^2_{\mathbb{L}^2}ds.\nonumber\\
\end{eqnarray}
Than by application of Gronwall's lemma, we get
$ \forall t\in [0, T, \, e^{-r(t)}|w(t)|^2_{\mathbb{L}^{2}} = 0, \, P-a.s.$ as
$P( e^{-c\int_0^{t}|u^1(s)|^{\frac{\alpha}{\alpha-1}}_{\mathbb{H}^{1, 2}}ds} <\infty)= 1 $.
The proof is achieved once we remark that thanks to \eqref{cond-solu-torus-H1},
$ P( \int_0^T|u^1(s)|^{\frac{\alpha}{\alpha-1}}_{\mathbb{H}^{1, 2}}ds <\infty)= 1$.}
\del{\noindent To prove the $ H^1-$continuity in \eqref{Eq-H-1-regu-2D-torus}, we use again \cite[Proposition VII.3.2.2]{Metivier-book-SPDEs-88}, with the gelfand triple \eqref{Gelfand-triple-Torus} and argue as above. In particular, we have 
\begin{itemize}
\item thanks to \eqref{est-B-estimatoion-q=2-eta-to-use} with $ \eta=1$ and by interpolation than by application \eqref{cond-solu-torus-H1} (remark that $ 4\frac{\alpha-1}{4-\alpha}<4$), we infer that $ B(u(\cdot, \omega))\in L^\frac{\alpha+2}{4-\alpha}(0, T; \mathbb{H}^{1-\frac\alpha2, 2}(\mathbb{T}^2)),\; P-a.s.$. In fact,
\begin{eqnarray}\label{cont-1-power-funct-alpha}
\mathbb{E} \int_0^T|B(u(s))|_{\mathbb{H}^{1-\frac\alpha2, 2}}^{\frac{\alpha+2}{4-\alpha}}ds &\leq & c \mathbb{E} \int_0^T|u(s)|^{2\frac{\alpha+2}{4-\alpha}}_{\mathbb{H}^{2-\frac\alpha2, 2}}ds
\leq c \mathbb{E} \int_0^T\Big(|u(s)|^{\frac{4-\alpha}{\alpha+2}}_{\mathbb{H}^{1+\frac\alpha2, 2}}
|u(s)|^{2\frac{\alpha-1}{\alpha+2}}_{\mathbb{H}^{1, 2}}\Big)^{2\frac{\alpha+2}{4-\alpha}}ds\nonumber\\
&\leq& c \mathbb{E} \int_0^T|u(s)|^{2}_{\mathbb{H}^{1+\frac\alpha2, 2}}
|u(s)|^{4\frac{\alpha-1}{4-\alpha}}_{\mathbb{H}^{1, 2}}ds
\del{\leq c \mathbb{E}\sup_{[0, T]}|u(s)|^2_{\mathbb{H}^{1, 2}}}<\infty. 
\end{eqnarray}
\item thanks to $\alpha\leq 2$, we have $ \frac{\alpha+2}{4-\alpha} \leq 2$ and
\begin{equation}\label{cont-1-1}
\mathbb{E} \int_0^T|A_\alpha u(s)|_{\mathbb{H}^{1-\frac\alpha2, 2}}^{\frac{\alpha+2}{4-\alpha}}ds \leq c \mathbb{E} \int_0^T|u(s)|_{\mathbb{H}^{1+\frac\alpha2, 2}}^{\frac{\alpha+2}{4-\alpha}}ds
\leq c (1+\mathbb{E}\int_0^T|u(s)|^2_{\mathbb{H}^{1+\frac\alpha2, 2}}ds)<\infty. 
\end{equation}

\item Now, we prove that\del{To prove that the trajectories of the 
continuous martingale $ M(t):= \int_0^tG(u(s))dW(s)$ belong $ P-a.s.$ to} $ M\in L^{\frac{2+\alpha}{2(\alpha-1)}}(\Omega, L^\infty(0, T; \mathbb{H}^{1+\frac\alpha2, 2}(\mathbb{T}^2)))$, we use 
Burkholdy-Davis-Gandy inequality, Assumption $ (\mathcal{C})$: (\eqref{Eq-Cond-Linear-Q-G}
with $ q=2$ and $ \delta=1$), \eqref{cond-solu-torus-H1}, we obtain 
\begin{eqnarray}\label{cont-2}
\mathbb{E}\sup_{[0, T]} |\int_0^tG(u(s))dW(s)|^2_{\mathbb{H}^{1, 2}}
&\leq& c \mathbb{E}\int_0^T|G(u(s))|^2_{L_Q(\mathbb{H}^{1, 2})}ds
\leq c \mathbb{E}\int_0^T(1+|u(s)|^2_{\mathbb{H}^{1, 2}})ds \nonumber\\
&\leq& c (1+ \mathbb{E}\sup_{[0, T]}|u(s)|^2_{\mathbb{H}^{1, 2}})<\infty.
\end{eqnarray}
\end{itemize}
Hence from \eqref{cond-solu-torus-H1}, \eqref{cont-1}, \eqref{cont-1-1}  and \eqref{cont-2}, 
we establish the existence of a subset $ \Omega'\subset \Omega$ (independent of "t"), 
such that $ P(\Omega')=0$ and
$ F(u(\cdot, \omega))\in L^2(0, T; \mathbb{H}^{-1, 2}(\mathbb{T}^2)),\; \text{and} \; u(\cdot, \omega),  M(\cdot, \omega) 
\in L^\infty(0, T; \mathbb{H}^{1, 2}(\mathbb{T}^2))\; \forall \omega \in \Omega'^c$. These ingredients are enough to apply  \cite[Proposition VII.3.2.2]{Metivier-book-SPDEs-88}\del{\cite[Proposition 2.5.]{Sundar-Sri-large-deviation-NS-06}}, hence we get the result. }

\subsection*{Proof of the pathwise uniqueness.}\label{sec-subsection-uniqueness-Torus}
Let $ u^1$ and $ u^2$ be two weak solutions of Equation \eqref{Main-stoch-eq} satisfying \eqref{eq-set-solu-weak-torus} 
and \eqref{cond-solu-torus-H1}. Let $ w:=  u^1- u^2$, then  $ w$ satisfies the following equation
\begin{eqnarray}\label{eq-uniq-w}
 w(t) = \int_0^t \big(-A_\alpha w(s)+ B(w(s), u^1(s))+ B(u^2(s), w(s))\big)ds + \int_0^t\big(G(u^1(s))- G(u^2(s))\big)dW(s).\nonumber\\
\end{eqnarray}
For $ N>0$, we define the stopping times\del{, see Appendix \ref{append-stop-time} for the proof}, $ \tau_N^i: \inf\{t\in (0, T); |u^i(t)|_{\mathbb{L}^{2}}> N\}\wedge T, i=1, 2$, with the understanding that $ \inf(\emptyset)= +\infty$ and define
$ \tau_N:=\min_{i\in\{1,2\}}\{ \tau_N^i\}$.
Using Ito formula for the product $ e^{-r(t)}|w(t)|^2_{\mathbb{L}^{2}}$, with $(r(t), t\in [0, T])$ being a
 positive real stochastic process to be defined later, Property \eqref{Eq-3lin-propnull},
Assumption $(\mathcal{C})$ (\eqref{Eq-Cond-Lipschitz-Q-G} with $ q=2$ and $\delta =0$, locally Lipschitz), 
Estimate \eqref{3linear-H1+alpha2-H-1}, Young inequality  and arguing as in the proof of \eqref{Eq-Ito-n-weak-estimation-torus}\del{
\eqref{Eq-Ito-n-weak-estimation-1-Torus}} with the replacement of
the spaces $ \mathbb{H}^{1, 2}(\mathbb{T}^2)$ and $ \mathbb{H}^{1+\frac\alpha2, 2}(\mathbb{T}^2)$ by
$ \mathbb{L}^{2}(\mathbb{T}^2)$ respectively $ \mathbb{H}^{\frac\alpha2, 2}(\mathbb{T}^2)$,  we infer that for $ 1\leq \alpha< 2$ 
(here we omit to writ the proof for the dissipative regime, as it is classical.)
\del{\begin{eqnarray}\label{Torus-uniquenss-ito-formula-H-1}
\mathbb{E}\!\!\!\!&{}&\!\!\!\!e^{-r(t\wedge \tau_N)}|w(t\wedge \tau_N)|^2_{\mathbb{L}^{2}}+
2\mathbb{E}\int_0^{t\wedge \tau_N}e^{-r(s)}|w(s)|^2_{\mathbb{H}^{\frac\alpha2, 2}}ds \nonumber\\
&\leq& \mathbb{E}\int_0^{t\wedge \tau_N}e^{-r(s)}|| G(u^1(s))- G(u^2(s))||^2_{L_Q(\mathbb{L}^2)}ds\nonumber\\
& -& \mathbb{E}\int_0^{t\wedge \tau_N} e^{-r(s)} \big(2\langle B(w(s)),  u^1(s)\rangle-
r'(s)|w(s)|^2_{\mathbb{L}^{2}}\big)ds\nonumber\\
& \leq & c_N\mathbb{E}\int_0^{t\wedge \tau_N}e^{-r(s)}\big(|w(s)|^2_{\mathbb{L}^2}+
|u^1(s)|_{\mathbb{H}^{1, 2}}|w(s)|^{\frac2\alpha}_{\mathbb{H}^{\frac{\alpha}{2}, 2}}
|w(s)|^{2\frac{\alpha-1}\alpha}_{{\mathbb{L}^{2}}}-
r'(s)|w(s)|^2_{\mathbb{L}^{2}}\big)ds\nonumber\\
& \leq & c_N
\mathbb{E}\int_0^{t\wedge \tau_N}e^{-r(s)}\big(|w(s)|^2_{\mathbb{L}^2}+ 2 c|w(s)|^2_{\mathbb{H}^{\frac\alpha2, 2}}+
2c_1|u^1(s)|^{\frac{\alpha}{\alpha-1}}_{\mathbb{H}^{1, 2}}|w(s)|^2_{\mathbb{L}^{2}}-
r'(s)|w(s)|^2_{\mathbb{L}^{2}}\big)ds.\nonumber\\
\end{eqnarray}}
\begin{eqnarray}\label{Torus-uniquenss-ito-formula-H-1}
\mathbb{E}\!\!\!\!&{}&\!\!\!\!e^{-r(t\wedge \tau_N)}|w(t\wedge \tau_N)|^2_{\mathbb{L}^{2}}+
2\mathbb{E}\int_0^{t\wedge \tau_N}e^{-r(s)}|w(s)|^2_{\mathbb{H}^{\frac\alpha2, 2}}ds \nonumber\\
&\leq& \mathbb{E}\int_0^{t\wedge \tau_N}e^{-r(s)}|| G(u^1(s))- G(u^2(s))||^2_{L_Q(\mathbb{L}^2)}ds\nonumber\\
& -& \mathbb{E}\int_0^{t\wedge \tau_N} e^{-r(s)} \big(2\langle B(w(s)),  u^1(s)\rangle-
r'(s)|w(s)|^2_{\mathbb{L}^{2}}\big)ds\nonumber\\
& \leq & c_N\mathbb{E}\int_0^{t\wedge \tau_N}e^{-r(s)}\big(|w(s)|^2_{\mathbb{L}^2}+
| u^1|_{\mathbb{H}^{1+\frac\alpha2, 2}}|w|^{\frac{2-\alpha}{\alpha}}_{\mathbb{H}^{\frac{\alpha}{2}, 2}}
|w|^{\frac{3\alpha-2 }{\alpha}}_{{\mathbb{L}^{2}}}-
r'(s)|w(s)|^2_{\mathbb{L}^{2}}\big)ds\nonumber\\
& \leq & c_N
\mathbb{E}\int_0^{t\wedge \tau_N}e^{-r(s)}\big(|w(s)|^2_{\mathbb{L}^2}+ 2 c|w(s)|^2_{\mathbb{H}^{\frac\alpha2, 2}}+
2c_1|u^1(s)|^{\frac{2\alpha}{3\alpha-2}}_{\mathbb{H}^{1+\frac\alpha2, 2}}|w(s)|^2_{\mathbb{L}^{2}}-
r'(s)|w(s)|^2_{\mathbb{L}^{2}}\big)ds.\nonumber\\
\end{eqnarray}
We choose $ c<1$ and $ r'(s)= 2c_1|u^1(s)|^{\frac{2\alpha}{3\alpha-2}}_{\mathbb{H}^{1+\frac\alpha2, 2}}$ and replace in
\eqref{Torus-uniquenss-ito-formula-H-1}, we end up with the simple formula
\begin{eqnarray}
\mathbb{E}e^{-r(t\wedge \tau_N)}|w(t\wedge \tau_N)|^2_{\mathbb{L}^{2}}&+& 2(1-c)\mathbb{E}\int_0^{t\wedge \tau_N}
e^{-r(s)}|w(s)|^2_{\mathbb{H}^{\frac\alpha2, 2}}ds\nonumber\\
&\leq& c_N\mathbb{E}\int_0^{t}e^{-r(s\wedge \tau_N)}|w(s\wedge \tau_N)|^2_{\mathbb{L}^2}ds.
\end{eqnarray}
By application of Gronwall's lemma, we get
$ \forall t\in [0, T]$ the random variable $|w(t\wedge \tau_N)|^2_{\mathbb{L}^{2}} = 0, \, P-a.s.$ as much as 
$ P( \int_0^{t\wedge \tau_N }|u^1(s)|^{\frac{2\alpha}{3\alpha-2}}_{\mathbb{H}^{1+\frac\alpha2, 2}}ds <\infty)= 1$. 
This last statement is confirmed thanks to 
\eqref{cond-solu-torus-H1} and the condition $ 1\leq \alpha<2$. 
The proof is then achieved once we remark that thanks to Chebyshev inequality and \eqref{cond-solu-torus-H1},
we have $ \lim_{N\rightarrow \infty}\tau_N = T, P-a.s.$
\subsection*{Proof of the space regularity.}
In the aim to get Estimate \eqref{propty-of-2D-global-mild-sol},\del{(the first term of this estmate),}\del{\eqref{est-u-torus-H-1-q-1}} we use the regularization effect of the vorticity. Let $ (u(t), t\in [0, T])$ be a
weak solution of Equation \eqref{Main-stoch-eq} in the sense of Definition
\ref{def-variational solution}. Thanks to Appendix \ref{sec-Passage Velocity-Vorticity}, the $ curl u \del{\in L^{2}(\mathbb{T}^2)}$ is a weak solution 
of  Equation \eqref{Eq-vorticity-Torus-2-diff}. We know from \cite{Debbi-scalar-active} that this equation\del{ \eqref{Eq-vorticity-Torus-2-diff}} admits a unique global solution  which is simultaneously weak and mild and satisfies
\begin{equation}\label{eq-est-theta-brut}
 \mathbb{E}\Big(\sup_{[0, T]}| \theta(t)|_{L^q}^q + \int_0^T
 | \theta(t)|_{H^{\frac\alpha2, 2}}^2dt\Big)<\infty,
\end{equation}
for $q_0, q $ and $ \alpha$
being characterized as in $(6.3.2.)$\del{Theorem \ref{Main-theorem-strog-Torus}} and  
provided that $ curl u_0$\del{$ curl u_0\in L^p(\Omega, \mathcal{F}_0, P; L_1^{q_0}(O))$} fulfills  \eqref{eq-curl-u-0-torus} and 
$ \tilde{G}$, defined by \eqref{inte-g-tilde}, satisfies the Lipschitz and the growth conditions, i.e.  $ \tilde{G}$ satisfies \eqref{Eq-Cond-Lipschitz-Q-G} and \eqref{Eq-Cond-Linear-Q-G}, with   
$R_Q(\mathbb{L}^2, \mathbb{H}^{\delta, q})$ in the LHSs and 
$ \mathbb{H}^{\delta, q}$ in the RHSs are replaced by 
$R_Q(L^2, L^q) $ and $ L^q$ respectively. As \eqref{eq-curl-u-0-torus} is fulfilled by assumption, we check that the two latter conditions are also satisfied. In fact, thanks to the definition of $ \tilde{G}$, Assumption $(\mathcal{C})$, with $ \delta =1$
and Lemma \ref{lem-R1-bounded},\del{the fact that the
operator $\mathcal{R}^{1}$ is bounded, on $\mathbb{H}^{s, q}(\mathbb{T}^2)$,} we get
\begin{eqnarray}
|| \tilde{G}(\theta)||_{R_\gamma(L^2, L^q)}&=&|\big[\sum_{k\in\Sigma}| curl G(\mathcal{R}^1(\theta))Q^\frac12 e_k|^2\big]^\frac12|_{L^q}
\leq c|\big[\sum_{k\in\Sigma}| \partial_j G(\mathcal{R}^1(\theta))Q^\frac12 e_k|^2\big]^\frac12|_{L_2^q}\nonumber\\
&\leq & c ||G(\mathcal{R}^1(\theta))||_{R_\gamma(\mathbb{L}^2, \mathbb{H}^{1, q})}\leq c(1+|\mathcal{R}^1(\theta)|_{\mathbb{H}^{1, q}})\leq c(1+|\theta|_{L^q}). 
\del{\sum_{k\in\Sigma}| curl \sigma_k(\mathcal{R}^1(\theta(s)))|_{L^q}\nonumber\\
&\leq &c \sum_{k\in\Sigma}| \sigma_k(\mathcal{R}^1(\theta(s)))|_{H^{1,q}}
\leq c (1+| \mathcal{R}^1(\theta(s))|_{H^{1,q}})\leq c (1+| \theta(s)|_{H^{1,q}}).}
\end{eqnarray} 
By the same way, we prove the Lipschitz condition. Estimate \eqref{propty-of-2D-global-mild-sol} follows from \eqref{eq-est-theta-brut} and Lemma \ref{lem-basic-curl-gradient}.

\del{\section{Global existence and uniqueness of the mild solution of the multi-dimensional FSNSEs.}\label{sec-global-mild-solution}
Recall that we have constructed in Section \ref{sec-1-approx-local-solution}, see also Appendix \ref{appendix-local-solution}, a local mild solution $ (u, \tau_\infty)$, where $ \tau_\infty $  is defined by 
\eqref{Eq-def-tau-n-delta} and \eqref{Eq-def-tau-delta} and $ u(t)=u_N(t),$ for $ t\leq \tau_N$\footnote{we use capital N to avoid confusion with the Faedo-Galerkin approximations.}, provided that $ \alpha \in(1+\frac dq, 2]$, $ q>d$ and $u_0$ and $ G$ satisfying assumptions $ (\mathcal{B})$ and $ (\mathcal{C})$ respectively. The solution can start from an $ \mathbb{L}^q-$ initial data, therefore we take  $ \delta =0$. To prove that the local solution is global, it is sufficient to prove that the $ \mathbb{L}^q-$norm does not explode,  i.e. for $ P-a.s.$, $ \tau_\infty =T$.

In the case $ O=\mathbb{T}^2$, thanks to Lemma \ref{lem-basic-curl-gradient} and \cite[Theorem 2.6]{Debbi-scalar-active} and according to the cases cited in $ (3.6.2)$, we infer that
\begin{equation}\label{est-nabla-u-theta-Global-existence}
 \exists c>0, s.t. \forall N\in \mathbb{N}_0, \mathbb{E}\sup_{[0, \tau_N]}|\nabla u(t)|^q_{q}\leq c \mathbb{E}\sup_{[0, \tau_N]}|\theta(t)|^q_{L^{q}}\leq c<\infty.
\end{equation}
Thus both the $\mathbb{L}^q$  and the $\mathbb{H}^{1, q}$  norms do not explode, therefore the solution is global. The estimation \eqref{est-nabla-u-theta-Global-existence} is also enough to prove the uniqueness as we shall show for the general case bellow, see also \cite{Debbi-scalar-active}. The regularity in \eqref{propty-of-2D-global-mild-sol} is a consequence of the application of \cite[Theorem 2.6]{Debbi-scalar-active} and Lemma \ref{lem-basic-curl-gradient} as we have seen in Section \ref{sec-Torus}.

\noindent For the general case, we adapt to the stochastic framework, the method used in \cite{Giga-al-Globalexistence-2001} and than apply Hasminskii's criteria \cite{BrzezniakDebbi1,  Brzez-Beam-eq, Hasminski-book-80}.\del{ It is easily seen that Thanks to Assumption $ (\mathcal{C}_b)$ the OU process, $ (z(t\wedge \tau_\infty), t\in[0, T])$ defined by  \eqref{Eq-z-t}, enjoys
\begin{equation}\label{cond-z-nabla-z}
 \mathbb{E}\sup_{[0, \tau)}(|z(s)|^2_{\mathbb{L}^q}+ |\nabla z(s)|^2_{q}) 
\leq c<\infty.
\end{equation}} Using Lemma \ref{Giga-Mikayawa-solutionLr-NS} and Lemma \ref{Lem-semigroup}, the continuity of the 
Helmholtz projection and  H\"older inequality, we get $P-a.s$ for all $ N\in \mathbb{N}_0$ and  $ t\in [0, T]$, 
\begin{eqnarray}\label{est-b-global-mild-solu-0}
|\int_0^{t\wedge\tau_N} e^{-(t\wedge\tau_N-s)A_\alpha}B(u(s))ds|_{\mathbb{L}^q}
&\leq & c \int_0^{t\wedge\tau_N} |A_\alpha^\frac d{\alpha q} e^{-(t\wedge\tau_N-s)A_\alpha}|_{\mathcal{L}(\mathbb{L}^{\frac q2})} |\Pi((u(s)\nabla)u(s))|_{\mathbb{L}^{\frac q2}}ds\nonumber\\
&\leq& c\int_0^{t\wedge\tau_N} (t\wedge\tau_N-s)^{-\frac d{\alpha q}}  |u(s)|_{\mathbb{L}^q}|\nabla u(s)|_{q}ds.
\end{eqnarray}
Therefore, as the local solution $ (u, \tau_\infty)$ satisfies Equation \eqref{Eq-Mild-Solution-stoped} up to $ \tau_\infty$, we infer that \del{for all $N \in \mathbb{N}_0 $, we have $P-a.s$ for all $ N\in \mathbb{N}_0$ and $ t\in [0, T]$,\del{ Thus for all $ t\leq \tau_\infty$ we have}
\begin{equation}\label{v-mild-sol-aux-pbm}
 |u(t\wedge\tau_N)|_{\mathbb{L}^q}= |e^{-tA_\alpha}u_0|_{\mathbb{L}^q}+ |\int_0^t e^{-(t-s)A_\alpha}B(u(s))|_{\mathbb{L}^q} + |z(s)|_{\mathbb{L}^q}.
\end{equation}}
\begin{eqnarray}\label{est-b-global-mild-solu-0}
|u(t\wedge\tau_N)|_{\mathbb{L}^q}&\leq& c\Big(|u_0|_{\mathbb{L}^q}+ c\int_0^{t\wedge\tau_N} (t\wedge\tau_N-s)^{-\frac d{\alpha q}}  |u(s)|_{\mathbb{L}^q}|\nabla u(s)|_{q}ds+ |z(t\wedge\tau_N)|_{\mathbb{L}^q}\Big),\nonumber\\
\end{eqnarray}
where the stopped Ornstein-Uhlenbeck process $( z(t\wedge\tau_N), t\in [0, T])$ is obtained as in  Remark \ref{Rem-1}. By application of  Gronwall's lemma, we conclude that  $P-a.s$ for all $ N\in \mathbb{N}_0$ and $ t\in [0, T]$, 
\del{ for all $ t\leq \tau_N$,}
\begin{eqnarray}
|u(t\wedge\tau_N)|_{\mathbb{L}^q}&\leq& c\Big(|u_0|_{\mathbb{L}^q}+ \sup_{[0, \tau_\infty)}|z(s)|_{\mathbb{L}^q}\Big)\exp{c\int_0^{t\wedge\tau_N} (t\wedge\tau_N -s)^{-\frac d{\alpha q}}|\nabla u(s)|_{q}ds}.
\end{eqnarray}
\noindent We apply the increasing function $\ln^+x=\max\{0,\ln x\}$  to
 both sides of the above inequality and than we use the classical
inequalities for $ a, b > 0$, $\ln^+(a+b)\leq \ln^+(a)+ \ln^+(b)$, $
\ln^+(ab)\leq \ln^+(a)+\ln^+(b)$ and $a\leq 1+a^2$, the elementary property 
\begin{equation}\label{elementary-property}
\exists c>0,\;\;s.t.\; ln^+x \leq x +c,\;\; \forall x>0.
\end{equation}
\noindent Estimate  \eqref{Eq-nabla-z-t}\del{{cond-z-nabla-z}} and  Condition \eqref{cond-global-mild-solu}, we get for all $ N\in \mathbb{N}_0$ and $ t\in[0, T]$,
\begin{eqnarray}
\mathbb{E}\ln^+|u(t\wedge \tau_N)|_{\mathbb{L}^q}&\leq& c\Big(\mathbb{E}\ln^+|u_0|_{\mathbb{L}^q}+ \mathbb{E}\ln^+\sup_{[0, \tau_\infty)}|z(s)|_{\mathbb{L}^q} +\mathbb{E}\sup_{[0, \tau_\infty)}\int_0^t (t-s)^{-\frac d{\alpha q}}|\nabla u(s)|_{q}ds\Big)\nonumber\\
&\leq& c\Big(1+\mathbb{E}\ln^+|u_0|_{\mathbb{L}^q}\Big)<\infty.
\end{eqnarray}
Now, we introduce the function $ V: u\in \mathbb{L}^q(O) \mapsto \mathbb{R}_+ \ni V(u)=ln^+(|u|_{\mathbb{L}^q})$. 
Obviously  $V$ is uniformly continuous on bounded sets and from the 
calculus above, it is easy to check that $ V$ is a Lyapunov function.  In fact, 
\begin{eqnarray} V &\geq&  0 \; \text{on}\;\;  \mathbb{L}^q(O), \label{2.1}
\\
q_N &:=& \inf_{|u|_{\mathbb{L}^q\ge N}} ln^+(|u|_{\mathbb{L}^q}) \to \infty \text{ as
 }\;N\to\infty, \label{2.2}
\\
\mathbb{E} ln^+(|u_0|_{\mathbb{L}^q}) &<& c+ \mathbb{E}|u_0|_{\mathbb{L}^q}<\infty, \label{2.3}
\\ 
\text{and} &&\nonumber\\
\mathbb{E} V( u(t\land \tau_N)) &\leq&  c\bigl(1+\mathbb{E} V(u_0) \bigr)
\text{ for all } t\ge 0,\;  N\in\mathbb{ N}_0. \label{2.5}
\end{eqnarray}
Consequently, 
\begin{eqnarray}\label{eq-prob-tau-n=0}
\lim_{N\rightarrow +\infty} P \bigl\{ \tau_N<t\bigr\} &\leq&
 \lim_{N\rightarrow +\infty} \frac1{q_N}\mathbb{E}
\bold 1_{\{\tau_N <t\}} V( u(t\land \tau_N))\nonumber\\
&\leq& \lim_{N\rightarrow +\infty} \frac1{q_N}c\bigl(1+\mathbb{E} V( u_0)\bigr)=0.
\end{eqnarray}
\noindent By this we achieve the proof of the global existence of the mild solution. The next step is to prove the uniqueness. Let  $ u^1$ and  $ u^2$  be two global mild solutions
satisfying\del{ the estimation \eqref{eq-last-est} and} $ P(u^1(0) = u^2(0)= u_0)=1$  and let $
u^0:= u^1- u^2$. We define for $ R>0$ the stopping times, see Appendix \ref{append-stop-time} for the proof,
\begin{equation}
 \tau^i_R := \inf\{t\in (0, T), \;\; s.t. \;\; |u^i(t) |_{\mathbb{L}^q}> R\}, \;\;\; \inf(\emptyset)=+\infty\;\;\; \text{and}\;\;\; \tau_R:= \tau^1_R\wedge \tau^2_R.
\end{equation}
\del{with the understanding that $ \inf(\emptyset)=+\infty$. } We show that for all $ R>0 $, the stopped processes  of $( u^1(t), t\in [0, T\wedge \tau_R])$ and $ ( u^2(t), t\in [0, T\wedge \tau_R]) $  are modifications of each other.
Than thanks \del{to the facts,
\begin{equation}\label{eq-lim-tau-n}
\lim_{R\rightarrow \infty}P(\tau_R<T) \leq \sum_{i=1}^2 \lim_{R\rightarrow \infty}P(\tau^i_R<T)
\leq \sum_{i=1}^2 \lim_{R\rightarrow \infty} \frac{\mathbb{E}|u^i(t)|^k_{L^q}}{R^k}=0,
\end{equation}}
to \eqref{eq-prob-tau-n=0} we conclude that the same is true for $ (u^1(t), t\in[0, T])$ and $ (u^2(t), t\in[0, T])$.
The stopped process $u^0$ satisfies the following equation $ P-a.s$ for all $ t\in [0, T]$, see \cite{BrzezniakDebbi1, Brzez-Beam-eq, Debbi-scalar-active},
\begin{equation}\label{eq-theta-0}
u^0(t\wedge \tau_R)= I^1_{\tau_R}(t\wedge \tau_R) +I^2_{\tau_R}(t\wedge \tau_R),
\end{equation}
where
\begin{equation}
 I^1_{\tau_R}(t\wedge \tau_R):=\int_0^{t\wedge \tau_R}
 e^{(t\wedge \tau_R-s)A_\alpha}(B(u^1(s\wedge \tau_R))-B(u^2(s\wedge \tau_R)))ds
\end{equation}
and
\begin{equation}
I^2_{\tau_R}(t\wedge \tau_R):= \int_0^{t}
1_{[0, \tau_R)}(s) e^{(t-s)A_\alpha}(G(u^1(s\wedge \tau_R))-G(u^2(s\wedge \tau_R)))dW(s).
\end{equation}
 First let us remark that 
\begin{equation}\label{eq-I-1-out-Tau}
1_{[0, \tau_R]}(t) I^1_{\tau_R}(t\wedge \tau_R)= 1_{[0, \tau_R]}(t)\int_0^{t}
 1_{[0, \tau_R]}(s)e^{(t-s)A_\alpha}(B(u^1(s\wedge \tau_R))-B(u^2(s\wedge \tau_R)))ds.
\end{equation}
Now let $ p>2$. Using formulas \eqref{formula-B-v1-B-v2} and \eqref{eq-I-1-out-Tau}, a similar calculation as in the proof of 
Proposition \ref{Prop-Main-I} and H\"older inequality in $ L^1(0, t; t^{-\gamma}dt)$ with $ \gamma := \frac1\alpha(1+\frac dq)<1$, 
we estimate $I^1_{\tau_R}(t\wedge \tau_R)$ as follow
\begin{eqnarray}\label{eq-est-B-theta-0}
\mathbb{E}\!\!\!\!\!&|&\!\!\!\!\!1_{[0, \tau_R]}(t)I^1_{\tau_R}(t\wedge \tau_R)|_{\mathbb{L}^q}^p
= \mathbb{E} |1_{[0, \tau_R]}(t)\int_0^{t}
 e^{(t-s)A_\alpha}1_{[0, \tau_R]}(s)(B(u^1(s\wedge \tau_R))-B(u^2(s\wedge \tau_R)))ds|^p_{\mathbb{L}^q} \nonumber\\
&\leq& c\int_0^{t}(t-s)^{-\gamma}\mathbb{E}\Big[|1_{[0, \tau_R]}(s)
(u^1(s\wedge \tau_R)-u^2(s\wedge \tau_R))|^p_{\mathbb{L}^q}(\sup_{[0, \tau_R]}(|u^1(s)|_{L^q}^p+ |u^2(s)|_{\mathbb{L}^q}^p)) \Big]ds \nonumber\\
&\leq& cR^p\int_0^{t}(t-s)^{-\gamma}\mathbb{E}|1_{[0, \tau_R]}(s)
(u^1(s\wedge \tau_R)-u^2(s\wedge \tau_R))|^p_{\mathbb{L}^q} ds. \nonumber\\
\end{eqnarray}
\noindent Furthermore, using \cite[Proposition 4.2.]{Neerven-Evolution-Eq-08} and a similar calculation 
as in Lemma \ref{lem-est-z-t}, in particular Assumption $(\mathcal{C})$, we infer the existence of
$ c_R>0, \; 0<\gamma_1<\frac12$ such that
\begin{eqnarray}\label{eq-est-G-theta-0}
\mathbb{E}\!\!\!\!\!&[&\!\!\!\!\!1_{[0, \tau_R]}(t)I^2_{\tau_R}(t\wedge \tau_R)|_{\mathbb{L}^q}^p\nonumber\\
&=&\mathbb{E} |1_{[0, \tau_R]}(t)\int_0^{t}
 e^{(t-s)A_\alpha}1_{[0, \tau_R]}(s)(G(u^1(s\wedge \tau_R))-G(u^2(s\wedge \tau_R)))dW(s)|^p_{\mathbb{L}^q} \nonumber\\
&\leq& c\mathbb{E}\Big(\int_0^{t}(t-s)^{-\gamma_1}||1_{[0, \tau_R]}(s)(
G(u^1(s\wedge \tau_R))-G(u^2(s\wedge \tau_R)))||^2_{R_\gamma(\mathbb{L}^2, \mathbb{L}^q)}ds \Big)^\frac p2 \nonumber\\
&\leq& c_R\int_0^{t}(t-s)^{-\gamma_1}\mathbb{E}|1_{[0, \tau_R]}(s)
(u^1(s\wedge \tau_R)-u^2(s\wedge \tau_R))|^p_{\mathbb{L}^q} ds. 
\end{eqnarray}
Using  \eqref{eq-est-B-theta-0} and \eqref{eq-est-G-theta-0} and \eqref{eq-theta-0}, we infer the existence of
an $L^1-$integrable function  $ \psi: (0, T]\rightarrow \mathbb{R}_+$ such that
\begin{equation}\label{eq-theta-0-with-phi}
\mathbb{E}|1_{[0, \tau_R]}u^0(t\wedge \tau_R)|_{\mathbb{L}^q}^p\leq c_R\int_0^{t}\psi(t-s)\mathbb{E}|1_{[0, \tau_R]}(s)
u^0(s\wedge \tau_R)|^p_{\mathbb{L}^q} ds.
\end{equation}
Using Gronwall Lemma, we get for all $ t\in [o, T]$,
\begin{equation}
 | u^1(t\wedge\tau_R)- u^2(t\wedge\tau_R)|_{\mathbb{L}^q}=0. \;\;\; a.s., \;\;\; \forall t\in [0, T].
\end{equation}
In particular, thanks to Estimate \eqref{eq-prob-tau-n=0}, we infer that for all $ t\in [o, T]$,
\begin{equation}
  u^1(t)= u^2(t)=0. \;\;\; a.s. \;\;\; \del{\forall t\in [0, T]}
\end{equation}
\del{\noindent To complete the proof of the uniqueness let us recall that Estimation \eqref{eq-last-est} is satisfied by the construction of the solution ($ \mathbb{E}\sup_{[0, T]}|u(t)|^p_{\mathbb{L}^q}<\infty$).}

\del{\noindent The whole calculus above remains valid to prove the results in $(3.6.3)$, in the exception that we have not \eqref{est-nabla-u-theta-Global-existence}. Hence, to estimate \eqref{eq-lastln-u}, we use \eqref{cond-mild-3d} and \eqref{elementary-property}.} }

\section{Martingale solution of the multi-dimensional FSNSEs.}\label{sec-Marting-solution}
In this section, we prove Theorem \ref{Main-theorem-martingale-solution-d}. The main ingredients are Faedo-Galerkin approximations,
compactness, Skorokhod embedding theorem and the representation theorem. In particular, once we prove Lemma \ref{lem-bounded-W-gamma-p} bellow, we can follow the same scheme e.g. as in \cite{Capinski-peaszat-Mart-SNSE-01, Flandoli-Gatarek-95}, see also similar calculus for the fractional stochastic scalar active equation in \cite{Debbi-scalar-active}. Thus we omit to give full details.\del{ in Appendix \ref{Appendix-Martingale-solu}.}

\begin{lem}\label{lem-bounded-W-gamma-p}
The sequence $ (u_n)_n$ of solutions of the equations \eqref{FSBE-Galerkin-approxi} is uniformly bounded in the space
\begin{eqnarray}\label{Eq-W-}
L^2(\Omega, W^{\gamma, 2}(0, T; \mathbb{H}^{-\delta', 2}(O))
\del{H_d^{-\delta', 2}(O))}\cap L^2(0, T; \mathbb{H}^{\frac\alpha2, 2}(O))),
\end{eqnarray}
where  $ \delta'\geq_1\max\{\alpha, 1+\frac d{2}\}$ and $ \gamma <\frac12$.
\end{lem}
\begin{proof}
Thanks to Lemma \ref{lem-unif-bound-theta-n-H-1-domain}, it is sufficient to prove that $ (u_n(t), t\in [0, T])$ is
uniformly bounded in $L^2(\Omega, W^{\gamma, 2}(0, T; \mathbb{H}^{-\delta', 2}(O))$. We recall that the Besov-Slobodetski space $W^{\gamma,
p}(0, T; E)$, with $ E$ being a Banach space, $ \gamma \in (0, 1)$ and $ p\geq 1$, is the space
of all $ v\in L^P(0, T; E) $ such that
\begin{eqnarray}
||v||_{W^{\gamma, p}}:= \left(\int_0^T|v(t)|_E^pdt+
\int_0^T\int_0^T\frac{|v(t)-v(s)|_E^p}{|t-s|^{1+\gamma p}}
dtds\right)^{\frac1p}<\infty.
\end{eqnarray}
\noindent As  $(u_n(t), t\in [0, T])$ is the strong solution  of the finite dimensional  stochastic
differential equation \eqref{FSBE-Galerkin-approxi}, then  $u_n(t)$  is the solution of the stochastic integral equation 
\begin{equation}\label{FSBE-Integ-solu-Galerkin-approxi}
u_n(t)= P_nu_0 + \int_0^t(-A_\alpha u_n(r) + P_nB(u_n(r))dr + \int_0^tP_nG(u_n(r))\,dW_n(r),\; a.s.,\\
\end{equation}
for all $t\in [0, T]$. We denote by 
\begin{equation}\label{Eq-Drift-term}
I(t):=  \int_0^t(-A_\alpha u_n(r) + P_nB(u_n(r))dr
\end{equation}
and
\begin{equation}\label{Eq-Drift-term}
J(t):= \int_0^tP_nG
(u_n(r))\,dW_n(r).
\end{equation}
\noindent We prove that $ I(\cdot)$ is uniformly bounded in  $L^2(\Omega; W^{\gamma, 2}(0, T; \mathbb{H}^{-\delta', 2}(O))$ 
and that the stochastic term $ J(\cdot )$ is uniformly bounded in $ L^2(\Omega; W^{\gamma, 2}(0, T;  \mathbb{L}^2(O))$, for all $ \gamma<\frac12$.\del{ as the stochastic term $ J$ is more regular then the drift term, see e.g. \cite[Lemma 2.1]{Flandoli-Gatarek-95}. \\}
Let  $ \phi \in \mathbb{H}^{{\delta'}, 2}(O)$, using Identity \eqref{Eq-3lin-propsym}, we get
\begin{eqnarray}
| {}_{\mathbb{H}^{-{\delta'}, 2}}\langle P_nB(u_n(r)), \phi\rangle_{\mathbb{H}^{{\delta'}, 2}}|
&=& |\langle u_n(r) \cdot
\nabla P_n\phi, u_n(r)\rangle_{\mathbb{L}^{2}}|\nonumber\\
&\leq& |\nabla P_n\phi|_{L^\infty}| u_n(r)|^2_{\mathbb{L}^{2}}.
\end{eqnarray}
Thanks to \cite[Remark 4 p 164, Theorem 3.5.4.ps.168-169 and Theorem 3.5.5 p 170]{Schmeisser-Tribel-87-book} for $ O= \mathbb{T}^d$, to
\cite[Theorem 7.63  and point 7.66]{Adams-Hedberg-94} for $ O$ being a bounded domain and to
the condition $ {\delta'}>1+\frac d{2}$,
we deduce for  $ 0<\epsilon < \delta'-1-\frac d{2}$,
 $$ |\nabla P_n\phi|_{L^\infty} \leq c |\nabla P_n\phi|_{H^{\epsilon+\frac d2, 2}} \leq c
|\phi|_{H_d^{1+\epsilon+\frac d2, 2}} \leq c |\phi|_{\mathbb{H}^{\delta', 2}}.$$
Therefore,
\begin{eqnarray}\label{}
 |P_nB(u_n(r))|_{\mathbb{H}^{-\delta', 2}}\leq c |u_n(r)|_{\mathbb{L}^{2}}^2
\end{eqnarray}
 and
\begin{eqnarray}\label{eq-unif-int-I(t)}
\int_0^T|I(t)|_{\mathbb{H}^{-\delta', 2}}^2dt&\leq& c\int_0^T\int_0^t \big(|(-A_\alpha
u_n(r)|^2_{\mathbb{H}^{-\delta', 2}} + |P_nB(u_n(r))|^2_{\mathbb{H}^{-\delta', 2}}\big)drdt\nonumber\\
&\leq& c\int_0^T\int_0^t\big(|
u_n(r)|^2_{\mathbb{L}^{ 2}} + |u_n(r)|^4_{\mathbb{L}^{2}}\big)drdt.
\end{eqnarray}
Moreover, using H\"older inequality and arguing as before, we get for $ t\geq s > 0$,
\begin{eqnarray}\label{eq-unif-int-I(t)-I(s)}
|I(t)- I(s)|^2_{\mathbb{H}^{-\delta', 2}}&=& |\int_s^t(-A_\alpha u_n(r) +
P_nB(u_n(r))dr|^2_{\mathbb{H}^{-{\delta'}, 2}}\nonumber\\
&\leq & C(t-s)\left(\int_s^t(| u_n(r)|^2_{\mathbb{L}^{2}} +
|u_n(r)|^4_{\mathbb{L}^{2}})dr \right).
\end{eqnarray}
From \eqref{eq-unif-int-I(t)},  \eqref{eq-unif-int-I(t)-I(s)} and \eqref{Eq-Ito-n-weak-estimation-1-bounded}, 
we have for  $ \gamma <\frac12$,
\begin{eqnarray}\label{eq-unif-int-I(t)x2}
\mathbb{E}\big(\int_0^T|I(t)|_{\mathbb{H}^{-\delta', 2}}^2dt&+&\int_0^T\int_0^T
\frac{|I(t)- I(s)|^2_{\mathbb{H}^{-\delta',
2}}}{|t-s|^{1+2\gamma }} dtds\big)^{\frac12} \nonumber\\
&\leq& C\mathbb{E}\left(\int_0^T(| u_n(r)|^2_{\mathbb{L}^{2}} +
|u_n(r)|^4_{\mathbb{L}^{2}})dr
\right)^\frac12 \leq C<\infty.
\end{eqnarray}
Now, we estimate the stochastic term $ J$. 
Using the stochastic isometry, the contraction property of $ P_n$ and Assumption $(\mathcal{C})$,( Condition  \eqref{Eq-Cond-Linear-Q-G} with $ q=2$ and $ \delta =0$), we get
\begin{eqnarray}
\int_0^T\mathbb{E}|\int_0^tP_nG(u_n(r))dW_n(r)|_{\mathbb{L}^{ 2}}^2dt&\leq&
C\int_0^T\mathbb{E}\int_0^t||G(u_n(r))||^2_{L_Q(\mathbb{L}^2)}drdt\nonumber\\
&\leq&
C\int_0^T\mathbb{E}\int_0^t(1+|u_n(r)|^2_{\mathbb{L}^2})drdt \leq c<\infty.
\end{eqnarray}
Moreover, for $ t\geq s> 0$ and $ \gamma <\frac12$, the same ingredients
above yield to
\begin{eqnarray}
\mathbb{E}\int_0^T\int_0^T\frac{|J(t)-
J(s)|^2_{\mathbb{L}^{2}}}{|t-s|^{1+2\gamma }} dtds &\leq&
C\mathbb{E}\int_0^T\int_0^T\frac{\int_s^t||G(u_n(r))||^2_{L_Q(\mathbb{L}^2)}dr}{|t-s|^{1+2\gamma
}} dtds \nonumber\\
&\leq& C\mathbb{E}\sup_{[0, T]}(1+|u_n(t)|^2_{\mathbb{L}^{2}})
\int_0^T\int_0^T|t-s|^{-2\gamma } dtds \leq c <\infty.
\end{eqnarray}
The proof of the lemma is now completed.
\end{proof}

To prove the existence of a martingale solution, we use \del{consider the Gelfand triplet \eqref{Gelfand-triple-Domain} and 
 use lemmas \ref{lem-unif-bound-theta-n-H-1-domain} and \ref{lem-bounded-W-gamma-p} 
and }the following compact embedding, see \cite[Theorem 2.1]{Flandoli-Gatarek-95},
\begin{equation}
 W^{\gamma, 2}(0, T; \mathbb{H}^{-\delta', 2}(O))\cap \mathbb{L}^2(0, T; \mathbb{H}^{\frac\alpha2, 2}(O)) 
 \hookrightarrow L^2(0, T; \mathbb{L}^2(O)).
\end{equation}
\del{$ L^2(0, T; \mathbb{H}^{\frac\alpha2, 2}(O)) \hookrightarrow L^2(0, T; \mathbb{L}^2(O))$,} Therefore, we deduce that the sequence of laws $ (\mathcal{L}(u_n))_n$  is tight on $ L^2(0, T; \mathbb{L}^2(O))$. 
Thanks to Prokhorov's theorem there exists a  subsequence, still denoted $ (u_n)_n$, for which  the sequence of laws $ (\mathcal{L}(u_n))_n$ converges  weakly on $ L^2(0, T; \mathbb{L}^2(O))$  to a probability measure $ \mu$. By Skorokhod's embedding theorem, we can construct a probability basis $ (\Omega_*, F_*, \mathbb{F}_*,  P_*)$  and a sequence of $ L^2(0, T; \mathbb{L}^2(O))\cap C([0, T]; \mathbb{H}^{-\delta', 2}(O))-$random variables
$ (u^*_n)_n$ and $ u^*$ such that  $\mathcal{L}(u^*_n) = \mathcal{L}(u_n), \forall n \in \mathbb{N}_0$,  $\mathcal{L}(u^*) = \mu$ and
$ u^*_n \rightarrow u^* a.s.$ in $ L^2(0, T; \mathbb{L}^2(O))\cap C([0, T]; \mathbb{H}^{-\delta', 2}(O))$. Moreover,  $ u^*_n(\cdot, \omega) \in C([0, T]; H_n)$. Thanks to  Lemma \ref{lem-unif-bound-theta-n-H-1-domain} and to the equality in law, we infer that the sequence  $ u^*_n$ converges weakly in $  L^2(\Omega\times [0, T]; \mathbb{H}^{\frac\alpha2, 2}(O))$ and weakly-star in $ L^p(\Omega, L^\infty([0, T]; \mathbb{L}^{2}(O))$ to a limit $ u^{**}$. It is easy to see that $u^{*} = u^{**},\;  dt\times dP-a.e.$ and
\begin{eqnarray}\label{eq-bound-u-*-n-u-*}
\mathbb{E}_*\sup_{[0, T]}| u^*(s)|^p_{\mathbb{L}^2}+ \mathbb{E}_*\int_0^T| u^*(s)|^2_{ \mathbb{H}^{\frac\alpha2, 2}}ds \leq c<\infty.
\end{eqnarray}
 
\del{\begin{equation}
u^{*}(\cdot, \omega)\in L^2(0, T; \mathbb{H}^{\frac\alpha2, 2}(O))\cap L^\infty(0, T; \mathbb{L}^2(O)).
\end{equation} }
We introduce the filtration 
\begin{equation}
(\mathit{G}_n^*)_t:= \sigma\{u^{*}_n(s), s\leq t\}
\end{equation}
and construct (with respect to $ (\mathit{G}_n^*)_t$) the time continuous square integrable martingale
$ (M_n(t), t\in [0, T])$ with trajectories in
$ C([0, T]; \mathbb{L}^2(O))$ by
\begin{equation}
M_n(t):= u_n^*(t) - P_nu_0+\int_0^t A_\alpha u_n^*(s) ds -\int_0^t P_nB(u_n^*(s))ds.
\end{equation}
The equality in law yields to the fact that the quadratic variation is given by 
\begin{equation}
\langle\langle M_n\rangle\rangle_t= \int_0^tP_nG(u^*_n(s))QG(u^*_n(s))^*ds,
\end{equation}
where $ G(u^*_n(s))^*$ is the adjoint of $G(u^*_n(s))$. We prove that, for $ a.s.$, $ M_n(t)$ converges weakly in $ \mathbb{H}^{-\delta', 2}(O)$ to  the martingale $ M(t)$, for all $ t\in [0, T]$, where $ M(t)$ is given by 
\begin{equation}\label{eq-M(t)}
M(t):= u^*(t) - u_0+\int_0^t A_\alpha u^*(s) ds -\int_0^t B(u^*(s))ds.
\end{equation}
\del{In fact, for all $ v\in V_2:= \mathbb{H}^{\delta', 2}$ with $ \delta'>1+\frac d2$, we have $ P-a.s.$,  $ \langle P_nu_0, v\rangle $  converges to $ \langle u_0, v\rangle $, thanks to the fact that $ P_n \rightarrow I$ in $ \mathbb{L}^2(O)$, $ \langle u_n^*(t), v\rangle $  converges to $ \langle u^*(t), v\rangle $ as a consequence of the 
the a.s. convergence in $ L^2(0, T; \mathbb{L}^2(O))$ (in fact, we speak about the convergence of a subsequence but as usual we keep the same notation) and  the weak convergence and the continuity in $\mathbb{H}^{-\delta', 2}(O))$, the term $\int_0^t A_\alpha u_n^*(s) ds$
converges thanks to the weak convergence $ L^2(0, t; \mathbb{L}^2(O))$, for all $ t\in [0, T]$ and the elementry inequality $ \langle A_\alpha u_n^*(t), v\rangle = \langle u_n^*(t), A_\alpha v\rangle $ with $ v \in \mathbb{H}^{\delta', 2}(O)$ and  $\delta'>1+\frac d2>\alpha $. The convergence of $\int_0^t \langle B(u_n^*(s)), v\rangle ds$ is completely described in \cite[Appendix 2]{Flandoli-Gatarek-95}, in particular the condition $ \delta'>1+\frac d2$ implies that $ \partial_j v \in C^0(O)$, which we need to do the calculus. To prove that $ M(t)$ is  a quadratic martingale, we see that for all $ \phi \in C_b(L^2(0, s; \mathbb{L}^2(O)))$ and $ v\in \mathcal{D}(O)$
\begin{equation}
\mathbb{E}(\langle M(t)- M(s), v\rangle \phi(u^*|_{[0, s]}))= \lim_{n\rightarrow +\infty} \mathbb{E}(\langle M_n(t)- M_n(s), v\rangle \phi(u^*|_{[0, s]}))=0
\end{equation}
and 
\begin{eqnarray}
&{}&\mathbb{E}(\langle M(t), v\rangle \langle M(t), y\rangle - \langle M(s), v\rangle \langle M(s), y\rangle -\int_s^t\langle G^*(u^*(r)P_nv, G^*(u^*(r)P_ny \rangle dr)\phi(u^*|_{[0, r]}))\nonumber\\
&=& \lim_{n\rightarrow +\infty} \mathbb{E}(\langle M_n(t), v\rangle \langle M_n(t), y\rangle - \langle M_n(s), v\rangle \langle M_n(s), y\rangle -\int_s^t\langle G^*(u_n^*(r)P_nv, G^*(u_n^*(r)P_ny \rangle dr)\phi(u_n^*|_{[0, r]}))\nonumber\\
&=& 0
\end{eqnarray}}
Some of the main ingredients are the $a.s.$ convergence of $ u_n^*$ in $ L^2(0, T; \mathbb{L}^2)$, $ \partial_j\phi \in C^0 $ and therefore we can estimate $\int_0^t \langle B(u_n^*(s)), v\rangle ds$ by $\int_0^t|B(u_n^*(s))|_{L^1} |v|_{C^1} ds$. Now we apply the representation theorem \cite[Theorem 8.2]{DaPrato-Zbc-92}, we infer that there exists a probability basis $ (\Omega^*, \mathcal{F}^*, P^*, \mathbb{F}^*, W^*)$ such that  
\begin{equation}
M(t)=\int_0^tG(u^*(s))W^*(ds).
\end{equation}
If moreover, $ \alpha \in [\alpha_0(d):= 1+\frac{d-1}{3}, 2]$, then thanks to 
Burkholdy-Davis-Gandy inequality, Assumption \eqref{Eq-Cond-Linear-Q-G}, with $q=2, \delta =0$ and  \eqref{eq-bound-u-*-n-u-*}
\begin{eqnarray}\label{cont-2-martg-stochas}
\mathbb{E}\sup_{[0, T]} |\int_0^tG(u^*(s))dW^*(s)|^2_{\mathbb{L}^{2}}
&\leq& c \mathbb{E}\int_0^T|G^*(u^*(s))|^2_{L_Q(\mathbb{L}^{2})}ds
\leq c (1+ \mathbb{E}\sup_{[0, T]}|u^*(s)|^2_{\mathbb{L}^{2}})<\infty.\nonumber\\
\end{eqnarray}
Further more, using Estimate \eqref{B-u-v-h-alpha-2-d} with $ \eta=0$, the Sobolev embedding 
$ \mathbb{H}^{\frac{d+2-\alpha}{4}, 2}(O)\subset \mathbb{H}^{\frac{\alpha}{2}, 2}(O)$, ( $ 1+\frac{d-1}{3}\leq \alpha \leq 2$) 
and the boundedness of the  operator $ A_\alpha: \mathbb{H}^{\frac{\alpha}{2}, 2}(O) \rightarrow \mathbb{H}^{\frac{-\alpha}{2}, 2}(O)$, we get
\begin{eqnarray}\label{cont-1-martg}
\mathbb{E}\int_0^T\big( |A_\alpha u^*(s)|_{\mathbb{H}^{-\frac\alpha2, 2}}&+& |B(u^*(s))|_{\mathbb{H}^{-\frac\alpha2, 2}}\big)ds \leq c
\mathbb{E} \int_0^T\big( |u^*(s)|_{\mathbb{H}^{\frac\alpha2, 2}} + |u^*(s)|^2_{\mathbb{H}^{\frac{d+2-\alpha}{4}, 2}}\big)ds\nonumber\\
&\leq& c (1+\mathbb{E} \int_0^T|u^*(s)|^2_{\mathbb{H}^{\frac\alpha2, 2}}ds)<\infty. 
\end{eqnarray}
Therefore, using the densely embedding $ \mathbb{H}^{\delta', 2}(O) \hookrightarrow \mathbb{H}^{\frac\alpha2, 2}(O)$, we can see the equality \eqref{Eq-weak-Solution} in the $\mathbb{H}^{\frac\alpha2, 2}- \mathbb{H}^{-\frac\alpha2, 2}-$duality.

\vspace{0.25cm}

{\bf Pathewise uniqueness $ (\ref{Main-theorem-martingale-solution-d}.2)$.}
To prove the uniqueness of the martingale solution under the condition \eqref{uniquness-con-martg}, we follow the scheme of the uniqueness in  
Section \ref{sec-Torus} taking into account the changes of the norms. Let
$ u^1$ and $ u^2$ be two martingale solutions on the same probability basis $ (\Omega^*, \mathbb{F}^*, P^*, W^*)$ and 
such that $u^1$ satisfies \eqref{uniquness-con-martg}.\del{\subsection*{Proof of pathwise uniqueness for the martingale solution.}
\label{sec-subsection-uniqueness-martingale}.} We define  $ \tau_N^i$, $ \tau_N$  and  $ w:=  u^1- u^2$ as in Section \ref{sec-Torus}. 
Then $ w$ satisfies\del{ in $ \mathbb{H}^{-\frac\alpha2, 2}(O)$}
Equation \eqref{eq-uniq-w}  with $ W$ replaced by $ W^*$. We use 
 Ito formula to the product $ e^{-r(t)}|w(t)|^2_{\mathbb{L}^{2}}$, with $(r(t):=
c\int_0^t|u^1(s)|^{\frac{4\alpha}{3\alpha-d-2}}_{\mathbb{H}^{\frac{d+2-\alpha}{4}, 2}}ds, t\in [0, T])$, 
Identity  \eqref{formula-B-v1-B-v2},
\del{Property \eqref{Eq-3lin-propnull},}condition \eqref{Eq-Cond-Lipschitz-Q-G}, with $ q=2, \delta=0$, Estimate \eqref{B-u-v-h-alpha-2-d} with $ \eta=0$ 
and argue as around  \eqref{Torus-uniquenss-ito-formula-H-1} we get the proof of the uniqueness. Combinning this latter result with Yamada-Watabnabe theorem \cite{Kurtz-Yamada-07, Ondrejat-thesis-03}, the global existence of a unique weak-strong solution follows.

\del{\section{{} {d}D-Fractional stochastic Navier-Stokes equation on Bounded domain and on the Torus with no smooth data}\label{sec-Domain}}

\del{\section{Maximal local weak-strong solutions for the multi-dimensional FSNSEs.}\label{sec-Domain}}

\section{Similarity of the 2D-FSNSE and the 3D-NSE.\del{ weak-strong solutions for the multi-dimensional FSNSEs.}}\label{sec-Domain}

In this section, we \del{prove \del{the first part of Theorem \ref{Main-theorem-boubded-2}}$(3.8.1)$. But first let us} illustrate the fact  that the 2D-FSNSE exhibits the same difficulty to prove the existence of the global solution as the 3D-NSE.\del{ This part will be also used in Section \ref{sec-global-weak-solution}.} We follow a similar calculus as 
in Section \ref{sec-Torus}, replacing Property \eqref{vanishes-bilinear-tous-H1} by Property
\eqref{Eq-3lin-propnull}\del{is intrinsic for that  calculus and is valid neither for bounded domains of $ \mathbb{R}^d$ nor for the $ 3D-$torus, we use the property
\eqref{Eq-3lin-propnull}.} and considering\del{ for $ d\in \{2, 3\}$ and $ \alpha_0(d):= 1+\frac{d-1}3\leq \alpha\leq 2$,} 
the  densely continuously embedding Gelfand triple \del{\eqref{Gelfand-triple-Domain}.}
\begin{equation}\label{Gelfand-triple-Domain}
 V_2=V:= \mathbb{H}^{\frac\alpha2, 2} (O)\hookrightarrow \mathbb{L}^{2}(O)\hookrightarrow \mathbb{H}^{-\frac\alpha2, 2}(O)=:V^*. 
\end{equation}
 We obtain the following Lemma
\begin{lem}\label{lem-unif-bound-theta-n-H-1-domain}
Let $ d\in \{2, 3\}$, $  \alpha_0(d):= 1+\frac{d-1}3\leq \alpha\leq 2$ and $ u_0\in L^{p}(\Omega, \mathbb{L}^{2}(O)), p\geq 4$ and let $ G$ satisfying Assumption $ (\mathcal{C})$ (\eqref{Eq-Cond-Linear-Q-G} with $ q=2$ and $ \delta =0$). Then
the solutions $ (u_n(t), t\in [0, T])$ of the equations \eqref{FSBE-Galerkin-approxi},  $ n\in \mathbb{N}_0$,
satisfy the following estimates

\begin{eqnarray}\label{Eq-Ito-n-weak-estimation-1-bounded}
\sup_{n}\mathbb{E}\Big(\sup_{[0, T]}|u_n(t)|^{p}_{\mathbb{L}^{ 2}}&+&
\int_0^T|u_n(t)|^{p-2}_{\mathbb{L}^{ 2}}\Big(|u_n(t)|^2_{\mathbb{H}^{\frac\alpha2, 2}} +
|u_n(t)|_{\mathbb{H}^{\beta, q_1}}^2 \Big)dt \nonumber\\
&+& \int_0^T|u_n(t)|^4_{\mathbb{L}^{ 2}}dt+ \int_0^T|u_n(t)|^{\frac{\alpha}{\eta}}_{\mathbb{H}^{\eta, 2}}dt\Big)<\infty,
\end{eqnarray}
where $ \beta\leq \frac\alpha2-\frac d2+\frac d{q_1}$, $ 2\leq q_1<\infty$ and $ \frac\alpha p<\eta\leq \frac\alpha2$.
\del{\item \begin{eqnarray}\label{Eq-B-n-weak-estimation-1-bounded}
\sup_{n}\left(\mathbb{E}\int_0^T |P_nB_n(u_n(t))|^{\frac{2\alpha}{4-\alpha}}_{\mathbb{H}^{-\frac\alpha2, 2}}+
\int_0^T |A^\frac\alpha2 u_n(t))|^{2}_{\mathbb{H}^{-\frac\alpha2, 2}} \right) dt <\infty.\nonumber\\
\end{eqnarray}}
\begin{eqnarray}\label{Eq-B-n-weak-estimation-1-bounded}
\sup_{n}\left(\mathbb{E}\int_0^T (|P_nB(u_n(t))|_{\mathbb{H}^{-\frac\alpha2, 2}}+| A^\frac\alpha2 u_n(t))|_{\mathbb{H}^{-\frac\alpha2, 2}})^{\frac{2\alpha}{d+2-\alpha}}dt \right) <\infty.
\end{eqnarray}
\end{lem}
\begin{proof}
The proof of \eqref{Eq-Ito-n-weak-estimation-1-bounded} follows exactly as for \eqref{Eq-Ito-n-weak-estimation-1-Torus} by replacing
the spaces $ \mathbb{H}^{1, 2}(\mathbb{T}^2)$ and $ \mathbb{H}^{1+\frac\alpha2, 2}(\mathbb{T}^2)$ respectively by
$ \mathbb{L}^{2}(O)$ and $ \mathbb{H}^{\frac\alpha2, 2}(O)$.
For the first term in the Estimate \eqref{Eq-B-n-weak-estimation-1-bounded}, we use the contraction property of $ P_n$, 
Estimate\del{ \eqref{B-u-v-h-alpha-2-d} with $ \eta =0$ (which coincides with} \eqref{Eq-B-H-alpha-2-est}\del{\eqref{Eq-B-H-alpha-2-est}}
and the Sobolev interpolation (recall that thanks to the condition $1+ \frac{d-1}3\leq  \alpha \leq 2$, we have 
the following embedding 
$ \mathbb{H}^{\frac\alpha2, 2}(O) \hookrightarrow \mathbb{H}^{\frac{d+2-\alpha}{4}, 2}(O)\hookrightarrow \mathbb{L}^2(O)$), we end up, for $1+ \frac{d-1}3<  \alpha \leq 2$,  with
\begin{eqnarray}\label{unif-estint-B-u}
\mathbb{E}\int_0^T |P_nB(u_n(t))|_{\mathbb{H}^{-\frac\alpha2}}^{\frac{2\alpha}{d+2-\alpha}}dt &\leq& c
\mathbb{E}\int_0^T |u_n(t)|_{\mathbb{H}^{\frac{d+2-\alpha}{4}, 2}}^{\frac{4\alpha}{d+2-\alpha}} dt \nonumber\\
&\leq& c
\mathbb{E}\int_0^T \big(|u_n(t)|^{\frac{d+2-\alpha}{2\alpha}}_{\mathbb{H}^{\frac\alpha2, 2}} 
|u_n(t)|^{\frac{3\alpha-d-2}{2\alpha}}_{\mathbb{L}^{2}}\big)^{\frac{4\alpha}{d+2-\alpha}} dt\nonumber\\
&\leq& c
\mathbb{E}\int_0^T |u_n(t)|^{2}_{\mathbb{H}^{\frac\alpha2, 2}} 
|u_n(t)|^{2\frac{3\alpha-d-2}{d+2-\alpha}}_{\mathbb{L}^{2}} dt.
\end{eqnarray}
The last term in the RHS of \eqref{unif-estint-B-u} is uniformly bounded thanks to \eqref{Eq-Ito-n-weak-estimation-1-bounded} and the condition
$ 2\frac{3\alpha-d-2}{d+2-\alpha}\leq p$. But this last is guaranteed thanks to $ 2\frac{3\alpha-d-2}{d+2-\alpha}\leq 4\leq p$. The case $1+ \frac{d-1}3=\alpha$ is easily obtained by application of Estimation \eqref{Eq-B-H-alpha-2-est}.
The second term in the RHS of \eqref{Eq-B-n-weak-estimation-1-bounded} is uniformly bounded thanks to
the fact that $ A: V:= D(A^\frac\alpha4) \rightarrow V^*$ is bounded, the condition $ \alpha \leq 1+\frac d2$ which yileds to 
$ \frac{2\alpha}{d+2-\alpha} \leq 2$ and thus we get
\begin{eqnarray}
\mathbb{E}\int_0^T |A^\frac\alpha2 u_n(t)|^{\frac{2\alpha}{d+2-\alpha}}_{\mathbb{H}^{-\frac\alpha2, 2}} dt
\leq c\mathbb{E}\int_0^T |u_n(t)|^{\frac{2\alpha}{d+2-\alpha}}_{\mathbb{H}^{\frac\alpha2, 2}} dt\leq c
\mathbb{E}\int_0^T (1+| u_n(t)|^{2}_{\mathbb{H}^{\frac\alpha2, 2}} )dt <\infty.\nonumber\\
\end{eqnarray}
Finaly we apply  Estimate \eqref{Eq-Ito-n-weak-estimation-1-bounded}.
\end{proof}

\noindent {\bf Existence of the solution.}
Assume that $1+ \frac{d-1}3<\alpha \leq 2$. Thanks to \eqref{Eq-Ito-n-weak-estimation-1-bounded}  and \eqref{Eq-B-n-weak-estimation-1-bounded}, we conclude the existence of
a subsequence, which is still denoted by $(u_n)_n$,\del{ and adapted processes $ u, F_2, G_2 $, such that}
\begin{equation}\label{eq-u-first-belonging}
 u\in L^2(\Omega\times [0, T]; \mathbb{H}^{\frac\alpha2, 2}(O))\cap
L^p(\Omega, L^\infty([0, T]; \mathbb{L}^{2}(O))),
\end{equation}
\begin{equation}
F_2 \in L^{\frac{2\alpha}{d+2-\alpha}}(\Omega\times [0, T];
\mathbb{H}^{-\frac\alpha2, 2} (O))\;\; \text{and}\;\; 
G_2\in L^2(\Omega\times [0, T]; L_Q(\mathbb{L}^{ 2} (O))), s.t.
\end{equation}
\begin{itemize}
\item (1') $u_n \rightarrow u$ weakly in $ L^2(\Omega\times [0, T]; \mathbb{H}^{\frac\alpha2, 2}(O)))$.
\item (2') $u_n \rightarrow u$ weakly-star in $ L^p(\Omega, \mathbb{L}^\infty([0, T]; \mathbb{L}^{ 2}(O)))$,
\item (3') $P_nF(u_n):= A^\frac\alpha2 u_n + P_nB(u_n)\rightarrow F_2$ weakly in $ L^{\frac{2\alpha}{d+2-\alpha}}(\Omega\times [0, T];
\mathbb{H}^{-\frac\alpha2, 2} (O))$.
\item (4')$u_n \rightarrow u$ weakly in $ L^{\frac{\alpha}{\eta}}(\Omega\times [0, T]; \mathbb{H}^{\eta, 2}(O))$, for all
$ \frac\alpha p<\eta \leq \frac\alpha2 $. \del{\footnote{Remark that $ \frac43 < \alpha \leq 2 \Leftrightarrow 1\leq \frac{2\alpha}{4-\alpha}\leq 2$}}
\item (5') $P_nG(u_n)\rightarrow G_2$ weakly in $ L^2(\Omega\times [0, T]; L_Q(\mathbb{L}^{ 2} (O)))$.

\end{itemize}
To prove the existence of a weak-strong solution of \eqref{Main-stoch-eq}, we can follow the same scheme as in Section \ref{sec-Torus} with the replacement of the spaces $ \mathbb{H}^{1, 2}(\mathbb{T}^2)$ and $ \mathbb{H}^{1+\frac\alpha2, 2}(\mathbb{T}^2)$ by
$ \mathbb{L}^{2}(O)$ and $ \mathbb{H}^{\frac\alpha2, 2}(O)$ respectively.  We construct a
process $ \tilde{\tilde {u}}$  as in \eqref{eq-def-u-tilde}, with $ F_1$ and $ G_1$ are replaced by $ F_2$ respectively  $ G_2$.
The proof of the statement $ u= \tilde{\tilde{u}},\; dt\times dP-a.e.$ can be done exactly as in Section \ref{sec-Torus}
with the brackets now stand for the $ V-V^*$-duality.
To check the  main key estimates, we use \eqref{formula-B-v1-B-v2}, \eqref{Eq-3lin-propnull}, H\"older inequality, 
\eqref{Eq-B-H-alpha-2-est}\del{{Eq-B-H-alpha-2-est}}, Sobolev interpolation\del{(recall that
$ 1+\frac{d-1}{3}< \alpha <2 \Rightarrow \mathbb{H}^{\frac\alpha2, 2}(O) \hookrightarrow \mathbb{H}^{\frac{d+2-\alpha}{4}, 2}(O) $)}
and Young inequality, we get
\begin{eqnarray}\label{ineq-B-u-v-v-local}
|{}_{V^*}\langle B(u)-B(v), u-v\rangle_{V} |&= & |{}_{V^*}\langle B(u-v, v), u-v\rangle_{V}|\leq
| B(u-v, v)|_{\mathbb{H}^{-\frac\alpha2, 2}}|u-v|_{\mathbb{H}^{\frac\alpha2, 2}}\nonumber\\
&\leq& c|v|_{\mathbb{H}^{\frac{d+2-\alpha}{4}, 2}} |u-v|_{\mathbb{H}^{\frac\alpha2, 2}}|u-v|_{\mathbb{H}^{\frac{d+2-\alpha}{4}, 2}}\nonumber\\
&\leq & c|v|_{\mathbb{H}^{\frac{d+2-\alpha}{4}, 2}} |u-v|^{\frac{d+2+\alpha}{2\alpha}}_{\mathbb{H}^{\frac\alpha2, 2}}
|u-v|_{\mathbb{L}^2}^{\frac{3\alpha-d-2}{2\alpha}}\nonumber\\
&\leq& c|v|^{\frac{4\alpha}{3\alpha-d-2}}_{\mathbb{H}^{\frac{d+2-\alpha}{4}, 2}}|u-v|^{2}_{\mathbb{L}^2} + \frac12|u-v|^{2}_{\mathbb{H}^{\frac\alpha2, 2}}.
\end{eqnarray}
Using the semigroup property of $ (A^\beta)_{\beta\geq0}$ and Assumption $ (\mathcal{C})$ ( \eqref{Eq-Cond-Lipschitz-Q-G}, with $ \delta =0$, $ q=2$ and $ C_R:=c$), we confirm\del{ 
\begin{eqnarray}\label{eq-monoto-local}
 -2{}_{V^*}\langle A_\alpha(u-v), u-v\rangle_{V} &+&  2{}_{V^*}\langle B(u)-B(v), u-v\rangle_{V} + 
|| G(u) - G(v)||_{L_Q(\mathbb{L}^2)}\nonumber\\
&\leq&
 -|u-v|^{2}_{\mathbb{H}^{\frac\alpha2, 2}} + c(1+|v|^{\frac{4\alpha}{3\alpha-d-2}}_{\mathbb{H}^{\frac{d+2-\alpha}{4}, 2}})
|u-v|^{2}_{\mathbb{L}^2}.
\end{eqnarray}
Consequently, we get}
\begin{itemize}
 \item $ (\mathcal{K}'_1)$- The local monotonicity property: There exists a constant $ c>0$ such that $\forall u, v \in \mathbb{H}^{\frac\alpha2}(O)$,
\begin{eqnarray}\label{eq-a-alpha-B-G-d}
 -2{}_{V^*}\langle A_\alpha(u-v), u-v\rangle_{V} &+&  2{}_{V^*}\langle B(u)-B(v), u-v\rangle_{V} +
|| G(u) - G(v)||_{L_Q(\mathbb{L}^2)}\nonumber\\
&\leq & r'(t)| u-v|^2_{\mathbb{L}^2}.
\end{eqnarray}
\end{itemize}
where $ r'(t):= c(1+ |v(t)|^{\frac{4\alpha}{3\alpha-d-2}}_{\mathbb{H}^{\frac{d+2-\alpha}{4}, 2}})$ and $c>0$ is a constant relevantly chosen.

\del{The main obstacle which prevent us in this stage to follow  the same steps as in Section \ref{sec-Torus} is the fact that we can not prove 
 for  $ v\in L^2(\Omega\times[0, T], \mathbb{H}^{\frac\alpha2, 2}(O))\cap  L^p(\Omega, L^\infty([0, T]; \mathbb{L}^{2}(O))$, that
$ v \in L^{\frac{4\alpha}{3\alpha-d-2}}(\Omega \times[0, T]; \mathbb{H}^{\frac{d+2-\alpha}{4}, 2}(O))$, unless $ \alpha \geq 1+\frac d2$. 
In fact, by interpolation,  we get
\begin{equation}\label{eq-impossible-local}
\mathbb{E}\int_0^T |v(t)|^{\frac{4\alpha}{3\alpha-d-2}}_{\mathbb{H}^{\frac{d+2-\alpha}{4}, 2}}dt \leq 
\mathbb{E}\int_0^T |v(t)|^{2\frac{d+2-\alpha}{3\alpha-d-2}}_{\mathbb{H}^{\frac\alpha2, 2}}|v(t)|^{2}_{\mathbb{L}^{2}} dt.
\end{equation}
Our claim here is that thanks to \eqref{Eq-Ito-n-weak-estimation-1-bounded}, a sufficient condition for the convergence of the integral in the RHS of 
\eqref{eq-impossible-local} is that $ 2\frac{d+2-\alpha}{3\alpha-d-2} \leq 2 \Leftrightarrow  \alpha \geq 1+\frac d2$. 
Remark that the classical values $ d=2, \alpha =2$ and $ d=3, \alpha=\frac52$ known in the literature for the
dD-NSEs are special cases of our condition above. 
The obstacle mentioned in \eqref{eq-impossible-local} is well known for the classical 3D-NSE  but not for the 2D-NSE. 
As we are using completely different calculus than the classical one, we need here to prove that our technique is optimal. Simultaneousily, we shall prove that our claim above is true.
In fact,  let $ \alpha =2=d$, then by interpolation, we get
\begin{equation}\label{eq-interp-alpha=2}
 |v|^{\frac{4\alpha}{3\alpha-d-2}}_{\mathbb{H}^{\frac{d+2-\alpha}{4}, 2}}\leq c  
( |v|^{\frac12}_{\mathbb{H}^{1, 2}}|v|^{\frac12}_{\mathbb{L}^{2}})^{4}
\leq  c|v|^{2}_{\mathbb{H}^{1, 2}}|v|^{2}_{\mathbb{L}^{2}}.
\end{equation}
Replacing \eqref{eq-interp-alpha=2} in \eqref{eq-impossible-local} and using the fact that $ u \in L^2(\Omega\times[0, T], \mathbb{H}^{1, 2}(O))\cap  L^p(\Omega, L^\infty([0, T]; \mathbb{L}^{2}(O))\cap L^4(\Omega\times[0, T], \mathbb{H}^{\frac{d+2-\alpha2}{4}, 2}(O))$ and the interpolation , thanks to \eqref{Eq-Ito-n-weak-estimation-1-bounded},  the fact that 

Follow the same machenery as in Section \ref{sec-Torus},  we prove the existence and uniquness of global solution for the 2D-NSE.}

\vspace{0.25cm}
The main obstacle which prevent us in this stage to follow  the same steps as in Section \ref{sec-Torus} is the fact that we are unable to prove that the solution \del{  $ u\in L^2(\Omega\times[0, T], \mathbb{H}^{\frac\alpha2, 2}(O))\cap  L^p(\Omega, L^\infty([0, T]; \mathbb{L}^{2}(O))$ is not enough to get} 
$ u \in L^{\frac{4\alpha}{3\alpha-d-2}}(\Omega \times[0, T]; \mathbb{H}^{\frac{d+2-\alpha}{4}, 2}(O))$, unless we suppose that $ \alpha \geq 1+\frac d2$.  In fact, under the condition $ 2\frac{d+2-\alpha}{3\alpha-d-2} \leq 2 \Leftrightarrow  \alpha \geq 1+\frac d2$ and using the interpolation and Estimate \eqref{Eq-Ito-n-weak-estimation-1-bounded},  we conclude that 
\begin{equation}\label{eq-impossible-local}
\sup_{n}\mathbb{E}\int_0^T |u_n(t)|^{\frac{4\alpha}{3\alpha-d-2}}_{\mathbb{H}^{\frac{d+2-\alpha}{4}, 2}}dt \leq 
c\sup_{n}\mathbb{E}\int_0^T |u_n(t)|^{2\frac{d+2-\alpha}{3\alpha-d-2}}_{\mathbb{H}^{\frac\alpha2, 2}}|u_n(t)|^{2}_{\mathbb{L}^{2}} dt<\infty.
\end{equation}
Remak that under the condition $\alpha \geq 1+\frac d2$, the regime is either dissipative or hyperdissipative. The proof of the existence and the uniqueness of the global solution for the dD-FSNSE under these two regimes is  classical.\del{for which we know that the calculus is easier to get  $ (3)'$. Than following} In particular, one can follow the same machinery as in Section \ref{sec-Torus} with the relevant changes mentioned above. The obstacle mentioned in \eqref{eq-impossible-local} is similar to the well known one for the classical 3D-NSE  but not for the 2D-NSE. To support more our claim mentioned in the begining of this section and in Section \ref{sec-intro}, we emphasize that the  2D-SNSE is\del{or 2D-NSE are} covered by our technique and this proves that this latter is optimal. Moreover, we can remark also that the values,  ($ \alpha\geq 1+\frac d2$),  ($ d=2, \alpha =2$) and ($ d=3, \alpha\geq\frac52$), known in the literature for the
dD-NSEs emerge in our setting in a natural way.\del{are special cases of our condition above.}

\del{one can see that the classical  2D-SNSE or 2D-NSE are 
covered by our method and this proves that our techniquethe optimality of our technique.}

\del{t is of great importance here to emphasis the optimality of our technique as  the classical  2D-SNSE or 2D-NSE are 
covered. by our method, this proves that our technique is optimal.}

\vspace{0.25cm}


\del{According to the calculus above, we can deduce that, if we follow this technique, it might be possible that we are only able to prove the existence of a $L^2-$ local maximal solution. This last is still again difficult, see the discussion in the end of this section. But the important deduction we can get, in particular from \eqref{eq-impossible-local}-\eqref{eq-impossible-local}, is the role could be played by the space $ \mathbb{H}^{\frac{d+2-\alpha}{4}, 2}(O)$, see also Condition  \eqref{uniquness-con-martg}. Thus, we shall look for local maximal solutions in this space. 

In this aim we use an approximation approach. We consider Equation \eqref{Main-stoch-eq} with $ B$  replaced by $ B\pi_m$ and $ \pi_m$ is defined in Lemma \ref{lem-Lipschitz-pi-n} with $  X:=\mathbb{H}^{\frac{d+2-\alpha}{4}, 2}(O)$ and with initila condition 
$ u_0$ satisfies Assumption $ (\mathcal{B}_\alpha)$. i.e. we consider, for $m \in \mathbb{N}_0$, the equations\del{   
Otherwise, we consider Equation \eqref{Eq-approx-n} with $ G$ globally Lipschitz, therefore, we do not need to use $G(\pi_m) $,}
\begin{equation}\label{Eq-approx-m}
\Bigg\{
\begin{array}{lr}
 du^m= \big(-
A_{\alpha}u^m(t) + B(\pi_m u^m(t))\big)dt+ G(u^m(t))dW(t), \; 0< t\leq T,\\
u_m(0)= u_0 \in L^2(\Omega, \mathcal{F}_0, P, \mathbb{H}^{\frac{d+2-\alpha}{4}, 2}(O)).
\end{array}
\end{equation}
First, let us mention that, using  Formula \eqref{construction-of-fract-bounded} and the fact that $ div \pi_m u=0$, it is easy to see that  $ \pi_m$ is well defined. Moreover, $ B\pi_m$ is globaly Lipschitz and we have for all  $(u, v)\in D(B(\cdot, \cdot))$,
\begin{eqnarray}\label{eq-relation-B-B-m}
B(\pi_m(u), \pi_m(v)) &=& \mathcal{X}_m(u, v)B(u, v), 
\end{eqnarray}
where 
\begin{eqnarray}\label{eq-x-m}
\mathcal{X}_m(u, v):&=&
\min\{1, m{|u|^{-1}_{\mathbb{H}^{\frac{d+2-\alpha}{4}, 2}}},
 m{|v|^{-1}_{\mathbb{H}^{\frac{d+2-\alpha}{4}, 2}}}, m^2{|u|^{-1}_{\mathbb{H}^{\frac{d+2-\alpha}{4}, 2}}|v|^{-1}_{\mathbb{H}^{\frac{d+2-\alpha}{4}, 2}}}\}
\end{eqnarray}
We have the following existence and  uniqueness result:
\begin{lem}\label{Lem-approx-weak-solution}
Assume that $ u_0$ satisfies Assumption $ (\mathcal{B}_\alpha)$ and $ G$ satisfies $ (\mathcal{C})$ with $ delta$ replaced by $ \frac{d+2-\alpha}{4}$, $ q=2$ and the constant in the RHS of \eqref{Eq-Cond-Lipschitz-Q-G} is independant of $ R$ ($ G$ is globally Lipschitz). Then, for all $ m\in \mathbb{N}_0$, Equation \eqref{Eq-approx-m} admits a unique continuous $\mathbb{H}^{\frac{d+2-\alpha}{4}, 2}$-valued $ \mathbb{F}-$adapted  process $ (u^m(t), t\in [0, T])$ in the sense of Definition \eqref{Def-solution-variational}, with $ V_2= V=\mathbb{H}^{\frac{d+2-\alpha}{4}+\frac\alpha2, 2}(O)$, $ H=\mathbb{H}^{\frac{d+2-\alpha}{4}, 2}(O)$ and $ V_1= V^*=\mathbb{H}^{\frac{d+2-\alpha}{4}-\frac\alpha2, 2}(O)$. Moreover, $ (u^m(t), t\in [0, T])$ is  a $\mathbb{H}^{\frac{d+2-\alpha}{4}+\frac\alpha2, 2}-$valued progressively measurable $ dt\otimes P$ satisfying \eqref{eq-weak-H-d-alpha-solu}.
\end{lem} 

\begin{proof}
In the aim to prove the existence and the uniqueness of the global weak solution of Equation \eqref{Eq-approx-m}, for all $ m\in \mathbb{N}_0$,
we use the variational approach, see e.g. \cite{Krylov-Rozovski-monotonocity-2007, Metivier-book-SPDEs-88, Rockner-Pevot-06}. Or equivalently we can follow either the scheme above or the scheme in Section \ref{sec-Torus} with the relevant changes mentioned before. In particular, we apply \cite[Definition 4.2.1 \&  Theorem 4.2.4 (One can see also Theorem 4.2.5)]{Rockner-Pevot-06}. Therefore, we check for all $ u, v, \theta \in \mathbb{H}^{\frac\alpha2, 2}(O)$, the following properties (let us for simplicity of notations, denote $ \frac{d+2-\alpha}{4}$ by $ \beta$)

{\bf $(a_1)$  The monotonicity.} Using \eqref{formula-B-v1-B-v2}, H\"older inequality, \eqref{B-u-v-h-alpha-2-d} with $ \eta =0$ (or \eqref{Eq-B-H-alpha-2-est}), Lemma \ref{lem-Lipschitz-pi-n}, Sobolev interpolation (recall that
$ 1+\frac{d-1}{3}< \alpha <2$)\del{(recall that
$ 1+\frac{d-1}{3}< \alpha <2 \Rightarrow \mathbb{H}^{\frac\alpha2, 2}(O) \hookrightarrow \mathbb{H}^{\frac{d+2-\alpha}{4}, 2}(O) $)}
and Young inequality, we infer that
\del{\begin{eqnarray}\label{ineq-B-u-v-v-local-n}
|{}_{V^*}\langle B(\pi_nu)&-&B(\pi_nv), u-v\rangle_{V} |\nonumber\\
&= & |{}_{V^*}\langle B(\pi_nu-\pi_nv, \pi_nv), u-v\rangle_{V}
+{}_{V^*}\langle B(\pi_nu, \pi_nu-\pi_nv), u-v\rangle_{V}|\nonumber\\
&\leq & |u-v|_{\mathbb{H}^{\frac\alpha2, 2}}\big(
| B(\pi_nu-\pi_nv, \pi_nv)|_{\mathbb{H}^{-\frac\alpha2, 2}}+| B(\pi_nu-\pi_nv, \pi_nv)|_{\mathbb{H}^{-\frac\alpha2, 2}}\big)\nonumber\\
&\leq & c|u-v|_{\mathbb{H}^{\frac\alpha2, 2}}
|\pi_nu-\pi_nv|_{\mathbb{H}^{\frac{d+2-\alpha}{4}, 2}}
\big(|\pi_nu|_{\mathbb{H}^{\frac{d+2-\alpha}{4}, 2}}+|\pi_nv|_{\mathbb{H}^{\frac{d+2-\alpha}{4}, 2}}\big)\nonumber\\
&\leq& c|v|_{\mathbb{H}^{\frac{d+2-\alpha}{4}, 2}} |u-v|_{\mathbb{H}^{\frac\alpha2, 2}}|u-v|_{\mathbb{H}^{\frac{d+2-\alpha}{4}, 2}}\nonumber\\
&\leq & c|v|_{\mathbb{H}^{\frac{d+2-\alpha}{4}, 2}} |u-v|^{\frac{d+2+\alpha}{2\alpha}}_{\mathbb{H}^{\frac\alpha2, 2}}
|u-v|_{\mathbb{L}^2}^{\frac{3\alpha-d-2}{2\alpha}}\nonumber\\
&\leq& c|v|^{\frac{4\alpha}{3\alpha-d-2}}_{\mathbb{H}^{\frac{d+2-\alpha}{4}, 2}}|u-v|^{2}_{\mathbb{L}^2} + \frac12|u-v|^{2}_{\mathbb{H}^{\frac\alpha2, 2}}.
\end{eqnarray}}
\begin{eqnarray}\label{ineq-B-u-v-v-local-m}
{}_{V^*}\langle B(\pi_mu)&-&B(\pi_mv), u-v\rangle_{V} \nonumber\\
&= & {}_{V^*}\langle B(\pi_mu, \pi_mu-\pi_mv), u-v\rangle_{V}+{}_{V^*}\langle B(\pi_mu-\pi_mv, \pi_mv), u-v\rangle_{V} \nonumber\\
&\leq & |u-v|_{\mathbb{H}^{\frac\alpha2, 2}}\big(
| B(\pi_mu, \pi_mu-\pi_mv)|_{\mathbb{H}^{-\frac\alpha2, 2}}+| B(\pi_mu-\pi_mv, \pi_mv)|_{\mathbb{H}^{-\frac\alpha2, 2}}\big)\nonumber\\
&\leq & c|u-v|_{\mathbb{H}^{\frac\alpha2, 2}}
|\pi_mu-\pi_mv|_{\mathbb{H}^{\frac{d+2-\alpha}{4}, 2}}
\big(|\pi_mu|_{\mathbb{H}^{\frac{d+2-\alpha}{4}, 2}}+|\pi_mv|_{\mathbb{H}^{\frac{d+2-\alpha}{4}, 2}}\big)\nonumber\\
&\leq& cm |u-v|_{\mathbb{H}^{\frac\alpha2, 2}}|u-v|_{\mathbb{H}^{\frac{d+2-\alpha}{4}, 2}}\nonumber\\
\del{&\leq & cn|u-v|^{\frac{d+2+\alpha}{2\alpha}}_{\mathbb{H}^{\frac\alpha2, 2}}
|u-v|_{\mathbb{L}^2}^{\frac{3\alpha-d-2}{2\alpha}}\nonumber\\}
&\leq& cm^{\frac{4\alpha}{3\alpha-d-2}}|u-v|^{2}_{\mathbb{L}^2} + \frac12|u-v|^{2}_{\mathbb{H}^{\frac\alpha2, 2}}.
\end{eqnarray}
Therefore, using the semigroup property of $ (A_\beta)_{\beta\geq 0}$, 
\eqref{ineq-B-u-v-v-local-m} and Assumption $ (\mathcal{C})$ (with $ q=2$, $ \delta=0$ and $c_R=c$), we get the monotonicity property\del{we conclude the monotonicity property of the term  $ -A_\alpha+B\pi_m$. In particular, thanks to Assumption $ (\mathcal{C})$,  Estimation \eqref{eq-monoto-local} with $ B\pi_m$ and $ m$ are in the place of $ B$ and $ |v|_{\mathbb{H}^{\frac{d+2-\alpha}{4}, 2}}$ respectively, holds. }
\begin{eqnarray}
-2{}_{V^*}\langle A_\alpha (u-v), u-v\rangle_{V} &+& 2{}_{V^*}\langle B(\pi_m u)- B(\pi_m v), u-v\rangle_{V} \nonumber\\ 
&+& ||G(u)- G(v)||^2_{HS(\mathbb{L}^2)}\leq  c_m|u-v|^2_{\mathbb{L}^{2}}.
\end{eqnarray}

{\bf $(a_2)$  The coercitivity.} \del{ Thanks to  H\"older inequality, \eqref{B-u-v-h-alpha-2-d} with $ \eta =0$ (or \eqref{Eq-B-H-alpha-2-est}), Lemma \ref{lem-Lipschitz-pi-n}, interpolation and Young inequality, we infer that 
\begin{eqnarray}\label{eq-coercive}
{}_{V^*}\langle B(\pi_m u), u\rangle_{V}&\leq &
|u|_{\mathbb{H}^{\frac\alpha2, 2}}|B(\pi_m u)|_{\mathbb{H}^{-\frac\alpha2, 2}} \leq  
|u|_{\mathbb{H}^{\frac\alpha2, 2}}|\pi_m u|^2_{\mathbb{H}^{\frac{d+2-\alpha}4, 2}}\nonumber\\
&\leq & m |u|_{\mathbb{H}^{\frac\alpha2, 2}}| u|_{\mathbb{H}^{\frac{d+2-\alpha}4, 2}}\leq 
m |u|^{\frac{\alpha+d+2}{2\alpha}}_{\mathbb{H}^{\frac\alpha2, 2}}|u|^ {\frac{3\alpha-d-2}{2\alpha}}_{\mathbb{L}^{2}}
\nonumber\\
&\leq & \frac12|u|^2_{\mathbb{H}^{\frac\alpha2, 2}} + m^{\frac{4\alpha}{3\alpha-d-2}} 2^{\frac{\alpha}{\alpha+d+2}}|u|^ {2}_{\mathbb{L}^{2}}.
\end{eqnarray}
Now i}It is easy, using the semigroup property of $ (A_\beta)_{\beta\geq 0}$,\del{\eqref{eq-coercive}}\eqref{eq-relation-B-B-m}, \eqref{eq-x-m},  \eqref{Eq-3lin-propnull} and Assumption $ (\mathcal{C})$(with $q=2$, $ \delta=0$ and $c_R=c$) that  
\begin{eqnarray}
-2{}_{V^*}\langle A_\alpha u, u\rangle_{V} + 2{}_{V^*}\langle B(\pi_m u), u\rangle_{V} + ||G(u)||^2_{L_Q(\mathbb{L}^2)}+ 2|u|^2_{\mathbb{H}^{\frac\alpha2, 2}}&\leq & c_m(1+|u|^2_{\mathbb{L}^{2}}).\nonumber\\
\end{eqnarray}

{\bf $(a_3)$  The growth.} Thanks to  \eqref{Eq-B-H-alpha-2-est}, Lemma \ref{lem-Lipschitz-pi-n} and Sobolev embedding, we infer that 
\begin{eqnarray}\label{eq-growth}
| A_\alpha u|_{\mathbb{H}^{-\frac\alpha2, 2}} + |B(\pi_m u)|_{\mathbb{H}^{-\frac\alpha2, 2}} &\leq & (m+1)|u|_{\mathbb{H}^{\frac\alpha2, 2}}. 
\end{eqnarray}

{\bf $(a_4)$  The hemicontinuity.} The hemicintinuity is obtained thanks to the continuity of the following real functions 
on $ \mathbb{R}$
\begin{equation}
s\mapsto |u+sv|_{\mathbb{H}^{\frac{d+2-\alpha}{4}, 2}}\;\;\text{and}\;\;\; 
s\mapsto \mathcal{X}_m(u+sv, u+sv).
\end{equation}
\del{and
\begin{equation}
s\mapsto \phi_{u, v, \theta}(s):=
\Big\{
\begin{array}{lr}
{}_{V^*}\langle B(u+sv), \theta\rangle_{V}, \;\; |u+sv|_{\mathbb{H}^{\frac{d+2-\alpha}{4}, 2}}\leq m\nonumber\\
\frac{m^2}{|u+sv|^2_{\mathbb{H}^{\frac{d+2-\alpha}{4}, 2}}}\;\;{}_{V^*}\langle B(u+sv), \theta\rangle_{V},\geq m \;\; |u+sv|_{\mathbb{H}^{\frac{d+2-\alpha}{4}, 2}}. \nonumber\\
\end{array}
\end{equation}}
\del{\begin{equation}
s\mapsto {}_{V^*}\langle B(\pi_m (u+sv)), \theta\rangle_{V}= 
\phi_{u, v}(s) ,
\end{equation}}
\del{Now it is easy to check, using \eqref{Eq-B-H-alpha-2-est}, that in addition to the monotonicity property, the coefficients of the approximated equation satisfy the hemicontinuity, the coercitivity and the growth properties.} 
\noindent Consequently, for all $ m \in \mathbb{N}_0$, there exists a unique global solution $ (u^m(t), t\in [0, T])$ of Equation \eqref{Eq-approx-m} in the sense of Definition \ref{def-variational solution},  i.e. $ (u^m(t), t\in [0, T])$ satisfies, Equation \eqref{Eq-weak-Solution}, with $ V_2= V=\mathbb{H}^{\frac\alpha2, 2}(O)$, $ V_1= \mathbb{H}^{-\delta', 2}(O)$, $ \delta'>1+\frac d2$, \eqref{eq-set-solu-weak} and \eqref{eq-mart-l-2solu} with the constant $ c$ in the RHS of the later estimate depends on $m$ (or equivalently $ (u^m(t), t\in [0, T])$ satisfies \eqref{eq-weak-l-2solu} up to $ T$. \del{
\begin{equation}
\mathbb{E}\sup_{[0, T]}|u^m(t)|^p_{\mathbb{L}^2}+ \mathbb{E}\int_0^T\del{|u^m(t)|^p_{\mathbb{L}^2}}|u^m(t)|^2_{\mathbb{H}^{\frac\alpha2, 2}}dt\leq c_m<\infty.
\end{equation}}The $\mathbb{H}^{-\delta', 2}$-continuity is obtained by a standard way using Section \ref{sec-nonlinear-prop}. In particular,  one of the main ingredients of the proof  is Formula \eqref{eq-B-estimatoion-q=2}\del{{eq-est-H-1-d-q}}. Thanks to Remark \ref{Rem-1}, we deduce that the trajectories of 
$ (u^m(t), t\in [0, T])$ are $\mathbb{L}^2-$weakly continuous, thus  the random time  
\begin{equation}\label{eq-stop-time-weak-solu-L2}
 \xi_m:=\inf\{t\in (0, T), s.t.\;\; |u^m(t)|_{\mathbb{H}^{\frac{d+2-\alpha}{4}, 2}} > m\}\wedge T,
\end{equation}
with the understanding that $ \inf(\emptyset)=+\infty$, is a stopping time, see complete proof in Appendix \ref{append-stop-time}. The sequence $(\xi_m)_m $ is an increasing sequence, hence the following random time exists and is then a predictable stopping time.
\begin{equation}
 \xi := \lim_{m\rightarrow +\infty}\xi_m.
\end{equation}
Thanks to the uniqueness of $ (u^m(t), t\in [0, T])$, we can define the process \del{( as for the mild solution)} $ (u(t), \;  t\in [0, \xi) )$ by 
$ (u(t):=u^m(t), t\in [0, \xi_m))$. Remark that as much as $ u(t)$ stays in the ball $ B_{\mathbb{H}^{\frac{d+2-\alpha}{4}, 2}}(0, m)$, it is a solution of our main equation \eqref{Main-stoch-eq}, hence  $ (u, \xi)$ is a maximal local weak-strong solution.
\del{\noindent Now, we assume that $ u\in L^2(\Omega\times[0, T], \mathbb{H}^{\frac\alpha2, 2}(O)) $,  
$ v\in L^2(\Omega\times[0, T], \mathbb{H}^{\frac\alpha2, 2}(O)) \cap  
L^{\frac{4\alpha}{3\alpha-d-2}}(\Omega\times[0, T], \mathbb{H}^{\frac{d+2-\alpha}{4}, 2}(O))$ and define 
$ r'(t):= c(1+ |v(t)|^{\frac{4\alpha}{3\alpha-d-2}}_{\mathbb{H}^{\frac{d+2-\alpha}{4}, 2}})$. Let $ \tilde{\xi}$ be any predictable stopping time.
We follow a similar calculus as in 
Section \ref{sec-Torus}, taking in consideration the changes mentioned in the beginning of this section and integrating on $ (0, \tilde{\xi})$.
Then, we end up with the estimation
\begin{eqnarray}\label{eq-key-leq-0}
\int_0^T\psi(t)dt\mathbb{E}\big\{\int_0^{t\wedge \tilde{\xi}}\!\!\!\!\!\!&{}&\!\!\!\!\!\! e^{-r(s)}
\big(-r'(s\wedge \tilde{\xi})|u(s\wedge \tilde{\xi})- v(s\wedge \tilde{\xi})
|^2_{\mathbb{L}^2}+ 2|| G_2(s\wedge \tilde{\xi}) - G(v(s\wedge \tilde{\xi}))||^2_{L_Q(\mathbb{L}^2)}\nonumber\\
&+& {}_{V^*}\langle F_2(s\wedge \tilde{\xi}) -
F(v(s\wedge \tilde{\xi})), u(s\wedge \tilde{\xi})- v(s\wedge \tilde{\xi}))\rangle_{V}\big)ds\big\}\leq 0.
\end{eqnarray}
\noindent We define, for $ N\in \mathbb{N}_0$, the predictable stopping time 
\begin{equation}
 \xi_N:=\inf\{t\in [0, T], s.t. |u(t)|_{\mathbb{H}^{\frac{d+2-\alpha}{4}, 2}} \geq N\}\wedge T,
\end{equation}
with the understanding that $ \inf(\emptyset)=+\infty$. Let 
\begin{equation}
 \xi := \lim_{N\rightarrow +\infty}\xi_N.
\end{equation}
It is well known that $ \xi$ exists thanks to the monotonicity of $ \xi_N$.
We replace in \eqref{eq-key-leq-0},  $ \tilde{\xi}$ by  $ \xi_N$ and take $ v(t\wedge \xi_N)= u(t\wedge \xi_N)$, 
\del{in $ L^2([\Omega\times0, T], \mathbb{H}^{\frac\alpha2, 2}(O))$,}we conclude from \eqref{eq-key-leq-0}, that for all fixed $ N$,
\del{$|| G(s) - G(v(s))||^2_{L_Q}$}$ G_2(s\wedge \xi_N)= G(u(s\wedge \xi_N)), \; ds\times dP-a.e.$. 
To get the equality $ F_2(s\wedge \xi_N)= F(u(s\wedge \xi_N)))$, we consider Estimation \eqref{eq-key-leq-0} without the last term and introduce
$ \tilde{v} \in L^\infty(\Omega\times[0, T], \mathbb{H}^{\frac\alpha2}(\mathbb{T}^2))$, the
parameter $ \lambda \in [-1, +1]$ and argue as in Section \ref{sec-Torus}, we get \del{. Then, replacing $ v$  and $ r(s)$ by 
$ u-\lambda\tilde{v}$ respectively
$ r'_\lambda(s):= c(1+|u-\lambda\tilde{v}|^{\frac{2\alpha}{3\alpha-2}}_{\mathbb{H}^{1+\frac\alpha2}})$, we get
\begin{eqnarray}\label{eq-equality-F-Fn}
\mathbb{E}\int_0^Te^{-r_\lambda(s)}\big(-r'_\lambda(s)\lambda^2|\tilde{v}(s)|^2_{\mathbb{L}^2}+ 2\lambda\langle F(s) -
F(u(s)-\lambda \tilde{v}(s)), \tilde{v}(s))\rangle_{\mathbb{L}^2}\big)ds\leq 0.\nonumber\\
\end{eqnarray}
Dividing on $ \lambda<0$ and on $ \lambda>0$, we conclude that if the limit of
the LHS of \eqref{eq-equality-F-Fn} exists, when $ \lambda \rightarrow 0$, then it vanishes.
For the first term in the LHS of \eqref{eq-equality-F-Fn}, we use the fact that
$ \tilde{v} \in L^\infty(\Omega\times[0, T], \mathbb{H}^{1+\frac\alpha2}(\mathbb{T}^2))$ and then we calculate the integral. Then it is easy that
the limit of this term vanishes. For the second term we use the dominated convergence theorem, we get}

\begin{eqnarray}\label{eq-equality-F-Fn-1}
\mathbb{E}\int_0^{T\wedge \xi_N}e^{-r_0(s\wedge \xi_N)}{}_{V^*}\langle F_2(s\wedge \xi_N) -
F(u(s\wedge \xi_N)), \tilde{v}(s\wedge \xi_N))\rangle_{V}ds= 0.
\end{eqnarray}
The justification of the application of the dominated convergence theorem  follows as in Section \ref{sec-Torus} using \eqref{Eq-B-H-alpha-2-est}.}\del{for $ s\leq \xi_N$
\begin{eqnarray}
|{}_{V^*}\langle F_2(s) &-&
F(u(s)-\lambda \tilde{v}(s)), \tilde{v}(s))\rangle_{V}|\leq
 |\tilde{v}(s)|_{\mathbb{H}^{\frac\alpha2, 2}}\big( |F_2(s)|_{\mathbb{H}^{-\frac\alpha2, 2}}+
|u(s)|_{\mathbb{H}^{\frac\alpha2, 2}}+ |\tilde{v}(s)|_{\mathbb{H}^{\frac\alpha2, 2}}\nonumber\\
&+& |B(u(s)-\lambda \tilde{v}(s))|_{\mathbb{H}^{-\frac\alpha2, 2}} \nonumber\\
&\leq&
 |\tilde{v}(s)|_{\mathbb{H}^{\frac\alpha2, 2}}\big( |F_2(s)|_{\mathbb{H}^{-\frac\alpha2, 2}}+
|u(s)|_{\mathbb{H}^{\frac\alpha2, 2}}+ |\tilde{v}(s)|_{\mathbb{H}^{\frac\alpha2, 2}}+
|u(s)|_{\mathbb{H}^{1, 2}}|u(s)|_{\mathbb{L}^{ 2}} \nonumber\\
&+& |\tilde{v}(s)|_{\mathbb{H}^{1, 2}}|\tilde{v}(s)|_{\mathbb{L}^{ 2}}
+|u(s)|_{\mathbb{H}^{1, 2}}|\tilde{v}(s)|_{\mathbb{L}^{ 2}} + |\tilde{v}(s)|_{\mathbb{H}^{1, 2}}|u(s)|_{\mathbb{L}^{ 2}}.
\end{eqnarray}By this procedure, we have constructed a sequence of progressively measurable processes $ (u_N)_{N\in \mathbb{N}_1}$ resolving up to a stoping time $ \xi_n$ te dD-FSNSE} 
\del{\section{Global existence and uniqueness of solution for 2D-FSNSEs.}\label{sec-global-mild-weak-solution}
We assume that $d=2$, $2< q<\infty$, $ p\geq 4$ and $ u_0\in L^p(\Omega, \mathbb{H}^{1, q}(O))$.  
We prove the global existence and the uniqueness of  mild  and weak solutions for multiplicative stochastic 2D-FSNSE.}
\end{proof}}

\section{Global existence and uniqueness of the weak solution of the multi-dimensional FSNSEs.}\label{sec-global-weak-solution}
\del{We denote by\del{ $ \mathfrak{V}$,  } $ \mathscr{E}$ the set of progressively measurable stochastic processes $ v$, such that 
$(v, \tau)$ is a maximal weak solution for  \eqref{Main-stoch-eq} with $\tau$ being a stopping time. 
Thanks to Appendix \ref{sec-Passage Velocity-Vorticity}, $ \theta := curl u$ is a local weak solution for the scalar active equation 
\eqref{Eq-vorticity-Torus-2-diff}

In this section, we prove the global existence and the uniqueness of the weak solution for the 2D-FNSE \eqref{Main-stoch-eq}.  
Thanks to Section \ref{sec-Domain}, we have   constructed a local weak solution $ (u, \xi)$ satisfying \eqref{est-local-u-tau-principle} and up to $ \xi_N$ for $ N\in \mathbb{N}_0$.
Under hypothesis in $ 8.3.2.$\del{ of Theorem \ref{Main-theorem-boubded-2} for this case ($d=2$) and} and using Appendix \ref{sec-Passage Velocity-Vorticity}
and \cite{Debbi-scalar-active}, we 
infer that (here we take, for simplicity, $ q= q_0\geq 6$)  
\begin{equation}\label{est-nabla-u-theta-Global-existence}
 \exists c>0, s.t. \forall N\in \mathbb{N}_0,  \mathbb{E}\sup_{[0, \xi_N)}|\nabla u(t)|^q_{L_{2\times 2}^q}\leq c \mathbb{E} \sup_{[0, \xi_N]}
|\theta(t)|^q_{L^q}\leq c <\infty.
\end{equation}
Recall $ \theta:= curl u$. Now, we change the calculus in \eqref{ineq-B-u-v-v-local} as follow. Using H\"older twice ($ 1/q+1/q'=1/2$), 
Gaglairdo-Nirenberg and than Young inequalities, we get
\begin{eqnarray}\label{ineq-B-u-v-v-global}
|{}_{V^*}\langle B(u)&-&B(v), u-v\rangle_{V} |\leq
| (u-v)\nabla v|_{L^2_2}|u-v|_{\mathbb{L}^{2}}
\leq
|u-v|_{L^{q'}_2}|\nabla v|_{L^q_2}|u-v|_{\mathbb{L}^{2}}\nonumber\\
&\leq& c|\nabla v|_{L^q_{2\times 2}}|u-v|^{2-\frac4{\alpha q}}_{\mathbb{L}^{2}}|u-v|^{\frac4{\alpha q}}_{\mathbb{H}^{\frac\alpha2, 2}}
\leq  c|\nabla v|^{\frac{1}{1-\frac{2}{\alpha q}}}_{L^q_{2\times 2}}|u-v|^{2}_{\mathbb{L}^{2}}+ c|u-v|^{2}_{\mathbb{H}^{\frac\alpha2, 2}}.
\end{eqnarray}
We take $ r'(t):= c(1+ |\nabla v(t)|^{\frac{1}{1-\frac{2}{\alpha q}}}_{L^q_{2\times 2}})$ with relevant constant $ c>0$ and 
$ v\in \mathscr{E}$, where $\mathscr{E}$ is the set of progressively measurable stochastic processes satisfying 
$ v \in  L^2(\Omega\times[0, T]; \mathbb{H}^{\frac\alpha2, 2}(\mathbb{T}^2))$ and 
$ \nabla v \in L^{(1-\frac{2}{\alpha q})^{-1}}(\Omega\times[0, T]; L_{2\times 2}^{q}(\mathbb{T}^2))$. We claim that  
 $\mathscr{E} \neq  \varnothing$. In fact, let us define for $ N$ fixed,
 $ (v_N(t), t\in [0, T])$, by 
  $v_N(t):= u(t\wedge \xi_N), \forall t\in [0, T])$, where 
$ (u, \xi)$ is the local solution constructed in Section \ref{sec-Domain}. It is easy to check, thanks to 
\eqref{est-nabla-u-theta-Global-existence} and the fact that $ q\geq 4$ and $ \frac43\leq \alpha \leq 2$, that $ v_N\in \mathscr{E}$. Indeed, we get
\begin{eqnarray}
\mathbb{E}\int_0^T|\nabla v_N(t)|^{\frac{1}{1-\frac{2}{\alpha q}}}_{L^q_{2\times 2}} dt &\leq &
 \mathbb{E}\int_0^{T}|\nabla u(t\wedge \xi_N)|^{\frac{1}{1-\frac{2}{\alpha q}}}_{L^q_{2\times 2}} dt
 \leq c\mathbb{E}\int_0^T|\theta(t\wedge \xi_N)|^{\frac{1}{1-\frac{2}{\alpha q}}}_{L^q})dt \nonumber\\
& \leq & c\mathbb{E}\int_0^T(1+|\theta(t)|^q_{L^q})dt<\infty.
\end{eqnarray}
\del{\begin{equation}
 \mathbb{E}\int_0^T|\nabla u(t)|^{\frac{1}{1-\frac{2}{\alpha q}}}_{L^q_{2\times 2}} dt
\leq c\mathbb{E}\int_0^T|\theta(t)|^{\frac{1}{1-\frac{2}{\alpha q}}}_{L^q})dt \leq c\mathbb{E}\int_0^T(1+|\theta(t)|^q_{L^q})dt<\infty.
\end{equation}}
Than, we apply the whole machinery as in Section \ref{sec-Torus} and  estimation as in Section \ref{sec-Domain} 
to prove the existence of the weak (strong in probability) solution of 
Equation \eqref{Main-stoch-eq} in the set $\mathscr{E}$.\del{We justify the 
application of the dominated convergence theorem as follow,
The only remained calculus is to give the justification to the 
application of the dominated convergence theorem as in \eqref{just-domin-Torus},} To prove the uniqueness of the solution in the set 
$\mathscr{E}$, we follow the steps as in Section \ref{sec-Torus}. In particular, in Formula \eqref{Torus-uniquenss-ito-formula-H-1}, 
we estimate the term $ \langle B(w(s)), u^1(s)\rangle = - \langle B(w(s), u^1(s)), w(s)\rangle$ using \eqref{ineq-B-u-v-v-global}. 
The existence and the uniqueness hold, therefore the local solution constructed in Section \ref{sec-Domain} is global and unique.

\begin{remark}
For $ d=3$, we can follow \cite{Giga-al-Globalexistence-2001}......
\end{remark}}

In this section, we prove the global existence and the uniqueness of the weak solution for the dD-FSNSE \eqref{Main-stoch-eq}.  For $ O=\mathbb{T}^2$, we have thanks to the  conditions in (\ref{Main-theorem-boubded-2}.1) and arguing as in the proof of the regularity in Section \ref{sec-Torus}, we infer that the maximal solution  $ (u, \xi)$ satisfies 
\begin{equation}\label{eq-2D-weak-sol-up-xi}
 \mathbb{E}\sup_{[0, \xi)}|u(t)|^q_{\mathbb{H}^{1, q}}+ \mathbb{E}\int_0^{T\wedge \xi}|u(t)|^2_{\mathbb{H}^{1+\frac\alpha2, 2}}dt\leq c<\infty.
\end{equation}
We denote by $\mathscr{E}$ the set of predictable stochastic processes $ (v(t), t\in [0, T])$ (or the extsension of $ v$ in the case $v$ is defined up to a stopping time)\del{(we can also take  an extension of $v$ if this latter is only defined up to a stopping time)} satisfying that there exists a stopping time $ \tau$ such that 
$ v \in  L^2(\Omega\times[0, \tau); \mathbb{H}^{\frac\alpha2, 2}(\mathbb{T}^2))$ and the process $ (\nabla v(t), t\in [0, \tau)) $ can be extended (we keep the some notation) to
$ \nabla v \in L^{(1-\frac{2d}{\alpha q})^{-1}}(\Omega\times[0, T]; L^{q}(\mathbb{T}^2))$, with the norm of $ \nabla v$ in this space is uniformly bounded, i.e. independently of the extension. We claim that  
 $\mathscr{E} \neq  \varnothing$. In fact, let us define,
 $ (\tilde{v}(t), t\in [0, T])$, by 
  $\tilde{v}(t):= u(t\wedge \xi), \forall t\in [0, T])$, where $ (u, \xi)$ is our maximal local solution. We have, for $ q$ characterized as in $(\ref{Main-theorem-strog-Torus}.3)$, (bellow d=2)
\begin{eqnarray}
\del{\mathbb{E}\int_0^T|\nabla v_N(t)|^{\frac{1}{1-\frac{2d}{\alpha q}}}_{L^q_{2\times 2}} dt &\leq &}
 \mathbb{E}\int_0^{T}|\nabla u(t\wedge \xi_N)|^{\frac{1}{1-\frac{2d}{\alpha q}}}_{q} dt
 \leq c\mathbb{E}\int_0^T|\theta(t\wedge \xi_N)|^{\frac{1}{1-\frac{2d}{\alpha q}}}_{L^q})dt \leq c\mathbb{E}\int_0^T(1+|\theta(t)|^q_{L^q})dt<\infty.\nonumber\\
\end{eqnarray}
Therfore\del{it is easy to check thanks to Estimate  \eqref{eq-bale-kato-majda-con} that} $ \tilde{v}\in \mathscr{E}$. Remark that the condition $ \frac{1}{1-\frac{2d}{\alpha q}}\leq q\Leftrightarrow 1+\frac{2d}{\alpha}\leq q$, see Remark \ref{Rem-2}. Now, we shall look for a solution in the set $\mathscr{E}$. We can go back to the first part of Section \ref{sec-Domain} and we repeat the same calculus until Estimate \eqref{ineq-B-u-v-v-local}, which we treat now  as follow. Using H\"older twice ($ 1/q+1/q'=1/2$), 
Gaglairdo-Nirenberg and than Young inequalities, we get (recall $V:=\mathbb{H}^{\frac\alpha2, 2} (O)$)
\begin{eqnarray}\label{ineq-B-u-v-v-global}
|{}_{V^*}\langle B(u)&-&B(v), u-v\rangle_{V} |\leq
| (u-v)\nabla v|_{L^2_2}|u-v|_{\mathbb{L}^{2}}
\leq
|u-v|_{L^{q'}_2}|\nabla v|_{q}|u-v|_{\mathbb{L}^{2}}\nonumber\\
&\leq& c|\nabla v|_{q}|u-v|^{2-\frac{2d}{\alpha q}}_{\mathbb{L}^{2}}|u-v|^{\frac{2d}{\alpha q}}_{\mathbb{H}^{\frac\alpha2, 2}}
\leq  c|\nabla v|^{\frac{1}{1-\frac{2d}{\alpha q}}}_{q}|u-v|^{2}_{\mathbb{L}^{2}}+ c|u-v|^{2}_{\mathbb{H}^{\frac\alpha2, 2}}.
\end{eqnarray}
We take $ r'(t):= c(1+ |\nabla v(t)|^{\frac{1}{1-\frac{2d}{\alpha q}}}_{q})$ with relevant constant $ c>0$. Than, we apply the whole machinery as in Section \ref{sec-Torus} and  the estimations as in Section \ref{sec-Domain} to get the existence of the global solution.\del{, we prove then the existence of global weak solution.}\del{to prove the existence of the weak (strong in probability) solution of Equation \eqref{Main-stoch-eq} in the set $\mathscr{E}$. We justify the application of the dominated convergence theorem as follow,
The only remained calculus is to give the justification to the 
application of the dominated convergence theorem as in \eqref{just-domin-Torus}} To prove the uniqueness of the solution in the set 
$\mathscr{E}$, we follow the steps as in Section \ref{sec-Torus}. In particular, in Formula \eqref{Torus-uniquenss-ito-formula-H-1}, 
we estimate the term $ \langle B(w(s)), u^1(s)\rangle = - \langle B(w(s), u^1(s)), w(s)\rangle$ using \eqref{ineq-B-u-v-v-global}. 
The existence and the uniqueness hold, therefore the local solution \del{constructed in Section \ref{sec-Domain}} is global and unique. The estimate \eqref{propty-of-2D-global-mild-sol} is obtained from \eqref{eq-2D-weak-sol-up-xi}.\del{as in the proof of the regularity in Section \ref{sec-Torus}.}

\vspace{0.25cm}

For the general case $ (\ref{Main-theorem-boubded-2}.2) $,\del{we know from Section \ref{sec-Domain}, that there exists at least one maximal local solution. I} if a maximal local weak solution enjoys \eqref{eq-bale-kato-majda-con}, then we have $\mathscr{E} \neq  \varnothing $ and thus we follow the proof above (for $ O=\mathbb{T}^2$) to get the results. If a maximal local weak solution enjoys Condition \eqref{eq-other-bale-kato-majda-con}, then the set $ \mathscr{E}_1 \neq  \varnothing$, where $ \mathscr{E}_1$ is the set of predictable  stochastic processes $ (v(t), t\in [0, T])$ (or the extsension of $ v$ in the case $v$ is defined up to a stopping time) satisfying that there exists a predictable stopping time $ \tau$ such that $ v \in  L^2(\Omega\times[0, \tau); \mathbb{H}^{\frac\alpha2, 2}(\mathbb{T}^2))$ and can be extended (we keep the some notation) to
$ v \in L^{\frac{4\alpha}{3\alpha-d-2}}(\Omega\times[0, T]; \mathbb{H}^{\frac{d+2-\alpha}{4}, 2}(O))$ uniformly, i.e. with the norm of $ v$ in this space is uniformly bounded independently of the extension. Now, we can continue from Estimate \eqref{ineq-B-u-v-v-local}\del{{eq-monoto-local}} and follow the proof as above and as in Section \ref{sec-Torus}.

\del{   Now, for $ O=\mathbb{T}^2$, the condition \eqref{eq-bale-kato-majda-con} follows from Lemma \ref{lem-basic-curl-gradient} and \cite[Theorem 2.6.]{Debbi-scalar-active}. In fact, 
\noindent The verification of the application of \cite[Theorem 2.6.]{Debbi-scalar-active} and of the obtention of the regularity  \eqref{propty-of-2D-global-mild-sol} are done exactly as in Section \ref{sec-global-mild-solution}.}

\del{We start doing the calculus in a general setting. 
In Section \ref{sec-Domain}, we have constructed  a local maximal weak solution $ (u, \xi)$ satisfying \eqref{eq-weak-l-2solu} up to $ \xi_N$ for all $ N\in \mathbb{N}_0$. We assume that $ (u, \xi)$ enjoys also Estimate  \eqref{eq-bale-kato-majda-con}.  We go back to the firts part of Section \ref{sec-Domain} and we repeat the same calculus until Estimate \eqref{ineq-B-u-v-v-local}, which we treat now  as follow. Using H\"older twice ($ 1/q+1/q'=1/2$), 
Gaglairdo-Nirenberg and than Young inequalities, we get (recall $V:=\mathbb{H}^{\frac\alpha2, 2} (O)$)
\begin{eqnarray}\label{ineq-B-u-v-v-global}
|{}_{V^*}\langle B(u)&-&B(v), u-v\rangle_{V} |\leq
| (u-v)\nabla v|_{L^2_2}|u-v|_{\mathbb{L}^{2}}
\leq
|u-v|_{L^{q'}_2}|\nabla v|_{q}|u-v|_{\mathbb{L}^{2}}\nonumber\\
&\leq& c|\nabla v|_{q}|u-v|^{2-\frac{2d}{\alpha q}}_{\mathbb{L}^{2}}|u-v|^{\frac{2d}{\alpha q}}_{\mathbb{H}^{\frac\alpha2, 2}}
\leq  c|\nabla v|^{\frac{1}{1-\frac{2d}{\alpha q}}}_{q}|u-v|^{2}_{\mathbb{L}^{2}}+ c|u-v|^{2}_{\mathbb{H}^{\frac\alpha2, 2}}.
\end{eqnarray}
We take $ r'(t):= c(1+ |\nabla v(t)|^{\frac{1}{1-\frac{2d}{\alpha q}}}_{q})$ with relevant constant $ c>0$ and 
$ v\in \mathscr{E}$, where $\mathscr{E}$ is the set of progressively measurable stochastic processes satisfying 
$ v \in  L^2(\Omega\times[0, T]; \mathbb{H}^{\frac\alpha2, 2}(\mathbb{T}^2))$ and 
$ \nabla v \in L^{(1-\frac{2}{\alpha q})^{-1}}(\Omega\times[0, T]; L^{q}(\mathbb{T}^2))$. We claim that  
 $\mathscr{E} \neq  \varnothing$. In fact, let us define for $ N$ fixed,
 $ (v_N(t), t\in [0, T])$, by 
  $v_N(t):= u(t\wedge \xi_N), \forall t\in [0, T])$, where 
$ (u, \xi)$ is our local solution\del{ (constructed in Section \ref{sec-Domain}).} Therfore, it is easy to check thanks to Estimate  \eqref{eq-bale-kato-majda-con} that $ v_N\in \mathscr{E}$. 
Than, we apply the whole machinery as in Section \ref{sec-Torus} and  estimation as in Section \ref{sec-Domain} 
to prove the existence of the weak (strong in probability) solution of 
Equation \eqref{Main-stoch-eq} in the set $\mathscr{E}$.\del{We justify the 
application of the dominated convergence theorem as follow,
The only remained calculus is to give the justification to the 
application of the dominated convergence theorem as in \eqref{just-domin-Torus},} To prove the uniqueness of the solution in the set 
$\mathscr{E}$, we follow the steps as in Section \ref{sec-Torus}. In particular, in Formula \eqref{Torus-uniquenss-ito-formula-H-1}, 
we estimate the term $ \langle B(w(s)), u^1(s)\rangle = - \langle B(w(s), u^1(s)), w(s)\rangle$ using \eqref{ineq-B-u-v-v-global}. 
The existence and the uniqueness hold, therefore the local solution constructed in Section \ref{sec-Domain} is global and unique.

Now, for $ O=\mathbb{T}^2$, the condition \eqref{eq-bale-kato-majda-con} follows from Lemma \ref{lem-basic-curl-gradient} and \cite[Theorem 2.6.]{Debbi-scalar-active}. In fact, for $ q\geq 4$,
\begin{eqnarray}
\del{\mathbb{E}\int_0^T|\nabla v_N(t)|^{\frac{1}{1-\frac{2}{\alpha q}}}_{L^q_{2\times 2}} dt &\leq &}
 \mathbb{E}\int_0^{T}|\nabla u(t\wedge \xi_N)|^{\frac{1}{1-\frac{2}{\alpha q}}}_{q} dt
 \leq c\mathbb{E}\int_0^T|\theta(t\wedge \xi_N)|^{\frac{1}{1-\frac{2}{\alpha q}}}_{L^q})dt \leq c\mathbb{E}\int_0^T(1+|\theta(t)|^q_{L^q})dt<\infty.\nonumber\\
\end{eqnarray}
\noindent The verification of the application of \cite[Theorem 2.6.]{Debbi-scalar-active} and of the obtaintion of the regularity  \eqref{propty-of-2D-global-mild-sol} are done exactly as in Section \ref{sec-global-mild-solution}.}

 \del{we conclude that $ \nabla u$ enjoys \eqref{est-nabla-u-theta-Global-existence} for $ \tau_N$ here is replaced by $ \xi_m$.Appendix \ref{sec-Passage Velocity-Vorticity}
and \cite{Debbi-scalar-active}, we  infer that (here we take, for simplicity, $ q= q_0\geq 6$)  
\begin{equation}\label{est-nabla-u-theta-Global-existence}
 \exists c>0, s.t. \forall N\in \mathbb{N}_0,  \mathbb{E}\sup_{[0, \xi_N)}|\nabla u(t)|^q_{L_{2\times 2}^q}\leq c \mathbb{E} \sup_{[0, \xi_N]}
|\theta(t)|^q_{L^q}\leq c <\infty.
\end{equation}
Recall $ \theta:= curl u$.}


\del{In this section, we prove Theorem \ref{Main-theorem-martingale-solution-d}.\del{ The method is quiet standard, see for Banach spaces 
\cite{Debbi-scalar-active} and for Hilbert space \cite{Flandoli-Gatarek-95}.} The main ingredients are Faedo-Galerkin approximations,
compactness and Skorokhod embedding theorem, see for similar calculus in
\cite{Debbi-scalar-active, Flandoli-Gatarek-95}.

\begin{lem}\label{lem-bounded-W-gamma-p}
The sequence $ (u_n)_n$ of solutions of Equations \eqref{FSBE-Galerkin-approxi} is uniformly bounded in the space
\begin{eqnarray}\label{Eq-W-}
L^2(\Omega, W^{\gamma, 2}(0, T; \del{\mathbb{H}^{-\delta', 2}(O))}
H_d^{-\delta', 2}(O))\cap L^2(0, T; \mathbb{H}^{\frac\alpha2, 2}(O))),
\end{eqnarray}
where  $ \delta'\geq_1\max\{\alpha, 1+\frac d{2}\}$ and $ \gamma <\frac12$.
\end{lem}
\begin{proof}
Thanks to Lemma \ref{lem-unif-bound-theta-n-H-1-domain}, it is sufficient to prove that $ (u_n(t), t\in [0, T])$ is
uniformly bounded in $L^2(\Omega, W^{\gamma, 2}(0, T; \mathbb{H}^{-\delta', 2}(O))$. We recall that the Besov-Slobodetski space $W^{\gamma,
p}(0, T; E)$, with $ E$ being a Banach space, $ \gamma \in (0, 1)$ and $ p\geq 1$, is the space
of all $ v\in L^P(0, T; E) $ such that
\begin{eqnarray}
||v||_{W^{\gamma, p}}:= \left(\int_0^T|v(t)|_E^pdt+
\int_0^T\int_0^T\frac{|v(t)-v(s)|_E^p}{|t-s|^{1+\gamma p}}
dtds\right)^{\frac1p}<\infty.
\end{eqnarray}
\noindent As  $(u_n(t), t\in [0, T])$ is the strong solution  of the finite dimensional  stochastic
differential equation \eqref{FSBE-Galerkin-approxi}, then  $u_n(t)$  is the solution of the stochastic integral equation 
\begin{equation}\label{FSBE-Integ-solu-Galerkin-approxi}
u_n(t)= P_nu_0 + \int_0^t(-A_\alpha u_n(r) + P_nB(u_n(r))dr + \int_0^tP_nG(u_n(r))\,dW_n(r),\; a.s.,\\
\end{equation}
for all $t\in [0, T]$. We denote by 
\begin{equation}\label{Eq-Drift-term}
I(t):=  \int_0^t(-A_\alpha u_n(r) + P_nB(u_n(r))dr
\end{equation}
and
\begin{equation}\label{Eq-Drift-term}
J(t):= \int_0^tP_nG
(u_n(r))\,dW_n(r).
\end{equation}
\noindent We prove that $ I(\cdot)$ is uniformly bounded in  $L^2(\Omega; W^{\gamma, 2}(0, T; H_d^{-\delta', 2}(O))$ 
and that the stochastic term $ J(\cdot )$ is uniformly bounded in $ L^2(\Omega; W^{\gamma, 2}(0, T;  \mathbb{L}^2(O))$, for all $ \gamma<\frac12$.\del{ as the stochastic term $ J$ is more regular then the drift term, see e.g. \cite[Lemma 2.1]{Flandoli-Gatarek-95}. \\}
Let  $ \phi \in H_d^{{\delta'}, 2}(O)$, using Identity \eqref{Eq-3lin-propsym}, we get
\begin{eqnarray}
| {}_{H_d^{-{\delta'}, 2}}\langle P_nB(u_n(r)), \phi\rangle_{H_d^{{\delta'}, 2}}|
&=& |\langle u_n(r) \cdot
\nabla P_n\phi, u_n(r)\rangle_{L_d^{2}}|\nonumber\\
&\leq& |\nabla P_n\phi|_{L^\infty}| u_n(r)|^2_{\mathbb{L}^{2}}.
\end{eqnarray}
Thanks to \cite[Remark 4 p 164, Theorem 3.5.4.ps.168-169 and Theorem 3.5.5 p 170]{Schmeisser-Tribel-87-book} for $ O= \mathbb{T}^d$, to
\cite[Theorem 7.63  and point 7.66]{Adams-Hedberg-94} for $ O$ being a bounded domain and to
the condition $ {\delta'}>1+\frac d{2}$,
we deduce for  $ 0<\epsilon < \delta'-1-\frac d{2}$,
 $$ |\nabla P_n\phi|_{L^\infty} \leq c |\nabla P_n\phi|_{H^{\epsilon+\frac d2, 2}} \leq c
|\phi|_{H_d^{1+\epsilon+\frac d2, 2}} \leq c |\phi|_{H_d^{\delta', 2}}.$$
Therefore,
\begin{eqnarray}\label{}
 |P_nB(u_n(r))|_{H_d^{-\delta', 2}}\leq c |u_n(r)|_{\mathbb{L}^{2}}^2
\end{eqnarray}
 and
\begin{eqnarray}\label{eq-unif-int-I(t)}
\int_0^T|I(t)|_{H_d^{-\delta', 2}}^2dt&\leq& c\int_0^T\int_0^t \big(|(-A_\alpha
u_n(r)|^2_{\mathbb{H}^{-\delta', 2}} + |P_nB(u_n(r))|^2_{H_d^{-\delta', 2}}\big)drdt\nonumber\\
&\leq& c\int_0^T\int_0^t\big(|
u_n(r)|^2_{\mathbb{L}^{ 2}} + |u_n(r)|^4_{\mathbb{L}^{2}}\big)drdt.
\end{eqnarray}
Moreover, using H\"older inequality and arguing as before, we get for $ t\geq s > 0$,
\begin{eqnarray}\label{eq-unif-int-I(t)-I(s)}
|I(t)- I(s)|^2_{H_d^{-\delta', 2}}&=& |\int_s^t(-A_\alpha u_n(r) +
P_nB(u_n(r))dr|^2_{\mathbb{H}^{-{\delta'}, 2}}\nonumber\\
&\leq & C(t-s)\left(\int_s^t(| u_n(r)|^2_{\mathbb{L}^{2}} +
|u_n(r)|^4_{\mathbb{L}^{2}})dr \right).
\end{eqnarray}
From \eqref{eq-unif-int-I(t)},  \eqref{eq-unif-int-I(t)-I(s)} and \eqref{Eq-Ito-n-weak-estimation-1-bounded}, 
we have for  $ \gamma <\frac12$,
\begin{eqnarray}\label{eq-unif-int-I(t)x2}
\mathbb{E}\big(\int_0^T|I(t)|_{H_d^{-\delta', 2}}^2dt&+&\int_0^T\int_0^T
\frac{|I(t)- I(s)|^2_{H_d^{-\delta',
2}}}{|t-s|^{1+2\gamma }} dtds\big)^{\frac12} \nonumber\\
&\leq& C\mathbb{E}\left(\int_0^T(| u_n(r)|^2_{\mathbb{L}^{2}} +
|u_n(r)|^4_{\mathbb{L}^{2}})dr
\right)^\frac12 \leq C<\infty.
\end{eqnarray}
Now, we estimate the stochastic term $ J$. 
Using the stochastic isometry, the contraction property of $ P_n$ and\del{ Assumption $
(\mathcal{C})$} Condition  \eqref{Eq-Cond-Linear-Q-G}, we get
\begin{eqnarray}
\int_0^T\mathbb{E}|\int_0^tP_nG(u_n(r))dW_n(r)|_{\mathbb{L}^{ 2}}^2dt&\leq&
C\int_0^T\mathbb{E}\int_0^t||G(u_n(r))||^2_{L_Q(\mathbb{L}^2)}drdt\nonumber\\
&\leq&
C\int_0^T\mathbb{E}\int_0^t(1+|u_n(r)|^2_{\mathbb{L}^2})drdt \leq c<\infty.
\end{eqnarray}
Moreover, for $ t\geq s> 0$ and $ \gamma <\frac12$, the same ingredients
above yield to
\begin{eqnarray}
\mathbb{E}\int_0^T\int_0^T\frac{|J(t)-
J(s)|^2_{\mathbb{L}^{2}}}{|t-s|^{1+2\gamma }} dtds &\leq&
C\mathbb{E}\int_0^T\int_0^T\frac{\int_s^t||G(u_n(r))||^2_{L_Q(\mathbb{L}^2)}dr}{|t-s|^{1+2\gamma
}} dtds \nonumber\\
&\leq& C\mathbb{E}\sup_{[0, T]}(1+|u_n(t)|^2_{\mathbb{L}^{2}})
\int_0^T\int_0^T|t-s|^{-2\gamma } dtds \leq c <\infty.
\end{eqnarray}
The proof of the Lemma is now completed.
\end{proof}
\begin{remark}
It is also possible, see e.g. \cite{Flandoli-Gatarek-95}, to 
\del{replace in \eqref{Eq-W-}, }prove the boundedness of the sequence $ (u_n)_n$ in 
\begin{eqnarray}
L^2(\Omega, W^{\gamma, 2}(0, T; D(A^{-\delta'})\cap L^2(0, T; \mathbb{H}^{\frac\alpha2, 2}(O))),
\end{eqnarray}

\end{remark}
\del{\noindent {\bf Proof of the existence of martingale solution
\del{Theorem \ref{Main-theorem-martingale-solution-d}}}}
To prove the existence of a martingale solution, we consider the Gelfand triplet \eqref{Gelfand-triple-Domain} and 
 use lemmas \ref{lem-unif-bound-theta-n-H-1-domain} and \ref{lem-bounded-W-gamma-p} 
and the compact embedding, see \cite[Theorem 2.1]{Flandoli-Gatarek-95}
\begin{equation}
 W^{\gamma, 2}(0, T; H_d^{-\delta', 2}(O))\cap \mathbb{L}^2(0, T; \mathbb{H}^{\frac\alpha2, 2}(O)) 
 \hookrightarrow L^2(0, T; \mathbb{L}^2(O)).
\end{equation}
\del{$ L^2(0, T; \mathbb{H}^{\frac\alpha2, 2}(O)) \hookrightarrow L^2(0, T; \mathbb{L}^2(O))$,}   Then we deduce
 that the sequence of laws $ (\mathcal{L}(u_n))_n$  is tight on $ L^2(0, T; \mathbb{L}^2(O))$. 
 Thanks to Prokhorov's theorem
there exists a  subsequence, still denoted $ (u_n)_n$, for which  the sequence of laws $ (\mathcal{L}(u_n))_n$ converges  weakly, on
$ L^2(0, T; \mathbb{L}^2(O))$,  to a probability measure $ \mu$. By Skorokhod's embedding theorem, we can construct a probability basis
$ (\Omega^*, F^*, \mathbb{F}^*,  P^*)$  and a sequence of $ L^2(0, T; \mathbb{L}^2(O))\cap C([0, T]; H_d^{-\delta', 2}(O))-$random variables
$ (u^*_n)_n$ and $ u^*$ such that  $\mathcal{L}(u^*_n) = \mathcal{L}(u_n), \forall n \in \mathbb{N}_0$,  $\mathcal{L}(u^*) = \mu$ and
$ u^*_n \rightarrow u^* a.s.$ in $ L^2(0, T; \mathbb{L}^2(O))\cap C([0, T]; H_d^{-\delta', 2}(O))$. Therefore, we conclude that
 for all $ n\in \mathbb{N}$,
\begin{eqnarray}\label{eq-bound-u-*-n-u-*}
\mathbb{E}\sup_{[0, T]}| u^*_n(s)|^p_{\mathbb{L}^2}+ \mathbb{E}\int_0^T| u^*_n(s)|_{ \mathbb{H}^{\frac\alpha2, 2}}ds
+\mathbb{E}\sup_{[0, T]}| u^*(s)|^p_{\mathbb{L}^2}+ \mathbb{E}\int_0^T| u^*(s)|^2_{ \mathbb{H}^{\frac\alpha2, 2}}ds \leq c<\infty.\nonumber\\
\end{eqnarray}
Consequently, the sequence  $ u^*_n$ converges weakly in $  L^2(\Omega\times [0, T]; \mathbb{H}^{\frac\alpha2, 2}(O))$ to a limit $ u^{**}$.
It is easy to see that $u^{*} = u^{**},  P\times dt a.e.$. We construct, as in \cite{Flandoli-Gatarek-95}, the process $ (M_n(t), t\in [0, T])$ with trajectories in
$ C([0, T]; \mathbb{L}^2(O))$ by
\begin{equation}
M_n(t):= u_n^*(t) - P_nu_0+\int_0^t A_\alpha u_n^*(s) ds -\int_0^t P_nB(u_n^*(s))ds.
\end{equation}
...>>>>>>>>>>>>>>>>>>>>>>>>>>>>>>>
We follow the same steps as in \cite{Flandoli-Gatarek-95} (hence we omit here the details), we end up with the statement that $ u^*$ is solution
in $ V^*:= \mathbb{H}^{\frac\alpha2, 2}(O)$ to Equation \eqref{Eq-weak-Solution} with $ W$ being replaced by $ W^*$.

\del{It is easy to see, thanks to Equation \eqref{FSBE-Galerkin-approxi} and to the equality in Law of $ u_n $ and $ u_n^*$,
that $ (M_n(t), t\in [0, T])$ is a square integrable martingale with respect to the filtration 
$(\mathcal{G}_n)_t:=\sigma\{u_n^*(s), s\leq t\}$. The remain part of the proof follows the same steps as in \cite{Flandoli-Gatarek-95},
hence we omit here.
\del{$ \mathbb{G}^* := (\mathcal{G}^*_t)_t$,  with $ \mathcal{G}^*_t $ is the $ \sigma-$algebra generated by
$ \cup_n (\mathcal{G}_n)_t:=\sigma\{u_n^*(s), s\leq t\}$.}}

The continuity of the trajectories of the solution $ u $ follows by a similar calculus as in 
Section \ref{sec-Torus} with the Gelfand triplet in \eqref{Gelfand-triple-Domain} and application of 
\cite[Proposition 2.5.]{Sundar-Sri-large-deviation-NS-06}. In deed,  we have thanks to 
Burkholdy-Davis-Gandy inequality, Assumption \eqref{Eq-Cond-Linear-Q-G}, \eqref{eq-bound-u-*-n-u-*}, we get
\begin{eqnarray}\label{cont-2-martg-stochas}
\mathbb{E}\sup_{[0, T]} |\int_0^tG(u(s))dW^*(s)|^2_{\mathbb{L}^{2}}
&\leq& c \mathbb{E}\int_0^T|G^*(u^*(s))|^2_{L_Q(\mathbb{L}^{2})}ds
\leq c (1+ \mathbb{E}\sup_{[0, T]}|u^*(s)|^2_{\mathbb{L}^{2}})<\infty.\nonumber\\
\end{eqnarray}
Moreover, using \eqref{Eq-B-H-alpha-2-est}, the Sobolev embedding 
$ \mathbb{H}^{\frac{d+2-\alpha}{4}, 2}(O)\subset \mathbb{H}^{\frac{\alpha}{2}, 2}(O)$, ( $ 1+\frac{d-1}{3}\leq \alpha \leq 2$) 
and the boundedness of the  operator $ A_\alpha: \mathbb{H}^{\frac{\alpha}{2}, 2}(O) \rightarrow \mathbb{H}^{\frac{-\alpha}{2}, 2}(O)$, we get
\begin{eqnarray}\label{cont-1-martg}
\mathbb{E}\int_0^T\big( |A_\alpha u^*(s)|_{\mathbb{H}^{-\frac\alpha2, 2}}&+& |B(u^*(s))|_{\mathbb{H}^{-\frac\alpha2, 2}}\big)ds \leq c
\mathbb{E} \int_0^T\big( |u^*(s)|_{\mathbb{H}^{\frac\alpha2, 2}} + |u^*(s)|^2_{\mathbb{H}^{\frac{d+2-\alpha}{4}, 2}}\big)ds\nonumber\\
&\leq& c (1+\mathbb{E} \int_0^T|u^*(s)|^2_{\mathbb{H}^{\frac\alpha2, 2}}ds)<\infty. 
\end{eqnarray}
>>>>>>>>>>>>>>>>>>>>>>>>>>>>>>>>>>>>>>>>>>>>}
\del{\begin{equation}\label{uniquness-con-martg}
 P^*( u (\cdot, \omega)\in L^{\frac{4\alpha}{3\alpha-d-2}}(0, T; \mathbb{H}^{\frac{d+2-\alpha}{4}, 2}(O)))=1.
\end{equation} \eqref{3linear-H1-H-1}, Young inequality,  we infer that
\begin{eqnarray}\label{Torus-uniquenss-ito-formula}
\mathbb{E}\!\!\!\!&{}&\!\!\!\!e^{-r(t\wedge \tau)}|w(t\wedge \tau)|^2_{\mathbb{L}^{2}}+
2\mathbb{E}\int_0^{t\wedge \tau}e^{-r(s)}|w(s)|^2_{\mathbb{H}^{\frac\alpha2, 2}}ds \nonumber\\
&\leq& \mathbb{E}\int_0^{t\wedge \tau}e^{-r(s)}|| G(u^1(s))- G(u^2(s))||^2_{HS_Q(\mathbb{L}^2)}\nonumber\\
& -& \mathbb{E}\int_0^{t\wedge \tau} e^{-r(s)} \big(2\langle B(w(s),  w(s)),  u^1(s)\rangle-
r'(s)|w(s)|^2_{\mathbb{L}^{2}}\big)ds\nonumber\\
& \leq & c_N\mathbb{E}\int_0^{t\wedge \tau}e^{-r(s)}\big(|w(s)|^2_{\mathbb{L}^2}+
|u^1(s)|_{\mathbb{H}^{?, 2}}|w(s)|^{\frac2\alpha}_{\mathbb{H}^{\frac{\alpha}{2}, 2}}
|w(s)|^{2\frac{1-\alpha}\alpha}_{{\mathbb{L}^{2}}}-
r'(s)|w(s)|^2_{\mathbb{L}^{2}}\big)ds\nonumber\\
& \leq & c_N
\mathbb{E}\int_0^{t\wedge \tau}e^{-r(s)}\big(|w(s)|^2_{\mathbb{L}^2}+ 2 c|w(s)|^2_{\mathbb{H}^{\frac\alpha2, 2}}+
2c_1|u^1(s)|^{\frac{\alpha}{\alpha-1}}_{\mathbb{H}^{1, 2}}|w(s)|^2_{\mathbb{L}^{2}}-
r'(s)|w(s)|^2_{\mathbb{L}^{2}}\big)ds.\nonumber\\
\end{eqnarray}
Now, we choose $ c<1$ and $ r'(s)= 2c_1|u^1(s)|^{\frac{\alpha}{\alpha-1}}_{\mathbb{H}^{1, 2}}$ and replace in
\eqref{Torus-uniquenss-ito-formula},
we end up by the simple formula
\begin{eqnarray}
\mathbb{E}e^{-r(t\wedge \tau)}|w(t\wedge \tau)|^2_{\mathbb{L}^{2}}&+& 2(1-c)\mathbb{E}\int_0^{t\wedge \tau}
e^{-r(s)}|w(s)|^2_{\mathbb{H}^{\frac\alpha2, 2}}ds
\leq c_N\mathbb{E}\int_0^{t\wedge \tau}e^{-r(s)}|w(s)|^2_{\mathbb{L}^2}ds.\nonumber\\
\end{eqnarray}
Than by application of Gronwall's lemma, we get
$ \forall t\in [0, T, \, e^{-r(t\wedge \tau)}|w(t\wedge \tau)|^2_{\mathbb{L}^{2}} = 0\, P-a.s.$ as
$P( e^{-c\int_0^{t\wedge \tau}|u^1(s)|^{\frac{\alpha}{\alpha-1}}_{\mathbb{H}^{1, 2}}ds} <\infty)= 1 $.
The proof is achieved once we remark that thanks to Chebyshev inequality and \eqref{cond-on the solution-Torus},
we have $ \lim_{N\rightarrow \infty}\tau_N = T a.s.$  and
$ P( \int_0^T|u^1(s)|^{\frac{\alpha}{\alpha-1}}_{\mathbb{H}^{1, 2}}ds <\infty)= 1$
\del{$ P( e^{-2c_1|u^1(s)|^{\frac{\alpha}{\alpha-1}}_{\mathbb{H}^{1, 2}}}<\infty)= 1$,
we conclude that
$\forall t\in [0, T] w(t)= 0, P-a.s.$}}

\del{\begin{lem}\label{lem-bounded-W-gamma-p}
The sequence $ (u_n)_n$ of solutions of Equations \eqref{FSBE-Galerkin-approxi} is uniformly bounded in the space
\begin{eqnarray}\label{Eq-W-}
L(\Omega, W^{\gamma, 2}(0, T; \mathbb{H}^{-\delta', 2}(O))\cap \mathbb{L}^2(0, T; \mathbb{H}^{\frac\alpha2, 2}(O)),
\end{eqnarray}
where  $ \delta'\geq_1\max\{\alpha, 1+\frac d{2}\}$.
\end{lem}
\begin{proof}
Thanks to Lemma \ref{lem-unif-bound-theta-n-l-2}, it is sufficient to prove that $ (u_n(t), t\in [0, T])$ is
uniformly bounded in $L(\Omega, W^{\gamma, 2}(0, T; \mathbb{H}^{-\delta', 2}(O))$. We recall that the Besov-Slobodetski space $W^{\gamma,
p}(0, T; E)$, with $ E$ being a Banach space, $ \gamma \in (0, 1)$ and $ p\geq 1$, is the space
of all $ v\in L^P(0, T; E) $ such that
\begin{eqnarray}
||v||_{W^{\gamma, p}}:= \left(\int_0^T|v(t)|_E^pdt+
\int_0^T\int_0^T\frac{|v(t)-v(s)|_E^p}{|t-s|^{1+\gamma p}}
dtds\right)^{\frac1p}<\infty.
\end{eqnarray}
\noindent As  $(u_n(t), t\in [0, T])$ is the solution  of the finite dimensional  stochastic
differential equation then, for $t\in [0, T]$,  we rewrite  $u_n(t)$  as
\begin{equation}\label{FSBE-Integ-solu-Galerkin-approxi}
u_n(t)= P_nu_0 + \int_0^t(-A_\alpha u_n(r) + P_nB_n(u_n(r))dr + \int_0^tP_nG_n(u_n(r))\,dW_n(r)\; a.s.\\
\end{equation}
We denote
\begin{equation}\label{Eq-Drift-term}
I(t):=  \int_0^t(-A_\alpha u_n(r) + P_nB_n(u_n(r))dr
\end{equation}
and
\begin{equation}\label{Eq-Drift-term}
J(t):= \int_0^tP_nG_n(u_n(r))\,dW^n(r).
\end{equation}
\noindent We prove that $ I(\cdot)$ is uniformly bounded in  $ W^{\gamma, 2}(0, T; \mathbb{H}^{-\delta', 2}(O))$
and that the stochastic term $ J(\cdot )$
is uniformly bounded in $ W^{\gamma, 2}(0, T;  L^2(\mathbb{T}^d)$, for all $ \gamma<\frac12$. \\
Let  $ \phi \in H^{{\delta'}, q^*}(\mathbb{T}^d)$. Recall that thanks to the choice of $ \delta'$ and to
\cite[Remark 4 p 164, Theorem 3.5.4.ps.168-169 and Theorem 3.5.5 p 170]{Schmeisser-Tribel-87-book}, we have
$  H^{{\delta'}, q^*}(\mathbb{T}^d) \hookleftarrow  L^2(\mathbb{T}^d)$. Therefore by use of Identity \eqref{Eq-3lin-propsym},
 we get
\begin{eqnarray}
| {}_{H^{-{\delta'}, q}}\langle P_nB_n(u_n(r)), \phi\rangle_{H^{{\delta'}, q^*}}|
&=& \mathcal{X}(|u_n(r)|_{L^2})|{}_{L^{2}}\langle \mathcal{R}^{\gamma, \sigma}u_n(r) \cdot \nabla P_n\phi, u_n(r)\rangle_{L^{2}}|\nonumber\\
&\leq& |\nabla P_n\phi|_{L^\infty}| \mathcal{R}^{\gamma, \sigma}u_n(r)|_{L^{2}}|u_n(r)|_{L^2}.
\end{eqnarray}
Thanks to \cite[Remark 4 p 164, Theorem 3.5.4.ps.168-169 and Theorem 3.5.5 p 170]{Schmeisser-Tribel-87-book} and to
the condition $ {\delta'}>1+\frac d{q^*}$,
we deduce for  $ 0<\epsilon < \delta'-1-\frac d{q^*}$,
 $$ |\nabla P_n\phi|_{L^\infty} \leq c |\nabla P_n\phi|_{H^{\epsilon+\frac d2, 2}} \leq c |\phi|_{H^{1+\epsilon+\frac d2, 2}}
\leq c |\phi|_{H^{{\delta'}, q^*}}.$$
Using Lemma \ref{lem-R-bounded-H-s} and the fact that $ 2\leq q <\infty$, we infer
\begin{eqnarray}\label{}
 |P_nB_n(u_n(r))|_{H^{-{\delta'}, q}}\leq c |u_n(r)|_{L^{q}}^2.
\end{eqnarray}
Hence,
\begin{eqnarray}\label{eq-unif-int-I(t)}
\int_0^T|I(t)|_{H^{-{\delta'}, q}}^2dt&\leq& c\int_0^T\int_0^t \big(|(-A_\alpha
u_n(r)|^2_{H^{-{\delta'}, q}} + |P_nB(u_n(r))|^2_{H^{-{\delta'}, q}}\big)drdt\nonumber\\
&\leq& c\int_0^T\int_0^t\big(|
u_n(r)|^2_{L^{ q}} + |u_n(r)|^4_{L^{q}}\big)drdt.
\end{eqnarray}
Moreover, using H\"older inequality and arguing as before, we get for $ t\geq s > 0$,
\begin{eqnarray}\label{eq-unif-int-I(t)-I(s)}
|I(t)- I(s)|^2_{H^{-{\delta'}, q}}&=& |\int_s^t(-A_\alpha u_n(r) +
P_nB_n(u_n(r))dr|^2_{H^{-{\delta'}, q}}\nonumber\\
&\leq & C(t-s)\left(\int_s^t(| u_n(r)|^2_{L^{ q}} +
|u_n(r)|^4_{L^{q}})dr \right).
\end{eqnarray}
Hence, from \eqref{eq-unif-int-I(t)} and \eqref{eq-unif-int-I(t)-I(s)}, we have for  $ \gamma <\frac12$,
\begin{eqnarray}\label{eq-unif-int-I(t)x2}
\mathbb{E}(\int_0^T|I(t)|_{H^{-{\delta'}, q}}^2dt&+&\int_0^T\int_0^T\frac{|I(t)- I(s)|^2_{H^{-\delta',
q}}}{|t-s|^{1+2\gamma }} dtds)^{\frac12} \nonumber\\
&\leq& C\mathbb{E}\left(\int_0^T(| u_n(r)|^2_{L^{ q}} +
|u_n(r)|^4_{L^{q}})dr
\right)^\frac12 \leq C<\infty.
\end{eqnarray}
The RHS of \eqref{eq-unif-int-I(t)x2} is uniformly bounded thanks to Estimation \eqref{Eq-Ito-n-weak-estimation-Lq}.
Now, we estimate the stochastic term $ J$. This last is more regular
then the drift term. We  prove $ J \in L(\Omega, W^{\gamma, 2}(0, T;
L^{q}(\mathbb{T}^d))$.
Using the stochastic isometry and Assumption $
(\mathcal{A})$, we get
\begin{eqnarray}
\int_0^T\mathbb{E}|\int_0^tP_nG_n(u_n(r))dW^n(r)|_{L^{ q}}^2dt&\leq&
C\int_0^T\mathbb{E}\int_0^t||G(u_n(r))Q^\frac12||^2_{R_\gamma(L^2, L^q)}drdt\nonumber\\
&\leq&
C\int_0^T\mathbb{E}\int_0^t(1+|u_n(r)|^2_{L^q})drdt <\infty.
\end{eqnarray}
Moreover, for $ t\geq s> 0$ and $ \gamma <\frac12$, the same ingredients
above yield to
\begin{eqnarray}
\mathbb{E}\int_0^T\int_0^T\frac{|J(t)-
J(s)|^2_{L^{2}}}{|t-s|^{1+2\gamma }} dtds &\leq&
C\mathbb{E}\int_0^T\int_0^T\frac{\int_s^t||G(u_n)Q^\frac12||^2_{R_\gamma(L^2, L^q)}}{|t-s|^{1+2\gamma
}} dtds \nonumber\\
&\leq& C\mathbb{E}\sup_{[0, T]}(1+|u_n(t)|^2_{L^{q}})
\int_0^T\int_0^T|t-s|^{-2\gamma } dtds <\infty.
\end{eqnarray}
The proof is now completed.
\end{proof}
\noindent {\bf Proof of  Theorem \ref{Main-theorem-mild-solution-3}.}
\noindent To prove  Theorem \ref{Main-theorem-mild-solution-3}, we can proceed by tow methods which seem, following our technique,
equivalent.
\begin{itemize}
 \item {\bf Method 1. Hilbert space setting.}  Thanks to Lemmas  \ref{lem-unif-bound-theta-n-L-2} and
\ref{lem-bounded-W-gamma-p}  with $ q=2$, we prove the first statement of Theorem \ref{Main-theorem-mild-solution-3}, i.e.
we prove the existence of a martingale solution
$(u^*(t), t\in [0, T])$ for Equation  \eqref{Main-stoch-eq} with $ (\sigma, \gamma)\in C_a$, in the sense of
Definition \eqref{def-martingle-solution}. Indeed, one can follow
\cite{Flandoli-Gatarek-95} combined with Lemma \ref{lem-R-bounded-H-s} and Proposition \ref{prop-Ben-q=2}.
To prove the second statement, we prove that $(u^*(t), t\in [0, T])$ satisfies
Equation \eqref{Eq-weak-Solution},  with $ \varphi \in H^{{\eta}, q^*}(\mathbb{T}^d)$ with $ q^*$ being the conjugate of $ q>2$.
To get the estimations \eqref{eq-set-solu-weak} and \eqref{Eq-reg-theta-mild}, we use Lemma \ref{lem-unif-bound-theta-n-l-2}.

\item {\bf Method 2. Banach space setting.} The idea here is to  extend the techniques used for the Hilbert space setting to Banach
spaces $ L^q(\mathbb{T}^d)$ with $ q>2 $. Let $ 2\leq q\leq_\infty q_0$. We prove that the
family of Laws of solutions $ (\mathcal{L}(u_n))_n$ is tight in $
L^2(0, T; L^{q}(\mathbb{T}^d))$. In fact, this latter is deduced thanks to Lemmas
 \ref{lem-unif-bound-theta-n-l-2} and  \ref{lem-bounded-W-gamma-p} and  the following embedding, see \cite[Theorem 2.1]{Flandoli-Gatarek-95} and
\cite[Remark 4 p 164, Theorem 3.5.4.ps.168-169 and Theorem 3.5.5 p 170]{Schmeisser-Tribel-87-book},
\begin{equation}
 L^2(0, T; H^{\frac\alpha2, 2})\cap W^{\gamma,
2}(0, T; H^{-{\delta'}, q}) \hookrightarrow L^2(0, T; L^{q}),
\end{equation}
provided $ d(1-\frac2q)\leq \alpha \leq 2$. This means that $ q $ and $d$ should enjoy the relation
$ d-2\leq \frac{2d}{q}$. This last is trivial for  $ d\leq \alpha$ and $ q\in[2, +\infty)$.
For $ d>\alpha$, we assume $  q\in [2, \frac{2d}{d-\alpha}]$. Therefore, under the conditions in Theorem \ref{Main-theorem-mild-solution-3} and
thanks to Skorokhod Theorem, we infer the existence of a probabilistic basis
$ (\Omega^*, \mathcal{F}^*, \mathbb{P}^*, \mathbb{F}^*, W^*)$ and  $L^2(0, T; L^{q})\cap C([0, T]; H^{-{\delta'}, q})$-valued random variables
$u^*$  and $(u^*_n)_n$ defined on this basis such that   $
\mathcal{L}(u_n^*) = \mathcal{L}(u_n) $ and  $u^*_n \rightarrow u^*, \;\; P-a.s. $
in $ L^2(0, T; L^{q})\cap C([0, T]; H^{-{\delta'}, q})$. Moreover, thanks to the above equality in law,
$(u_n^*(t), t\in[0, T])$ satisfies \eqref{Eq-Ito-n-weak-estimation-1} for all $ n\in \mathbb{N}_0$,
\eqref{Eq-Ito-n-weak-estimation-Lq} and \eqref{Eq-Ito-n-weak-estimation}.
Hence we deduce that $(u^*(t), t\in[0, T])$ satisfies  \eqref{eq-set-solu-weak} and \eqref{Eq-reg-theta-mild}.
To prove that $(u^*(t), t\in[0, T])$ is a weak solution of Equation \eqref{Main-stoch-eq} in the sense
of Definition \ref{def-variational solution}, i.e.  $ u^*$ satisfies Equation \eqref{Eq-weak-Solution}, we use Lemma \ref{lem-R-bounded-H-s} and
proceed as in \cite{Flandoli-Gatarek-95}.
\end{itemize}

\noindent {\bf Proof of Part 3. of Theorem \ref{Main-theorem-mild-solution}}

\noindent Let us remark, that thanks to Lemma  \ref{lem-unif-bound-theta-n-L-2} and  Lemma
\ref{lem-bounded-W-gamma-p}  with $ q=2$ in this later, we can prove Part 3. of
Theorem \ref{Main-theorem-mild-solution}, see, for this technique in Hilbert space, e.g. \cite{Flandoli-Gatarek-95, Rockner-quasi-geostro-1}.
As our method to prove the existence of the global mild solution is to use the result of the existence of the martingale solution,
we are then interested to extend this result for $ L^q-$spaces with $ q \geq 2$. We can proceed by two ways.
One way, is to prove that the $ L^2$-martingale solution  satisfies equation \eqref{Eq-weak-Solution} with $ \phi \in H^{\delta, q}$.
The second way is to prove the existence of martingale solution in the $ L^2(0, T; L^q).$ The techniques we are using seems equivalent in this task,
hence we will follow the second way.  It is obvious that our technique to prove
the martingale solution in $ L^q-$spaces is still valid for $ q=2$ without any conditions on the space dimension.

\noindent  We prove that the
family of laws of the solutions $ (\mathcal{L}(u_n))_n$ is tight in $
L^2(0, T; \mathbb{L}^{q}(O))$. In fact, the tightness of  $(\mathcal{L}(u_n))_n$ is deduced thanks to Lemmas
 \ref{lem-unif-bound-theta-n-l-2} and  \ref{lem-bounded-W-gamma-p} and  the following embedding
\begin{equation}
 L^2(0, T; H^{\frac\alpha2, 2})\cap W^{\gamma,
2}(0, T; H^{-\delta, q^*}) \hookrightarrow L^2(0, T; \mathbb{L}^{q}),
\end{equation}
provided $ \alpha \geq d(1-\frac2q)$*, see, e.g. \cite[Theorem 2.1]{Flandoli-Gatarek-95},
\cite[Theorem 7.63. p 221 + 7.66, p 222]{Adams-Hedberg-94}.   Therefore, we infer the existence of a probabilistic basis
$ (\Omega^*, \mathcal{F}^*, \mathbb{P}^*, \mathbb{F}^*)$ and an $L^2(0, T; \mathbb{L}^{q})\cap C([0, T]; H^{-\delta, q^*})$-valued random variables
$u^*$  and $u^*_n$ defined on this basis and such that \del{$
\mathcal{L}(u^*) = \mu$,}  $
\mathcal{L}(u_n^*) = \mathcal{L}(u_n) $ and  $u^*_n \rightarrow u^*, \;\; P-a.s. $
in $ L^2(0, T; \mathbb{L}^{q})\cap C([0, T]; H^{-\delta, q^*})$. Moreover, thanks to the above equality in law, the sequence,
$(u_n^*)_n$ satisfies \eqref{Eq-Ito-n-weak-estimation-1}, \eqref{Eq-Ito-n-weak-estimation-Lq} and \eqref{Eq-Ito-n-weak-estimation}.
Hence, we deduce that $u^*_n \rightarrow u^*, \;\; \mathbb{P}^*-a.s. $ weakly in $ L^2(\Omega\times [0, T]; H^{\frac\alpha2, 2})$
\footnote{In fact,$u^*_n \rightarrow u^*$ weakly to other limit which could be easily identified with $u^*$} and
\begin{equation}
 u^*(\cdot, \omega)\in L^2(0, T;H^{\frac\alpha2, 2}(O))\cap  L^\infty(0, T;\mathbb{L}^{q}(O)), \;\;\; P-a.s..
\end{equation}
Using the classical scheme in particular Skorokhod Theorem see \cite{Flandoli-Gatarek-95},  we prove $ u^*$ is a weak solution of the equation \eqref{Main-stoch-eq} in the sense
of Definition \ref{def-variational solution}, i.e.  $ u^*$ satisfies equation \eqref{Eq-weak-Solution}.

\noindent {\bf End of the proof} }


\del{\section{special Section Resolution of an auxiliary problem}\label{sec-2-approxim}

This definition is a simulation of Definition 4.4 in \cite{Daprato-Debussche-Martingale-Pbm-2-3-NS08}
\begin{defn}
We say that an $\mathcal{F}-$adapted process is a local weak (strong) solution for Equation \eqref{Main-stoch-eq}, if  there exists an increasing
sequence of stopping time $ (\tau_N)_N$ such that Equation \eqref{Eq-weak-Solution} is satisfied for all $ t\leq \tau_N$.
\end{defn}

\begin{lem}
 \begin{equation}
  |\rangle B(u-v, v), u-v\langle|\leq c|u-v|_{\mathbb{H}^{\frac\alpha2, 2}}^2 + |u-v|_{\mathbb{L}^2}^2|v|^{\frac{2}{\theta}}_{\mathbb{H}^{\frac\alpha2, q}}
 \end{equation}
with $ \theta := 2+\frac{2d}{\alpha}(\frac{1q} +\frac{1-s}{d})$, s>>>>>>>>>>>
\end{lem}

In this section, we introduce the following auxiliary problem SPDE
\begin{equation}\label{FSB-approx-Xi}
\Bigg\{
\begin{array}{lr}
dv(t)= (-A_\alpha v(t) + B(\xi(t), v(t))dt + G(v(t))\,dW(t), \; 0< t\leq T,\\
v(0) = u_0,
\end{array}
\end{equation}
where $ \xi_t$ is an $ \mathcal{F}-$adapted process satisfying for some $ \delta \geq  \frac{2d-1-\alpha}{4}$,
\begin{equation}
 \mathbb{E}[\sup_{[0, T]}|\xi(t)|_{H^{\delta, 2}}^p]<\infty.
\end{equation}
We will use the monotonicity method to prove the existence of a variational solution to , for more details about this method see e.g.
\cite{Krylov-Rozovski-monotonocity-2007, Pardoux-thesis, Rockner-Pevot-06} and  Appendix \ref{appendix-Monotonocity}. But, we will only give the
following definition of a variational solution of equation \eqref{FSB-approx-Xi}, see \cite[Definition 3.4.]{Krylov-Rozovski-monotonocity-2007}.
Let us give the Gelfant triple
\begin{equation}\label{Gelfant-triple}
 V \hookrightarrow H\tilde{=}  H^*\hookrightarrow  V^*,
\end{equation}
where $ H $ is a Hilbert space.
\begin{defn}\label{Def-solution-variational}
A strong continuous $ L^2-$valued $ \mathcal{F}_t-$adapted process $ (u(t))_{t\in [0, T]}$ is called a solution of equation \eqref{FSB-approx-Xi} iff

\begin{itemize}
 \item $ u \in V $ a.e. in $ (t, \omega)$ and for some given $ p>0$,
\begin{equation}\label{cond-Def-Kry-Rozo}
 \mathbb{E}\left(| u(t)|_{L^2}^p +\int_0^T|u(t)|_V^2dt \right).
\end{equation}
\item there exists a set $ \Omega'\subset \Omega $ of probability one on which for all $ t\in [0, T]$, the following equality holds in $ V^*$,
\begin{equation}\label{Eq-Def-Kry-Rozo}
 u(t)= u_0+ \int_0^t \left(-A_\alpha u(s) + B(\xi(s), u(s)) \right) ds + \int_0^t G(u(s))\,dW(s).
\end{equation}
\end{itemize}
\end{defn}

Let us here make precise the specificity of the application of the monotonicity method on the fractional dissipative problems, relatively
to the classical onces.  It is well known that the main ingredient of this
method is first to consider a  Gelfant triple \eqref{Gelfant-triple},  such that the drift term (let us denote by $ A(t, \omega, v)$) is well
defined and bounded as an operator from $V$ to $ V^*$.  In the case of the second order differential equations (classical case)
(e. g.  $ A= \Delta$ with linearity $ B(v)$ given by the gradient such as in our case, the Gelfant triplet is taken as
 \begin{equation}
 V= D((A^\frac12))=H_0^{1, 2}(\mathbb{T}^d) \hookrightarrow  H:= L^2(\mathbb{T}^d))^d \tilde{=}  H^*\hookrightarrow  V^*:= H^{-1, 2}(\mathbb{T}^d).
 \end{equation}
 It is obvious that $ V\subset D(B) =H^{1, 2}(\mathbb{T}^d)$. Hence,  the restriction of
$ B $ on $ V$ denoted also by $ B$ and the operator $ A$ are both  well defined and bounded from $V$ to $ V^*$.

In the fractional case, we can consider the Gelfant triple
\begin{equation}
 V=D(A_\alpha^\frac12)=H_0^{\frac\alpha2, 2} \hookrightarrow H^{0, 2}\tilde{=}  H^*\hookrightarrow  V^*:= H^{-\frac\alpha2, 2}.
\end{equation}
Hence, it is also obvious that $A_\alpha: V=H_0^{\frac\alpha2, 2} \rightarrow V^*=H^{-\frac\alpha2, 2}$  is bounded, however for $ \alpha<2$,
the space $ V $ is larger than the domain of definition ob $ B$. In particular,
$ H^{1}_0(\mathbb{T}^d) \subsetneq V$. Consequently,  to be able to apply the  monotonicity theorem,
we need to extend uniquely the operator $ B $ to a  bounded operator on $ V$. This will be  the content of the following proposition \del{This will be  the content of the following two propositions.
The first proposition is more relevant for the small regularization effect of the operator $ \mathcal{R}^{\gamma, \sigma}$, i.e. small $ \gamma $,
in particular for $ \gamma=1 $. In the second proposition, we investigate more the large regularization effect}
\begin{prop}
 Let $ d\in \{1, 2, 3\}$ and  $ \alpha > \frac{d+2}{3}$. Then the operator $ B $ is extended uniquely to a bounded operator
(we keep the same notation)  $ B:V \rightarrow V^*$.

\end{prop}

\begin{proof}
For generality reasons, we consider the bilinear form $ B( u, v)$ as defined in \eqref{B-theta1-theta2-divergence}.
Thanks to the fact that $ \sigma \in \sum_d^0$ and Lemma \ref{Lem-classic}, it is easy to deduce that  there exists a constant $ c>0$, s.t.
\begin{equation}
|B(u, v)|_{V^*} \leq \ c\sum_{j =1}^d| \mathcal{R}_j^{\gamma, \sigma}u \cdot  v|_{H^{\frac{2- \alpha}2}}.
\end{equation}
Using the pointwise multiplication in Theorem \ref{theor-1-SR},  Lemma \ref{lem-R-bounded-H-s} and the condition that $ \alpha > \frac{d+2}{3} $,
we  get
\begin{equation}\label{Eq-B-u-v-V-dual}
|B(u, v)|_{V^*} \leq \ c|u |_{H^{\frac{d+2- \alpha}4, 2}} |v|_{H^{\frac{d+2- \alpha}4, 2}}\leq c |u|_V|v|_V<\infty.
\end{equation}

\end{proof}
\begin{remark}
Let us remark that the above method is also applied for $ d>3$, but then $ \alpha >2$. Let us summarize:
\begin{itemize}
\item $ d=1 $ and $ \alpha >1$,
\item $ d=2 $ and $ \frac43\leq  \alpha \leq 2$,
\item $ d=3$ and $ \frac53\leq  \alpha \leq 2$,
\item $ d=4 $ and  $ \alpha = 2$,
\end{itemize}
\end{remark}

Our main result in this section is the following Theorem,
\begin{theorem}\label{Main-theorem-approxi-2}
 Equation \eqref{FSB-approx-Xi} admits a unique solution $ u$ (in the sense of Definition \ref{Def-solution}).
Moreover,  for all $ p>1$, there exists $ c>0$  independent of $ \xi$, such that
\begin{equation}\label{est-solution-variational}
 \mathbb{E}\left(\sup_{[0, T]}| u(t)|_{L^2}^p +\int_0^T|u(t)|_V^2dt \right)\leq c<\infty.
\end{equation}
\end{theorem}

\begin{proof}
\noindent It is easy to see that the coefficients of equation \eqref{FSB-approx-Xi} satisfy the assumptions  $ (H_1)-(H_3)$, with $ p'=2$
and $ \lambda =1$.  We are enable to get  $ (H_4')$. Then we try to get other estimation. Indeed, using the first inequality in
\eqref{Eq-B-u-v-V-dual}  and Young inequality, we get
\begin{eqnarray}
 | A(v)|_{V^*} &\leq & c( | v|_{H^{\frac\alpha2, 2}}+ |\xi|_{H^{\frac{d+2-\alpha}{4}}} | v|_{H^{\frac{d+2-\alpha}{4}}})\nonumber\\
&\leq & c\left( 2|v|_{H^{\frac\alpha2, 2}}+ |v|_{L^2} |\xi_t|_{H^{\frac{d+2-\alpha}{4}}}^{\frac{2\alpha}{3\alpha-(d+2)}} \right).\nonumber\\
\end{eqnarray}
We replace the assumption $ (H_4')$ in Theorem \ref{Teorem-Krylov-Rozovsky} by  the assumption
\begin{itemize}
\item  $(H_4)$ (Boundedness of the growth of $ A(t, .)$),
\begin{equation}
 | A(v)|^2_{V^*}\leq  c(N_t^{2}+ K|v|_V^{2})(1+ |v|_H^{2}),
\end{equation}
\end{itemize}
where here $ N_t := |\xi_t|_{H^{\frac{d+2-\alpha}{4}}}^{\frac{2\alpha}{3\alpha-(d+2)}} \in L^2([0, T]\times\Omega; dt\times P)$.
Using the same proof as in  \cite{Lui-Roekner-nonlocal-monotonicity-10}, we infer the existence of a unique strong solution of equation
\eqref{FSB-approx-Xi} in the sense of definition \ref{Def-solution-variational} satisfying the estimation
\begin{equation}\label{est-solution-variational-c-xi}
 \mathbb{E}\left(\sup_{[0, T]}| v(t)|_{L^2}^p +\int_0^T|v(t)|_V^2dt \right)\leq c_\xi<\infty.
\end{equation}

One also can apply the result in  \cite{Zhang-random-coe-09}, to check all the details.

The last step is to prove that the LHS in \eqref{est-solution-variational-c-xi} is uniformly bounded (independently of $ \xi$). In fact, by
application of Ito formula,
see, e.g. \cite[Theorem 4.2.5]{ Roeckner-Pevot-06},  we get,

\begin{equation}
|u(t)|_{H} ^p= |u_0|_H^p+ 2\int_0^t\left(|u(s)|_H^{p-2}({}_{V^*}\langle -A_\alpha u(s)+ \xi_s\cdots \nabla u(s), u(s)\rangle_V)+ 2\int_0^t
|u(s)|_H^{p-2}||G(u(s))||_Q^2\right)ds + 2 \int_0^t{}_{V^*}\langle  u(s), G(u(s)) dW(s)\rangle_V.
\end{equation}
 Thanks to  property  \eqref{B-v-v-v} (the divergence free  property of $ \xi$), we infer
\begin{equation}
\mathbb{E}\left[\sup_{[0, T]}|u(t)|_{H} ^p +  \mathbb{E}\int_0^T|u(t)|_{H} ^{p-2}| u(s)|^2_Vds \right] \leq C(\mathbb{E}|u_0|_H^2+ 1).
\end{equation}

\end{proof}
  Thanks to our hypothesis $ $
as it is announced above.
The proof is a mixture of the proofs in \cite{Krylov-Rozovski-monotonocity-2007, Lui-Roekner-nonlocal-monotonicity-10, Roeckner-Pevot-06}, hence
some standard facts will be only omitted.

\noindent We use the Galerkin approximation. Let $ P_n$ be the projection of $ V^*$ onto the span $ \{e_1, e_2, \cdots, e_n\}$ and let
$ P_n$ be the projection of $\mathbb{E} $ onto the span $ \{h_1, h_2, \cdots, h_n\}$.
We set $ u^n:= \sum_{j=1}^n{}_{V^*}\langle u^n, e_j\rangle_Ve_j $. Then the following equation admits a unique solution,

\begin{equation}\label{Eq-Def-Kry-Rozo}
 u^n(t)= P_nu_0+ \int_0^t \left(-P_nA_\alpha u^n(s) + P_n(B(\xi(s), u^n(s))) \right) ds + \int_0^t \tilde{P}_nG(u^n(s))\,dW(s).
\end{equation}}

\appendix
\section{Equivalence between FSNS and SFNS equations.}\label{Appendix-Equivalence}
Recall that we have proved in Section \ref{sec-formulation} that Equation \eqref{Main-stoch-eq} with $ A_\alpha:= (A^S)^{\frac{\alpha}{2}}$ is well defined. This equation can be seen as the fractional version of the stochastic Navier-Stokes equation (FSNSE). 
A stochastic version of the fractional Navier-Stokes equation (SFNSE)
 can also be constructed by taking 
$ A_\alpha:= \Pi(-\Delta)^{\frac{\alpha}{2}}$ on $ \mathbb{L}^q(O)$. 
For simplicity, let us keep in mind for a short time that the two equations, FSNSE and SFNSE, are different. Later on, we shall prove that they are equivalent. By a fractional Navier-Stokes equation (FNSE), we mean Equation 
 \eqref{Eq-classical-SNSE-O-p}, with $ \Delta$ replaced by $ -(-\Delta)^\frac\alpha2$. The SFNSE is a 
stochastic perturbation of FNSE. Thanks to theorems \ref{Prop-1-Laplace-Stokes}-\ref{Lem-semigroup} 
and to the calculus above, the SFNSE is also well defined.\del{The main question now is whether or not the two equations FSNSE and SFNSE are equivalent. Let us  before moving to this last question,  
clarify some practical and theoretical matters for the two versions. 
Indeed,} As the derivation of equations describing physical phenomena is mainely based on the deterministic case it is intuitively seen that the stochastic version of the fractional Navier-Stokes equation is more suitable for physical modeling, rather than the fractional version of the stochastic Navier-Stokes equation, see e.g. \cite{Caffarelli-2009, SugKak, Sug, Sug-Frac-cal-89}. As we shall prove the equivalence of the two versions, we conclude that the FSNSE, seems intuitively more theoretical, is  of practical interest as well.

\noindent  In the case $ O= \mathbb{T}^d$, the operators  $ \Delta $ and $ div$ are commuting. 
Therefore, the Stokes operator $A^S$
is minus the Laplacian $ \Delta$ on $ \mathbb{L}^q( \mathbb{T}^d)$, see e.g.
\cite{Gigaweak-strong83},
\cite[p. 48]{Foias-book-2001}, \cite[p. 9]{Temam-NS-Functional-95} and \cite[p 105]{Temam-Inf-dim-88} for the torus,
 \cite{Kato-Ponce-86} for $ O=\mathbb{R}^d$  and see also a direct  proof in Section \ref{sec-Torus}. Combining this statement with the results in Theorem \ref{theorem-domains-A-D-A-S}, we conclude that  
\begin{eqnarray}\label{eq-A-S-alpha-Delta-alpha}
D((A^S)^\frac\alpha2) &=& D(\Pi_q(-\Delta)^\frac\alpha2\Pi_q) = H_d^{\alpha, q}(\mathbb{T}^d)\cap \mathbb{L}^q(\mathbb{T}^d),\nonumber\\
(A^S)^\frac\alpha2 u &=& \Pi(-\Delta)^\frac\alpha2 u= (-\Delta)^\frac\alpha2 u, \;\; \forall u \in D((A^S)^\frac\alpha2).
\end{eqnarray}
This proves that the FSNSE and the SFNSE  defined on the torus are ''equivalent''. 
In the case $ O\subset \mathbb{R}^d$  being bounded, the Stokes operator $ A^S$ is not equal to $ -\Delta$.  In fact,
 as  we can not in general expect that if $ u \in D(A^S)$ we also have
$ \Delta u\cdot \vec{n} =0 $ on $\partial O$ it is not obvious whether or not $ \Delta u \in \mathbb{L}^q(O)$. Our claim here is that $ (A^S)^\frac\alpha2 =  \Pi(A^D)^\frac\alpha2 \Pi$.
In deed, thanks to \eqref{def-A-D} and \eqref{def-A-S}, it is easy to deduce that 
 $ A^S = \Pi A^D\Pi$, see also \cite{Fujiwara-Morimoto-L-r-Helmholtz-decomposition-77}.\del{In fact, this latter is a simple 
consequence of \eqref{def-A-D} and \eqref{def-A-S}.
using \eqref{def-A-D} and \eqref{def-A-S}, we get
\begin{eqnarray}
 D(\Pi A^D\Pi)&=& D(A^D)\cap \mathbb{L}^q(O) = D(A^S)\nonumber\\
\Pi A^D u &=& -\Pi \Delta u = A^Su, \;\;\;  \forall u \in  D(\Pi A^D)\cap \mathbb{L}^q(O).
\end{eqnarray}
In the second step, we prove that $ (A^S)^\frac\alpha2 =  \Pi(A^D)^\frac\alpha2 \Pi$.}
Using Theorem \ref{theorem-domains-A-D-A-S}, we infer that
\begin{equation}
 D(\Pi (A^D)^\frac\alpha2\Pi)=  D((A^D)^\frac\alpha2)\cap \mathbb{L}^q(O) = D((A^S)^\frac\alpha2).
\end{equation}
\del{Recall the following definition of the negative power of $ A$, with  
$ A$ could be either $ A^S$ or $ A^D$,  see e.g. 
\cite{Giga-Doamian-fract-Stokes-Laplace, Pazy-83}, 
\begin{equation}\label{negative-fractional-A}
 A^{-\frac\alpha2} u = \frac1{2\pi i}\int_\Gamma z^{-\frac\alpha2}(A-zI_{L_d^q})^{-1} u dz, \; \forall u \in \mathbb{L}^q(O).
\end{equation}
where  $ \Gamma $ is the path running the resolvent set from $ \infty e^{-i\theta}$ to  $ \infty e^{i\theta}$, $ <0\theta <\pi$,  avoiding the 
negative real axis and the origin and such that the branch $  z^{-\frac\alpha2}$ is taken to be positive for real for real positive values of $z$ and $ I_X$ is the identity on the space $ X$.
The integral in the RHS of \eqref{negative-fractional-A-D} converges in the uniform operator topology.}
Moreover, using the definition of the negative power of $ A^S$ and $ A^D$ via the resolvent, see e.g. 
\cite{Giga-Doamian-fract-Stokes-Laplace, Pazy-83} and the definition of the Helmholtz projection, we infer that 
\begin{equation}\label{negative-fractional-A-D}
 (A^D)^{-\frac\alpha2} \Pi^{-1}u = \frac1{2\pi i}
 \int_\Gamma z^{-\frac\alpha2}(A^D-zI_{L_d^q})^{-1}\Pi^{-1}u dz,\;\;\; \forall u \in \mathbb{L}^q(O),
\end{equation}
where  $ \Gamma $ is the path running the resolvent set from $ \infty e^{-i\theta}$ to  
$ \infty e^{i\theta}$, $ 0<\theta <\pi$,  avoiding the 
negative real axis and the origin and such that the branch $  z^{-\frac\alpha2}$ is taken to be positive
for real  positive values of $z$ and $ I_X$ is the identity on the space $ X$.
The integral in the RHS of \eqref{negative-fractional-A-D} 
converges in the uniform operator topology. Furthermore, we have for all $ u \in \mathbb{L}^q(O)$,
\begin{eqnarray}\label{negative-fractional-A-D-last}
 (A^D)^{-\frac\alpha2} \Pi^{-1}u &=& \frac1{2\pi i}\int_\Gamma z^{-\frac\alpha2}(\Pi A^D-z\Pi)^{-1}u dz
 =\frac1{2\pi i}\int_\Gamma z^{-\frac\alpha2}(A^S-zI_{\mathbb{L}^q})^{-1}u dz
= :(A^S)^{-\frac\alpha2}.\nonumber\\
\end{eqnarray}
As the operators $(A^S)^{-1}$ and $(A^D)^{-1}$ are one-to-one this achieved the proof of the equivalence between the FSNSE and the 
SFNSE.

\del{\section{Local  mild solution of the multidimensional FSNSE.}\label{appendix-local-solution}

\noindent\del{ In this Section, we prove Theorem \ref{Main-theorem-mild-solution-d}.} Let us first recall the following result,
\begin{lem}\label{lem-Lipschitz-pi-n}
 \noindent  Let $ n$ be fixed and let $ X$ be a Banach space.
We define the application $ \pi_n: X \rightarrow  X$, by
\begin{equation}\label{Eq-Pi-n-X}
\pi_n(x)=
\Bigg\{
\begin{array}{lr}
x, \;\;  |x|_{X}\leq n,\\
nx|x|^{-1}_{X}, \;\;  |x|_{X}> n.
\end{array}
\end{equation}
\noindent Then $  \pi_n$ satisfies the following estimates,
\begin{equation}\label{inea-Pi-n}
 |\pi_n x- \pi_n y|_{X}\leq 2 |x-y|_{X}\;\;\; \text{and} \;\;\;
 |\pi_n x|_{X}\leq  \min\{n, |x|_{X}\}.
\end{equation}
\end{lem}
\noindent  For $ n\in \mathbb{N}_0$, we introduce the following sequence of equations
\begin{equation}\label{Eq-approx-n}
\Bigg\{
\begin{array}{lr}
 du_n= \left(-
A_{\alpha}u_n(t) + B(\pi_nu_n(t))\right)dt+ G(\pi_n(u_n(t)))dW(t), \; 0< t\leq T,\\
u_n(0)= u_0.
\end{array}
\end{equation}
\noindent For the fixed parameters $ 2< p <\infty$, $ 2< q\leq q_0$ and $ T_1\leq T$ (for simplicity reasons, we omit the subscription of all  these parameters),  we define the spaces $ \mathcal{E}_T$ and $ E_T$ as in \eqref{E-cal-T} and \eqref{E-T} respectively with $ X:= D(A_q^{\frac\delta2})$. Recall that as we have seen in Section \ref{sec-Domain}, $ \pi_n$ is well defined on $D(A_q^{\frac\delta2})$. Furthermore, we define the map  $ \Phi_{n}$ on $ \mathcal{E}_T$ by\del{ We denote by $ \mathcal{M}_{q, \delta, p}$
the space of adapted  stochastic
processes with values in $ D(A_q^{\frac {\delta}2})$,  (later, we use for simplicity, the
notation $\mathcal{M}_{\delta} $),  endowed  by the norm
\begin{equation}\label{norm-M}
|\theta |_{\mathcal{M}_{\delta}}:= \left(\mathbb{E}\sup_{[0,T_1]}|\theta(t)|^p_{D(A_q^{\frac{\delta}2})}\right)^\frac 1p<\infty.
\end{equation}}

\begin{equation}\label{Eq-def-Phi}
\Phi_{n}(u)(t) := e^{-A_\alpha t}u_0 + \int_0^te^{-A_\alpha (t-s)}B(\pi_n(u(s)))ds + \int_0^te^{-A_\alpha (t-s)}
G(\pi_n(u(s)))W(ds).
\end{equation}
\del{The aim now  is to prove that the map $ \Phi_{n} $ is a contraction on $ \mathcal{E}_{T}$ and than we apply the
fixed point Theorem. The claim is}

\noindent {\bf Claim}
\noindent Under the conditions of Theorem \ref{Main-theorem-mild-solution-d} the map $\Phi_{n}$  is well defined and it is a
contraction on $ \mathcal{E}_{T}$.

\noindent {\bf Proof of  the claim}
Let us denote by $ I_0,\; I_1(u), \; I_2(u)$ respectively the terms in the RHS of \eqref{Eq-def-Phi} and prove that $I_j(u)
\in \mathcal{E}_{T}$, for $ j \in\{0, 1, 2\}$.
In fact, $ I_0 \in \mathcal{E}_{T} $ thanks to \eqref{Eq-initial-cond} and to the boundedness of the operators of the semigroup. From Lemma \ref{lem-est-z-t} and the second
estimate in \eqref{inea-Pi-n}, we conclude that
$ I_2(u)  \in \mathcal{E}_{T}$.  In particular, there exists $ c_n >0$, such that
\begin{eqnarray}\label{est-2-term-B-0-2}
|I_2 (u)|_{\mathcal{E}_{T_1}} & \leq & c_nT_1 (1+|u|_{\mathcal{E}_{T_1}}).
\end{eqnarray}
Moreover, thanks to Proposition \ref{Prop-Main-I} and
the second estimate in  \eqref{inea-Pi-n}, we infer the existence of $ c, \mu>0$ such that
\begin{eqnarray}\label{est-2-term-B-0-1}
|I_1(u)|_{\mathcal{E}_{T_1}} &\leq& CT_1^{\mu}\left(\mathbb{E}\sup_{[0, T_1]}|\pi_nu(s)|^{2p}_{D(A_q^\frac{\delta}2)} \right)^\frac1p
\leq  ncT_1^{\mu}|u|_{\mathcal{E}_{T_1}}<\infty. \nonumber\\
\end{eqnarray}
This proves that $ \Phi_{n} $ is well defined. To prove that $ \Phi_{n} $ is a contraction, we first remark that
\begin{equation}\label{formula-B-v1-B-v2}
 B(u_1, u_1) - B(u_2, u_2) = B(u_1, u_1-u_2)+ B(u_1-u_2, u_2).
\end{equation}
Thanks to Proposition \ref{Prop-Main-I} and the first estimate in \eqref{inea-Pi-n}, we get\del{ ( see also a similar calculus in the proof of Theorem
\ref{prop-global-mild-solution} bellow)}
\begin{equation}\label{eq-lipschitz-B}
\Big(\mathbb{E} \sup_{[0, T_1]}|\int_0^te^{-A_\alpha (t-s)}\big(B(\pi_n(u_1(s)))-B(\pi_n(u_2(s)))\big)
ds|^p_{D(A_q^{\frac\delta2})}\Big)^\frac1p\leq cnT_1^{1-\frac1\alpha(1+\frac dq)}
|u_1-u_2|_{\mathcal{E}_{T_1}}.
\end{equation}
Using Assumption $(\mathcal{C})$, \cite[Proposition 4.2]{Neerven-Evolution-Eq-08}, the calculus as in Lemma \ref{lem-est-z-t}
 and the fact that $ \pi_n$ is globally Lipschitz, we get
\begin{equation}\label{eq-lipschitz-G}
\Big(\mathbb{E} \sup_{[0, T_1]}|\int_0^te^{-A_\alpha (t-s)}\big(G(\pi_n(u_1(s))) - G(\pi_n(u_2(s)))\big)W(ds)|^p_{D(A_q^{\frac\delta2})}\Big)^\frac1p
\leq c_n T_1|u_1-u_2|_{\mathcal{E}_{T_1}}.
\end{equation}
Consequently, thanks to \eqref{eq-lipschitz-B} and \eqref{eq-lipschitz-G} we infer that
\begin{equation}
 | \Phi_{n}(u_1) -\Phi_{n}(u_2)|_{\mathcal{E}_{T_1}}\leq C_nT_1^\mu | u_1 -u_2|_{\mathcal{E}_{T_1}},\;\; \mu>0.
\end{equation}
The time interval $ [0, T_1]$ is chosen such that $ C_nT_1^\mu <1$. Therefore by application of the fixed point theorem, we infer the existence
 of $ (u_n(t), t\in [0, T_1]) \in \mathcal{E}_{T}$  solution of \eqref{Eq-approx-n}, for all $ t \leq T_1$. Using the semigroup property of
$ (e^{-tA_\alpha}, t \in [0, T])$, we extend $  (u_n(t), t\in [0, T_1]) $ to be defined on $ [0, T]$.
\noindent We define the following stopping time, see Appendix \ref{append-stop-time} for the proof,
\begin{equation}\label{Eq-def-tau-n-delta}
 \tau_n := \inf\{t\in (0, T), \; s.t. |u_n(t) |_{D(A^{\frac\delta2})}> n\}\wedge T,
\end{equation}
with the understanding that $ \inf(\emptyset)=+\infty$.\del{and prove that $ \tau_n $ is a stopping time. Indeed, it is easy to check,  using Lemma \ref{lem-est-z-t}, Estimate \eqref{Eq-Pi-n-X} and a similar calculus as in \eqref{est-1-semi-group-second-term-1}, that the real  
$ \mathcal{F}_t-$adapted stochastic processes $ (\langle u_n(s), e_k\rangle_{\mathbb{L}^2})_{k\in\Sigma}$ are continuous for all $ k\in \Sigma$, see also Remark \ref{Rem-1}.
Therefore, \del{ the positive $ \mathcal{F}_t-$adapted stochastic real processes
\begin{equation}
X_j(\omega, s):= |\sum_{k\in \Sigma_j}|k|^{\delta} \langle u_n(\omega, s), e_k\rangle_{\mathbb{L}^2}e_k|_{\mathbb{L}^q},
\end{equation} }the positive $ \mathcal{F}_t-$adapted stochastic real processes
\begin{equation}
X_j(\omega, s):= |\sum_{k\in \Sigma_j}|k|^{\delta} \langle u_n(\omega, s), e_k\rangle_{\mathbb{L}^2}e_k|_{\mathbb{L}^q},
\end{equation} 
are also continuous for all $ j\in \mathbb{N}_1$, with $ (\Sigma_j)_j\subset \Sigma$  being an increasing sequence of finite subsets converging to $ \Sigma$. Therefore, see e.g. \cite[Propositions 4.5 or 4.6 ]{Revuz-Yor}, (Recall that we have assumed that the filtration $ (\mathcal{F}_t)_{t\in[0, T]}$ is right continuous), \del{therefore an optional time is also a stopping time.}  
\del{\begin{eqnarray}
\{\tau_n> t\} &=& \cap_{s\leq t}\{\omega, |u_n(s, \omega)|_{D(A_q^\frac\delta2)}\leq n\}\nonumber\\
&=& \cap_{s\leq t}\{\omega, |\sum_{k\in \Sigma}|k|^{\delta} \langle u_n(\omega, s), e_k\rangle_{\mathbb{L}^2}e_k|_{\mathbb{L}^q}\leq n\}
\nonumber\\
&=& \cap_{s\leq t}\cup_{m\in\mathbb{N}}\cap_{j\geq m}\{\omega, |\sum_{k\in \Sigma_j}|k|^{\delta} \langle u_n(\omega, s), e_k\rangle_{\mathbb{L}^2}e_k|_{\mathbb{L}^q}\leq n\}.
\end{eqnarray}}
\begin{eqnarray}
\{\tau_n\leq t\} &=& \cup_{s\leq t}\{\omega, |u_n(s, \omega)|_{D(A_q^\frac\delta2)}\geq n\}\nonumber\\
&=& \cup_{s\leq t}\{\omega, |\sum_{k\in \Sigma}|k|^{\delta} \langle u_n(\omega, s), e_k\rangle_{\mathbb{L}^2}e_k|_{\mathbb{L}^q}\geq n\}
\nonumber\\
&=& \cup_{s\leq t}\cap_{m\in\mathbb{N}}\cup_{j\geq m}\{\omega, |\sum_{k\in \Sigma_j}|k|^{\delta} \langle u_n(\omega, s), e_k\rangle_{\mathbb{L}^2}e_k|_{\mathbb{L}^q}\geq n\}
\nonumber\\
&=& \cup_{s\leq t}\cap_{m\in\mathbb{N}}\cup_{j\geq m}\{\omega, X_j(\omega, s):= \sum_{k\in \Sigma_j}|k|^{\delta} \langle u_n(\omega, s), e_k\rangle_{\mathbb{L}^2}e_k \in (B_{\mathbb{L}^q}(0, n))^c\}
\nonumber\\
&=& \cup_{s\in \mathbb{Q} \&\leq t}\cap_{m\in\mathbb{N}}\cup_{j\geq m}\{\omega, X_j(\omega, s) \in (B_{\mathbb{L}^q}(0, n))^c\} \in \mathcal{F}_t.
\end{eqnarray}
\del{Thanks to the continuity of the positive $ \mathcal{F}_t-$adapted process $X_j(\omega, s):=  \sum_{k\in \Sigma_j}|k|^{\delta} \langle u_n(\omega, s), e_k\rangle_{\mathbb{L}^2}e_k|_{\mathbb{L}^q}$ 
As the filtration $\mathcal{F}_t$ is right continuous, it is sufficient to prove that for all $ t\in [0, T]$, $ \{\tau_n\geq t\}\in \mathcal{F}_t$.}} The sequence $ (\tau_n)_n $ increases, hence there exists a stopping time $ \tau_\infty$ such that
\begin{equation}\label{Eq-def-tau-delta}
 \tau_\infty = \lim_{n\nearrow +\infty}\tau_n\leq T.
\end{equation}
Thanks to the uniqueness of the fixed point in $ \mathcal{E}_{T}$, we can construct $ (u(t), t< \tau_\infty)$,
by $ u(t)= u_n(t), \; \forall t < \tau_n$.
Then   $ (u(t), \; t < \tau_\infty)$  is solution of Equation \eqref{Main-stoch-eq} up to $ \tau_\infty$.
The estimate \eqref{eq-cond-norm-mild-solution} follows from the construction of the solution as an element of $ \mathcal{E}_{T}$ with $ X= D(A_q^\frac\delta2)$. The regularity of the trajectories in \eqref{eq-cond-cont-mild-solution} is deduced,
by a standard way, thanks to Lemma \ref{lem-est-z-t}, Proposition \ref{Prop-Main-I}, Assumption \eqref{Eq-initial-cond} and the strong continuity
of the semigroup $ (e^{-tA^{\frac{\alpha}{2}}})_{t\geq 0}$ with $ X_1:= H^{-\delta'', q}(O)$ and $ \delta'' \geq \alpha+1+\frac{d}{q}-\delta$.}

\del{\section{Passage Velocity-Vorticity forms of 2D-FSNSEs.}\label{sec-Passage Velocity-Vorticity}}
\del{We define the operator "curl" as follow
\begin{eqnarray}
 curl: v\in H_2^{s, q}(O) \rightarrow  H_1^{s-1, q}(O) \ni curlv:= \partial_1v_2-  \partial_2v_1,
\end{eqnarray}
with $ O $ being either $ \mathbb{T}^2$ or a bounded domain from $ \mathbb{R}^2$, $ 1<q<\infty$ and $ s\in \mathbb{R}$. 
The following result characterizes an intrinsic property between $ curlv$ and the Sobolev regularity of $ v$, 
\begin{lem}\label{lem-basic-curl-gradient}
 Let  $ \mathbb{R}^2$, $ 1<q<\infty$ and $ s\in \mathbb{R}$. Then there exists a constant $ c>0$ such that for all $ v \in H_2^{s+1, q}(O)$
\begin{equation}
 |\nabla v|_{H_2^{s, q}}\leq c |curl v|_{H_1^{s, q}}.
\end{equation}
\end{lem}
\begin{proof}
Let \begin{equation}
     \zeta_{jk}:= \partial_jv_k- \partial_kv_j.
    \end{equation}
By a straightforward calculus, we get
\begin{equation}
     \partial_jv_k=  \partial_j \partial_i\Delta^{-1}\zeta_{ik}.
    \end{equation}
For $ O=\mathbb{T}^2$, it is easy to see, using Marcinkiewicz's theorem, that the  pseudodifferential operators
operator $  \partial_j \partial_i\Delta^{-1}$, with symbol
$ k_ik_j|k|^{-2}$ are bounded  on 
$ H^{s, q}(\mathbb{T}^2)_1$, $ 1<q<\infty $ and $ s\in \mathbb{R}$.  Therefore,
\begin{eqnarray}
 | \nabla v|_{H_2^{s, q}}&\leq& c|\partial_jv_{k}|_{H_1^{s, q}} \leq c|\partial_j \partial_i\Delta^{-1}\zeta_{ik}|_{H_1^{s, q}}\nonumber\\
&\leq& c|\zeta_{ik}|_{H_1^{s, q}} \leq c |curl v|_{H_1^{s, q}}.
\end{eqnarray}
 See also \cite[Lemma 3.1]{Kato-Ponce-86} for the proof for $ O=\mathbb{R}^2$.
\end{proof}}
\del{ Later on, we unify the representation of $ u $ and $ \theta $ as follow
\begin{eqnarray}
\mathcal{R}^{1}:H_1^{s, q}(O) &\rightarrow & \mathbb{H}^{s, q}(O)\nonumber\\
\theta &\mapsto& \mathcal{R}^{1}\theta:= \nabla^\perp\cdot \Delta^{-1}\theta = \int_O \nabla^\perp_xg_{\mathbb{T}^2}(\cdot, y)\theta(t, y)dy.
\end{eqnarray}
Let us now separate the discussion of the cases $ O$ bounded and  $ O=\mathbb{T}^2$.
{\bf The case $ O=\mathbb{T}^2$.} 
In this case, the 
\begin{equation}
\left\{
 \begin{array}{lr}
  \Delta \psi = \theta
 \end{array}
\right.
\end{equation}
\begin{equation}\label{eq-def-u-theta}
u = \mathcal{R}^{1} \theta.
\end{equation}
{\bf The case $ O\subset \mathbb{R}^2$ bounded.} 
\begin{equation}
\left\{
 \begin{array}{lr}
  \Delta \psi = \theta
 \end{array}
\right.
\end{equation}
The following result characterizes an intrinsic property between $ curlv$ and the Sobolev regularity of $ v$, 
\begin{lem}\label{lem-basic-curl-gradient}
 Let  $ \mathbb{R}^2$, $ 1<q<\infty$ and $ s\in \mathbb{R}$. Then there exists a constant $ c>0$ such that for all $ v \in H_2^{s+1, q}(O)$
\begin{equation}
 |\nabla v|_{H_2^{s, q}}\leq c |curl v|_{H_1^{s, q}}.
\end{equation}
\end{lem}
\begin{proof}
Let \begin{equation}
     \zeta_{jk}:= \partial_jv_k- \partial_kv_j.
    \end{equation}
By a straightforward calculus, we get
\begin{equation}
     \partial_jv_k=  \partial_j \partial_i\Delta^{-1}\zeta_{ik}.
    \end{equation}
For $ O=\mathbb{T}^2$, it is easy to see, using Marcinkiewicz's theorem, that the  pseudodifferential operators
operator $  \partial_j \partial_i\Delta^{-1}$, with symbol
$ k_ik_j|k|^{-2}$ are bounded  on 
$ H^{s, q}(\mathbb{T}^2)_1$, $ 1<q<\infty $ and $ s\in \mathbb{R}$.  Therefore,
\begin{eqnarray}
 | \nabla v|_{H_2^{s, q}}&\leq& c|\partial_jv_{k}|_{H_1^{s, q}} \leq c|\partial_j \partial_i\Delta^{-1}\zeta_{ik}|_{H_1^{s, q}}\nonumber\\
&\leq& c|\zeta_{ik}|_{H_1^{s, q}} \leq c |curl v|_{H_1^{s, q}}.
\end{eqnarray}
 See also 
 \cite[Lemma 3.1]{Kato-Ponce-86} for the proof for $ O=\mathbb{R}^2$.
\end{proof}

In this section we consider only the case $ d=2$. \del{The passage from velocity to vorticity forms and vise versa is done in 4 steps. 

{\bf Step 1. General framework.} } We defined the operator "curl" as follow
\del{\begin{eqnarray}
 curl: v\in \mathbb{H}^{s, q}(O) \rightarrow  H_1^{s-1, q}(O) \ni \theta = curlv:= \partial_1v_2-  \partial_2v_1,
\end{eqnarray}}
\begin{eqnarray}
 curl: v\in H_2^{s, q}(O) \rightarrow  H_1^{s-1, q}(O) \ni \theta = curlv:= \partial_1v_2-  \partial_2v_1,
\end{eqnarray}
with $ O $ being either $ \mathbb{T}^2$ or a bounded domain from $ \mathbb{R}^2$, $ 1<q<\infty$ and $ s\in \mathbb{R}$.
To recuperate the  velocity $ u$ of the fluid once the vorticity $ \theta$ is known, we introduce the stream function $ \psi$,
which is the solution of the 
Poisson equation endowed with relevant boundary conditions in the case $ O\subset \mathbb{R}^2$ being bounded. In particular, 
for  $ O =\mathbb{T}^2$, the problem is formulated as follow,
\begin{equation}\label{recuper-eq-u-theta-Torus}
\left\{
 \begin{array}{lr}
 \Delta \psi = \theta,\\
u= \nabla^\perp \psi, \;\;\; \text{and} \;\; \nabla^\perp:=(-\frac{\partial }{\partial x_2}, \frac{\partial
}{\partial x_1}).
 \end{array}
\right.
\end{equation}
\del{\begin{equation}\label{recuper-eq-u-theta}
\left\{
 \begin{array}{lr}
 \Delta \psi = \theta,\\
\psi/\partial O =0,\\
u= \nabla^\perp \psi, \;\;\; \text{and} \;\; \nabla^\perp:=(-\frac{\partial }{\partial x_2}, \frac{\partial
}{\partial x_1}).
 \end{array}
\right.
\end{equation}
\begin{equation}\label{psi-theta-u}
 \Delta \psi = \theta, \;\;\; \text{and} \;\; u= \nabla^\perp \psi, 
\end{equation}}
Thanks to the vanishing average condition, Poisson equation is well posed. The velocity $ u $ is then obtained by a direct calculus using the 
second equation in \eqref{recuper-eq-u-theta-Torus}.
For  $ O\subset \mathbb{R}^2$ bounded, we consider \eqref{recuper-eq-u-theta-Torus} and we conclude from Dirichlet boundary conditions of the 
velocity $ u$ that 
$ \psi$ should satisfy vanishing Neumann boundary conditions. Therefore, $ \psi = const.$ on $ \partial O$. We suppose that this constant is null, 
for further discussion, see e.g. \cite{Marchioro-Puvirenti-Vortex-84}. Let us denote by $ A_1$ either the Laplacian on $ O \subset \mathbb{R}^2 $ 
with Dirichlet boundary conditions or  a component of Stokes operator in the case $O= \mathbb{T}^2$,  then we formulate the recuperation  problem 
for both cases as 
\begin{equation}\label{recuper-eq-u-theta}
\left\{
\begin{array}{lr}
 A_1 \psi = \theta,\\
u= \nabla^\perp \psi.
\end{array}
\right.
\end{equation}
The problem, as mentioned before, is well posed, see also Section \ref{sec-formulation} and we have
\begin{equation}\label{repres-u-theta}
 u(t, x) = \nabla^\perp A_1^{-1}\theta(t, x)= \int_O \nabla^\perp_xg_O(x, y)\theta(t, y)dy,
\end{equation}
where $ g_O$ is the Green function corresponding to the Poisson equation with Dirichlet boundary conditions. 
In the case $  O= \mathbb{T}^2 $, the Green function  $g_{\mathbb{T}^2}$ is explicitly given by
\begin{equation}
 g_{\mathbb{T}^2}(x, y):= -\frac1{(2\pi)^2}\sum_{k\in\mathbb{Z}^2_0}\frac{1}{|k|^2}e^{k\cdot(x-y)},\;\; x, y \in \mathbb{T}^2.
\end{equation}
Moreover,  the operator 
\del{\begin{eqnarray}
\mathcal{R}^{1}:H_1^{s, q}(\mathbb{T}^2) &\rightarrow & H_2^{s, q}(\mathbb{T}^2)\nonumber\\
\theta &\mapsto& \mathcal{R}^{1}\theta:= \nabla^\perp\cdot \Delta^{-1}\theta
\end{eqnarray}}
\begin{eqnarray}
\mathcal{R}^{1}:H_1^{s, q}(O) &\rightarrow & H_2^{s+1, q}(O)\nonumber\\
\theta &\mapsto& \mathcal{R}^{1}\theta:= u= \nabla^\perp\cdot A_1^{-1}\theta = \int_O \nabla^\perp_{\cdot}g_{O}(\cdot, y)\theta(y)dy.
\end{eqnarray}
is well defined and bounded for all $ 1<q<\infty$ and $ s\in \mathbb{R}$. In fact,
\begin{eqnarray}
\del{ |u|_{\mathbb{H}^{s+1, q}}&\leq& }|u|_{H_2^{s+1, q}} &\leq& c | \nabla^\perp A_1^{-1}\theta|_{H_2^{s+1, q}}
\leq c|\partial_j A_1^{-1}\theta|_{H_2^{s+1, q}} \leq c|A_1^{-1}\theta|_{H_2^{s+2, q}}\leq c|\theta|_{H_2^{s, q}}.
\end{eqnarray}
The proof of this statement 
for a more large class of operators on  $\mathbb{T}^d, d\in \mathbb{N}_0$ which includes 
 $ \mathcal{R}^{1}$,  can be found in \cite{Debbi-scalar-active} and for the case $ O=\mathbb{R}^2$, one can  see \cite[Lemma 3.1]{Kato-Ponce-86}.
 For the convenience of the reader, let us mention here that  $ \mathcal{R}^{1} $ is a 
pseudoodifferential operator of Calderon-Zygmund Reisz type. Otherwise, we can rewrite  
$\mathcal{R}^{1} = -\mathcal{R}^\perp (-\Delta)^{-\frac12}$, where $ \mathcal{R}$ is Riesz transform and $ v^\perp := (-v_2, v_1)$. 
 One can also use the representation of $ u$ via Green function (recall that $ \Delta_xg_O(x, y)= \delta_x(y)$). \\}
\section{The Biot-Savart's law and the corresponding fractional stochastic vorticity equation.}\label{sec-Passage Velocity-Vorticity}
In this appendix we consider only the case $ d=2$, for the multidimensional case see e.g. \cite[Chap.3]{Chemin-Book-98}, \cite[Chap.2]{Majda-Bertozzi-02}, \cite{Marchioro-Puvirenti-Vortex-84} and \cite{Mikulevicius-H1-NS-solution-2004}. The Biot-Savart law determines the  velocity $ u$ from the vorticity $ \theta$. This law is given as a pseudo-differential operator of order $-1$ in the cases $ O=\mathbb{R}^d$ and $ O=\mathbb{T}^d$. The case $ O\subsetneq\mathbb{R}^d$, as mentioned before is much involved. here we give a survey and some results about this law in the cases  $ O\subsetneq\mathbb{R}^d$ and $ O=\mathbb{T}^d$ than we move to the derivation of the stochastic vorticity equation for the case $ O=\mathbb{T}^d$. A generalization of the Biot-Savart's law to a nonlocal pseudo-differential operators of fractional order $ \gamma \leq 0$ has been investigated in \cite{Debbi-scalar-active}.\del{The passage from velocity to vorticity forms and vise versa is done in 4 steps. 

{\bf Step 1. General framework.} } We define the operator "curl" as follow, see Preliminary Notations,
\del{\begin{eqnarray}
 curl: v\in \mathbb{H}^{s, q}(O) \rightarrow  H_1^{s-1, q}(O) \ni \theta = curlv:= \partial_1v_2-  \partial_2v_1,\;\; 1<q<\infty\;\; s\in \mathbb{R}.
\end{eqnarray}}
\begin{eqnarray}
 curl: v\in H_2^{\beta, q}(O) \rightarrow  H_1^{\beta-1, q}(O) \ni \theta = curlv:= \partial_1v_2-  \partial_2v_1,\; \beta\in \mathbb{R}, \; 1<q<\infty.\nonumber\\
\end{eqnarray}
\del{with $ O $ being either $ \mathbb{T}^2$ or a bounded domain from $ \mathbb{R}^2$, $ 1<q<\infty$ and $ s\in \mathbb{R}$.
To recuperate the  velocity $ u$ of the fluid once the vorticity $ \theta$ is known, }We introduce the stream function $ \psi$,
 as the solution of the 
Poisson equation endowed with a relevant boundary condition in the case $ O\subset \mathbb{R}^2$ being bounded. In deed,
we conclude from Dirichlet boundary condition of the 
velocity $ u$ and the third equation in \eqref{recuper-eq-u-theta-Torus} bellow, that 
$ \psi$ should satisfy vanishing Neumann boundary condition. Therefore, $ \psi/ \partial O= const.$\del{ on $ \partial O$.} We suppose that this constant is null, 
for further discussion, see e.g. \cite{ Marchioro-Puvirenti-Vortex-84}. The problem is then formulated as follow,
\begin{equation}\label{recuper-eq-u-theta-Torus}
\left\{
 \begin{array}{lr}
 \Delta \psi = \theta,\\
\psi/\partial O =0,\\
u= \nabla^\perp \psi, \;\;\; \text{and} \;\; \nabla^\perp:=(-\frac{\partial }{\partial x_2}, \frac{\partial
}{\partial x_1}).
 \end{array}
\right.
\end{equation}
The formulation  \eqref{recuper-eq-u-theta-Torus} is still valid for  $O= \mathbb{T}^2$ without the boundary condition.
\del{\begin{equation}\label{recuper-eq-u-theta}
\left\{
 \begin{array}{lr}
 \Delta \psi = \theta,\\
\psi/\partial O =0,\\
u= \nabla^\perp \psi, \;\;\; \text{and} \;\; \nabla^\perp:=(-\frac{\partial }{\partial x_2}, \frac{\partial
}{\partial x_1}).
 \end{array}
\right.
\end{equation}
\begin{equation}\label{psi-theta-u}
 \Delta \psi = \theta, \;\;\; \text{and} \;\; u= \nabla^\perp \psi, 
\end{equation}}Let us denote by $ A_1$ either the Laplacian on $O= \mathbb{T}^2$ or the Laplacian  
with Dirichlet boundary condition on $ O \subset \mathbb{R}^2 $,  then we formulate the recuperation  problem 
for both cases as 
\begin{equation}\label{recuper-eq-u-theta}
\left\{
\begin{array}{lr}
 A_1 \psi = \theta,\\
u= \nabla^\perp \psi.
\end{array}
\right.
\end{equation}
Problem \eqref{recuper-eq-u-theta} is well posed, see also Section \ref{sec-formulation}. Recall that for  $ O=\mathbb{T}^2$, the wellposdness is guaranteed   
thanks to the vanishing average condition for the torus.\\ The velocity $ u $ is obtained by a direct calculus,\del{ using the 
second equation in \eqref{recuper-eq-u-theta-Torus}. We get}
\begin{equation}\label{repres-u-theta}
 u(t, x) = \nabla^\perp A_1^{-1}\theta(t, x)= \int_O \nabla^\perp_xg_O(x, y)\theta(t, y)dy,
\end{equation}
where $ g_O$ is the Green function corresponding to the Poisson equation with Dirichlet boundary conditions for $ O$ bounded. 
In the case $  O= \mathbb{T}^2 $, the Green function  $g_{\mathbb{T}^2}$ is explicitly given by
\begin{equation}
 g_{\mathbb{T}^2}(x, y):= -\frac1{(2\pi)^2}\sum_{k\in\mathbb{Z}^2_0}\frac{1}{|k|^2}e^{k\cdot(x-y)},\;\; x, y \in \mathbb{T}^2.
\end{equation}
Moreover,
\begin{lem}\label{lem-R1-bounded}
The operator 
\del{\begin{eqnarray}
\mathcal{R}^{1}:H^{s, q}(\mathbb{T}^2) &\rightarrow & H_2^{s, q}(\mathbb{T}^2)\nonumber\\
\theta &\mapsto& \mathcal{R}^{1}\theta:= \nabla^\perp\cdot \Delta^{-1}\theta
\end{eqnarray}}
\begin{eqnarray}
\mathcal{R}^{1}:H_1^{\beta, q}(O) &\rightarrow & H_2^{\beta+1, q}(O)\nonumber\\
\theta &\mapsto& \mathcal{R}^{1}\theta:= u= \nabla^\perp\cdot A_1^{-1}\theta = \int_O \nabla^\perp_{\cdot}g_{O}(\cdot, y)\theta(y)dy
\end{eqnarray}
is well defined and bounded for all $ 1<q<\infty$ and $ \beta\in \mathbb{R}$. 
\end{lem}
\begin{proof}
In fact,
\begin{eqnarray}
\del{ |u|_{\mathbb{H}^{s+1, q}}&\leq& }|u|_{H_2^{\beta+1, q}} &\leq& c | \nabla^\perp A_1^{-1}\theta|_{H_2^{\beta+1, q}}
\leq c|\partial_j A_1^{-1}\theta|_{H_1^{\beta+1, q}} \leq\del{ c|A_1^{-1}\theta|_{H^{\beta+2, q}}\leq} c|\theta|_{H^{\beta, q}}.
\end{eqnarray}
One can also use the representation of $ u$ via Green function (recall that $ \Delta_xg_O(x, y)= \delta_x(y)$). For 
$ O=\mathbb{T}^2$, it is also convenient to remark that $ \mathcal{R}^{1} $ is a 
pseudo-differential operator of Calderon-Zygmund Reisz type, see 
definitions, results and further discussions in \cite{Debbi-scalar-active}. In deed, we can rewrite  
$\mathcal{R}^{1} = -\mathcal{R}^\perp (-\Delta)^{-\frac12}$, where $ \mathcal{R}$ is Riesz transform and $ \mathcal{R}^\perp := (-\mathcal{R}_2, \mathcal{R}_1)$. 
The proof of the above statement 
for a larger class of operators on  $\mathbb{T}^d, d\in \mathbb{N}_0$ which includes 
 $ \mathcal{R}^{1}$  can be found in \cite{Debbi-scalar-active}.\del{ For the case $ O=\mathbb{R}^2$ see e.g. \cite[Lemma 3.1]{Kato-Ponce-86}.}
\end{proof}

The following result characterizes an intrinsic property between $ curlv$ and the Sobolev regularity of $ v$.\del{ Recall that $ \nabla v$ is a matrix and we use the following notation for the norm matrix
$ | \nabla v|_{H_{2\times d}^{s, q}}:= | \partial_j v_i|_{H^{s, q}}$.}
\del{It is an easy consequence of the above discussion.} 
\begin{lem}\label{lem-basic-curl-gradient}
 Let  $O \subset \mathbb{R}^2$ bounded or $O = \mathbb{T}^2$, $ 1<q<\infty$ and $ \beta\in \mathbb{R}$. 
Then there exists a constant $ c>0$ such that for all $ v \in H_2^{\beta+1, q}(O)$
\begin{equation}
 c|\nabla v|_{H^{\beta, q}}\leq |curl v|_{H^{\beta, q}}\leq |\nabla v|_{H^{\beta, q}}.
\end{equation}
\del{Moreover, for $O \subset \mathbb{R}^2$ bounded or $O = \mathbb{T}^2$, $ 1<q<\infty$ and $ \beta\in \mathbb{R}$, there exists a constant $ c>0$ such that for all $ v \in \mathbb{H}_d^{\beta+1, q}(O)$
\begin{equation}
 |\nabla v|_{H_{2\times d}^{\beta, q}}\leq c |curl v|_{H^{\beta, q}}.
\end{equation}}
\end{lem}
\begin{proof}
Using the definition of the curl operator, the sobolev spaces and Lemma 
\ref{lem-R1-bounded}, we infer that there exists $ c>0$ such that 
\begin{eqnarray}
|\nabla v|_{H^{\beta, q}}&\leq & |\partial_j v_i|_{H^{\beta, q}}
\leq c|v|_{H_{2}^{\beta+1, q}}\leq c|curl v|_{H^{\beta, q}}.
\end{eqnarray}
Moreover, we have, 
\begin{eqnarray}
|curl v|_{H^{s, q}} &\leq & |\partial_j v_i|_{H^{\beta, q}}
\leq |\nabla v|_{H^{\beta, q}}.
\end{eqnarray}
\end{proof}
\begin{remark}
If we assume that $ v $ is of divergence free, i.e. $ v \in \mathbb{H}^{\beta, q}(\mathbb{T}^d)$, $d\in \mathbb{N}_1$, $O = \mathbb{T}^2$, $ 1<q<\infty$ and $ \beta \in \mathbb{R}_+$, then it is easy to adapt the proof of  \cite[Lemma 3.1]{Kato-Ponce-86}. 
\end{remark}

 \del{Now, we turn to the stochastic term. Let $(u(t), t \in [0, T])$ be a solution of FSNSE satisfying 
 \eqref{cond-solu-torus-H1}.  We denote by
\begin{equation}\label{eq-def-sigma-k}
\sigma^k(u):= G(u)Q^{\frac12}e_k= q_{k}^\frac12G(u)e_k,  for\;\;  k\in\Sigma.
\end{equation}
First, we claim that for  $P-a.s.$ the following stochastic integral $ \int_0^t  \sum_{k\in\Sigma}curl \sigma^k(u(s))d\beta_k(s)$ 
is well defined and
\begin{eqnarray}
 curl \int_0^tG(u(s))dW(s)= \int_0^t  \sum_{k\in\Sigma} curl \sigma^k(u(s))d\beta_k(s), \forall  t \in [o, T].
\end{eqnarray}

\noindent In fact, using the stochastic isometry,\del{ Lemma \ref{lem-basic-curl-gradient},} Assumption $ (\mathcal{C})$
and \eqref{cond-solu-torus-H1} we infer that for $ \beta$ either equals $ 1$, or $0$.
\del{\begin{eqnarray}\label{eq-def-sima-integ}
 \mathbb{E}\!\!\!\!&|&\!\!\!\!\int_0^t \sum_{k\in\Sigma} curl\sigma^k(u(s))d\beta_k(s)|_{L^2}^2\leq
\mathbb{E}\int_0^t\sum_{k\in\Sigma} |curl \sigma^k(u(s))|_{L^2}^2ds \nonumber\\
&\leq&
c\mathbb{E}\int_0^t\sum_{k\in\Sigma}|\partial_j \sigma_i^k(u(s))|_{L^2}^2ds
\del{\leq
c\mathbb{E}\int_0^t |G(u(s))Q^\frac12|_{HS(H^{1,2}) }^2ds \nonumber\\}
\leq c\mathbb{E}\int_0^t\sum_{k\in\Sigma}  
|\sigma^k(u(s))|_{H^{1,2}}^2ds\nonumber\\
&\leq&
c\mathbb{E}\int_0^t |G(u(s))|_{L_Q(H^{1,2}) }^2ds 
\leq
c\int_0^t(1+ \mathbb{E}|u(s)|_{\mathbb{H}^{1, 2}}^2)ds <\infty.
\end{eqnarray}}
\begin{eqnarray}\label{eq-def-sima-integ}
 \mathbb{E}\!\!\!\!&|&\!\!\!\!\int_0^t \sum_{k\in\Sigma} curl\sigma^k(u(s))d\beta_k(s)|_{H^{\beta-1, 2}}^2\leq
\mathbb{E}\int_0^t\sum_{k\in\Sigma} |curl \sigma^k(u(s))|_{H^{\beta-1, 2}}^2ds \nonumber\\
&\leq&
c\mathbb{E}\int_0^t\sum_{k\in\Sigma}|\partial_j \sigma_i^k(u(s))|_{H^{\beta-1, 2}}^2ds
\del{\leq
c\mathbb{E}\int_0^t |G(u(s))Q^\frac12|_{HS(H^{1,2}) }^2ds \nonumber\\}
\leq c\mathbb{E}\int_0^t\sum_{k\in\Sigma}  
|\sigma^k(u(s))|_{H^{\beta,2}}^2ds\nonumber\\
&\leq&
c\mathbb{E}\int_0^t |G(u(s))|_{L_Q(H^{\beta,2}) }^2ds 
\leq
c\int_0^t(1+ \mathbb{E}|u(s)|_{\mathbb{H}^{\beta, 2}}^2)ds <\infty.
\end{eqnarray}
\noindent Moreover, thanks to  \eqref{eq-W-n} and \eqref{eq-def-sigma-k}, we infer on one hand that 
\begin{eqnarray}
curl \sum_{k\in\Sigma_n} \int_0^t  \sigma^k(u(s))d\beta_k(s) \rightarrow
curl  \int_0^t \sum_{k\in\Sigma_n} \sigma^k(u(s))d\beta_k(s),\,\,  in \, L^2(\Omega, \mathcal{D}'(O)),  \nonumber\\
\end{eqnarray}
where  $ (\Sigma_n)_n$ is a sequence of subsets converging to $\Sigma$ and $ \mathcal{D}'(O)$ is the dual of $ \mathcal{D}(O)$. 
On the other hand, using the linearity of the operator $ curl$, the stochastic isometry identity, 
\eqref{eq-W-n},  Assumption $ (\mathcal{C})$, \eqref{cond-solu-torus-H1} and \eqref{eq-def-sima-integ}, we end up with
\begin{eqnarray}
curl \sum_{k\in\Sigma_n} \int_0^t  \sigma^k(u(s))d\beta_k(s) &=&\sum_{k\in\Sigma_n} \int_0^t curl \sigma^k(u(s))d\beta_k(s)\nonumber\\
&{}& \rightarrow
 \int_0^t \sum_{k\in\Sigma}curl  \sigma^k(u(s))d\beta_k(s),\,\,  in \, L^2(\Omega, \mathcal{D}'(O)). 
\end{eqnarray}
The uniqueness of the limit confirm the result. We use the following notation for the stochastic integral
\begin{eqnarray}\label{inte-g-tilde}
\int_0^t \tilde{G}(\theta(s))dW(s) := \sum_{k\in\Sigma} \int_0^t curl \sigma_k(u(s))d\beta_k(s)
& =& \sum_{k\in\Sigma} \int_0^t curl \sigma_k(\mathcal{R}^{1}(\theta(s))d\beta_k(s).\nonumber\\
\end{eqnarray}
The same calculus above is still valid for Orstein-Uhlenbeck stochastic process.\\
\noindent Using the definition of Helmholtz projection, in particular, the fact that $ \mathcal{Y}_q \subset Ker(Curl)$, we prove 
\begin{equation}
 curl B(u) = u\cdot \nabla \theta.
\end{equation}}
Now, we derive the stochastic vorticity equation. Let $(u, \tau)$ be a maximal weak solution of FSNSE satisfying \eqref{cond-solu-torus-H1}, 
 up to the stopping time $ \tau $.  
First, we claim that for  $P-a.s.$ the following stochastic integral 
$ \int_0^{t\wedge \tau}  \sum_{k\in\Sigma}curl \sigma^k(u(s))d\beta_k(s)$, with  $\sigma^k(u(s))$ is given by
\begin{equation}\label{eq-def-sigma-k}
\sigma^k(u):= G(u)Q^{\frac12}e_k= q_{k}^\frac12G(u)e_k,\;  for\;\;  k\in\Sigma,
\end{equation}
 is well defined and
\begin{eqnarray}
 \int_0^{t\wedge \tau} \sum_{k\in\Sigma} curl \sigma^k(u(s))d
 \beta_k(s)= curl \int_0^{t\wedge \tau}G(u(s))dW(s),\; \forall  t \in [o, T].
\end{eqnarray}

\noindent In fact, using the stochastic isometry,\del{ Lemma \ref{lem-basic-curl-gradient}} Assumption $ (\mathcal{C})$ (\eqref{Eq-Cond-Linear-Q-G}, with $ 2\leq q<\infty$ and $ \delta \in\{0, 1\}$)
and \eqref{cond-solu-torus-H1}, we infer that for $ \beta$  equals either $ 1$ or $0$,
\begin{eqnarray}\label{eq-def-sima-integ}
 \mathbb{E}\!\!\!\!&|&\!\!\!\!\int_0^{t\wedge \tau} \sum_{k\in\Sigma} curl\sigma^k(u(s))d\beta_k(s)|_{H^{\beta-1, q}}^2\leq
\del{\mathbb{E}\int_0^t\sum_{k\in\Sigma} |curl \sigma^k(u(s))|_{H^{\beta-1, q}}^2ds \nonumber\\
&\leq&}
c\mathbb{E}\int_0^{t\wedge \tau}\sum_{k\in\Sigma}|\partial_j \sigma_i^k(u(s))|_{H^{\beta-1, q}}^2ds\nonumber\\
\del{\leq
c\mathbb{E}\int_0^t |G(u(s))Q^\frac12|_{HS(H^{1,2}) }^2ds \nonumber\\}
&\leq& c\mathbb{E}\int_0^{t\wedge \tau}\sum_{k\in\Sigma}  
|\sigma^k(u(s))|_{H^{\beta, q}}^2ds
\leq
c\mathbb{E}\int_0^{t\wedge \tau} |G(u(s))|_{R_Q(H^{\beta, q}) }^2ds \nonumber\\
&\leq&
c\int_0^{t\wedge \tau}(1+ \mathbb{E}|u(s)|_{\mathbb{H}^{\beta, q}}^2)ds <\infty.
\end{eqnarray}
\noindent Moreover, thanks to  \eqref{eq-W-n} and \eqref{eq-def-sigma-k}, we infer on one hand that 
\begin{eqnarray}
curl \sum_{k\in\Sigma_n} \int_0^{t\wedge \tau}  \sigma^k(u(s))d\beta_k(s) \rightarrow
curl  \int_0^{t\wedge \tau} \sum_{k\in\Sigma} \sigma^k(u(s))d\beta_k(s),\,\,  in \, L^2(\Omega, \mathcal{D}'(O)),  \nonumber\\
\end{eqnarray}
where  $ (\Sigma_n)_n$ is a sequence of subsets converging to $\Sigma$ and $ \mathcal{D}'(O)$ is the dual of $ \mathcal{D}(O)$. 
On the other hand, using the linearity of the operator $ curl$, the stochastic isometry identity, 
\eqref{eq-W-n},  Assumption $ (\mathcal{C})$, \eqref{cond-solu-torus-H1} and \eqref{eq-def-sima-integ}, we end up with
\begin{eqnarray}
curl \sum_{k\in\Sigma_n} \int_0^{t\wedge \tau}  \sigma^k(u(s))d\beta_k(s) &=&\sum_{k\in\Sigma_n} \int_0^{t\wedge \tau} 
curl \sigma^k(u(s))d\beta_k(s)\nonumber\\
&{}& \rightarrow
 \int_0^{t\wedge \tau} \sum_{k\in\Sigma}curl  \sigma^k(u(s))d\beta_k(s),\,\,  in \, L^2(\Omega, \mathcal{D}'(O)). 
\end{eqnarray}
The uniqueness of the limit confirm the result. We use the following notation\del{ for the stochastic integral
\begin{eqnarray}\label{inte-g-tilde}
\int_0^{t\wedge \tau} \tilde{G}(\theta(s))dW(s) := \sum_{k\in\Sigma} \int_0^{t\wedge \tau} curl \sigma^k(u(s))d\beta_k(s)
& =& \sum_{k\in\Sigma} \int_0^{t\wedge \tau} curl \sigma^k(\mathcal{R}^{1}(\theta(s))d\beta_k(s).\nonumber\\
\end{eqnarray}}
\begin{eqnarray}\label{inte-g-tilde}
\tilde{G}(\theta):= curl G(\mathcal{R}^{1}(\theta)).
\end{eqnarray}
\noindent Using the definition of Helmholtz projection, in particular, the fact that $ \mathcal{Y}_q \subset Ker(Curl)$, an elementary calculus
yields to 
\begin{equation}
 curl B(u) = u\cdot \nabla \theta.
\end{equation}

 Now, we  assume that $ O= \mathbb{T}^2$,  using Fourier transform, it is easy to prove that 
\begin{equation}
 curl A_\alpha u = (-\Delta)^{\frac\alpha2}curl u, \; \; \forall u\in D(A_\alpha).
\end{equation}
In fact the relation above is also true for all $ u\in \mathbb{H}^{\beta+\alpha}(\mathbb{T}^2),  \beta\in \mathbb{R}$.
Applying the operator curl on the integral representation of Equation \eqref{Main-stoch-eq} stopped at the  stopping time $ \tau$
 and using the calculus above, we infer that if $ (u, \tau)$ is a local weak solution of \eqref{Main-stoch-eq}, then
 $ \theta:= curl u$ is a weak (strong in probability) solution of 
\begin{equation}\label{Eq-vorticity-Torus-2-diff}
\left\{
\begin{array}{lr}
 d\theta(t)= \left(- A_{\alpha}\theta(t) + u(t)\cdot \nabla \theta(t)\right)dt+ \tilde{G}(\theta(t))dW(t), \; 0< t\leq \tau.\\
\theta(t) = curl u_0. 
\end{array}
\right.
\end{equation}
By the same way, we can  prove that if $ (u, \tau)$ is a local mild solution of \eqref{Main-stoch-eq}, then the same calculus above is still valid\del{ for Orstein-Uhlenbeck stochastic process and \del{$ u$ is solution of \eqref{Eq-Mild-Solution}}} and $ \theta:= curl u$ is a mild solution 
to equation \eqref{Eq-vorticity-Torus-2-diff}. In the proof of this case, we use the commutativity property between the operators $ \partial_j$ and the 
semigroup $ (e^{-tA_\alpha})_{t\geq0}$. A general formula of Equation \eqref{Eq-vorticity-Torus-2-diff} has been studied in \cite{Debbi-scalar-active}.
\del{For the case $ O=\mathbb{R}^2$, we use the integral formula of the fractional Laplacian. } 

\del{For the case of bounded domain $ O\subset \mathbb{R}^2$, we prove, roughly speaking, a weak commutativity  between the the operator $ A_\alpha$
and the curl. It is easy to see that 
\begin{equation}
 \langle curl A^S u, \phi\langle =  \langle A^S curl u, \phi\langle, \;\;  \forall \phi \in C_0^\infty,
\end{equation}
with the notation  $ A^S$ stands for both 1D and 2D  Stokes operator.}

\del{the local solution $ (u, \tau)$ of \eqref{Eq-weak-Solution}. Using
the commutativity of $ A^\frac\alpha2$ and $ \partial_j, j=1,2,$, which is inherited from the commutativity of
the Laplacian $ A^S= -\Delta$ and $ \partial_j, j=1,2,$}
\del{satisfies in $ V^*$
\begin{equation}\label{Eq-vorticity-Torus-2}
\theta(t)= curl u_0 + \int_0^t\left(-A_{\alpha}\theta(s) + u(s)\cdot \nabla \theta(s)\right)ds+ \int_0^t \tilde{G}(\theta(s))dW(s), \; 0< t\leq T.\\
\end{equation}
Or equivalently, $ \theta $ is a weak (strong in probability) solution\del{ (in distribution sense)} of
\begin{equation}\label{Eq-vorticity-Torus-2-diff}
\left\{
\begin{array}{lr}
 d\theta(t)= \left(- A_{\alpha}\theta(t) + u(t)\cdot \nabla \theta(t)\right)dt+ \tilde{G}(\theta(t))dW(t), \; 0< t\leq T.\\
\theta(t) = curl u_0. 
\end{array}
\right.
\end{equation}}

\del{\section{Stopping times.}\label{append-stop-time}
\del{In this paper, we have defined in many places random times and claimed that they are stopping time. Some of them are easily seen that 
Here, we give the proof. Recall that  Remark \ref{Rem-1} confirms that the mild and weak solutions are weakly continuous.

\begin{lem}
Let $(Y(t), t\in [0, T])$ be a weakly continuous process.  We define the following random time
\begin{equation}\label{Eq-def-tau-n-delta-Y}
 \tau_n := \inf\{t\in (0, T), \; s.t. |Y(t)|_{D(A_q^{\frac\beta2})}\geq n\}\wedge T,
\end{equation}
with the understanding that $ \inf(\emptyset)=+\infty$. Then that $ \tau_n $ is a stopping time. 
\end{lem}

\begin{proof}
\del{Indeed, it is easy to check,  using Lemma \ref{lem-est-z-t}, Estimation \eqref{Eq-Pi-n-X} and a similar calculus as in \eqref{est-1-semi-group-second-term-1}, that the real  
$ \mathcal{F}_t-$adapted stochastic processes $ (\langle u_n(s), e_k\rangle_{\mathbb{L}^2})_{k\in\Sigma}$ are continuous for all $ k\in \Sigma$, see also Remark \ref{Rem-1}.}

Thanks to the weak continuity of $ Y$, the real $ \mathcal{F}_t-$adapted stochastic processes $ (\langle u_n(s), e_k\rangle_{\mathbb{L}^2})_{k\in\Sigma}$ are continuous for all $ k\in \Sigma$. Therefore, \del{the positive $ \mathcal{F}_t-$adapted stochastic real processes
\begin{equation}
X_j(\omega, s):= |\sum_{k\in \Sigma_j}|k|^{\delta} \langle u_n(\omega, s), e_k\rangle_{\mathbb{L}^2}e_k|_{\mathbb{L}^q},
\end{equation} }the positive $ \mathcal{F}_t-$adapted stochastic real processes
\begin{equation}
X_j(\omega, s):= |\sum_{k\in \Sigma_j}|k|^{\delta} \langle u_n(\omega, s), e_k\rangle_{\mathbb{L}^2}e_k|_{\mathbb{L}^q},
\end{equation} 
are also continuous for all $ j\in \mathbb{N}_1$, with $ (\Sigma_j)_j\subset \Sigma$  being an increasing sequence of finite subsets converging to $ \Sigma$. Therefore, see e.g. \cite[Propositions 4.5 or 4.6 ]{Revuz-Yor}, (Recall that we have assumed that the filtration $ (\mathcal{F}_t)_{t\in[0, T]}$ is right continuous), \del{therefore an optional time is also a stopping time.}  
\del{\begin{eqnarray}
\{\tau_n> t\} &=& \cap_{s\leq t}\{\omega, |u_n(s, \omega)|_{D(A_q^\frac\delta2)}\leq n\}\nonumber\\
&=& \cap_{s\leq t}\{\omega, |\sum_{k\in \Sigma}|k|^{\delta} \langle u_n(\omega, s), e_k\rangle_{\mathbb{L}^2}e_k|_{\mathbb{L}^q}\leq n\}
\nonumber\\
&=& \cap_{s\leq t}\cup_{m\in\mathbb{N}}\cap_{j\geq m}\{\omega, |\sum_{k\in \Sigma_j}|k|^{\delta} \langle u_n(\omega, s), e_k\rangle_{\mathbb{L}^2}e_k|_{\mathbb{L}^q}\leq n\}.
\end{eqnarray}}
\begin{eqnarray}
\{\tau_n\leq t\} &=& \cup_{s\leq t}\{\omega, |u_n(s, \omega)|_{D(A_q^\frac\delta2)}\geq n\}\nonumber\\
&=& \cup_{s\leq t}\{\omega, |\sum_{k\in \Sigma}|k|^{\delta} \langle u_n(\omega, s), e_k\rangle_{\mathbb{L}^2}e_k|_{\mathbb{L}^q}\geq n\}
\nonumber\\
&=& \cup_{s\leq t}\cap_{m\in\mathbb{N}}\cup_{j\geq m}\{\omega, |\sum_{k\in \Sigma_j}|k|^{\delta} \langle u_n(\omega, s), e_k\rangle_{\mathbb{L}^2}e_k|_{\mathbb{L}^q}\geq n\}
\nonumber\\
&=& \cup_{s\leq t}\cap_{m\in\mathbb{N}}\cup_{j\geq m}\{\omega, X_j(\omega, s):= \sum_{k\in \Sigma_j}|k|^{\delta} \langle u_n(\omega, s), e_k\rangle_{\mathbb{L}^2}e_k \in (B_{\mathbb{L}^q}(0, n))^c\}
\nonumber\\
&=& \cup_{s\in \mathbb{Q} \&\leq t}\cap_{m\in\mathbb{N}}\cup_{j\geq m}\{\omega, X_j(\omega, s) \in (B_{\mathbb{L}^q}(0, n))^c\} \in \mathcal{F}_t.
\end{eqnarray}
\del{Thanks to the continuity of the positive $ \mathcal{F}_t-$adapted process $X_j(\omega, s):=  \sum_{k\in \Sigma_j}|k|^{\delta} \langle u_n(\omega, s), e_k\rangle_{\mathbb{L}^2}e_k|_{\mathbb{L}^q}$ 
As the filtration $\mathcal{F}_t$ is right continuous, it is sufficient to prove that for all $ t\in [0, T]$, $ \{\tau_n\geq t\}\in \mathcal{F}_t$.}

\end{proof}}
In this work, we have defined in several places hitting times and we have claimed that they are stopping times.  The proofs of some of them are easily deduced thanks to the continuity of the norm-process and the right continuity of the filtration. This is the case for example for the stopping times\del{ $\tau_N^i$} defined in the proof of the uniqueness of the solution in Section \ref{sec-Torus}.\del{ the stopping times
$ \tau_N^i: \inf\{t\leq T; |u^i(t)|_{\mathbb{L}^{2}}\geq N\}\wedge T, i=1, 2$} Other proofs are consequences of the right continuity of the filtration and the $X-$weak continuity of the process.  Recall that Remark \ref{Rem-1} confirms that the local mild and weak solutions are $X-$weakly continuous. This is the case for the stopping time defined by \eqref{Eq-def-tau-n-delta} in Section \ref{appendix-local-solution}. Recall that this  stopping time is used implicitly in the proof of the local mild solution in Section \ref{sec-1-approx-local-solution}. Other proofs are more sophisticated as is the case for  $ \xi_N$ defined by \eqref{eq-stop-time-weak-solu-L2} in  Section \ref{sec-Domain}. In deed, in this case, the weak continuity concerns the $\mathbb{L}^2-$norm and the hitting time is defined via the $ \mathbb{H}^{\frac{d+2-\alpha}{4}, 2}-$norm. Here, we give the proof for this case.\del{ For simplicity, we consider the case $ q=2$. This is a direct proof of}   

\begin{lem}
Let $(Y(t), t\in [0, T])$ be an $\mathbb{L}^2-$weakly continuous  $ \mathcal{F}_t-$adapted $\mathbb{L}^2-$valued process.  We define the following random time
\begin{equation}\label{Eq-def-tau-n-delta-Y}
 \tau_n := \inf\{t\in (0, T), \; s.t.\;\;  |Y(t)|_{D(A_2^{\frac\delta2})}> n\}\wedge T,
\end{equation}
with the understanding that $ \inf(\emptyset)=+\infty$. Then $ \tau_n $ is a stopping time. 
\end{lem}

\begin{proof}
\del{Indeed, it is easy to check,  using Lemma \ref{lem-est-z-t}, Estimation \eqref{Eq-Pi-n-X} and a similar calculus as in \eqref{est-1-semi-group-second-term-1}, that the real  
$ \mathcal{F}_t-$adapted stochastic processes $ (\langle u_n(s), e_k\rangle_{\mathbb{L}^2})_{k\in\Sigma}$ are continuous for all $ k\in \Sigma$, see also Remark \ref{Rem-1}.}

Thanks to the weak continuity of $ Y$,\del{ the real $ \mathcal{F}_t-$adapted stochastic processes $ (\langle u_n(s), e_k\rangle_{\mathbb{L}^2})_{k\in\Sigma}$ are continuous for all $ k\in \Sigma$. Therefore, \del{the positive $ \mathcal{F}_t-$adapted stochastic real processes
\begin{equation}
X_j(\omega, s):= \sum_{k\in \Sigma_j}\lambda_k^{\beta} \langle u_n(\omega, s), e_k\rangle_{\mathbb{L}^2}^2,
\end{equation} }} the positive $ \mathcal{F}_t-$adapted stochastic processes
\begin{equation}
X_j(\omega, s):= \sum_{k\in \Sigma_j}\lambda_k^{\delta} \langle Y(\omega, s), e_k\rangle_{\mathbb{L}^2}^2
\end{equation} 
are continuous for all $ j\in \mathbb{N}_1$, with $ (\Sigma_j)_j\subset \Sigma$  being an increasing sequence of finite subsets converging to $ \Sigma$. Therefore,  
\begin{eqnarray}
\{\tau_n> t\} &=& \cap_{s\leq t}\{\omega, |Y(s, \omega)|_{D(A_2^\frac\delta2)}\leq n\}
\del{= \cap_{s\leq t}\{\omega, \sum_{k\in \Sigma}|k|^{2\beta} \langle Y(\omega, s), e_k\rangle_{\mathbb{L}^2}^2< n^2\}
\nonumber\\
=\cap_{j}\cap_{s\leq t}\{\omega, \sum_{k\in \Sigma_j}|k|^{2\beta} \langle Y(\omega, s), e_k\rangle_{\mathbb{L}^2}^2<\leq n^2\}.\nonumber\\}
=\cap_{j}\cap_{s\leq t}\{\omega, X_j(\omega, s)< n^2\}
= \cap_{j}\{ T_n^j>t\},\nonumber\\
\del{&=& \cap_{s\leq t}\cap_{j\geq m}\{\omega, |\sum_{k\in \Sigma_j}|k|^{2\beta} \langle u_n(\omega, s), e_k\rangle_{\mathbb{L}^2}^2\leq n\}.}
\end{eqnarray}
\del{\begin{eqnarray}
\{\tau_n\leq t\} &=& \cup_{s\leq t}\{\omega, |Y(s, \omega)|_{D(A^\frac\beta2)}\geq n\}\nonumber\\
&=& \cup_{s\leq t}\{\omega, |\sum_{k\in \Sigma}|k|^{\delta} \langle u_n(\omega, s), e_k\rangle_{\mathbb{L}^2}e_k|_{\mathbb{L}^q}\geq n\}
\nonumber\\
&=& \cup_{s\leq t}\cap_{m\in\mathbb{N}}\cup_{j\geq m}\{\omega, |\sum_{k\in \Sigma_j}|k|^{\delta} \langle u_n(\omega, s), e_k\rangle_{\mathbb{L}^2}e_k|_{\mathbb{L}^q}\geq n\}
\nonumber\\
&=& \cup_{s\leq t}\cap_{m\in\mathbb{N}}\cup_{j\geq m}\{\omega, X_j(\omega, s):= \sum_{k\in \Sigma_j}|k|^{\delta} \langle u_n(\omega, s), e_k\rangle_{\mathbb{L}^2}e_k \in (B_{\mathbb{L}^q}(0, n))^c\}
\nonumber\\
&=& \cup_{s\in \mathbb{Q} \&\leq t}\cap_{m\in\mathbb{N}}\cup_{j\geq m}\{\omega, X_j(\omega, s) \in (B_{\mathbb{L}^q}(0, n))^c\} \in \mathcal{F}_t.
\end{eqnarray}}
where\del{$ T_n^j:=\inf\{t, X_j \notin B_{\mathbb{R}_+}(0, n^2)\}$  (open ball)} $ T_n^j:=\inf\{t, X_j >n^2\}$. As $ X_j$ is an $ \mathcal{F}_t$-adapted and continuous and thanks to \cite[Proposition 4.6]{Revuz-Yor}, we conclude that $ T_n^j$ is an optional time with respect to the filteration $ \sigma(X_j(s), s\leq t)\subset \mathcal{F}_t$. Using the right continuity property of the filtration $\mathcal{F}_t$, see the assumption in Section \ref{sec-formulation}, we conclude that $ T_n^j$ is a stopping time with respect to this latter. Therefore, $ \{\tau_n\leq t\} = \{\tau_n> t\}^c \in \mathcal{F}_t$. 

\del{We can also use other tools like, \cite[Problems 2.6 \& 2.7]{Karatzas-Book},  \cite[Proposition 4.14]{Metivier-book-Mart-82} and \cite[Proposition 4.5]{Revuz-Yor}. In particualar, we can also use \cite[Theorem 1.6]{Metivier-book-Mart-82} to  prove that $X_j $, as an adapted continuous process, is progressively measurable therefore, thanks to \cite[Proposition 4.15]{Metivier-book-Mart-82} and to the right continuity the filtration $\mathcal{F}_t$, we confirm that for all $ n, j$ the random time $ T_n^j$ is a stopping time with respect to $\mathcal{F}_t$. }

\del{Thanks to the continuity of the trajectories of  $X_j$\del{$X_j(\omega, s):=  \sum_{k\in \Sigma_j}|k|^{\delta} \langle u_n(\omega, s), e_k\rangle_{\mathbb{L}^2}e_k|_{\mathbb{L}^q}$} 
As the filtration $\mathcal{F}_t$ is right continuous, it is sufficient to prove that for all $ t\in [0, T]$, $ \{\tau_n\geq t\}\in \mathcal{F}_t$.}
\end{proof} }

\del{\section{Complementary of the proof of the martingale solution.}\label{Appendix-Martingale-solu}
To prove the existence of a martingale solution, we use \del{consider the Gelfand triplet \eqref{Gelfand-triple-Domain} and 
 use lemmas \ref{lem-unif-bound-theta-n-H-1-domain} and \ref{lem-bounded-W-gamma-p} 
and }the following compact embedding, see \cite[Theorem 2.1]{Flandoli-Gatarek-95},
\begin{equation}
 W^{\gamma, 2}(0, T; \mathbb{H}^{-\delta', 2}(O))\cap \mathbb{L}^2(0, T; \mathbb{H}^{\frac\alpha2, 2}(O)) 
 \hookrightarrow L^2(0, T; \mathbb{L}^2(O)).
\end{equation}
\del{$ L^2(0, T; \mathbb{H}^{\frac\alpha2, 2}(O)) \hookrightarrow L^2(0, T; \mathbb{L}^2(O))$,} Therefore, we deduce that the sequence of laws $ (\mathcal{L}(u_n))_n$  is tight on $ L^2(0, T; \mathbb{L}^2(O))$. 
Thanks to Prokhorov's theorem there exists a  subsequence, still denoted $ (u_n)_n$, for which  the sequence of laws $ (\mathcal{L}(u_n))_n$ converges  weakly on $ L^2(0, T; \mathbb{L}^2(O))$  to a probability measure $ \mu$. By Skorokhod's embedding theorem, we can construct a probability basis
$ (\Omega^*, F^*, \mathbb{F}^*,  P^*)$  and a sequence of $ L^2(0, T; \mathbb{L}^2(O))\cap C([0, T]; \mathbb{H}^{-\delta', 2}(O))-$random variables
$ (u^*_n)_n$ and $ u^*$ such that  $\mathcal{L}(u^*_n) = \mathcal{L}(u_n), \forall n \in \mathbb{N}_0$,  $\mathcal{L}(u^*) = \mu$ and
$ u^*_n \rightarrow u^* a.s.$ in $ L^2(0, T; \mathbb{L}^2(O))\cap C([0, T]; \mathbb{H}^{-\delta', 2}(O))$. Moreover,  $ u^*_n(\cdot, \omega) \in C([0, T]; H_n)$. Thanks to  Lemma \ref{lem-unif-bound-theta-n-H-1-domain} and to the equality in law, we infer that
 for all $ n\in \mathbb{N}$,
\begin{eqnarray}\label{eq-bound-u-*-n-u-*}
\mathbb{E}\sup_{[0, T]}| u_n^*(s)|^p_{\mathbb{L}^2}+ \mathbb{E}\int_0^T| u_n^*(s)|^2_{ \mathbb{H}^{\frac\alpha2, 2}}ds \leq c<\infty.
\end{eqnarray}
Consequently, the sequence  $ u^*_n$ converges weakly in $  L^2(\Omega\times [0, T]; \mathbb{H}^{\frac\alpha2, 2}(O))$ to a limit $ u^{**}$. It is easy to see that $u^{*} = u^{**},  P\times dt a.e.$ and 
\begin{equation}
u^{*}(\cdot, \omega)\in L^2(0, T; \mathbb{H}^{\frac\alpha2, 2}(O))\cap L^\infty(0, T; \mathbb{L}^2(O)).
\end{equation} 
We introduce the filtration 
\begin{equation}
(\mathit{G}_n^*)_t:= \sigma\{u^{*}_n(s), s\leq t\}
\end{equation}
and construct with respect to the filtration $ (\mathit{G}_n^*)_t$ the time continuous square integrable martingale
$ (M_n(t), t\in [0, T])$ with trajectories in
$ C([0, T]; \mathbb{L}^2(O))$ by
\begin{equation}
M_n(t):= u_n^*(t) - P_nu_0+\int_0^t A_\alpha u_n^*(s) ds -\int_0^t P_nB(u_n^*(s))ds
\end{equation}
with the quadratic variation 
\begin{equation}
\langle\langle M_n\rangle\rangle_t= \int_0^tP_nG(u^*_n(s))QG(u^*_n(s))^*ds,
\end{equation}
where $ G(u^*_n(s))^*$ is the adjoint of $G(u^*_n(s))^*$. The proof yields as a consequence of the equality in law. The main task now is to prove that for $ a.s.$, $ M_n(t)$ converges weakly in $ \mathbb{H}^{-\delta', 2}(O)$ to a martingale $ M(t)$, for all $ t\in [0, T]$, where $ M(t)$ given by 
\begin{equation}\label{eq-M(t)}
M(t):= u^*(t) - u_0+\int_0^t A_\alpha u^*(s) ds -\int_0^t B(u^*(s))ds.
\end{equation}
In fact, for all $ v\in V_2:= \mathbb{H}^{\delta', 2}$ with $ \delta'>1+\frac d2$, we have $ P-a.s.$,  $ \langle P_nu_0, v\rangle $  converges to $ \langle u_0, v\rangle $, thanks to the fact that $ P_n \rightarrow I$ in $ \mathbb{L}^2(O)$, $ \langle u_n^*(t), v\rangle $  converges to $ \langle u^*(t), v\rangle $ as a consequence of the 
the a.s. convergence in $ L^2(0, T; \mathbb{L}^2(O))$ (in fact, we speak about the convergence of a subsequence but as usual we keep the same notation) and  the weak convergence and the continuity in $\mathbb{H}^{-\delta', 2}(O))$, the term $\int_0^t A_\alpha u_n^*(s) ds$
converges thanks to the weak convergence $ L^2(0, t; \mathbb{L}^2(O))$, for all $ t\in [0, T]$ and the elementry inequality $ \langle A_\alpha u_n^*(t), v\rangle = \langle u_n^*(t), A_\alpha v\rangle $ with $ v \in \mathbb{H}^{\delta', 2}(O)$ and  $\delta'>1+\frac d2>\alpha $. The convergence of $\int_0^t \langle B(u_n^*(s)), v\rangle ds$ is completely described in \cite[Appendix 2]{Flandoli-Gatarek-95}, in particular the condition $ \delta'>1+\frac d2$ implies that $ \partial_j v \in C^0(O)$, which we need to do the calculus. To prove that $ M(t)$ is  a quadratic martingale, we see that for all $ \phi \in C_b(L^2(0, s; \mathbb{L}^2(O)))$ and $ v\in \mathcal{D}(O)$
\begin{equation}
\mathbb{E}(\langle M(t)- M(s), v\rangle \phi(u^*|_{[0, s]}))= \lim_{n\rightarrow +\infty} \mathbb{E}(\langle M_n(t)- M_n(s), v\rangle \phi(u^*|_{[0, s]}))=0
\end{equation}
and 
\begin{eqnarray}
&{}&\mathbb{E}(\langle M(t), v\rangle \langle M(t), y\rangle - \langle M(s), v\rangle \langle M(s), y\rangle -\int_s^t\langle G^*(u^*(r)P_nv, G^*(u^*(r)P_ny \rangle dr)\phi(u^*|_{[0, r]}))\nonumber\\
&=& \lim_{n\rightarrow +\infty} \mathbb{E}(\langle M_n(t), v\rangle \langle M_n(t), y\rangle - \langle M_n(s), v\rangle \langle M_n(s), y\rangle -\int_s^t\langle G^*(u_n^*(r)P_nv, G^*(u_n^*(r)P_ny \rangle dr)\phi(u_n^*|_{[0, r]}))\nonumber\\
&=& 0
\end{eqnarray}
The main ingredeints are formula \eqref{eq-M(t)}, $ \partial_j\phi \in C^0 $ and therefore we can estimate $\int_0^t \langle B(u_n^*(s)), v\rangle ds$ by $\int_0^t|B(u_n^*(s))|_{L^1} |v|_{C^1} ds$.

\del{We follow the same steps as in \cite{Flandoli-Gatarek-95} (hence we omit here the details), we end up with the statement that $ u^*$ is a solution in $ V_2:= \mathbb{H}^{1+\frac d2, 2}(O)$ of Equation \eqref{Eq-weak-Solution} with $ W$ being replaced by $ W^*$. 

It is easy to see, thanks to Equation \eqref{FSBE-Galerkin-approxi} and to the equality in Law of $ u_n $ and $ u_n^*$,
that $ (M_n(t), t\in [0, T])$ is a square integrable martingale with respect to the filtration 
$(\mathcal{G}_n)_t:=\sigma\{u_n^*(s), s\leq t\}$. The remain part of the proof follows the same steps as in \cite{Flandoli-Gatarek-95},
hence we omit here.
\del{$ \mathbb{G}^* := (\mathcal{G}^*_t)_t$,  with $ \mathcal{G}^*_t $ is the $ \sigma-$algebra generated by
$ \cup_n (\mathcal{G}_n)_t:=\sigma\{u_n^*(s), s\leq t\}$.}}
}


\section{Some Sobolev inequalities.}\label{Appendix-Sobolev}
\del{\subsection{Sobolev pointwise multiplication}\label{Sobolev pointwise multiplication}
Let us recall the following classical result, see e.g. \cite[Theorem 1.4.6.1, p 190-191]{R&S-96}
\begin{theorem}\label{theor-1-SR}
 Let   $ s_1, s_2 \in \mathbb{R}$, satisfy  $ s_1\leq  s_2$ and  $ s_1+ s_2>0$. Then
\begin{itemize}
 \item (i) if $ s_2> \frac dq$, $ H^{s_1, 2}\cdot H^{s_2, 2} \hookrightarrow H^{s_1, 2}$.
\item (ii) if $ s_2< \frac dq$, $ H^{s_1, 2}\cdot H^{s_2, 2} \hookrightarrow H^{s_1+s_2-\frac dq, 2}$.
\end{itemize}
\end{theorem}}

\subsection{Sobolev pointwise multiplication on bounded sets}\label{Sobolev pointwise multiplication-Bounded-Domain}
Assume that $ O \subset \mathbb{R}^d$ is a bounded  $ C^\infty$ domain,(recall, domain means an open subset, see e.g.
\cite[5.2 p43]{Triebel-Structure-Function-vol97}). The notation
$A^s_{pq}(\mathbb{R}^d), s\in \mathbb{R}, 0<q\leq \infty, 0<p<\infty,$ stands either for Triebel-Lizorkin spaces $F^s_{pq}(\mathbb{R}^d) $ or for
Besov spaces $B^s_{pq}(\mathbb{R}^d)$, see
the definition in \cite[p.8]{R&S-96}. We know that, see e.g. \cite[Proposition Tr.6, 2.3.5, p 14]{R&S-96}, 
\begin{eqnarray}\label{main-relation-spaces}
F^s_{p2}(\mathbb{R}^d) &=& H^{s, p}(\mathbb{R}^d), \; 1<p<\infty, \; s\in \mathbb{R}, \nonumber\\
F^s_{pp}(\mathbb{R}^d) &=& B^s_{pp}(\mathbb{R}^d) =  W^{s, p}(\mathbb{R}^d), \; 1\leq p<\infty, \; 0<s\neq\; integer,
\end{eqnarray}
where $ H^{s, p}(\mathbb{R}^d)$ is the Bessel potential spaces or called also Sobolev spaces of fractional order and
$ W^{s, p}(\mathbb{R}^d), \; 1\leq p<\infty, \; 0<s\neq\; integer$  is Slobodeckij spaces. We define Triebel-Lizorkin and Besov spaces $ A^s_{pq}(O)$
on bounded sets by, see e.g. \cite[Definition 5.3. p 44]{Triebel-Structure-Function-vol97}
\begin{equation}
 A^s_{pq}(O)= \{ f \in D'(O); \;\; \text{there is a } g \in A^s_{pq}(\mathbb{R}^d),\; \text{with}\; g/O =f \; \text{in distribution sense}\},
\end{equation}
endowed with the norm
\begin{equation}
 |f|_{A^s_{pq}(O)}= \inf_{g\in A^s_{pq}(\mathbb{R}^d),\; g/O=f}|g |_{A^s_{pq}(\mathbb{R}^d)}.
\end{equation}
The relations in \eqref{main-relation-spaces} still also valid for
bounded sets, see e.g. \cite[5.8 p 52]{Triebel-Structure-Function-vol97}. Our main theorem is the following
\del{Let us recall the following classical result, see e.g. \cite[Theorem 1.4.6.1, p 190-191]{R&S-96}}
\begin{theorem}\label{Theo-pointwiseMulti-Bounded-Domain}
 Let\del{   $ s_1, s_2, s_3 \in \mathbb{R}$ and} $p, s, q, p_i, s_i, q_i, i=1,2$, such that the following pointwise multiplication is satisfied for
$ A^{s_i}_{p_iq_i}(\mathbb{R}^d)$
\begin{equation}\label{Eq-pointwise-R-d}
 |f_1f_2|_{A^{s}_{pq}}\leq c |f_1|_{A^{s_1}_{p_1q_1}} |f_2|_{A^{s_2}_{p_2q_2}}.
\end{equation}
Then Inequality \eqref{Eq-pointwise-R-d} is also valid for $ O \subset \mathbb{R}^d$ being a bounded open  $ C^\infty$ set.
\end{theorem}
\begin{proof}
Let $f_i\in A^{s_i}_{p_iq_i}(O)$, \del{and $g_i\in A^{s_i}_{p_iq_i}(\mathbb{R}^d)$, such that  $ g_i/O =f_i$}
then
\begin{equation}
 |f_1f_2|_{A^s_{pq}(O)}= \inf_{g\in A^s_{pq}(\mathbb{R}^d), g/O=(f_1f_2)}|g |_{A^s_{pq}(\mathbb{R}^d)}\leq
\inf_{g_i\in A^{s_i}_{p_iq_i}(\mathbb{R}^d), g_i/O=f_i} |g_1g_2|_{A^s_{pq}(\mathbb{R}^d)}.
\end{equation}
 Applying Estimate  \eqref{Eq-pointwise-R-d}, we infer that
\begin{equation}
 |f_1f_2|_{A^s_{pq}(O)}\leq c
\inf_{g_i\in A^{s_i}_{p_iq_i}(\mathbb{R}^d), g_i/O=f_i} (|g_1|_{A^{s_1}_{p_1q_1}(\mathbb{R}^d)}|g_2|_{A^{s_2}_{p_2q_2}(\mathbb{R}^d)})
\leq c|f_1|_{A^{s_1}_{p_1q_1}(O)}|f_2|_{A^{s_2}_{p_2q_2}(O)}
\end{equation}

\end{proof}

\subsection{Sobolev embedding }\label{lem-appendix-sobolev-embedding}
\begin{theorem} Let  $ O$ be  either the whole space $ \mathbb{R}^d$, or the torus $ \mathbb{T}^d$, or an arbitrary domain 
 $O\subset \mathbb{R}^d$.
 If $ t\leq s$ and $ 1<p\leq q\leq \frac{dp}{d-(s-t)p}<\infty$, then
\begin{equation}
 H^{s, p}(O) \hookrightarrow  H^{t, q}(O).
\end{equation}
\end{theorem}

\begin{proof}
For the proof see \cite[Theorem 7.63. p221 + 7.66 p222]{Adams-Hedberg-94}. For   $O= \mathbb{R}^d $ and $ q=\frac{dp}{d-sp}$, see \cite[Theorem 1, p 119 or Theorem 2 p 124]{Stein} and
\cite[Proposition 6.4. p 24]{Taylor-PDE-III}. For  $O= \mathbb{T}^d$, see e.g. \cite[pp 23-24]{Taylor-PDE-III}. 
\end{proof}
As a consequence, we have
\begin{equation}
 H^{\frac\alpha2, 2}(O) \hookrightarrow H^{\frac\alpha2-\frac d2+\frac dq, q}( O), \;\;\; \forall q \geq 2.
\end{equation}
See also the above result for  the Sobolev solenoidal spaces in \cite[Theorem 3.10]{Amann-solvability-NSE-2000}.

\del{\subsection{Sobolev embedding }\label{lem-appendix-sobolev-embedding}
\begin{theorem} We take  $ O$ to be  the whole space $ \mathbb{R}^d$, or the torus $ \mathbb{T}^d$,
or an arbitrary regular (see \cite{Adams-Hedberg-94} for the definition of the regularity) domain (open set)
 $O\subset \mathbb{R}^d$.
 If $ t\leq s$ and $ 1<p\leq q\leq \frac{dp}{d-(s-t)p}<\infty$, then
\begin{equation}
 H^{s, p}( O) \hookrightarrow  H^{t, q}(O).
\end{equation}
\end{theorem}

\begin{proof}
For the proof, in the case  $O= \mathbb{R}^d $,  see \cite[Theorem 7.63]{Adams-Hedberg-94}
and in the case $ O$ open regular domain, see \cite[Theorem 7.63. p221 + 7.66 p222]{Adams-Hedberg-94}.
Moreover, for   $O= \mathbb{R}^d $ and $ q=\frac{dp}{d-sp}$, see \cite[Theorem 1, p 119 or Theorem 2 p 124]{Stein} and
\cite[Proposition 6.4. p 24]{Taylor-PDE-III}. For  $O= O $, see e.g. \cite[pp 23-24]{Taylor-PDE-III}.
\end{proof}
As a consequence, we have
\begin{equation}
  H^{\frac\alpha2, 2}(O) \hookrightarrow H^{\frac\alpha2-\frac d2+\frac dq, q}( O), \;\;\; \forall q \geq 2.
\end{equation}}

\del{\eqref{B-u-v-h-alpha-2-d} with $ \eta =0$ (or \eqref{Eq-B-H-alpha-2-est})}

\section*{Acknowledgement}

 The author would like to express her sincere gratitude to Prof. Alexandra Lunardi, Prof. Franco Flandoli, Prof. Yoshikazu Giga, Prof. Marco Romito for their fruitful discussions and for pointing out to the author important references and to Prof. Lyazid Abbaoui and Prof. Ed Corrigan for the administrative facilities they have provided.

\end{document}